\documentclass[reqno,11pt]{book}
\usepackage{graphicx, amsmath, amsthm}
\usepackage{amssymb}  
\usepackage{color}
\usepackage{tikz}
\usetikzlibrary{arrows,shapes}
\usetikzlibrary{calc}

\input{mart.sty}

\begin{document}


\title{{\Huge Martingales et calcul stochastique} \\
\bigskip
\bigskip
\bigskip
\bigskip
{\Large Master 2 Recherche de Math\'ematiques}\\
\bigskip
\bigskip
{\Large Universit\'e d'Orl\'eans}}
\bigskip
\author{Nils Berglund}
\bigskip
\date{Version de Janvier 2014}   

\maketitle
\vfill
\newpage
\thispagestyle{empty}
\cleardoublepage
\thispagestyle{empty}
\setcounter{page}{-1}
\tableofcontents
\newpage
\thispagestyle{empty}


\part{Processus en temps discret}
\label{part_disc}


\chapter{Exemples de processus stochastiques}
\label{chap_ex}

D'une mani\`ere g\'en\'erale, un \defwd{processus stochastique \`a temps
discret}\/ est simplement une suite $(X_0, X_1, X_2, \dots)$ de variables
al\'eatoires, d\'efinies sur un m\^eme espace probabilis\'e. Avant de donner une
d\'efinition pr\'ecise, nous discutons quelques exemples de tels processus. 


\section{Variables al\'eatoires i.i.d.}
\label{sec_iid} 

Supposons que les variables al\'eatoires $X_0, X_1, X_2, \dots$ ont toutes la
m\^eme loi et sont ind\'ependantes. On dit alors qu'elles sont
\defwd{ind\'ependantes et identiquement distribu\'ees}\/,  abr\'eg\'e i.i.d.
C'est d'une certaine mani\`ere la situation la plus al\'eatoire que l'on puisse
imaginer. 
On peut consid\'erer que les $X_i$ sont d\'efinies chacune dans une copie
diff\'erente d'un espace probabilis\'e de base, et que le processus vit dans
l'espace produit de ces espaces. 

La suite des $X_i$ en soi n'est pas tr\`es int\'eressante. Par contre la suite
des sommes partielles
\begin{equation}
 \label{iid1}
S_n = \sum_{i=0}^{n-1} X_i 
\end{equation} 
admet plusieurs propri\'et\'es remarquables. En particulier, rappelons les
th\'eor\`emes de convergence suivants.
\begin{enum}
\item	La \defwd{loi faible des grands nombres}: Si l'esp\'erance $\expec{X_0}$
est finie, alors 
$S_n/n$ converge vers $\expec{X_0}$ en probabilit\'e, c'est-\`a-dire 
\begin{equation}
 \label{iid2}
\lim_{n\to\infty} \biggprob{\biggabs{\frac{S_n}{n}-\expec{X_0}} > \eps} = 0
\end{equation} 
pour tout $\eps>0$. 

\item	La \defwd{loi forte des grands nombres}: Si l'esp\'erance 
est finie, alors $S_n/n$ converge presque s\^urement vers $\expec{X_0}$,
c'est-\`a-dire
\begin{equation}
 \label{iid3}
\Bigprob{\lim_{n\to\infty}\frac{S_n}{n}=\expec{X_0}} = 1\;.
\end{equation} 

\item	Le \defwd{th\'eor\`eme de la limite centrale}: Si l'esp\'erance
et la variance $\variance(X_0)$ sont finies, alors
$(S_n-n\expec{X_0})/\sqrt{n\variance(X_0)}$ tend en loi vers une variable
normale centr\'ee r\'eduite, 
c'est-\`a-dire
\begin{equation}
 \label{iid4}
\lim_{n\to\infty} \biggprob{a \leqs
\frac{S_n-n\expec{X_0}}{\sqrt{n\variance(X_0)}} \leqs b}
= \int_a^b \frac{\e^{-x^2/2}}{\sqrt{2\pi}} \6x
\end{equation} 
pour tout choix de $-\infty\leqs a<b\leqs \infty$.
\end{enum}

Il existe d'autres th\'eor\`emes limites, comme par exemple des principes de
grandes d\'eviations et des lois du logarithme it\'er\'e. Le
th\'eor\`eme de Cram\'er affirme que sous certaines conditions de croissance sur
la loi des $X_i$, on a
\begin{equation}
 \label{iid5}
\lim_{n\to\infty} \frac{1}{n} \log 
\biggprob{\frac{S_n}{n} > x} = -I(x)\;, 
\end{equation}
o\`u la \defwd{fonction taux}\/ $I$ est la transform\'ee de Legendre de la
fonction g\'en\'eratrice des cumulants de $X_0$, \`a savoir
$\lambda(\theta)=\log\expec{\e^{\theta X_0}}$. C'est un principe de grandes
d\'eviations, qui implique que $\prob{S_n > nx}$ d\'ecro\^\i t
exponentiellement avec un taux $-n I(x)$. 

La loi du logarithme it\'er\'e, quant \`a elle, affirme que 
\begin{equation}
 \label{iid6}
\biggprob{
\limsup_{n\to\infty} \frac{S_n-n\expec{X_0}}{\sqrt{n \log(\log n)
\variance(X_0)}} = \sqrt{2}} = 1\;.
\end{equation} 
Les fluctuations de $S_n$, qui sont typiquement d'ordre $\sqrt{n\variance(X_0)}$
d'apr\`es le th\'eor\`eme de la limite centrale, atteignent donc avec
probabilit\'e $1$ la valeur $\sqrt{2n\log(\log n)\variance(X_0)}$, mais ne 
la d\'epassent, presque s\^urement, qu'un nombre fini de fois. 


\section{Marches al\'eatoires et \chaine s de Markov}
\label{sec_mc} 

La \defwd{marche al\'eatoire sym\'etrique sur $\Z^d$} est construite de la
mani\`ere suivante. On pose $X_0=0$ (l'origine de $\Z^d$), puis pour $X_1$ on
choisit l'un des $2d$ plus proches voisins de l'origine avec probabilit\'e
$1/2d$, et ainsi de suite. Chaque $X_n$ est choisi de mani\`ere \'equi\-probable
parmi les $2d$ plus proches voisins de $X_{n-1}$, ind\'ependamment des variables
$X_{n-2}, \dots, X_0$. On obtient ainsi une distribution de probabilit\'e sur
les lignes bris\'ees de $\Z^d$. 

La marche sym\'etrique sur $\Z^d$ est tr\`es \'etudi\'ee et relativement bien
comprise. On sait par exemple que 
\begin{enum}
\item	en dimensions $d=1$ et $2$, la marche est \defwd{r\'ecurrente}\/,
c'est-\`a-dire qu'elle revient presque s\^urement en z\'ero; 
\item	en dimensions $d\geqs 3$, la marche est
\defwd{transiente}\/, c'est-\`a-dire qu'il y a une probabilit\'e
strictement positive qu'elle ne revienne jamais en z\'ero.
\end{enum}
Il existe de nombreuses variantes des marches sym\'etriques sur $\Z^d$~:
Marches asym\'etriques, sur un sous-ensemble de $\Z^d$ avec r\'eflexion ou
absorption au bord, marches sur des graphes, sur des groupes\dots

\begin{figure}
\centerline{\includegraphics*[clip=true,width=150mm]{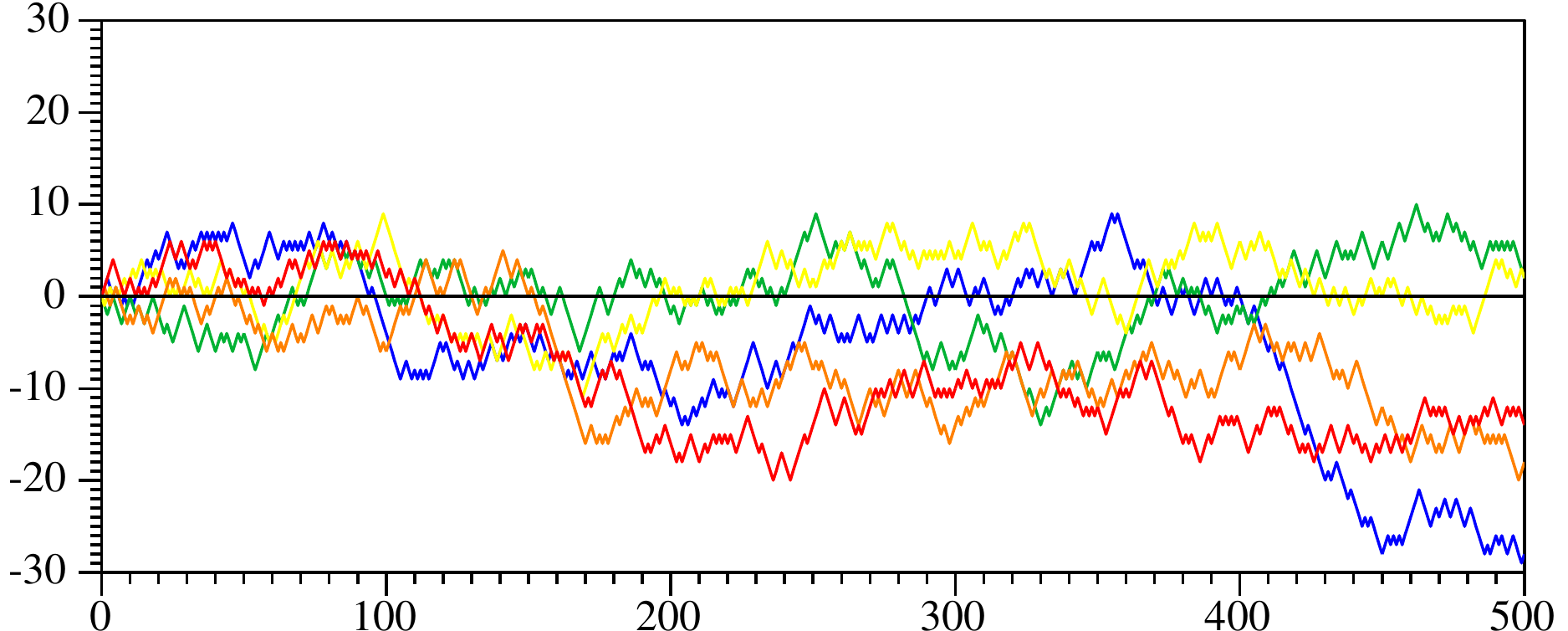}}
\caption[]{Cinq r\'ealisations d'une marche al\'eatoire unidimensionnelle
sym\'etrique.}
\label{fig_rw1}
\end{figure}

Tous ces exemples de marches font partie de la classe plus g\'en\'erale des
\defwd{\chaine s de Markov}\/. Pour d\'efinir une telle \chaine, on se donne un
ensemble d\'enombrable $\cX$ et une \defwd{matrice stochastique}\/
$P=(p_{ij})_{i,j\in\cX}$, satisfaisant les deux conditions
\begin{equation}
 \label{mc1}
p_{ij}\geqs 0 \quad\forall i,j\in\cX
\qquad
\text{et}
\qquad
\sum_{j\in\cX} p_{ij} = 1 \quad \forall i\in\cX \;.
\end{equation} 
On choisit alors une distribution initiale pour $X_0$, et on construit les
$X_n$ de telle mani\`ere que 
\begin{equation}
 \label{mc2}
\bigpcond{X_{n+1}=j}{X_n=i,X_{n-1}=i_{n-1},\dots,X_0=i_0}
=  \bigpcond{X_{n+1}=j}{X_n=i} = p_{ij}
\end{equation} 
pour tout temps $n$ et tout choix des $i_0, \dots, i_{n-1}, i, j \in\cX$. La
premi\`ere \'egalit\'e s'appelle la \defwd{propri\'et\'e de Markov}\/. Les
\chaine s de Markov ont un comportement un peu plus riche que les variables
i.i.d.\ parce que chaque variable d\'epend de celle qui la pr\'ec\`ede
imm\'ediatement. Toutefois, ces processus ont toujours une m\'emoire courte,
tout ce qui s'est pass\'e plus de deux unit\'es de temps dans le pass\'e
\'etant oubli\'e.  

Un exemple int\'eressant de \chaine\ de Markov, motiv\'e par la physique, est
le \defwd{mod\`ele d'Ehrenfest}\/. On consid\`ere deux r\'ecipients contenant
en tout $N$ mol\'ecules de gaz. Les r\'ecipients sont connect\'es par un petit
tube, laissant passer les mol\'ecules une \`a une dans les deux sens. Le but
est d'\'etudier comment \'evolue la proportion de mol\'ecules contenues dans
l'un des r\'ecipients au cours du temps. Si par exemple ce nombre est nul au
d\'ebut, on s'attend \`a ce qu'assez rapidement, on atteigne un \'equilibre
dans lequel le nombre de mol\'ecules dans les deux r\'ecipients est \`a peu
pr\`es \'egal, alors qu'il n'arrivera quasiment jamais qu'on revoie toutes les
mol\'ecules revenir dans le r\'ecipient de droite. 

\begin{figure}[hb]
\vspace{-3mm}
\begin{center}
\begin{tikzpicture}[->,>=stealth',auto,scale=0.9,node
distance=3.0cm, thick,main node/.style={circle,scale=0.7,minimum size=0.4cm,
fill=green!50,draw,font=\sffamily}]
  
  \pos{0}{0} \urntikz
  \pos{1.2}{0} \urntikz
  
  \node[main node] at(0.35,0.2) {};
  \node[main node] at(0.85,0.2) {};
  \node[main node] at(0.6,0.4) {};
  
  \pos{4}{0} \urntikz
  \pos{5.2}{0} \urntikz

  \node[main node] at(4.35,0.2) {};
  \node[main node] at(4.85,0.2) {};
  \node[main node] at(3.4,0.2) {};

  \pos{8}{0} \urntikz
  \pos{9.2}{0} \urntikz

  \node[main node] at(7.15,0.2) {};
  \node[main node] at(7.65,0.2) {};
  \node[main node] at(8.6,0.2) {};

  \pos{12}{0} \urntikz
  \pos{13.2}{0} \urntikz

  \node[main node] at(11.15,0.2) {};
  \node[main node] at(11.65,0.2) {};
  \node[main node] at(11.4,0.4) {};
  
  \node[minimum size=2.2cm] (0) at (0.1,0.5) {};
  \node[minimum size=2.2cm] (1) at (4.1,0.5) {};
  \node[minimum size=2.2cm] (2) at (8.1,0.5) {};
  \node[minimum size=2.2cm] (3) at (12.1,0.5) {};
  
  \path[shorten >=.3cm,shorten <=.3cm,every
        node/.style={font=\sffamily\footnotesize}]
    (0) edge [bend left,above] node {$1$} (1)
    (1) edge [bend left,above] node {$2/3$} (2)
    (2) edge [bend left,above] node {$1/3$} (3)
    (3) edge [bend left,below] node {$1$} (2)
    (2) edge [bend left,below] node {$2/3$} (1)
    (1) edge [bend left,below] node {$1/3$} (0)
    ;
\end{tikzpicture}
\end{center}
\vspace{-7mm} 
\caption[]{Le mod\`ele d'urnes d'Ehrenfest, dans le cas de $3$
boules.}
 \label{fig_ehrenfest0}
\end{figure}
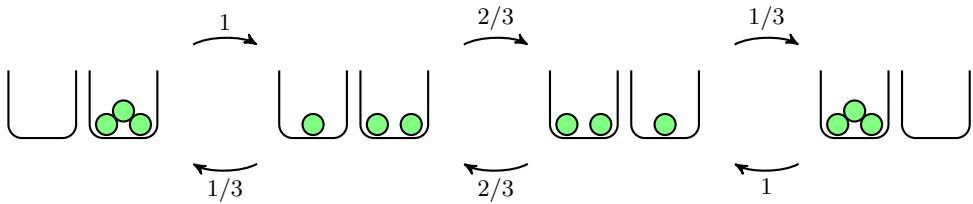

\begin{figure}[ht]
\centerline{\includegraphics*[clip=true,width=75mm]{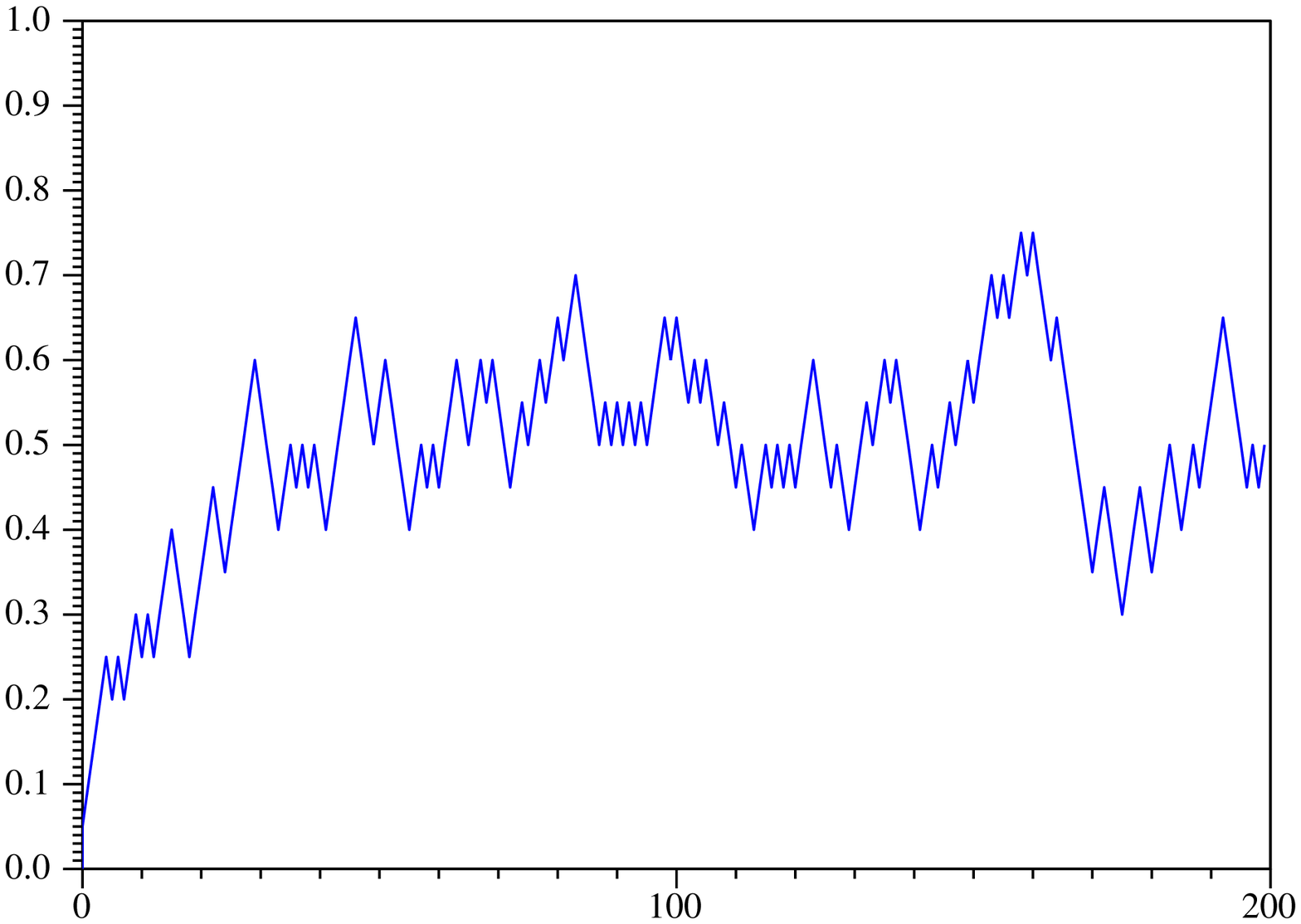}
\hspace{3mm}
\includegraphics*[clip=true,width=75mm]{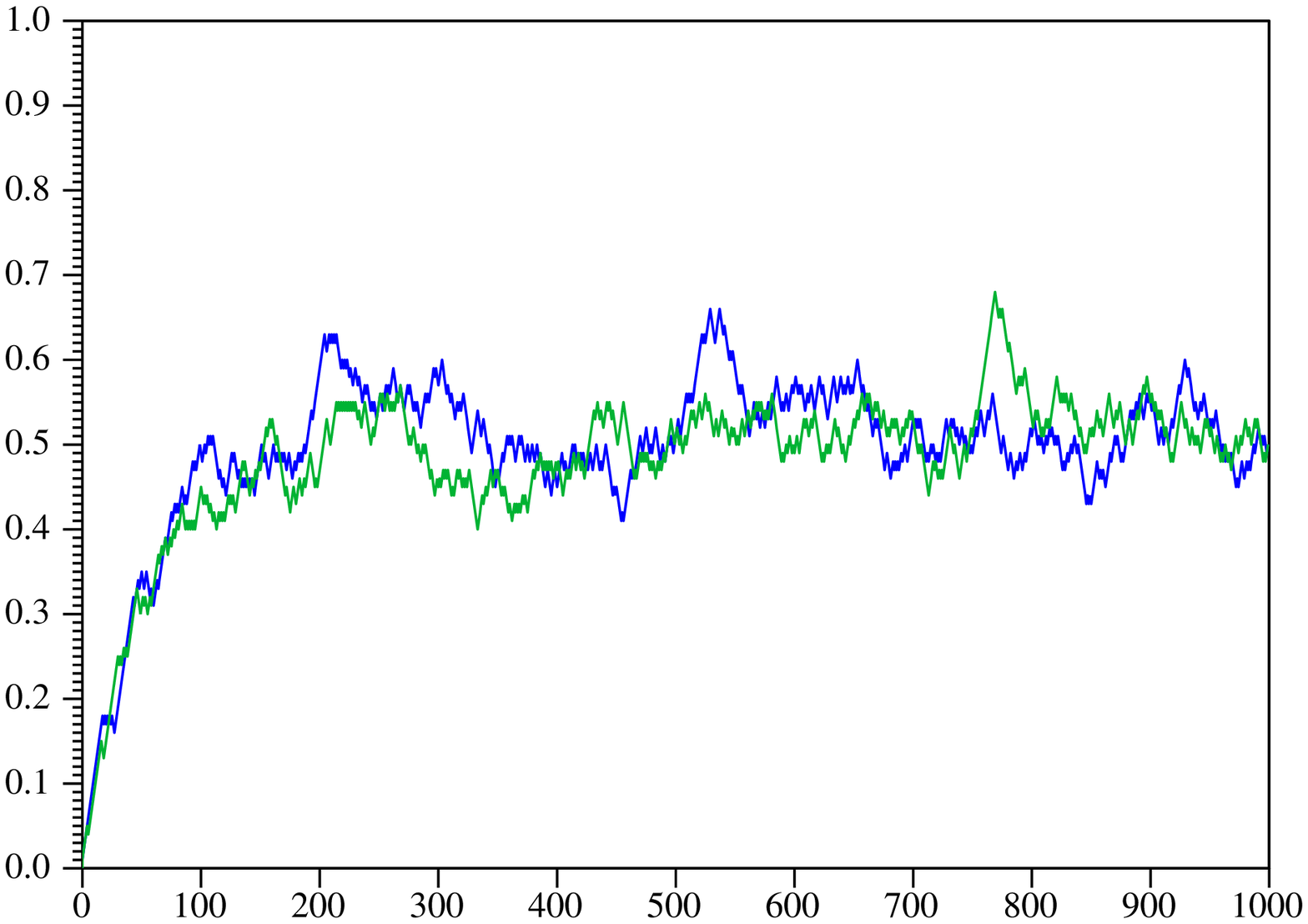}}
\caption[]{R\'ealisations du mod\`ele d'Ehrenfest, montrant la proportion de
boules dans l'urne de gauche en fonction du temps. Le nombre initial de boules
dans l'urne de gauche est $0$. La figure de gauche montre une r\'ealisation
pour $N=20$ boules, celle de droite montre deux r\'ealisations pour $N=100$
boules.}
\label{fig_ehrenfest}
\end{figure}

On mod\'elise la situation en consid\'erant deux urnes contenant en tout $N$
boules. A chaque unit\'e de temps, l'une des $N$ boules, tir\'ee uniform\'ement
au hasard, passe d'un r\'ecipient \`a l'autre. S'il y a $i$ boules dans
l'urne de gauche, alors on tirera avec probabilit\'e $i/N$ l'une de ces boules,
et le nombre de boules dans l'urne passera de $i$ \`a $i-1$. Dans le cas
contraire, ce nombre passe de $i$ \`a $i+1$. Si $X_n$ d\'esigne le nombre de
boules dans l'urne de gauche au temps $n$, l'\'evolution est donn\'ee par une
\chaine\ de Markov sur $\set{0,1,\dots,N}$ de probabilit\'es de transition 
\begin{equation}
 \label{mc3}
p_{ij} = 
\begin{cases}
\frac{i}{N} & \text{si $j=i-1$\;,} \\
1-\frac{i}{N} & \text{si $j=i+1$\;,} \\
0 & \text{sinon\;.}
\end{cases}
\end{equation} 

La \figref{fig_ehrenfest} montre des exemples de suites $\set{X_n/N}$.
On observe effectivement que $X_n/N$ approche assez rapidement la valeur $1/2$,
et tend \`a fluctuer autour de cette valeur sans trop s'en \'eloigner. En fait
on peut montrer que la loi de $X_n$ tend vers une loi binomiale $b(N,1/2)$, qui
est invariante sous la dynamique. L'esp\'erance de $X_n/N$ tend donc bien vers
$1/2$ et sa variance tend vers $1/4N$. De plus, le temps de r\'ecurrence moyen
vers un \'etat est \'egal \`a l'inverse de la valeur de la loi invariante en
cet \'etat. En particulier, le temps de r\'ecurrence moyen vers l'\'etat o\`u
toutes les boules sont dans la m\^eme urne est \'egal \`a $2^N$. Ceci permet
d'expliquer pourquoi dans un syst\`eme macroscopique, dont le nombre de
mol\'ecules $N$ est d'ordre $10^{23}$, on n'observe jamais toutes les
mol\'ecules dans le m\^eme r\'ecipient -- il faudrait attendre un temps d'ordre
$2^{10^{23}} \simeq 10^{3\cdot 10^{22}}$ pour que cela arrive. 


\section{Urnes de Polya}
\label{sec_polya} 

La d\'efinition du processus de l'urne de Polya est \`a priori semblable \`a
celle du mod\`ele d'Ehrenfest, mais elle r\'esulte en un comportement tr\`es
diff\'erent. On consid\`ere une urne contenant initialement $r_0\geqs1$ boules
rouges et $v_0\geqs1$ boules vertes. De mani\`ere r\'ep\'et\'ee, on tire une
boule de l'urne. Si la boule est rouge, on la remet dans l'urne, et on ajoute
$c\geqs1$ boules rouges dans l'urne. Si la boule tir\'ee est verte, on la remet
dans l'urne, ainsi que $c$ boules vertes. Le nombre $c$ est constant tout au
long de l'exp\'erience. 

Au temps $n$, le nombre total de boules dans l'urne est $N_n=r_0+v_0+nc$. 
Soient $r_n$ et $v_n$ le nombre de boules rouges et vertes au temps $n$, et 
soit $X_n=r_n/(r_n+v_n)=r_n/N_n$ la proportion de boules rouges. La
probabilit\'e de tirer une boule rouge vaut $X_n$, celle de tirer une boule
verte vaut $1-X_n$, et on obtient facilement les valeurs correspondantes de
$X_{n+1}$ en fonction de $X_n$ et $n$: 
\begin{equation}
 \label{polya1}
X_{n+1} = 
\begin{cases}
\dfrac{r_n+c}{r_n+v_n+c} = \dfrac{X_n N_n+c}{N_n+c}
& \quad\text{avec probabilit\'e $X_n$\;,} \\
&\\
\dfrac{r_n}{r_n+v_n+c} = \dfrac{X_n N_n}{N_n+c}
& \quad\text{avec probabilit\'e $1-X_n$\;.}
\end{cases} 
\end{equation} 

Le processus n'est plus \`a strictement
parler une \chaine\ de Markov, car l'ensemble des valeurs possibles de $X_n$
change au cours du temps.  

\begin{figure}
\centerline{\includegraphics*[clip=true,width=150mm]{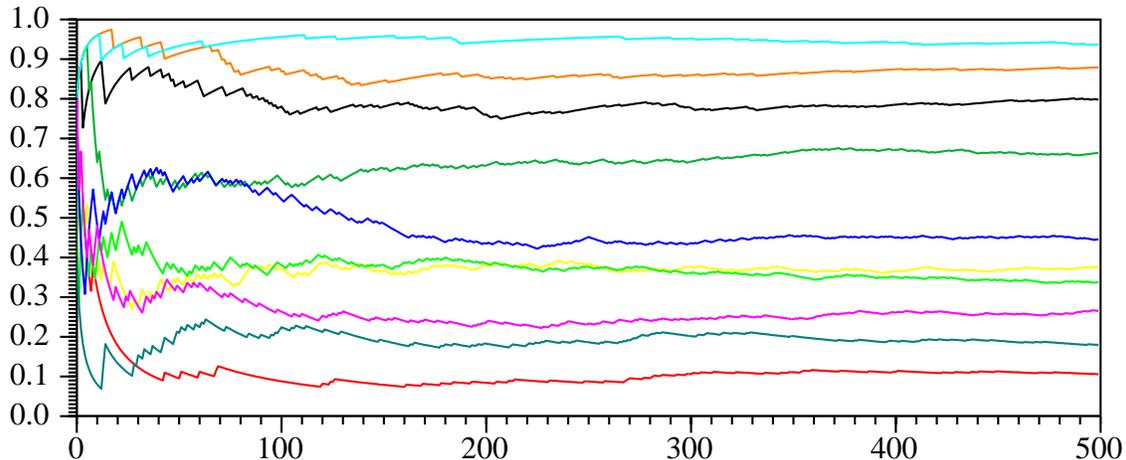}}
\caption[]{Dix r\'ealisations du processus d'urne de Polya, avec initialement
$r_0=2$ boules rouges et $v_0=1$ boule verte. A chaque pas, on ajoute $c=2$
boules.}
\label{fig_polya1}
\end{figure}

La \figref{fig_polya1} montre plusieurs r\'ealisations diff\'erentes du
processus $\set{X_n}_n$, pour les m\^emes valeurs initiales. Contrairement au
mod\`ele d'Ehrenfest, $X_n$ semble converger vers une constante, qui est
diff\'erente pour chaque r\'ealisation. Nous montrerons que c'est effectivement
le cas, en utilisant le fait que $X_n$ est une martingale born\'ee, qui dans ce
cas converge presque s\^urement. On sait par ailleurs que la loi limite de $X_n$
est une loi Beta.


\section{Le processus de Galton--Watson}
\label{sec_gw} 

Le processus de Galton--Watson est l'exemple le plus simple de processus
stochastique d\'ecrivant l'\'evolution d'une population. Soit $Z_0$ le nombre
d'individus dans la population au temps $0$ (souvent on choisit $Z_0=1$). 
On consid\`ere alors que chaque individu a un nombre al\'eatoire de
descendants. Les nombres de descendants des diff\'erents individus sont 
ind\'ependants et identiquement distribu\'es. 

Cela revient \`a supposer que 
\begin{equation}
 \label{gw1}
Z_{n+1} = 
\begin{cases}
\displaystyle
\sum_{i=1}^{Z_n} \xi_{i,n+1} & \text{si $Z_n\geqs 1$\;,} \\
0 &\text{si $Z_n=0$\;,} 
\end{cases}
\end{equation} 
o\`u les variables al\'eatoires $\set{\xi_{i,n}}_{i,n\geqs1}$, qui correspondent
au nombre de descendants de l'individu $i$ au temps $n$, sont i.i.d. On
supposera de plus que chaque $\xi_{i,n}$ admet une esp\'erance $\mu$ finie. 

Galton et Watson ont introduit leur mod\`ele dans le but d'\'etudier la
disparition de noms de famille (\`a leur \'epoque, les noms de famille \'etant
transmis de p\`ere en fils uniquement, on ne consid\`ere que les descendants
m\^ales). On observe \`a ce sujet que d\`es que $Z_n=0$ pour un certain $n=n_0$,
on aura $Z_n=0$ pour tous les $n\geqs n_0$. Cette situation correspond \`a
l'extinction de l'esp\`ece (ou du nom de famille). 

Nous montrerons dans ce cours que
\begin{itemiz}
\item	Si $\mu<1$, alors la population s'\'eteint presque s\^urement;
\item	Si $\mu>1$, alors la population s'\'eteint avec une probabilit\'e 
$\rho=\rho(\mu)$ comprise strictement entre $0$ et $1$.  
\end{itemiz}


\begin{figure}[t]
\centerline{\includegraphics*[clip=true,height=96mm]{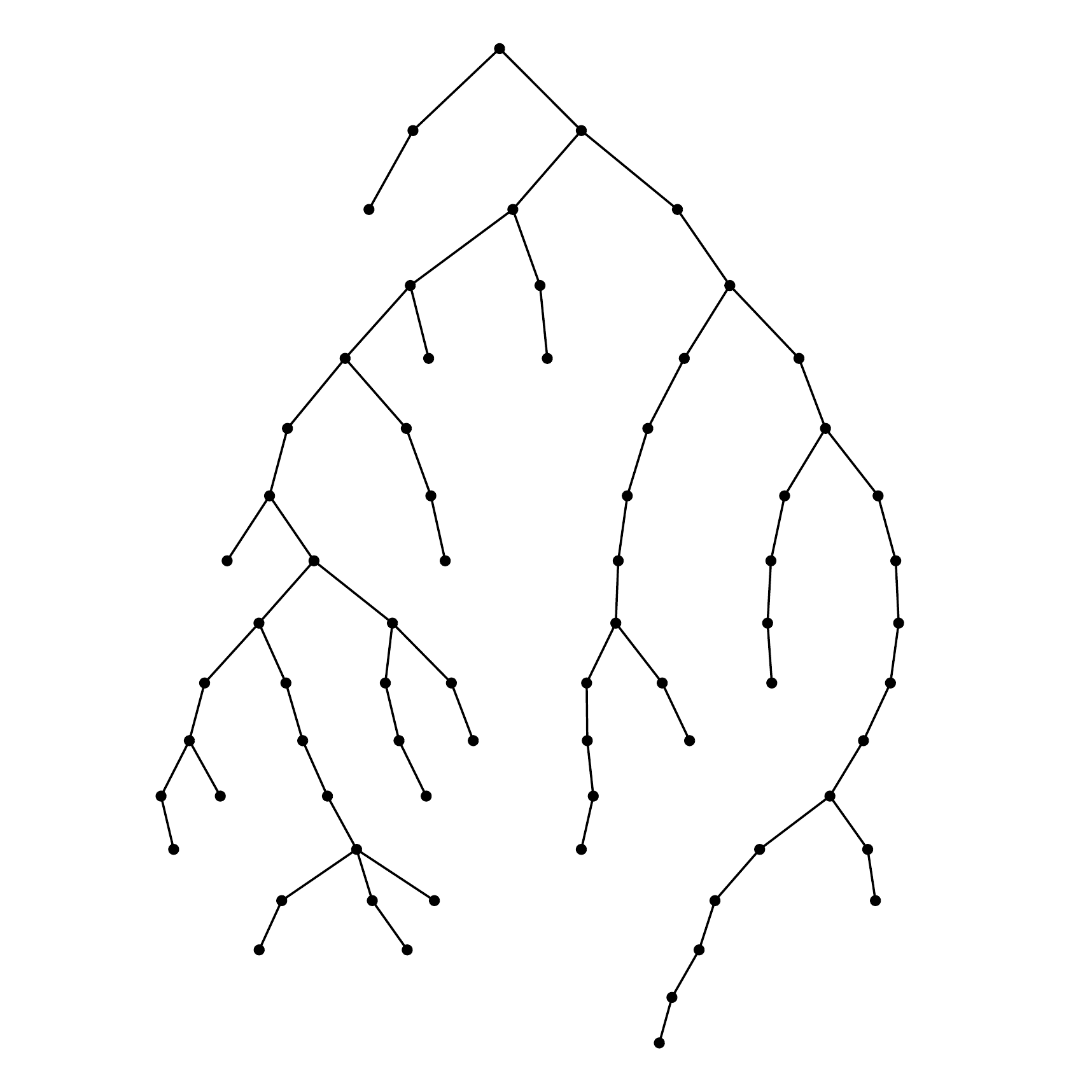}}
\caption[]{Une r\'ealisation du processus de Galton--Watson. La distribution
des descendants est binomiale, de param\`etres $n=3$ et $p=0.4$. Le temps
s'\'ecoule de haut en bas.
}
\label{fig_gw1}
\end{figure}


\begin{figure}[ht]
\centerline{\includegraphics*[clip=true,height=96mm]{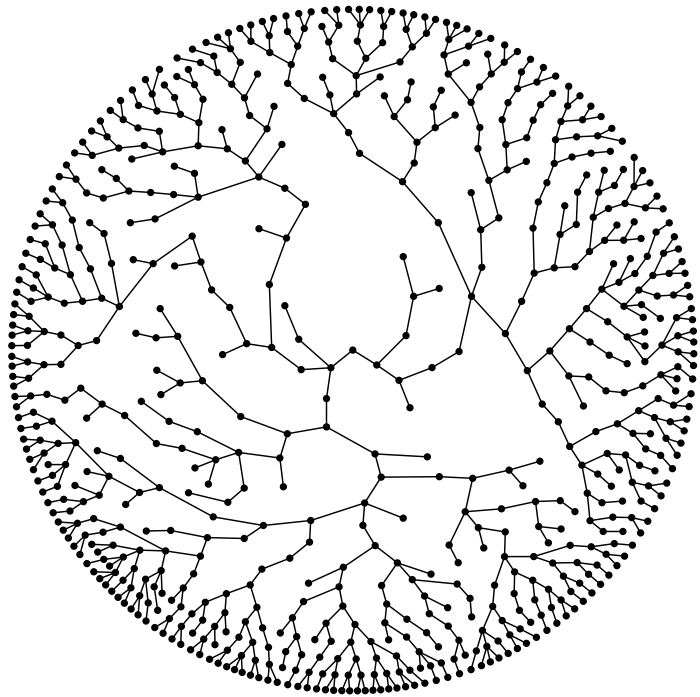}}
\caption[]{Une r\'ealisation du processus de Galton--Watson. La distribution
des descendants est binomiale, de param\`etres $n=3$ et $p=0.45$. L'anc\^etre
est au centre du cercle, et le temps s'\'ecoule de l'int\'erieur vers
l'ext\'erieur.}
\label{fig_gw2}
\end{figure}


\begin{figure}[ht]
\centerline{\includegraphics*[clip=true,height=96mm]{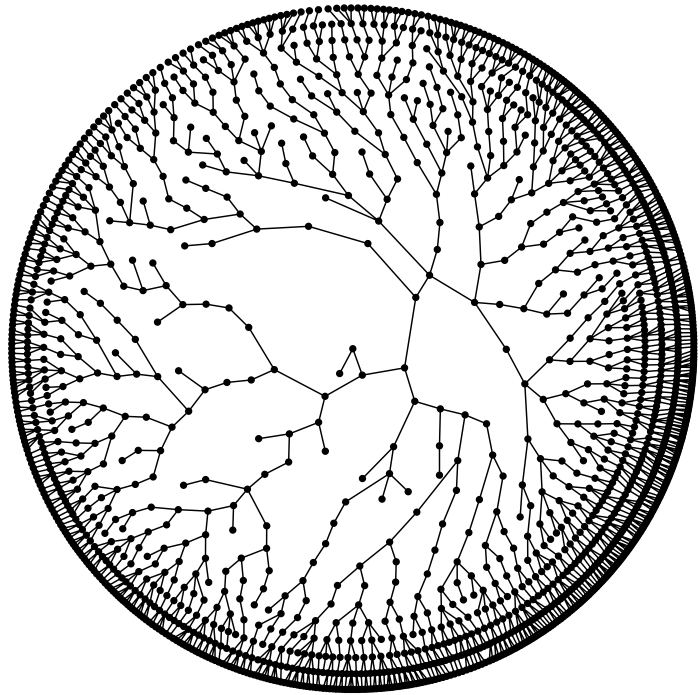}}
\caption[]{Une r\'ealisation du processus de Galton--Watson. La distribution
des descendants est binomiale, de param\`etres $n=3$ et $p=0.5$. }
\label{fig_gw3}
\end{figure}

La preuve s'appuie sur le fait que $X_n=Z_n/\mu^n$ est une martingale.


Les Figures~\ref{fig_gw1} \`a~\ref{fig_gw3} montrent des r\'ealisations
du processus de Galton--Watson sous forme d'arbre. Le temps s\'ecoule de haut
en bas. La distribution des descendants est binomiale $b(3,p)$ pour des valeurs
croissantes de $p$ (dans ce cas on a $\mu=3p$). Pour $p=0.4$, on observe
l'extinction de l'esp\`ece. Dans les deux autres cas, on a pris des exemples
dans lesquels l'esp\`ece survit, et en fait le nombre d'individus cro\^it
exponentiellement vite.

Il existe de nombreux mod\`eles d'\'evolution plus d\'etaill\'es que le
processus de Galton--Watson, tenant compte de la distribution g\'eographique,
de la migration, des mutations. 


\section{Marches al\'eatoires auto-\'evitantes}
\label{sec_nm} 

Les exemples pr\'ec\'edents jouissent tous de la propri\'et\'e de Markov,
c'est-\`a-dire que $X_{n+1}$ d\'epend uniquement de $X_n$, de l'al\'ea, et
\'eventuellement du temps $n$. Il est facile de modifier ces exemples afin de
les rendre non markoviens. Il suffit pour cela de faire d\'ependre chaque
$X_{n+1}$ de tous ses pr\'ed\'ecesseurs. 

Un exemple int\'eressant de processus non markovien est la \defwd{marche
al\'eatoire auto-\'evitante}\/, qui mod\'elise par exemple certains polym\`eres.
Il existe en fait deux variantes de ce processus. Dans la premi\`ere, on
d\'efinit $X_0$ et $X_1$ comme dans le cas de la marche sym\'etrique simple,
mais ensuite chaque $X_n$ est choisi uniform\'ement parmi tous les plus proches
voisins jamais visit\'es auparavant. La seconde variante, un peu plus simple 
\`a \'etudier, est obtenue en consid\'erant comme \'equiprobables toutes les
lignes bris\'ees de longueur donn\'ee ne passant jamais plus d'une fois au
m\^eme endroit. A strictement parler, il ne s'agit pas d'un processus
stochastique.

Une question importante pour les marches al\'eatoires auto-\'evitantes est le
comportement asymptotique du d\'eplacement quadratique moyen
$\expec{\norm{X_n}^2}^{1/2}$. On sait par exemple d\'emontrer rigoureusement
que pour $d\geqs5$, ce d\'eplacement cro\^\i t comme $n^{1/2}$, comme c'est le
cas pour les marches al\'eatoires simples. Pour les dimensions inf\'erieures,
la question de la vitesse de croissance est encore ouverte (le d\'eplacement
devrait cro\^\i tre plus rapidement que $n^{1/2}$ \`a cause des contraintes
g\'eom\'etriques). La conjecture (pour la seconde variante) est que la
croissance est en $n^{3/4}$ en dimension $2$, en $n^{\nu}$ avec $\nu\simeq 0.59$
en dimension $3$, et en $n^{1/2}(\log n)^{1/8}$ en dimension~$4$. 

\begin{figure}[t]
\centerline{\includegraphics*[clip=true,width=72mm]{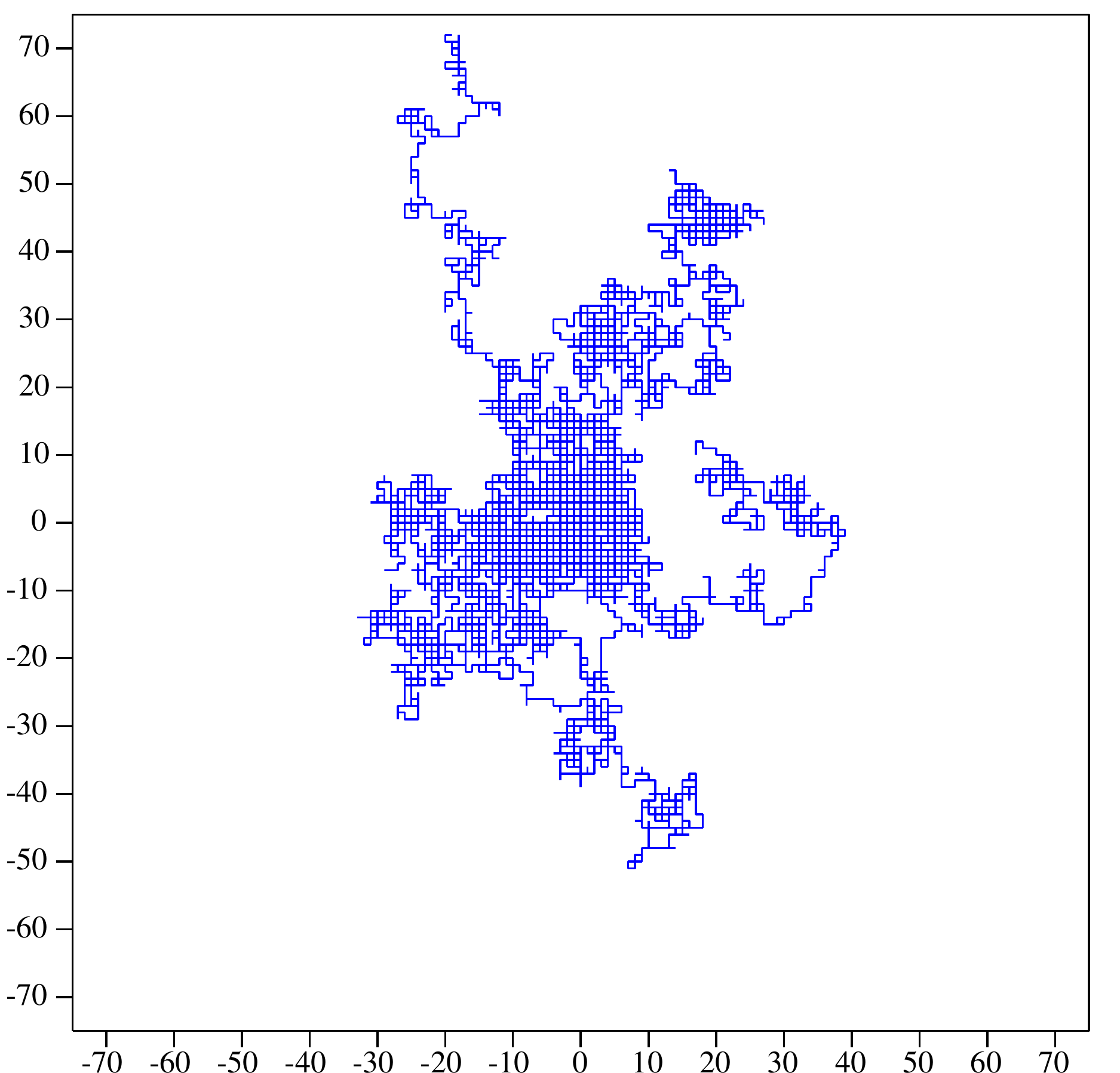}
\hspace{3mm}
\includegraphics*[clip=true,width=72mm]{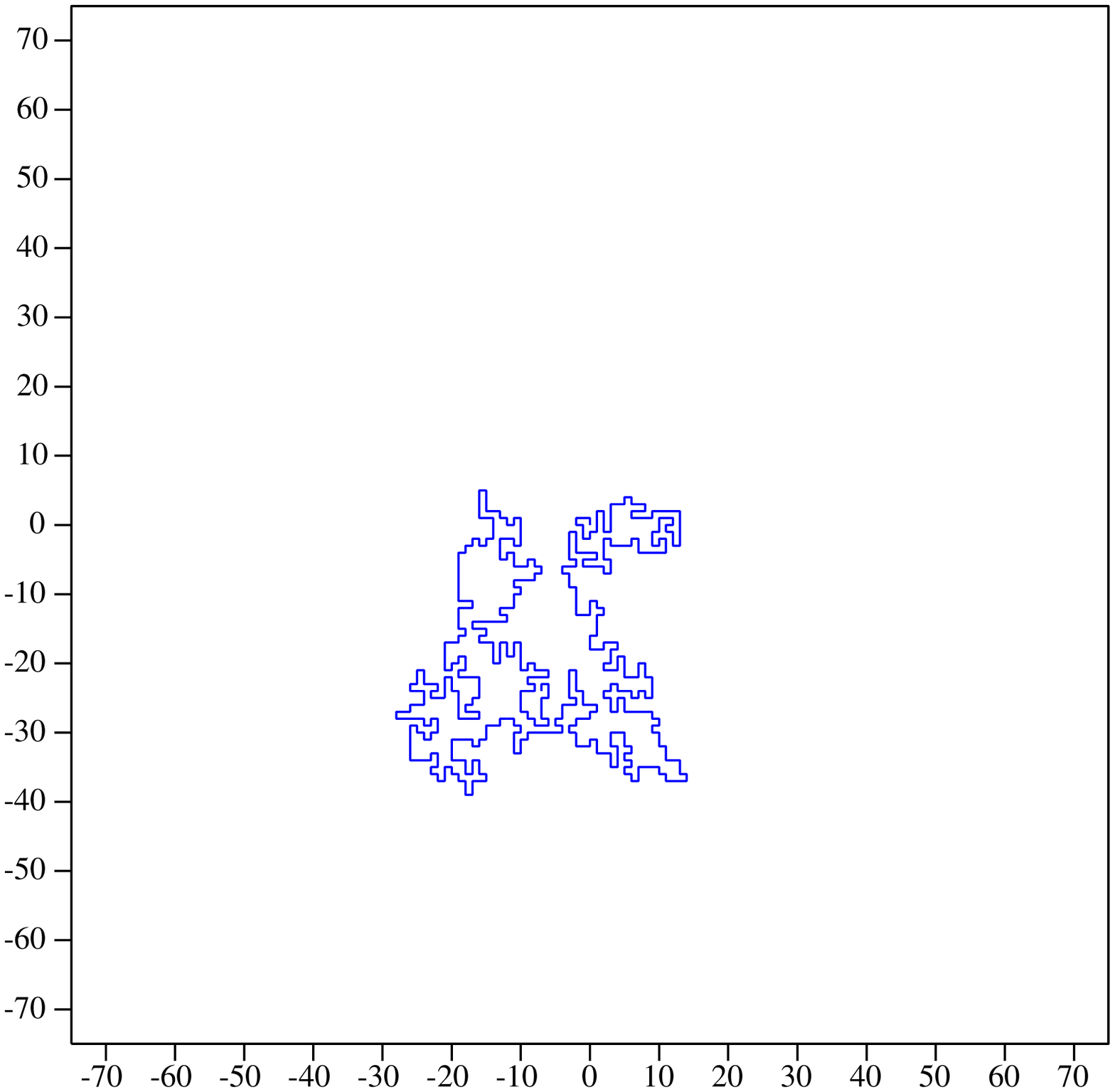}}
\caption[]{A gauche, une r\'ealisation d'une marche al\'eatoire
simple dans $\Z^2$. A droite, une r\'ealisation d'une marche
al\'eatoire auto-\'evitante dans $\Z^2$, se pi\'egeant apr\`es $595$ pas.}
\label{fig_sarw}
\end{figure}

\begin{figure}
\centerline{\includegraphics*[clip=true,width=72mm]{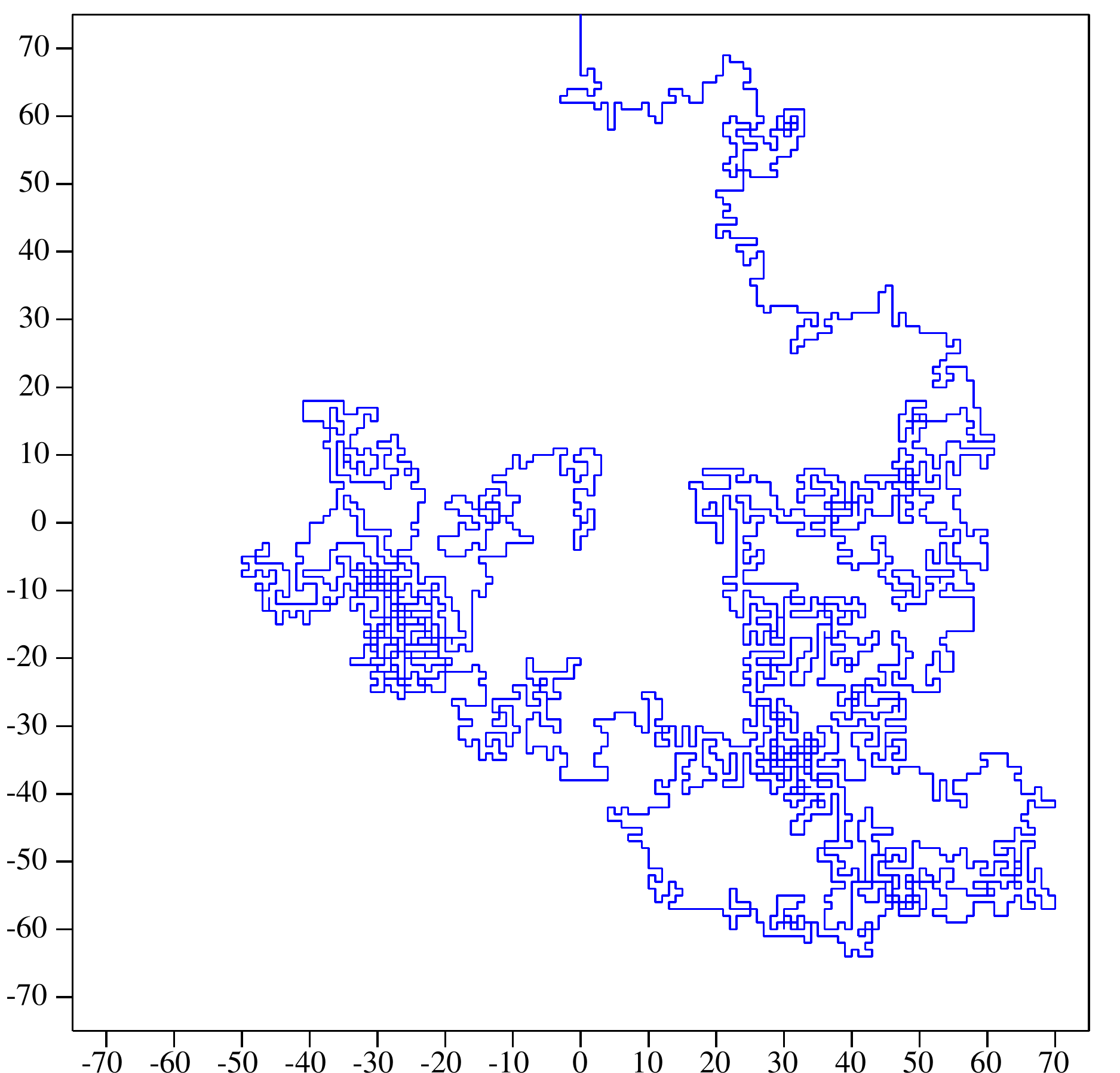}
\hspace{3mm}
\includegraphics*[clip=true,width=72mm]{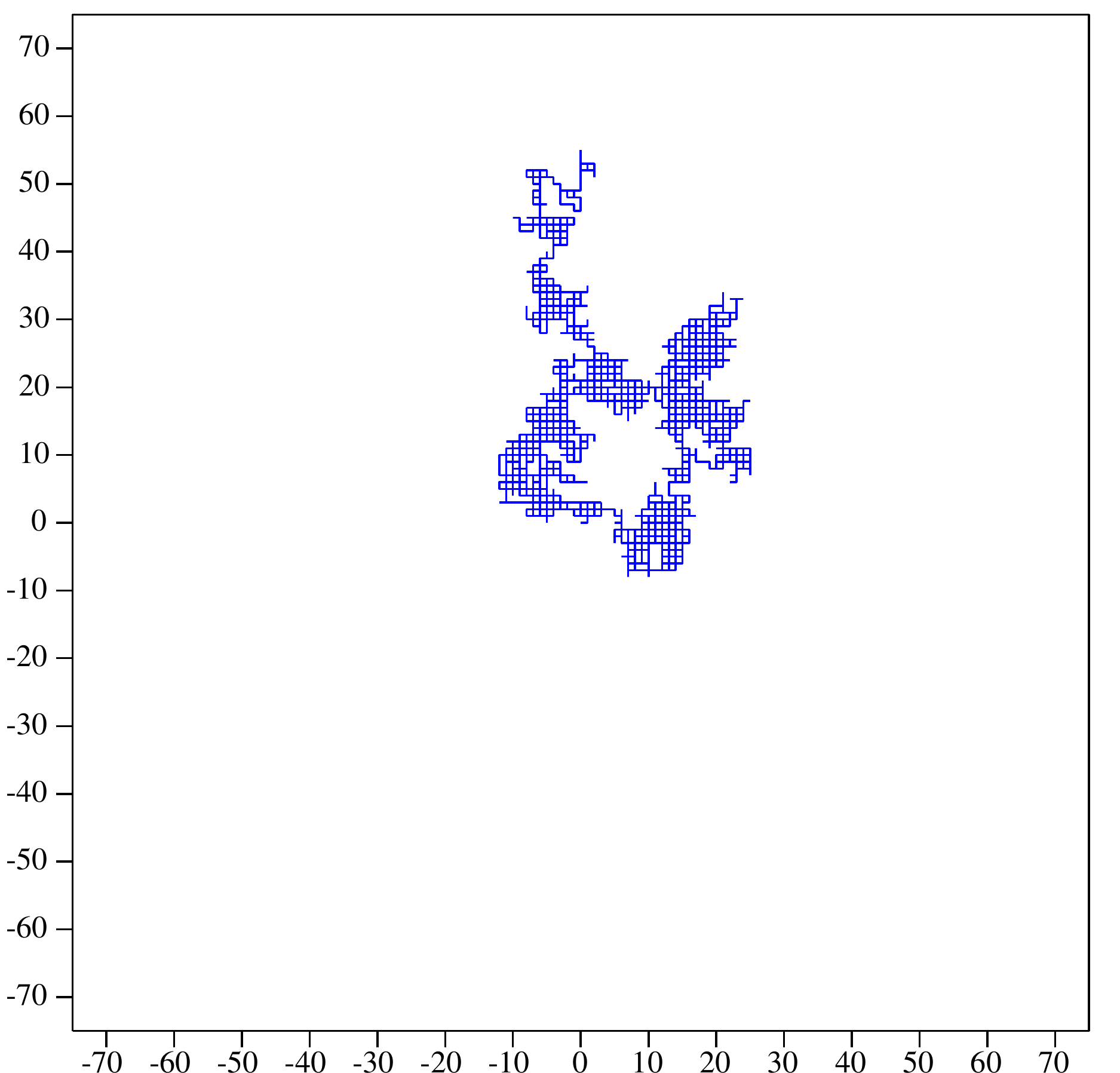}}
\caption[]{A gauche, une r\'ealisation d'une marche
al\'eatoire choisissant un site jamais visit\'e avec une probabilit\'e $10^5$
fois sup\'erieure \`a un site d\'ej\`a visit\'e. A droite, la marche choisit un
site d\'ej\`a visit\'e deux fois plus souvent qu'un site jamais visit\'e.}
\label{fig_sarw2}
\end{figure}

Parmi les variantes de ce mod\`ele, mentionnons les marches al\'eatoires
attir\'ees ou repouss\'ees par les sites d\'ej\`a visit\'es, c'est-\`a-dire
que quand elles passent \`a c\^ot\'e d'un tel site, la probabilit\'e d'y
revenir est soit plus grande, soit plus petite que la probabilit\'e de choisir
un site jamais encore visit\'e (\figref{fig_sarw2}). On peut \'egalement faire
d\'ependre les probabilit\'es du nombre de fois qu'un site a d\'ej\`a \'et\'e
visit\'e dans le pass\'e.


\section{Syst\`emes dynamiques}
\label{sec_sd} 

\begin{figure}[t]
\centerline{\includegraphics*[clip=true,height=96mm]{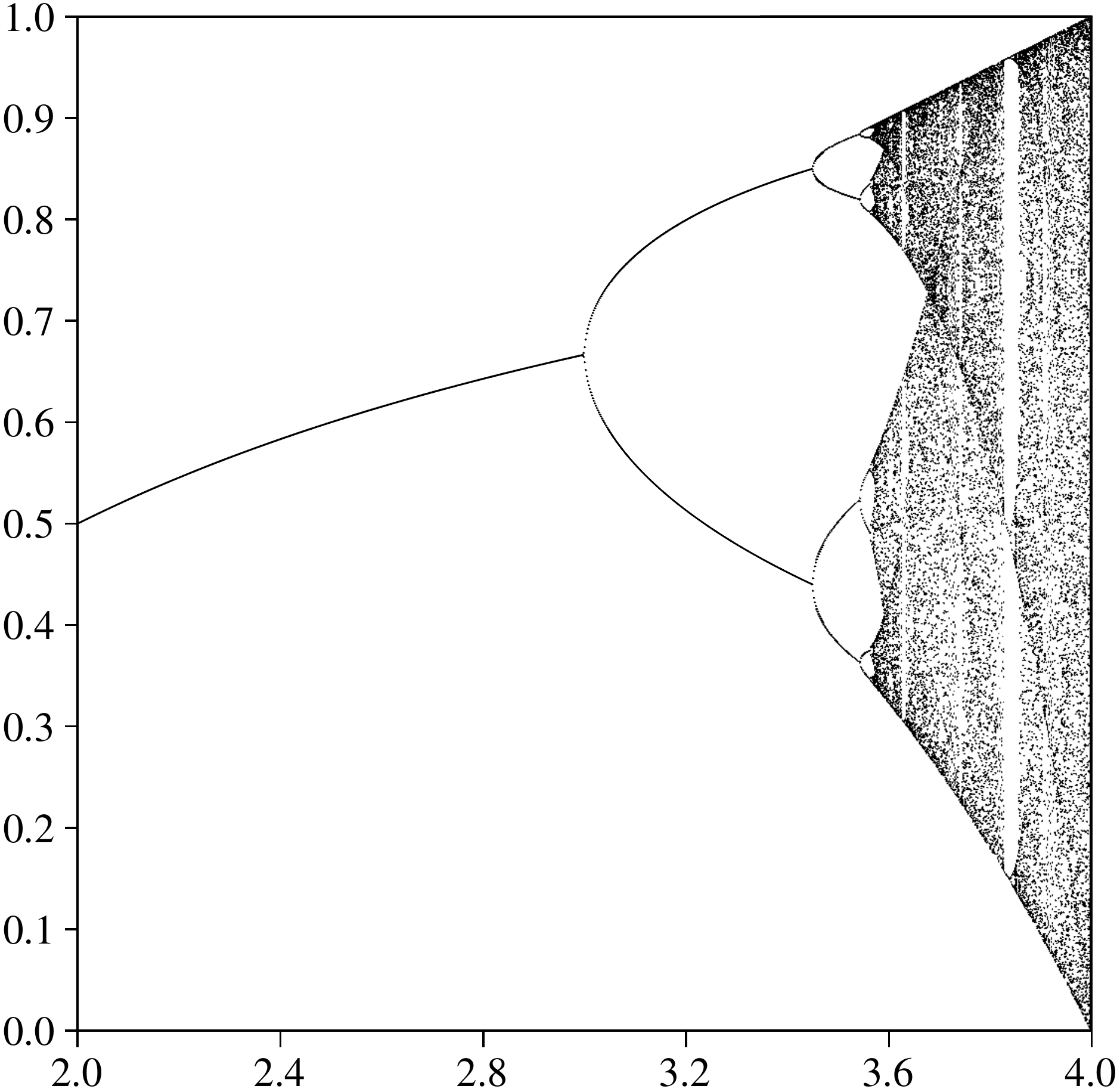}}
\caption[]{Diagramme de bifurcation de l'application logistique. Pour chaque
valeur du param\`etre $\lambda$ en abscisse, on a repr\'esent\'e les valeurs
$X_{1001}$ \`a $X_{1100}$. La condition initiale est toujours $X_0=1/2$.}
\label{fig_logistic}
\end{figure}

Les suites de variables i.i.d.\ forment les processus les plus al\'eatoires. A
l'autre extr\^eme, on trouve les syst\`emes dynamiques, d\'efinis par 
\begin{equation}
 \label{sd1}
X_{n+1} = f(X_n)\;, 
\end{equation} 
o\`u $f$ est une fonction fix\'ee. Dans ce cas, comme 
\begin{equation}
 \label{sd2}
\bigprob{X_{n+1}\in A} = \bigprob{X_n\in f^{-1}(A)}\;, 
\end{equation} 
les lois $\nu_n$ des variables al\'eatoires $X_n$ \'evoluent selon la r\`egle
simple
\begin{equation}
 \label{sd3}
\nu_{n+1} = \nu_n \circ f^{-1}\;.
\end{equation}
On pourrait penser que ces processus sont
plus simples \`a analyser, la seule source de hasard \'etant la distribution de
la condition initiale $X_0$. Dans certaines situations, par exemple si $f$ est
contractante, c'est effectivement le cas, mais de mani\`ere g\'en\'erale il n'en
est rien. Par exemple dans le cas de l'application logistique 
\begin{equation}
 \label{sd4}
f(x) = \lambda x (1-x)\;, 
\end{equation} 
le comportement de la suite des $X_n$ est chaotique pour certaines valeurs de
$\lambda$, en particulier pour $\lambda=4$. En fait l'existence d'une
composante probabiliste tend plut\^ot \`a simplifier l'\'etude de la suite des
$X_n$. Par exemple, pour $\lambda=4$ on conna\^it la mesure de probabilit\'e
invariante du processus, c'est-\`a-dire la mesure $\nu$ telle que $\nu \circ
f^{-1}=\nu$.


\chapter[Construction g\'en\'erale des processus]
{Construction g\'en\'erale de processus stochastiques \`a temps discret}
\label{chap_gen}


\section{Pr\'eliminaires et rappels}
\label{ssec_pr} 

Rappelons quelques notions de base de th\'eorie de la mesure et
de probabilit\'es.

Soit $\Omega$ un ensemble non vide, et soit $\cF$ une \defwd{tribu}\/
sur $\Omega$, c'est-\`a-dire une collection de sous-ensembles de $\Omega$, 
contenant $\Omega$ et stable par r\'eunion d\'enombrable et compl\'ementaire.
On dit que $(\Omega,\cF)$ forme un \defwd{espace mesurable}\/. 

Une \defwd{mesure}\/ sur $(\Omega,\cF)$ est une application
$\mu:\cF\to[0,\infty]$ telle que $\mu(\emptyset)=0$ et satisfaisant 
\begin{equation}
 \label{gen1}
\mu\biggpar{\bigcup_{n=1}^\infty A_n} = \sum_{n=1}^\infty \mu(A_n)
\end{equation} 
pour toute famille $\set{A_n}$ d'\'el\'ements deux \`a deux disjoints de $\cF$
(\defwd{$\sigma$-additivit\'e}\/). Si \mbox{$\mu(\Omega)=1$} on dit que c'est
une \defwd{mesure de probabilit\'e}\/, qu'on notera souvent $\fP$. Le triplet
$(\Omega,\cF,\fP)$ est appel\'e un \defwd{espace probabilis\'e}. 

Si $(E,\cE)$ est un espace mesurable, une application $f:\Omega\to E$
est dite \defwd{$\cF$-$\cE$-mesurable}\/ si elle satisfait 
\begin{equation}
 \label{gen2}
f^{-1}(A) \in \cF \quad \forall A\in\cE\;. 
\end{equation} 
Une \defwd{variable al\'eatoire \`a valeurs dans $(E,\cE)$}\/ sur un espace
probabilis\'e $(\Omega,\cF,\fP)$ est une application $\cF$-$\cE$-mesurable
$X:\Omega\to E$. Nous aurons souvent affaire au cas $E=\R$, avec $\cE$ la tribu
des bor\'eliens $\cB$. Dans ce cas nous dirons simplement que $X$ est
$\cF$-mesurable, et \'ecrirons $X\measurable\cF$. La \defwd{loi}\/ d'une
variable al\'eatoire $X$ \`a valeurs dans $(E,\cE)$ est l'application 
\begin{equation}
 \label{gen3}
\begin{array}{llll}
 \fP X^{-1} : &\cE &\to & [0,1] \\
              &A & \mapsto & \fP(X^{-1}(A)) = \prob{X\in A}\;.
\end{array}
\end{equation} 

L'\defwd{esp\'erance}\/ d'une variable al\'eatoire $X$ est d\'efinie
comme l'int\'egrale de Lebesgue 
\begin{equation}
 \label{gen4}
\expec{X} = \int_\Omega X \,\6\fP\;.  
\end{equation} 
Rappelons que cette int\'egrale est d\'efinie en approchant $X$ par une suite
de fonctions \'etag\'ees $X_n=\sum_i a_i\indicator{A_i}$, pour lesquelles
$\int_\Omega X_n \,\6\fP = \sum_i a_i\fP(A_i)$. 


\section{Distributions de dimension finie}
\label{ssec_dfin} 

Soit $(E,\cE)$ un espace mesurable. L'ensemble $E^n = E\times E\times
\dots \times E$ peut \^etre muni
d'une tribu $\cE^{\otimes n}$, d\'efinie comme la tribu engendr\'ee par
tous les \'ev\'enements du type 
\begin{equation}
\label{gps1}
A^{(i)} = \setsuch{\omega\in E^n}{\omega_i\in A}\;, 
\qquad A\in\cE\;,
\end{equation}
appel\'es \defwd{cylindres}.
On d\'enote par $E^\N$ l'ensemble des applications $x: \N\to E$,
c'est-\`a-dire l'ensemble des suites $(x_0, x_1, x_2, \dots)$ \`a valeurs
dans $E$. Cet ensemble peut \`a nouveau \^etre muni d'une tribu construite
\`a partir de tous les cylindres, not\'ee $\cE^{\otimes \N}$.

\begin{definition}[Processus stochastique, distributions de dimension finie]
\label{def_gps1}
\hfill
\begin{itemiz}
\item	Un \defwd{processus stochastique}\/ \`a valeurs dans $(E,\cE)$ est
une suite $\set{X_n}_{n\in\N}$ de variables al\'eatoires \`a valeurs dans
$(E,\cE)$, d\'efinies sur un m\^eme espace probabilis\'e $(\Omega,\cF,\fP)$
(autrement dit, chaque $X_n$ est une application $\cF$-$\cE$-mesurable de
$\Omega$ dans $E$). C'est donc \'egalement une variable al\'eatoire \`a
valeurs dans $(E^\N,\cE^{\otimes\N})$. 

\item	Soit $\Q$ une mesure de probabilit\'e sur
$(E^\N,\cE^{\otimes\N})$. Les \defwd{distributions de dimension finie}\/
de $\Q$ sont les mesures sur $(E^{n+1},\cE^{\otimes n+1})$ d\'efinies par 
\begin{equation}
\label{gps2}
\Q^{(n)} = \Q \circ (\pi^{(n)})^{-1}\;,
\end{equation}
o\`u $\pi^{(n)}$ est la projection $\pi^{(n)}:E^\N\to E^{n+1}$, $(x_0, x_1,
x_2, \dots) \mapsto (x_0, \dots, x_n)$.
\end{itemiz}
\end{definition}

On se convainc facilement que la suite des $\set{\Q^{(n)}}_{n\in\N}$
d\'etermine $\Q$ univoquement. Inversement, pour qu'une suite donn\'ee
$\set{\Q^{(n)}}_{n\in\N}$ corresponde effectivement \`a une mesure $\Q$,
les $\Q^{(n)}$ doivent satisfaire une \defwd{condition de compatibilit\'e}
: 

Soit $\ph_n$ la projection $\ph_n:E^{n+1}\to E^n$, $(x_0, \dots, x_n)
\mapsto (x_0, \dots, x_{n-1})$. Alors on a
$\pi^{(n-1)}=\ph_n\circ\pi^{(n)}$, donc pour tout $A\in\cE^{\otimes n}$, 
$(\pi^{(n-1)})^{-1}(A) = (\pi^{(n)})^{-1}(\ph_n^{-1}(A))$. La condition
de compatibilit\'e s'\'ecrit donc 
\begin{equation}
\label{gps3}
\Q^{(n-1)} = \Q^{(n)} \circ \ph_n^{-1}\;.
\end{equation}
Le diagramme suivant illustre la situation (toutes les projections
\'etant mesurables, on peut les consid\'erer \`a la fois comme
applications entre ensembles et entre tribus) : 


\vspace{-2mm}
\begin{center}
\begin{tikzpicture}[->,>=stealth',shorten >=2pt,shorten <=2pt,auto,node
distance=2.5cm, thick
]

  \node (N) {$(E^\N,\cE^{\otimes\N})$};
  \node (n+1) [below of=N] {$(E^{n+1},\cE^{\otimes n+1})$};
  \node (n) [below of=n+1] {$(E^n,\cE^{\otimes n})$};
  \node[node distance=0.9cm] (N-) [left of=N] {};
  \node[node distance=0.9cm] (n-) [left of=n] {};
  \node[node distance=0.8cm] (N+) [right of=N] {};
  \node[node distance=0.8cm] (n+) [right of=n] {};
  \node[node distance=1.1cm] (n+1+) [right of=n+1] {};
  \node[node distance=4.5cm] (01) [right of=n+1] {$[0,1]$};

   \path
     (N) edge [right] node {$\pi^{(n)}$} (n+1)
     (n+1) edge [right] node {$\ph_n$} (n)
     (N-) edge [bend right, distance=2.0cm, left] node {$\pi^{(n-1)}$} (n-)
     (N+) edge [above right] node {$\Q$} (01)
     (n+1+) edge [above] node {$\Q^{(n)}$} (01)
     (n+) edge [below right] node {$\Q^{(n-1)}$} (01)
;
\end{tikzpicture}
\end{center}


\section{Noyaux markoviens}
\label{ssec_nmark} 

Nous allons voir comment construire une suite de $\Q^{(n)}$ satisfaisant
la condition~\eqref{gps3}. 

\begin{definition}[Noyau markovien]
\label{def_gps2}
Soient $(E_1,\cE_1)$ et $(E_2,\cE_2)$ deux espaces mesurables. Un
\defwd{noyau marko\-vien de $(E_1,\cE_1)$ vers $(E_2,\cE_2)$}\/  est une
application $K: E_1\times\cE_2 \to [0,1]$ satisfaisant les deux conditions 
\begin{enum}
\item	Pour tout $x\in E_1$, $K(x,\cdot)$ est une mesure de
probabilit\'e sur $(E_2,\cE_2)$. 
\item	Pour tout $A\in\cE_2$, $K(\cdot,A)$ est une application
$\cE_1$-mesurable. 
\end{enum}
\end{definition}

\begin{example}
\label{ex_gps1}
\hfill
\begin{enum}
\item	Soit $\mu$ une mesure de probabilit\'e sur $(E_2,\cE_2)$. Alors
$K$ d\'efini par $K(x,A)=\mu(A)$ pour tout $x\in E_1$ est un noyau
markovien. 

\item	Soit $f:E_1\to E_2$ une application mesurable. Alors $K$ d\'efini
par  $K(x,A)=\indicator{A}(f(x))$ est un noyau markovien.

\item	Soit $\cX=\set{1,\dots,N}$ un ensemble fini, et posons
$E_1=E_2=\cX$ et $\cE_1=\cE_2=\cP(\cX)$. Alors $K$ d\'efini par 
\begin{equation}
\label{gps4}
K(i,A) = \sum_{j\in A} p_{ij}\;,
\end{equation}
o\`u $P=(p_{ij})_{i,j\in\cX}$ est une matrice stochastique, est un noyau
markovien.
\end{enum}
\end{example}

Si $\mu$ est une mesure de probabilit\'e sur $(E_1,\cE_1)$ et $K$ est un
noyau markovien de $(E_1,\cE_1)$ vers $(E_2,\cE_2)$, on d\'efinit une
mesure de probabilit\'e $\mu\otimes K$ sur $\cE_1\otimes\cE_2$ par 
\begin{equation}
\label{gps5}
(\mu\otimes K)(A) \defby \int_{E_1} K(x_1,A_{x_1})
\mu(\6x_1)\;,
\end{equation}
o\`u $A_{x_1} = \setsuch{x_2\in E_2}{(x_1,x_2)\in A} \in \cE_2$ est la
\defwd{section}\/ de $A$ en $x_1$. On v\'erifie que c'est bien une mesure
de probabilit\'e. Afin de comprendre sa signification, calculons ses
marginales. Soient $\pi_1$ et $\pi_2$ les projections d\'efinies par
$\pi_i(x_1,x_2)=x_i$, $i=1,2$. 
\begin{enum}
\item	Pour tout ensemble mesurable $A_1\in\cE_1$, on a 
\begin{align}
\nonumber
\bigpar{(\mu\otimes K)\circ\pi_1^{-1}}(A_1) 
&= (\mu\otimes K)(A_1\times E_2) \\
\nonumber
&= \int_{E_1} \indicator{A_1}(x_1)K(x_1,E_2) \mu(\6x_1) \\
&= \int_{A_1} K(x_1,E_2) \mu(\6x_1)
= \mu(A_1)\;,
\label{gps6}
\end{align}
o\`u on a utilis\'e le fait que la section $(A_1\times E_2)_{x_1}$ est
donn\'ee par $E_2$ si $x_1\in A_1$, et $\emptyset$ sinon. Ceci implique 
\begin{equation}
\label{gps7}
(\mu\otimes K)\circ\pi_1^{-1} = \mu\;.
\end{equation}
La premi\`ere marginale de $\mu\otimes K$ est donc simplement $\mu$. 

\item	Pour tout ensemble mesurable $A_2\in\cE_2$, on a 
\begin{align}
\nonumber
\bigpar{(\mu\otimes K)\circ\pi_2^{-1}}(A_2) 
&= (\mu\otimes K)(E_1\times A_2) \\
&= \int_{E_1} K(x_1,A_2) \mu(\6x_1) \;.
\label{gps8}
\end{align}
La seconde marginale de $\mu\otimes K$ s’interpr\`ete comme suit: c'est la
mesure sur $E_2$ obtenue en partant avec la mesure $\mu$ sur $E_1$, et en
\lq\lq allant de tout $x\in E_1$ vers $A_2\in\cE_2$ avec
probabilit\'e $K(x_1,A_2)$\rq\rq. 
\end{enum}

Enfin, par une variante du th\'eor\`eme de Fubini--Tonelli, on v\'erifie
que  pour toute fonction $(\mu\otimes K)$-int\'egrable $f:E_1\times
E_2\to\R$, on a  
\begin{equation}
\label{gps9}
\int_{E_1\times E_2} f \6\,(\mu\otimes K) 
= \int_{E_1} \biggpar{\int_{E_2} f(x_1,x_2) K(x_1,\6x_2)} \mu(\6x_1)\;.
\end{equation}

Nous pouvons maintenant proc\'eder \`a la construction de la suite
$\set{\Q^{(n)}}_{n\in\N}$ de distributions de dimension finie,
satisfaisant la condition de compatibilit\'e~\eqref{gps3}. Sur
l'espace mesurable $(E,\cE)$, on se donne une mesure de probabilit\'e
$\nu$, appel\'ee \defwd{mesure initiale}. On se donne pour tout $n\in\N$
un noyau markovien $K_n$ de $(E^{n+1},\cE^{\otimes n+1})$ vers $(E,\cE)$. On
d\'efinit alors la suite $\set{\Q^{(n)}}_{n\in\N}$ de mesures de
probabilit\'e sur $(E^{n+1},\cE^{\otimes n+1})$ r\'ecursivement par 
\begin{align}
\nonumber
\Q^{(0)} &= \nu \;,\\
\Q^{(n)} &= \Q^{(n-1)}\otimes K_{n-1}\;, 
& 
n&\geqs 1\;.
\label{gps10}
\end{align}
Par~\eqref{gps7}, on a $\Q^{(n)}\circ\ph_n^{-1}=(\Q^{(n-1)}\otimes
K_{n-1})\circ\ph_n^{-1}=\Q^{(n-1)}$, donc la condition de compatibilit\'e
est bien satisfaite. 

L'interpr\'etation de~\eqref{gps10} est simplement que chaque noyau
markovien $K_n$ d\'ecrit les probabilit\'es de transition entre les temps
$n$ et $n+1$, et permet ainsi de d\'efinir une mesure sur les
segments de trajectoire plus longs d'une unit\'e. Remarquons enfin qu'on
peut \'egalement construire pour tout $m, n$ un noyau $K_{n,m}$ de
$(E^n,\cE^{\otimes n})$ vers $(E^m,\cE^{\otimes m})$ tel que
$\Q^{(n+m)}=\Q^{(n)}\otimes K_{n,m}$. 

On peut noter par ailleurs que si $\psi_n:E^{n+1}\to E$ d\'esigne la
projection sur la derni\`ere composante $(x_0,\dots,x_n)\mapsto x_n$,
alors la formule~\eqref{gps8} montre que la loi de $X_n$, qui est
donn\'ee par la marginale $\nu_n=\Q^{(n)}\circ\psi_n^{-1}$, s'exprime
comme 
\begin{equation}
\label{gps10B}
\prob{X_n\in A} = 
\nu_n(A) = \int_{E^n} K_{n-1}(x,A) \Q^{(n-1)}(\6x)\;.
\end{equation}
La situation est illustr\'ee par le diagramme suivant :


\vspace{-2mm}
\begin{center}
\begin{tikzpicture}[->,>=stealth',shorten >=2pt,shorten <=2pt,auto,node
distance=2.5cm, thick
]

  \node (N) {$(E,\cE)$};
  \node (n+1) [below of=N] {$(E^{n+1},\cE^{\otimes n+1})$};
  \node (n) [below of=n+1] {$(E^n,\cE^{\otimes n})$};
  \node[node distance=0.8cm] (N+) [right of=N] {};
  \node[node distance=0.8cm] (n+) [right of=n] {};
  \node[node distance=1.1cm] (n+1+) [right of=n+1] {};
  \node[node distance=4.5cm] (01) [right of=n+1] {$[0,1]$};

   \path
     (n+1) edge [right] node {$\psi_n$} (N)
     (n+1) edge [right] node {$\ph_n$} (n)
     (N+) edge [above right] node {$\nu_n$} (01)
     (n+1+) edge [above] node {$\Q^{(n)}$} (01)
     (n+) edge [below right] node {$\Q^{(n-1)}$} (01)
;
\end{tikzpicture}
\end{center}


\section{Le th\'eor\`eme de Ionescu--Tulcea}
\label{ssec_til} 

Nous donnons maintenant, sans d\'emonstration, le r\'esultat g\'en\'eral
assurant la l\'egitimit\'e de toute la proc\'edure. 

\begin{theorem}[Ionescu--Tulcea]
Pour la suite de mesures $\set{\Q^{(n)}}_{n\in\N}$ construites
selon~\eqref{gps10}, il existe une unique mesure de probabilit\'e $\Q$ sur
$(E^\N,\cE^{\otimes \N})$ telle que $\Q^{(n)}=\Q\circ(\pi^{(n)})^{-1}$
pour tout $n$, c'est-\`a-dire que les $\Q^{(n)}$ sont les distributions de
dimension finie de $\Q$. 
\end{theorem}

\begin{example}
\label{ex_gps2}
\hfill
\begin{enum}
\item	{\bf Mesures produit:} On se donne une suite
$\set{\mu_n}_{n\in\N}$ de mesures de probabilit\'e sur l'espace mesurable
$(E,\cE)$. Soit, pour tout $n$, 
$\Q^{(n)}=\mu_0\otimes\mu_1\otimes\dots\otimes\mu_n$ la mesure produit. 
C'est une mesure de la forme ci-dessus, avec noyau markovien
\begin{equation}
\label{gps11}
K_n(x,A) = \mu_{n+1}(A) 
\qquad \forall x\in E^n, \forall A\in\cE\;.
\end{equation}
La relation~\eqref{gps10B} montre que la loi $\nu_n$ de $X_n$ est donn\'ee
par $\mu_n$. On dit que les variables al\'eatoires $X_n$ sont
\defwd{ind\'ependantes}. Si tous les $\mu_n$ sont les m\^emes, on dit
qu'elles sont \defwd{ind\'ependantes et identiquement distribu\'ees
(i.i.d.)}\/.

\item	{\bf Syst\`eme dynamique:} On se donne une application mesurable
$f:E\to E$, une mesure de probabilit\'e initiale $\nu$ sur $E$. Soit pour
tout $n$ le noyau markovien 
\begin{equation}
\label{gps12}
K_n(x,A) = \indicator{A}{f(x_n)}\;, 
\end{equation}
et construisons les $\Q^{(n)}$ comme ci-dessus. Alors la
formule~\eqref{gps10B} montre qu'on a pour tout $A\in E$ 
\begin{align}
\nonumber
\nu_{n+1}(A) 
&= \int_{E^{n+1}} \indicator{A}{f(x_n)} \Q^{(n)}(\6x) \\
\nonumber
&= \int_{E} \indicator{A}{f(x_n)} \nu_n(\6x_n) \\
&= \int_{f^{-1}(A)} \nu_n(\6x_n) 
= \nu_n(f^{-1}(A))\;.
\label{gps13}
\end{align}
Il suit que 
\begin{equation}
\label{gps14}
\nu_n = \nu\circ f^{-n}
\qquad 
\forall n\in\N\;.
\end{equation}
Cette situation correspond \`a un syst\`eme dynamique d\'eterministe. Par
exemple, si $f$ est bijective et $\nu=\delta_{x_0}$ est concentr\'ee en un
point, on a $\nu_n = \delta_{f^n(x_0)}$.

\item	{\bf \Chaine s de Markov:}
Soit $\cX$ un ensemble fini ou d\'enombrable, muni de la tribu $\cP(\cX)$,
et $P=(p_{ij})_{i,j\in\cX}$ une matrice stochastique sur $\cX$,
c'est-\`a-dire que $0\leqs p_{ij}\leqs 1$ $\forall i,j\in\cX$ 
et $\sum_{j\in\cX}p_{ij}=1$ $\forall i\in\cX$. On se donne une mesure de
probabilit\'e $\nu$ sur $\cX$, et une suite $\smash{\Q^{(n)}}$ construite
\`a partir de~\eqref{gps10} avec pour noyaux markoviens 
\begin{equation}
\label{gps15}
K_n(i_{[0,n-1]},i_n) = p_{i_{n-1}i_n}\;.
\end{equation}
Le processus stochastique de mesure $\Q$, dont les distributions de
dimension finie sont les $\Q^{(n)}$, est la \chaine\ de Markov sur $\cX$
de distribution initiale $\nu$ et de matrice de transition $P$. Dans ce
cas la relation~\eqref{gps10B} se traduit en
\begin{equation}
\label{gps16}
\prob{X_n\in A} = \nu_n(A) = \sum_{i\in\cX} \sum_{j\in A} p_{ij}
\nu_{n-1}(\set{i}) = \sum_{j\in A} \sum_{i\in\cX} \prob{X_{n-1}=i}p_{ij}\;.
\end{equation}
\end{enum}
\end{example}


\chapter{Filtrations, esp\'erance conditionnelle}
\label{chap_ec}


\section{Sous-tribus et filtrations}
\label{sec_st}

Soit $(\Omega,\cF)$ un espace mesurable. Une \defwd{sous-tribu}\/ de $\cF$ est
une sous-famille $\cF_1\subset\cF$ qui est \'egalement une tribu. 
On dit parfois que la tribu $\cF_1$ est plus grossi\`ere que la tribu $\cF$, et
que $\cF$ est plus fine que $\cF_1$. On notera
que si $X$ est une variable al\'eatoire r\'eelle, on a 
l'implication 
\begin{equation}
 \label{st0}
X\measurable\cF_1 \Rightarrow X\measurable\cF\;. 
\end{equation}
Une fonction non mesurable peut donc \^etre rendue mesurable en choisissant une
tribu plus fine. 

\begin{example}\hfill
\label{ex_st1}
\begin{enum}
\item	La plus petite sous-tribu de $\cF$ est la \defwd{tribu triviale}\/
$\cF_0=\set{\emptyset,\Omega}$. On remarquera que si $X$ est mesurable par
rapport \`a $\cF_0$, alors la pr\'eimage $X^{-1}(y)$ d'un point $y$ doit \^etre
\'egale soit \`a $\Omega$ tout entier, soit \`a l'ensemble vide, ce qui implique
que $X$ doit \^etre constante. En d'autres termes, les variables al\'eatoires
mesurables par rapport \`a la tribu triviale sont les fonctions constantes.

\item	Soit $X$ une variable al\'eatoire \`a valeurs dans un espace
mesurable $(E,\cE)$. Alors on v\'erifie ais\'ement que 
\begin{equation}
 \label{st1}
\sigma(X) \defby \setsuch{X^{-1}(A)}{A\in\cE}\;,
\end{equation} 
o\`u $X^{-1}(A)\defby\setsuch{\omega\in\Omega}{X(\omega)\in A}$, 
est une sous-tribu de $\cF$. C'est la plus petite sous-tribu de $\cF$ par
rapport \`a laquelle $X$ soit mesurable. 

\item	Si $X_1, \dots, X_n$ sont des variables al\'eatoires \`a valeurs dans
$(E,\cE)$, alors 
\begin{equation}
 \label{st1b}
\sigma(X_1,\dots,X_n)  \defby
\setsuch{(X_1,\dots,X_n)^{-1}(A)}{A\in\cE^{\otimes n}}\;,
\end{equation} 
o\`u $(X_1,\dots,X_n)^{-1}(A)\defby\setsuch{\omega\in\Omega}{(X_1(\omega),\dots,
X_n(\omega))\in A}$, est une sous-tribu de $\cF$. C'est la plus petite
sous-tribu de $\cF$ par rapport \`a laquelle les variables al\'eatoires
$X_1,\dots X_n$ soient toutes mesurables. 
\end{enum}
\end{example}

L'exemple ci-dessus donne une interpr\'etation importante de la notion de
sous-tribu. En effet, $\sigma(X)$ repr\'esente l'ensemble des \'ev\'enements
qu'on est potentiellement capable de distinguer en mesurant la variable $X$.
Autrement dit, $\sigma(X)$ est l'information que $X$ peut fournir sur l'espace
probabilis\'e. 

\begin{definition}[Filtration, processus adapt\'e] \hfill
\label{def_filtration}
\begin{enum}
\item	Soit $(\Omega,\cF,\fP)$ un espace probabilis\'e. Une\/
\defwd{filtration}\/ de $(\Omega,\cF,\fP)$ est une suite croissante de
sous-tribus 
\begin{equation}
 \label{st2}
\cF_0 \subset \cF_1 \subset \dots \subset \cF_n \subset \dots \subset \cF\;.
\end{equation} 
On dit alors que $(\Omega,\cF,\set{\cF_n},\fP)$ est un \defwd{espace
probabilis\'e filtr\'e}\/.

\item	Soit $\set{X_n}_{n\in\N}$ un processus stochastique sur
$(\Omega,\cF,\fP)$. On dit que le processus est\/ \defwd{adapt\'e \`a la
filtration $\set{\cF_n}$}\/ si $X_n$ est mesurable par rapport \`a $\cF_n$ pour
tout $n$. 
\end{enum}
\end{definition}

Un choix minimal de filtration adapt\'ee est la \defwd{filtration canonique}\/
(ou\/ \defwd{naturelle}\/)
\begin{equation}
 \label{st3}
\cF_n = \sigma(X_0,X_1,\dots,X_n)\;. 
\end{equation} 
Dans ce cas, $\cF_n$ repr\'esente l'information disponible au temps $n$, si
l'on observe le processus stochastique. 

\begin{example}
Consid\'erons la marche al\'eatoire sym\'etrique sur $\Z$. Dans ce cas, $E=\Z$
et $\cE=\cP(\Z)$. Le choix de $\Omega$ est arbitraire, il suffit de le prendre
\lq\lq assez grand\rq\rq\ pour distinguer toutes les r\'ealisations possibles
de la marche. Un choix pratique est l'ensemble $\Omega=\set{-1,1}^\N$ des
suites infinies de $-1$ et de $1$, avec la convention que le $n$i\`eme
\'el\'ement de la suite sp\'ecifie la direction du $n$i\`eme pas de la marche.
En d'autres termes,
\begin{equation}
 \label{st4}
X_n(\omega) = \sum_{i=1}^n \omega_i\;. 
\end{equation} 
La tribu associ\'ee est
$\cF=\cP(\set{-1,1})^{\otimes\N}=
\set{\emptyset,\set{-1},\set{1},\set{-1,1}}^{\otimes\N}$. 

Construisons maintenant la filtration naturelle. Tout d'abord, $X_0=0$ n'est
pas vraiment al\'eatoire. Nous avons pour tout $A\subset\Z$ 
\begin{equation}
 \label{st5}
X_0^{-1}(A) = \setsuch{\omega\in\Omega}{X_0=0\in A} = 
\begin{cases}
\emptyset & \text{si $0\notin A$\;,} \\
\Omega & \text{si $0\in A$\;,} 
\end{cases}
\end{equation}
de sorte que 
\begin{equation}
 \label{st6}
\cF_0 = \set{\emptyset,\Omega} 
\end{equation} 
est la tribu triviale. 

Pour $n=1$, on observe que
$(X_0,X_1)(\Omega)=\set{(0,-1),(0,1)}$. Il y a donc quatre cas \`a distinguer,
selon qu'aucun, l'un ou l'autre, ou les deux points appartiennent \`a
$A\subset\Z^2$.
Ainsi, 
\begin{equation}
 \label{st7}
(X_0,X_1)^{-1}(A) = 
\begin{cases}
\emptyset & \text{si $(0,-1)\notin A$ et $(0,1)\notin A$\;,} \\
\setsuch{\omega}{\omega_1=1} & \text{si $(0,-1)\notin A$ et $(0,1)\in A$\;,} \\
\setsuch{\omega}{\omega_1=-1} & \text{si $(0,-1)\in A$ et $(0,1)\notin A$\;,} \\
\Omega & \text{si $(0,-1)\in A$ et $(0,1)\in A$\;,}
\end{cases}
\end{equation} 
et par cons\'equent 
\begin{equation}
 \label{st8}
\cF_1 = \set{
\emptyset,\setsuch{\omega}{\omega_1=1},\setsuch{\omega}{\omega_1=-1},\Omega}\;. 
\end{equation} 
Ceci traduit bien le fait que $\cF_1$ contient l'information disponible au
temps $1$~: On sait distinguer tous les \'ev\'enements d\'ependant du premier
pas de la marche. Les variables al\'eatoires mesurables par rapport \`a $\cF_1$
sont pr\'ecis\'ement celles qui ne d\'ependent que de~$\omega_1$. 

Jusqu'au temps $n=2$, il y a quatre trajectoires possibles, puisque 
$(X_0,X_1,X_2)(\Omega)=\set{(0,-1,-2),(0,-1,0),(0,1,0),(0,1,2)}$. Il suit par
un raisonnement analogue que $\cF_2$ contient $2^4=16$ \'el\'ements, qui se
distinguent par quels points parmi ces quatre sont contenus dans $A$. Les
variables al\'eatoires mesurables par rapport \`a $\cF_2$
sont pr\'ecis\'ement celles qui ne d\'ependent que de $\omega_1$ et $\omega_2$. 

Il est maintenant facile de g\'en\'eraliser \`a des $n$ quelconques. 
\end{example}


\section{Esp\'erance conditionnelle}
\label{sec_ec}

Dans cette section, nous fixons un espace probabilis\'e $(\Omega,\cF,\fP)$ et
une sous-tribu $\cF_1\subset\cF$. Nous avons vu que $\cF_1$ repr\'esente une
information partielle sur l'espace, obtenue par exemple en observant une
variable al\'eatoire $X_1$. L'esp\'erance conditionnelle d'une variable
al\'eatoire $X$ par rapport \`a $\cF_1$ repr\'esente la meilleure estimation que
l'on puisse faire de la valeur de $X$ \`a l'aide de l'information contenue dans
$\cF_1$. 

\begin{definition}[Esp\'erance conditionnelle]
\label{def_condesp}
Soit $X$ une variable al\'eatoire r\'eelle sur $(\Omega,\cF,\fP)$ telle que
$\expec{\abs{X}}<\infty$. On appelle \defwd{esp\'erance conditionnelle de $X$
sachant $\cF_1$}\/, et on note $\econd{X}{\cF_1}$, toute variable al\'eatoire
$Y$ satisfaisant les deux conditions
\begin{enum}
\item	$Y\measurable\cF_1$, c'est-\`a-dire $Y$ est $\cF_1$-mesurable;
\item	pour tout $A\in\cF_1$, on a 
\begin{equation}
 \label{ec1}
\int_A X\,\6\fP = \int_A Y\,\6\fP\;.
\end{equation} 
\end{enum}
Si $Z$ est une variable al\'eatoire r\'eelle sur $(\Omega,\cF,\fP)$, nous
abr\'egeons $\econd{X}{\sigma(Z)}$ par $\econd{X}{Z}$.
\end{definition}

En fait, toute variable al\'eatoire $Y$ satisfaisant la d\'efinition est
appel\'ee une \defwd{version}\/ de $\econd{X}{\cF_1}$. Le r\'esultat suivant
tranche la question de l'existence et de l'unicit\'e de l'esp\'erance
conditionnelle. 

\begin{theorem}\hfill
\label{thm_mart1}
\begin{enum}
\item	L'esp\'erance conditionnelle\/ $\econd{X}{\cF_1}$ existe.
\item	L'esp\'erance conditionnelle est unique dans le sens que si\/ $Y$ et\/
$Y'$ sont deux versions de\/ $\econd{X}{\cF_1}$, alors $Y=Y'$ presque
s\^urement.
\item	On a\/ $\expec{\econd{X}{\cF_1}}=\expec{X}$ et\/ 
$\expec{\abs{\econd{X}{\cF_1}}} \leqs \expec{\abs{X}}$.
\end{enum}
\end{theorem}
\begin{proof}
Commen\c cons par prouver les assertions du point 3. La premi\`ere est un cas
particulier de~\eqref{ec1} avec $A=\Omega$. Pour montrer la seconde, soit
$Y=\econd{X}{\cF_1}$ et
$A=\set{Y>0}\defby\setsuch{\omega\in\Omega}{Y(\omega)>0}$. On a $A\in\cF_1$ par
mesurabilit\'e de $Y$. Or par~\eqref{ec1}, 
\begin{align}
\nonumber
\int_A Y\,\6\fP &= \int_A X\,\6\fP \leqs \int_A \abs{X}\,\6\fP\;, \\
\int_{A^c} -Y\,\6\fP &= \int_{A^c} -X\,\6\fP \leqs \int_{A^c} \abs{X}\,\6\fP\;.
 \label{ec2:1} 
\end{align} 
Le r\'esultat suit en ajoutant ces deux in\'egalit\'es.

Montrons l'unicit\'e. Si $Y$ et $Y'$ sont deux versions de
$\econd{X}{\cF_1}$, alors pour tout $A\in\cF_1$ 
\begin{equation}
 \label{ec2:2}
\int_A Y\,\6\fP = \int_A Y'\,\6\fP\;.
\end{equation} 
Prenons $A=\set{Y-Y'\geqs\eps}$ pour $\eps>0$. Alors 
\begin{equation}
 \label{ec2:3}
0 =  \int_A (Y-Y')\,\6\fP \geqs \eps \fP(A) = \eps\prob{Y-Y'\geqs\eps}\;.
\end{equation} 
Comme c'est vrai pour tout $\eps$, on a $Y\leqs Y'$ presque s\^urement. En
interchangeant les r\^oles de $Y$ et $Y'$, on obtient l'in\'egalit\'e inverse,
d'o\`u on d\'eduit l'\'egalit\'e presque s\^ure. 

Montrons finalement l'existence. Rappelons qu'une mesure $\nu$ sur
$(\Omega,\cF_1)$ est dite \defwd{absolument continue par rapport \`a la mesure
$\mu$}\/ si $\mu(A)=0$ implique $\nu(A)=0$ pour tout $A\in\cF_1$. On \'ecrit
alors $\nu\ll\mu$. Le th\'eor\`eme de Radon--Nikodym affirme qu'il existe alors
une fonction $f\measurable\cF_1$ telle que 
\begin{equation}
 \label{ec2:4}
\int_A f\,\6\mu = \nu(A)
\qquad\forall A\in\cF_1\;. 
\end{equation} 
La fonction $f$ est appel\'ee \defwd{d\'eriv\'ee de Radon--Nikodym}\/ et
not\'ee $\tdtot\nu\mu$. 

Supposons d'abord que $X\geqs0$. Posons $\mu=\fP$ et d\'efinissons $\nu$ par 
\begin{equation}
 \label{ec2:5}
\nu(A) = \int_A X\,\6\fP 
\qquad\forall A\in\cF_1\;,
\end{equation} 
de sorte que $\nu\ll\mu$. Nous avons $\tdtot\nu\mu\measurable\cF_1$ et pour
tout $A\in\cF_1$ 
\begin{equation}
\label{ec2:6}
\int_A X\,\6\fP  = \nu(A) = \int_A \dtot\nu\mu\,\6\fP\;.
\end{equation} 
Ceci montre que $\tdtot\nu\mu$ satisfait~\eqref{ec1}. De plus, prenant
$A=\Omega$ on voit que $\tdtot\nu\mu$ est int\'egrable. Par cons\'equent
$\tdtot\nu\mu$ est une version de $\econd{X}{\cF_1}$.

Finalement, un $X$ g\'en\'eral peut se d\'ecomposer $X=X^+-X^-$ avec $X^+,
X^-\geqs0$. Soit $Y_1=\econd{X^+}{\cF_1}$ et $Y_2=\econd{X^-}{\cF_1}$. Alors
$Y_1-Y_2$ est $\cF_1$-mesurable et int\'egrable et on a pour tout $A\in\cF_1$ 
\begin{equation}
 \label{ec2:7}
\int_A X\,\6\fP = \int_A X^+\,\6\fP - \int_A X^-\,\6\fP 
= \int_A Y_1\,\6\fP - \int_A Y_2\,\6\fP = \int_A (Y_1-Y_2)\,\6\fP\;.
\end{equation} 
Ceci montre que $Y_1-Y_2$ est une version de  $\econd{X}{\cF_1}$.
\end{proof}

Dans la suite, nous allons en g\'en\'eral ignorer l'existence de plusieurs
versions de l'esp\'e\-rance conditionnelle, puisque des variables al\'eatoires
\'egales presque partout sont indistinguables en pratique.

\begin{example}\hfill
\label{ex_ec1}
\begin{enum}
\item	Supposons que $X$ soit $\cF_1$-mesurable. Alors $\econd{X}{\cF_1}=X$
v\'erifie la d\'efinition. Cela traduit le fait que $\cF_1$ contient d\'ej\`a
toute l'information sur $X$. 

\item	L'autre extr\^eme est le cas de l'ind\'ependance. Rappelons que $X$ est
ind\'ependante de $\cF_1$ si pour tout $A\in\cF_1$ et tout bor\'elien
$B\subset\R$, 
\begin{equation}
 \label{ec3}
\fP\bigpar{\set{X\in B}\cap A} = \prob{X\in B} \fP(A)\;,
\end{equation}
et qu'alors on a $\expec{X\indicator{A}}=\expec{X}\fP(A)$. 
Dans ce cas nous avons $\econd{X}{\cF_1}=\expec{X}$, c'est-\`a-dire qu'en
l'absence de toute information, la meilleure estimation que l'on puisse faire
de $X$ est son esp\'erance. En particulier, on notera que toute variable
al\'eatoire est ind\'ependante de la tribu triviale $\cF_0$, et que par
cons\'equent on aura toujours $\econd{X}{\cF_0}=\expec{X}$.

Pour v\'erifier la premi\`ere condition de la d\'efinition, il suffit
d'observer que $\expec{X}$, \'etant une constante, est mesurable par rapport
\`a la tribu triviale $\cF_0$, donc aussi par rapport \`a $\cF_1$. Pour
v\'erifier la seconde assertion, prenons $A\in\cF_1$. Alors
$\indicator{A}\measurable\cF_1$ et par ind\'ependance 
\begin{equation}
 \label{ec4}
\int_A X\,\6\fP = \expec{X \indicator{A}} = \expec{X} \expec{\indicator{A}} 
= \expec{X} \fP(A) = \int_A \expec{X} \,\6\fP\;.
\end{equation} 

\item	Soit $\Omega_1,\Omega_2,\dots$ une partition de $\Omega$ telle que
$\fP(\Omega_i)$ soit strictement positif pour tout $i$. Soit
$\cF_1=\sigma(\Omega_1,\Omega_2,\dots)$ la tribu engendr\'ee par les $\Omega_i$.
Alors 
\begin{equation}
 \label{ec5}
\econd{X}{\cF_1}(\omega) = \frac{\expec{X\indicator{\Omega_i}}}{\fP(\Omega_i)}
= \frac{1}{\fP(\Omega_i)} \int_{\Omega_i} X\,\6\fP 
\qquad \forall \omega\in\Omega_i\;.
\end{equation} 
Dans ce cas, l'information contenue dans $\cF_1$ sp\'ecifie dans quel
$\Omega_i$ on se trouve, et la meilleure estimation de $X$ est donc sa moyenne
sur $\Omega_i$.

Pour le v\'erifier, observons d'abord que comme $\econd{X}{\cF_1}$ est
constante sur chaque $\Omega_i$, elle est mesurable par rapport \`a $\cF_1$. 
De plus, 
\begin{equation}
 \label{ec6}
\int_{\Omega_i} \econd{X}{\cF_1}\,\6\fP = \expec{X\indicator{\Omega_i}}
= \int_{\Omega_i} X\,\6\fP \;.
\end{equation} 
Comme $\cF_1$ est engendr\'ee par les $\Omega_i$, le r\'esultat suit par
$\sigma$-additivit\'e. 

\item	Consid\'erons une \chaine\ de Markov $\set{X_n}_{n\in\N}$ sur un
ensemble $\cX$, de matrice de transition $P=(p_{i,j})$, et soit $f:\cX\to\R$ une
fonction mesurable. Nous pr\'etendons que pour tout $n\in\N$, 
\begin{equation}
 \label{ec7}
\econd{f(X_{n+1})}{X_n}  = \sum_{j\in\cX} f(j) p_{X_n,j}\;.
\end{equation} 
En effet, cette expression \'etant constante sur tout ensemble o\`u $X_n$ est
constant, elle est mesurable par rapport \`a $\sigma(X_n)$. De plus, en
appliquant~\eqref{ec5} avec la partition
donn\'ee par $\Omega_i=\setsuch{\omega}{X_n=i}$, on a pour $\omega\in\Omega_i$
\begin{align}
\nonumber
\econd{f(X_{n+1})}{X_n}(\omega)  &= 
\frac{\expec{f(X_{n+1})\indexfct{X_n=i}}}{\prob{X_n=i}} \\
\nonumber
&= \sum_{j\in\cX} f(j) \frac{\prob{X_{n+1}=j,X_n=i}}{\prob{X_n=i}} \\
&= \sum_{j\in\cX} f(j) \pcond{X_{n+1}=j}{X_n=i} 
= \sum_{j\in\cX} f(j) p_{i,j}\;.
 \label{ec8}
\end{align} 
\end{enum}
\end{example}


\begin{prop}
L'esp\'erance conditionnelle a les propri\'et\'es suivantes~:\hfill
\begin{enum}
\item	Lin\'earit\'e~: $\econd{aX+Y}{\cF_1} = a\econd{X}{\cF_1} +
\econd{Y}{\cF_1}$. 
\item	Monotonie~: Si $X\leqs Y$ alors $\econd{X}{\cF_1} \leqs
\econd{Y}{\cF_1}$. 
\item	Convergence monotone~: Si $X_n\geqs0$ est une suite
croissante telle que $X_n\nearrow X$ avec $\expec{X}<\infty$ alors
$\econd{X_n}{\cF_1}\nearrow\econd{X}{\cF_1}$. 
\item	In\'egalit\'e de Jensen~: Si $\ph$ est convexe et $\expec{\abs{X}}$ et
$\expec{\abs{\ph(X)}}$ sont finies alors 
\begin{equation}
 \label{ec9} 
\ph(\econd{X}{\cF_1}) \leqs \econd{\ph(X)}{\cF_1}\;.
\end{equation} 
\item	Contraction dans $L^p$ pour $p\geqs 1$~: 
$\expec{\abs{\econd{X}{\cF_1}}^p}\leqs\expec{\abs{X}^p}$. 
\end{enum}
\end{prop}
%

\goodbreak

Nous laissons la preuve en exercice. 
Le r\'esultat suivant d\'ecrit comment les esp\'erances conditionnelles se
comportent par rapport \`a des sous-tribus. Il dit en r\'esum\'e que c'est
toujours la tribu la plus grossi\`ere qui l'emporte. 

\begin{prop}
\label{prop_econd1}
Si $\cF_1\subset\cF_2$, alors 
\begin{enum}
\item	$\econd{\econd{X}{\cF_1}}{\cF_2} = \econd{X}{\cF_1}$;
\item	$\econd{\econd{X}{\cF_2}}{\cF_1} = \econd{X}{\cF_1}$. 
\end{enum} 
\end{prop}
\begin{proof}
Pour montrer la premi\`ere identit\'e, il suffit de noter que
$\econd{X}{\cF_1}\measurable\cF_2$ et d'appliquer le premier point de
l'exemple~\ref{ex_ec1}. 
Pour la seconde relation, on observe que pour tout $A\in\cF_1\subset\cF_2$, 
\begin{equation}
 \label{ec10}
\int_A \econd{X}{\cF_1} \,\6\fP = \int_A X \,\6\fP 
= \int_A \econd{X}{\cF_2} \,\6\fP\;,
\end{equation} 
la premi\`ere \'egalit\'e suivant de
$\econd{X}{\cF_1}\measurable\cF_1$ et la seconde de 
$\econd{X}{\cF_2}\measurable\cF_2$.
\end{proof}

Le r\'esultat suivant montre que les variables al\'eatoires $\cF_1$-mesurables
se comportent comme des constantes relativement aux esp\'erances conditionnelles
par rapport \`a $\cF_1$. 

\begin{theorem}
\label{thm_ec2}
Si $X\measurable\cF_1$ et\/ $\expec{\abs{Y}}, \expec{\abs{XY}} <  \infty$,
alors 
\begin{equation}
 \label{ec11}
\econd{XY}{\cF_1} = X \econd{Y}{\cF_1}\;. 
\end{equation} 
\end{theorem}
\begin{proof}
Le membre de droite \'etant $\cF_1$-mesurable, la premi\`ere condition de la
d\'efinition est v\'erifi\'ee. Pour v\'erifier la seconde condition, nous
commen\c cons par consid\'erer le cas $X=\indicator{B}$ avec $B\in\cF_1$. Alors
pour tout $A\in\cF_1$ on a 
\begin{equation}
 \label{ec11:1}
\int_A \indicator{B}\, \econd{Y}{\cF_1} \,\6\fP 
= \int_{A\cap B} \econd{Y}{\cF_1} \,\6\fP 
= \int_{A\cap B} Y \,\6\fP 
= \int_A \indicator{B}\, Y \,\6\fP\;.
\end{equation} 
Ceci montre que la seconde condition est v\'erifi\'ee pour des fonctions
indicatrices. Par lin\'earit\'e, elle est aussi vraie pour des fonctions
\'etag\'ees $\sum a_i\indicator{B_i}$. Le th\'eor\`eme de convergence monotone
permet d'\'etendre le r\'esultat aux variables $X, Y\geqs 0$. Enfin, pour le
cas g\'en\'eral, il suffit de d\'ecomposer $X$ et $Y$ en leurs parties positive
et n\'egative. 
\end{proof}

Enfin, le r\'esultat suivant montre que l'esp\'erance conditionnelle peut
\^etre consid\'er\'ee comme une projection de
$L^2(\cF)=\setsuch{Y\measurable\cF}{\expec{Y^2}<\infty}$ dans $L^2(\cF_1)$. 

\begin{theorem}
\label{thm_ec3}
Si $\expec{X^2}<\infty$, alors $\econd{X}{\cF_1}$ est la variable
$Y\measurable\cF_1$ qui minimise $\expec{(X-Y)^2}$.  
\end{theorem}
\begin{proof}
Pour tout $Z\in L^2(\cF_1)$, le th\'eor\`eme pr\'ec\'edent montre que 
$Z\econd{X}{\cF_1}=\econd{ZX}{\cF_1}$. Prenant l'esp\'erance, on obtient 
\begin{equation}
 \label{ec12}
\bigexpec{Z\econd{X}{\cF_1}} = \bigexpec{\econd{ZX}{\cF_1}} = \expec{ZX}\;, 
\end{equation} 
ou encore 
\begin{equation}
 \label{ec13}
\bigexpec{Z\brak{X-\econd{X}{\cF_1}}} = 0\;. 
\end{equation} 
Si $Y=\econd{X}{\cF_1}+Z\measurable\cF_1$, alors 
\begin{equation}
 \label{ec14}
\bigexpec{(X-Y)^2} 
= \bigexpec{(X-\econd{X}{\cF_1}-Z)^2}
= \bigexpec{(X-\econd{X}{\cF_1})^2} + \bigexpec{Z^2}\;,
\end{equation} 
qui est minimal quand $Z=0$. 
\end{proof}


\section{Exercices}
\label{sec_exo_expcond}

\begin{exercice}
\label{exo_ec1} 
Dans une exp\'erience consistant \`a jeter deux t\'etra\`edres parfaitement
sym\'e\-triques, dont les faces sont num\'erot\'ees de $1$ \`a $4$, on
consid\`ere
les variables al\'eatoires $X$, \'egale \`a la somme des points, et $Y$, \'egale
\`a leur diff\'erence (en valeur absolue). 

\begin{enum}
\item	Sp\'ecifier un espace probabilis\'e permettant de d\'ecrire cette
exp\'erience.
\item	D\'eterminer la loi conjointe de $X$ et $Y$ ainsi que leurs
esp\'erances.
\item	Calculer $\econd{X}{Y}$ et $\econd{Y}{X}$.
\end{enum}
\end{exercice}

\begin{exercice}
\label{exo_ec2}
On jette un d\'e sym\'etrique, puis on jette une pi\`ece de monnaie autant de
fois que le d\'e indique de points. Soit $X$ le nombre de Pile obtenus.
D\'eterminer $\expec{X}$.
\end{exercice}

\begin{exercice}
\label{exo_ec3}
Soit $(\Omega,\cF,\fP)$ un espace probabilis\'e, et $\cF_1$ une sous-tribu de
$\cF$. Pour tout $A\in\cF$, on pose
$\pecond{A}{\cF_1}=\econd{\indicator{A}}{\cF_1}$. Montrer
l'in\'egalit\'e de
Bienaym\'e--Chebyshev
\[
\bigpecond{\set{\abs{X}\geqs a}}{\cF_1} \leqs \frac1{a^2} \econd{X^2}{\cF_1}\;.
\]
\end{exercice}

\begin{exercice}
\label{exo_ec4}
Soient $X, Y$ des variables al\'eatoires r\'eelles int\'egrables telles que 
$XY$ soit \'egalement int\'egrable. 
Montrer les implications~: \\
$X, Y$ ind\'ependantes $\Rightarrow \econd{Y}{X}=\expec{Y} \Rightarrow
\expec{XY}=\expec{X}\expec{Y}$. 

\noindent
{\bf Indication~:} Commencer par consid\'erer des fonctions indicatrices.
\end{exercice}

\begin{exercice}
\label{exo_ec5}
Soient $X$ et $Y$ des variables al\'eatoires \`a valeurs dans $\set{-1,0,1}$.
Exprimer les trois conditions de l'exercice~\ref{exo_ec4} \`a l'aide de la loi
conjointe de $X$ et $Y$. Donner des contre-exemples aux implications inverses. 

\noindent
{\bf Indication~:} On peut supposer $\expec{Y}=0$.
\end{exercice}

\begin{exercice}
\label{exo_ec6}
Soit $(\Omega,\cF,\fP)$ un espace probabilis\'e, et $\cF_1\subset\cF_2$ des
sous-tribus de $\cF$. 
\begin{enum}
\item	Montrer que 
\[
\bigexpec{\brak{X-\econd{X}{\cF_2}}^2}
+ \bigexpec{\brak{\econd{X}{\cF_2}-\econd{X}{\cF_1}}^2}
= \bigexpec{\brak{X-\econd{X}{\cF_1}}^2}
\]
\item	On pose $\varcond{X}{\cF_1}=\econd{X^2}{\cF_1} - \econd{X}{\cF_1}^2$.
Montrer que 
\[
\Variance(X) = \bigexpec{\varcond{X}{\cF_1}} +
\Variance\bigpar{\econd{X}{\cF_1}}\;.
\]
\item	Soit $Y_1, Y_2, \dots$ une suite de variables al\'eatoires i.i.d.\
d'esp\'erance $\mu$ et de variance $\sigma^2$. Soit $N$ une variable
al\'eatoire \`a valeurs dans $\N$, ind\'ependante de tous les $Y_i$. Soit
finalement $X=Y_1+Y_2+\dots+Y_N$. Montrer que 
\[
\Variance(X) = \sigma^2 \expec{N} + \mu^2 \Variance(N)\;.
\]
\item	D\'eterminer la variance de la variable al\'eatoire $X$ de
l'exercice~\ref{exo_ec2}.
\end{enum}
\end{exercice}

\begin{exercice} 
\label{exo_ec7}
Soit $(\Omega,\cF,\fP)$ un espace probabilis\'e, et $\cG$ une sous-tribu de
$\cF$. On consid\`ere deux variables al\'eatoires $X$ et $Y$ telles que 
\[
\econd{Y}{\cG}=X
\qquad\text{et}\qquad 
\expec{X^2}=\expec{Y^2}
\]

\begin{enum}
\item	Calculer $\varcond{Y-X}{\cG}$. 
\item	En d\'eduire $\Variance(Y-X)$. 
\item	Que peut-on en d\'eduire sur la relation entre $X$ et $Y$?
\end{enum}
\end{exercice}

\begin{exercice}[Le processus ponctuel de Poisson]
\label{exo_ec8}
Un processus ponctuel de Poisson d'intensit\'e $\lambda>0$ est un processus
$\set{N_t}_{t\geqs0}$ tel que 
\begin{enum}
\renewcommand{\labelenumi}{\roman{enumi}.}
\item	$N_0 = 0$;
\item	pour $t>s\geqs 0$, $N_t-N_s$ est ind\'ependant de $N_s$;
\item	pour $t>s\geqs 0$, $N_t-N_s$ suit une loi de Poisson de param\`etre
$\lambda(t-s)$: 
\[
\prob{N_t-N_s = k} = \e^{-\lambda(t-s)} \frac{(\lambda (t-s))^k}{k!}\;, 
\qquad k \in \N\;. 
\]
\end{enum}
On se donne des variables al\'eatoires i.i.d.\ $\set{\xi_k}_{k\in\N}$ \`a
valeurs dans $\N$ et de carr\'e int\'egrables. Soit 
\[
X_t = \sum_{k=1}^{N_t} \xi_k\;.
\]

\begin{enum}
\item	Calculer $\expec{X_t}$.
\item	Calculer $\Variance(X_t)$.  
\end{enum}
\end{exercice}


\chapter{Martingales}
\label{chap_mart}

Le nom \defwd{martingale}\/ est synonyme de \defwd{jeu \'equitable}\/,
c'est-\`a-dire d'un jeu o\`u le gain que l'on peut esp\'erer faire en tout
temps ult\'erieur est \'egal \`a la somme gagn\'ee au moment pr\'esent. En
probabilit\'es, on appelle donc martingale un processus stochastique
$\set{X_n}_n$ tel que l'esp\'erance conditionnelle $\econd{X_m}{X_n}$ est
\'egale \`a $X_n$ pour tout $m\geqs n$. Les martingales, ainsi que leurs
variantes les sous-martingales et les surmartingales, jouissent de nombreuses
propri\'et\'es qui les rendent tr\`es utiles dans l'\'etude de processus
stochastiques plus g\'en\'eraux. Nous allons voir quelques-unes de ces
propri\'et\'es dans ce chapitre et les suivants. 


\section{D\'efinitions et exemples}
\label{sec_mdef}

\begin{definition}[Martingale, sous- et surmartingale]
\label{def_martingale}
Soit $(\Omega,\cF,\set{\cF_n}_n,\fP)$ un espace probabilis\'e filtr\'e. Une
\defwd{martingale}\/ par rapport \`a la filtration $\set{\cF_n}_n$ est un
processus stochastique $\set{X_n}_{n\in\N}$ tel que 
\begin{enum}
\item	$\expec{\abs{X_n}}<\infty$ pour tout $n\in\N$;
\item	$\set{X_n}_n$ est adapt\'e \`a la filtration $\set{\cF_n}_n$;
\item	$\econd{X_{n+1}}{\cF_n} = X_n$ pour tout $n\in\N$.
\end{enum}
Si la derni\`ere condition est rempla\c c\'ee par $\econd{X_{n+1}}{\cF_n}
\leqs X_n$ on dit que $\set{X_n}_n$ est une\/ \defwd{surmartingale}, et si elle
est rempla\c c\'ee par $\econd{X_{n+1}}{\cF_n} \geqs X_n$ on dit que c'est
une\/ \defwd{sous-martingale}\/.
\end{definition}

On notera qu'en termes de jeu, une surmartingale est d\'efavorable au joueur,
alors qu'une sous-martingale lui est favorable (la terminologie vient de la
notion de fonction sous-harmonique).

\begin{example}\hfill
\label{ex_martingale} 
\begin{enum}
\item	Soit $\set{X_n}_n$ la marche al\'eatoire sym\'etrique sur $\Z$, et soit
$\cF_n=\sigma(X_0,\dots,X_n)$ la filtration canonique. On remarque que
$X_{n+1}-X_n$ est ind\'ependant de $\cF_n$, ce qui permet d'\'ecrire 
\begin{equation}
 \label{mart1}
\econd{X_{n+1}}{\cF_n} =  
\econd{X_n}{\cF_n} + \econd{X_{n+1}-X_n}{\cF_n} =  
X_n + \expec{X_{n+1}-X_n} = X_n\;,
\end{equation} 
montrant que la marche al\'eatoire sym\'etrique est une martingale. Cela
traduit le fait que si l'on sait o\`u se trouve la marche au temps $n$, alors
la meilleure estimation de sa position aux temps ult\'erieurs est donn\'ee par
celle au temps $n$.

\item	Soit $\cX\subset\R$ un ensemble d\'enombrable, et soit $\set{X_n}_n$
une \chaine\ de Markov sur $\cX$, de matrice de transition
$P=(p_{i,j})_{i,j\in\cX}$. Alors l'\'equation~\eqref{ec7} implique 
\begin{equation}
 \label{mart2}
\econd{X_{n+1}}{\cF_n} = \sum_{j\in\cX} j p_{X_n,j}\;.
\end{equation} 
Par cons\'equent la \chaine\ est une martingale si $\sum_j jp_{i,j}=i$ pour
tout $i\in\cX$, une surmartingale si $\sum_j jp_{i,j}\leqs i$ pour
tout $i\in\cX$, et une sous-martingale si $\sum_j jp_{i,j}\geqs i$ pour
tout $i\in\cX$. Comme $\sum_j p_{i,j}=1$, la condition d'\^etre une martingale
peut aussi s'\'ecrire $\sum_j (j-i)p_{i,j}=0$ (respectivement $\leqs 0$ pour
une surmartingale et $\geqs 0$ pour une sous-martingale).

\item	{\bf Urne de Polya~:} On consid\`ere une urne contenant $r$ boules
rouges et $v$ boules vertes. De mani\`ere r\'ep\'et\'ee, on tire une boule de
l'urne, puis on la remet en ajoutant un nombre fix\'e $c$ de boules de la
m\^eme couleur. Soit $r_n$ le nombre de boules rouges apr\`es le $n$i\`eme
tirage, $v_n$ le nombre de boules vertes, et $X_n=r_n/(r_n+v_n)$ la proportion
de boules rouges. On aura donc 
\begin{equation}
 \label{mart2b}
X_{n+1} = 
\begin{cases}
\dfrac{r_n+c}{r_n+v_n+c} & \text{avec probabilit\'e
$\dfrac{r_n}{r_n+v_n}=X_n$\;,}\\
\\
\dfrac{r_n}{r_n+v_n+c} & \text{avec probabilit\'e
$\dfrac{v_n}{r_n+v_n}=1-X_n$\;,}
\end{cases} 
\end{equation} 
de sorte que 
\begin{equation}
 \label{mart2c}
\econd{X_{n+1}}{X_n} = 
 \frac{r_n+c}{r_n+v_n+c}  \, \dfrac{r_n}{r_n+v_n} + 
 \frac{r_n}{r_n+v_n+c} \, \dfrac{v_n}{r_n+v_n} 
= \frac{r_n}{r_n+v_n} = X_n\;.
\end{equation} 
La suite $\set{X_n}_n$ est donc une martingale (par rapport \`a la filtration
canonique).
\end{enum}
\end{example}


\section{In\'egalit\'es}
\label{sec_min}

Commen\c cons par v\'erifier que le fait de ne faire intervenir que des temps
cons\'ecutifs $n$ et $n+1$ dans la d\'efinition n'est pas une restriction, mais
que la monotonie s'\'etend \`a tous les temps.

\begin{prop}
\label{prop_mart1}
Si $\set{X_n}_n$ est une surmartingale (respectivement une sous-martingale,
une martingale), alors 
\begin{equation}
 \label{mart3}
\econd{X_n}{\cF_m} \leqs X_m 
\qquad \forall n>m\geqs0
\end{equation}  
(respectivement $\econd{X_n}{\cF_m} \geqs X_m$, $\econd{X_n}{\cF_m} = X_m$). 
\end{prop}
\begin{proof}
Consid\'erons le premier cas. Le r\'esultat suit de la d\'efinition si $n=m+1$.
Si $n=m+k$ avec $k\geqs2$, alors par la Proposition~\ref{prop_econd1},
\begin{equation}
 \label{mart3:1}
\econd{X_{m+k}}{\cF_m} 
= \econd{\econd{X_{m+k}}{\cF_{m+k-1}}}{\cF_m} 
\leqs \econd{X_{m+k-1}}{\cF_m}\;.
\end{equation} 
Le r\'esultat suit alors par r\'ecurrence. Si $\set{X_n}_n$ est une
sous-martingale, il suffit d'observer que $\set{-X_n}_n$ est une surmartingale
et d'appliquer la lin\'earit\'e. Enfin, si $\set{X_n}_n$ est une martingale,
alors c'est \`a la fois une surmartingale et une sous-martingale, d'o\`u la
conclusion (ce raisonnement en trois pas est typique pour les martingales). 
\end{proof}

Un deuxi\`eme type important d'in\'egalit\'e s'obtient en appliquant une
fonction convexe \`a un processus. Nous \'enoncerons simplement les
r\'esultats dans l'un des trois cas sur- sous- ou martingale, mais souvent
d'autres in\'egalit\'es s'obtiennent dans les autres cas par lin\'earit\'e.

\begin{prop}\hfill
\label{prop_mart2} 
\begin{enum}
\item	Soit\/ $X_n$ une martingale par rapport \`a $\cF_n$, et soit $\ph$ une
fonction convexe telle que\/ $\expec{\abs{\ph(X_n)}}<\infty$  pour tout $n$.
Alors\/ $\ph(X_n)$ est une sous-martingale par rapport \`a\/~$\cF_n$. 
\item	Soit\/ $X_n$ une sous-martingale par rapport \`a\/ $\cF_n$ et $\ph$ une
fonction convexe croissante telle que\/ $\expec{\abs{\ph(X_n)}}<\infty$ pour
tout $n$. Alors\/ $\ph(X_n)$ est une sous-martingale par rapport \`a\/ $\cF_n$. 
\end{enum}
\end{prop}
\begin{proof}\hfill
\begin{enum}
\item	Par l'in\'egalit\'e de Jensen, 
$\econd{\ph(X_{n+1})}{\cF_n} \geqs \ph(\econd{X_{n+1}}{\cF_n}) = \ph(X_n)$. 
\item	Par l'in\'egalit\'e de Jensen, 
$\econd{\ph(X_{n+1})}{\cF_n} \geqs \ph(\econd{X_{n+1}}{\cF_n}) \geqs \ph(X_n)$. 
\qed
\end{enum}
\renewcommand{\qed}{}
\end{proof}

Plusieurs cas particuliers de fonctions $\ph$ joueront un r\^ole dans la suite.
Pour deux nombres r\'eels $a$ et $b$, nous notons $a\wedge b$ leur minimum et
$a\vee b$ leur maximum. De plus, $a^+=a\vee0$ d\'enote la partie positive de
$a$ et $a^-=(-a)\vee 0$ sa partie n\'egative.

\begin{cor} \hfill
\label{cor_mart}
\begin{enum}
\item	Si\/ $p\geqs1$ et\/ $X_n$ est une martingale telle que\/
$\expec{\abs{X_n}^p}<\infty$ pour tout\/ $n$, alors\/ $\abs{X_n}^p$ est une
sous-martingale.
\item	Si $X_n$ est une sous-martingale, alors $(X_n-a)^+$ est
une sous-martingale.
\item	Si $X_n$ est une surmartingale, alors $X_n\wedge a$ est une
surmartingale.
\end{enum}
\end{cor}


\section{D\'ecomposition de Doob, processus croissant}
\label{sec_mpc}

\begin{definition}[Processus pr\'evisible]
\label{def_previsible}
Soit\/ $\set{\cF_n}_{n\geqs0}$ une filtration. Une suite\/
$\set{H_n}_{n\geqs0}$ est un\/ \defwd{processus pr\'evisible} si
$H_n\measurable\cF_{n-1}$ pour tout $n\geqs1$.
\end{definition}

La notion de pr\'evisibilit\'e se comprend facilement dans le cadre de la
th\'eorie des jeux (et par cons\'equent \'egalement dans celle des
march\'es financiers, ce qui est essentiellement pareil). Supposons en effet
qu'un joueur mise de mani\`ere r\'ep\'et\'ee sur le r\'esultat d'une
exp\'erience al\'eatoire, telle que le jet d'une pi\`ece de monnaie. Une
strat\'egie est une mani\`ere de d\'ecider la somme mis\'ee \`a chaque tour, en
fonction des gains pr\'ec\'edents. Par exemple, le joueur peut d\'ecider de
doubler la mise \`a chaque tour, ou de miser une proportion fix\'ee de la somme
qu'il a gagn\'ee. Toute strat\'egie doit \^etre pr\'evisible, car elle ne peut
pas d\'ependre de r\'esultats futurs du jeu. De m\^eme, un investisseur d\'ecide
de la mani\`ere de placer ses capitaux en fonction de l'information disponible
au temps pr\'esent, \`a moins de commettre une d\'elit d'initi\'e.

Consid\'erons le cas particulier o\`u le joueur gagne un euro pour chaque euro
mis\'e si la pi\`ece tombe sur Pile, et perd chaque euro mis\'e si elle tombe
sur Face. Soit $X_n$ la somme totale qu'aurait gagn\'ee au temps $n$ un joueur
misant un Euro \`a chaque coup. Un joueur suivant la strat\'egie $H$ aura alors
gagn\'e au temps $n$ la somme 
\begin{equation}
 \label{mpc1}
(H\cdot X)_n \defby \sum_{m=1}^n H_m (X_m - X_{m-1})\;, 
\end{equation} 
puisque $X_m-X_{m-1}$ vaut $1$ ou $-1$ selon que la pi\`ece est tomb\'ee sur
Pile au Face lors du $n$i\`eme jet. Le r\'esultat suivant affirme qu'il
n'existe pas de strat\'egie gagnante dans ce jeu. 

\begin{prop}
\label{prop_mpc1}
Soit $\set{X_n}_{n\geqs0}$ une surmartingale et $H_n$ un processus
pr\'evisible, non-n\'egatif et born\'e pour tout $n$. Alors $(H\cdot X)_n$ est
une surmartingale.
\end{prop}
\begin{proof}
Comme $H_{n+1}\measurable\cF_n$ et $(H\cdot X)_n\measurable\cF_n$, la
lin\'earit\'e de l'esp\'erance conditionnelle implique 
\begin{align}
\nonumber
\econd{(H\cdot X)_{n+1}}{\cF_n} 
&= (H\cdot X)_n + \econd{H_{n+1}(X_{n+1}-X_n)}{\cF_n} \\
&= (H\cdot X)_n + H_{n+1} \econd{X_{n+1}-X_n}{\cF_n} 
\leqs (H\cdot X)_n\;,
\label{mpc2} 
\end{align}
puisque $H_{n+1}\geqs 0$ et $\econd{X_{n+1}-X_n}{\cF_n} =
\econd{X_{n+1}}{\cF_n}-X_n\leqs 0$.
\end{proof}

La d\'ecomposition de Doob permet d'\'ecrire une sous-martingale comme la
somme d'une martingale et d'un processus pr\'evisible. 

\begin{prop}[D\'ecomposition de Doob]
\label{prop_decDoob}  
Toute sous-martingale $\set{X_n}_{n\geqs0}$ peut \^etre \'ecrite d'une
mani\`ere unique comme $X_n=M_n+A_n$, o\`u $M_n$ est une martingale et
$A_n$ est un processus pr\'evisible croissant tel que $A_0=0$.  
\end{prop}
\begin{proof}
La preuve est parfaitement constructive. Supposons d'abord que la
d\'ecompo\-sition existe, avec $\econd{M_n}{\cF_{n-1}}=M_{n-1}$ et
$A_n\measurable\cF_{n-1}$. Alors n\'ecessairement 
\begin{align}
\nonumber
\econd{X_n}{\cF_{n-1}} 
&= \econd{M_n}{\cF_{n-1}} + \econd{A_n}{\cF_{n-1}} \\
&= M_{n-1} + A_n = X_{n-1} - A_{n-1} + A_n\;.
\label{mpc3} 
\end{align}
Il en r\'esulte les deux relations 
\begin{equation}
 \label{pcm4}
A_n = A_{n-1} + \econd{X_n}{\cF_{n-1}} - X_{n-1}  
\end{equation} 
et
\begin{equation}
 \label{pcm5}
M_n = X_n - A_n\;. 
\end{equation} 
Ces deux relations d\'efinissent $M_n$ et $A_n$ univoquement. En effet, $A_0=0$
et donc $M_0=X_0$ par hypoth\`ese, de sorte que tous les $M_n$ et $A_n$ sont
d\'efinis par r\'ecurrence. Le fait que $X_n$ est une sous-martingale implique
que $A_n\geqs A_{n-1}\geqs 0$. Par r\'ecurrence, $A_n\measurable\cF_{n-1}$.
Enfin, 
\begin{align}
\nonumber
\econd{M_n}{\cF_{n-1}}
&= \econd{X_n-A_n}{\cF_{n-1}} \\
&= \econd{X_n}{\cF_{n-1}} - A_n = X_{n-1} - A_{n-1} = M_{n-1}\;,
\label{pcm6} 
\end{align}
ce qui montre que $M_n$ est bien une martingale.
\end{proof}

Un cas particulier important se pr\'esente si $X_n$ est une martingale
telle que $X_0=0$ et $\expec{X_n^2}<\infty$ pour tout $n$. Dans ce cas, le
Corollaire~\ref{cor_mart} implique que $X_n^2$ est une sous-martingale. La
d\'ecomposition de Doob de $X_n^2$ s'\'ecrit alors 
\begin{equation}
 \label{pcm7}
X_n^2 = M_n + A_n\;,
\end{equation} 
o\`u $M_n$ est une martingale, et~\eqref{pcm4} implique 
\begin{equation}
 \label{pcm8}
A_n = \sum_{m=1}^n \econd{X_m^2}{\cF_{m-1}} - X_{m-1}^2 
=  \sum_{m=1}^n \econd{(X_m-X_{m-1})^2}{\cF_{m-1}}\;.
\end{equation} 

\begin{definition}[Processus croissant]
\label{def_crochet} 
Le processus pr\'evisible~\eqref{pcm8} est appel\'e le\/ \defwd{processus
croissant} ou le\/ \defwd{crochet}, ou encore le\/ \defwd{compensateur}  de
$X_n$, et not\'e\/ $\braket{X}_n$.
\end{definition}

Le processus croissant $\braket{X}_n$ est une mesure de la variance de la
trajectoire jusqu'au temps $n$. Nous verrons que sous certaines conditions,
$\braket{X}_n$ converge vers une variable al\'eatoire $\braket{X}_\infty$, qui
mesure alors la variance totale des trajectoires. 


\section{Exercices}
\label{sec_exo_mart}

\begin{exercice}
\label{exo_mart1} 
Soient $Y_1, Y_2, \dots$ des variables i.i.d., et soit $X_n=\prod_{m=1}^n Y_m$.
Sous quelle condition la suite $X_n$ est-elle une surmartingale? Une
sous-martingale? Une martingale? 
\end{exercice}

\goodbreak

\begin{exercice}
\label{exo_mart3} 
Soit $(\Omega,\cF,\set{\cF_n},\fP)$ un espace probabilis\'e filtr\'e, et
$X_n=\sum_{m=1}^n\indicator{B_m}$, avec $B_n\in\cF_n$ $\forall n$. 
\begin{enum}
\item	Montrer que $X_n$ est une sous-martingale.
\item	Donner la d\'ecomposition de Doob de $X_n$. 
\item	Particulariser au cas $\cF_n=\sigma(X_1,\dots,X_n)$.
\end{enum}
\end{exercice}

\goodbreak

\begin{exercice}[Urne de Polya]
\label{exo_mart4} 
On consid\`ere le mod\`ele d'urne de Polya, avec param\`etres $(r,v,c)$. 
\begin{enum}
\item	D\'eterminer la loi de la proportion de boules vertes $X_n$ dans le cas
$r=v=c=1$. 
\item	Dans le cas g\'en\'eral, exprimer le processus croissant
$\braket{X}_n$ en fonction des $X_m$ et du nombre total de boules $N_m$ aux
temps $m\leqs n$. 
\item	Montrer que $\lim_{n\to\infty}\braket{X}_n<\infty$. 
\end{enum}
\end{exercice}

\goodbreak

\begin{exercice}[Martingales et fonctions harmoniques]
\label{exo_mart6} 
Une fonction continue $f: \C\to\R$ est dite \emph{sous-harmonique}\/ si 
pour tout $z\in\C$ et tout $r>0$, 
\[
f(z) \leqs \frac{1}{2\pi} \int_0^{2\pi} f(z+r\e^{\icx\theta})\6\theta\;.
\]
$f$ est dite \emph{surharmonique}\/ si $-f$ est sous-harmonique, et
\emph{harmonique}\/ si elle est \`a la fois sous-harmonique et surharmonique. 

On fixe $r>0$. Soit $\set{U_n}_{n\geqs1}$ une suite de variables al\'eatoires
i.i.d.\ de loi uniforme sur $\setsuch{z\in\C}{\abs{z}=r}$, et soit 
$\cF_n$ la tribu engendr\'ee par $(U_1,\dots,U_n)$. Pour tout $n\geqs1$ on pose 
\[
X_n = U_1 + \dots + U_n
\qquad\text{et}\qquad
Y_n=f(X_n)\;.
\]

\begin{enum}
\item 	Sous quelle condition sur $f$ la suite $\set{Y_n}_{n\geqs1}$ est-elle
une sous-martingale, une surmartingale, ou une martingale?

\item 	Que se passe-t-il si $f$ est la partie r\'eelle d'une fonction
analytique? 
\end{enum}
\end{exercice}

\goodbreak

\begin{exercice}[Le processus de Galton--Watson]
\label{exo_mart5} 
On se donne des variables al\'eatoires i.i.d.\ $\set{\xi_{n,i}}_{n,i\geqs1}$
\`a valeurs dans $\N$. On note leur distribution $p_k=\prob{\xi_{n,i}=k}$,
leur esp\'erance $\mu>0$ et leur variance $\sigma^2$. 
On notera $\cF_n$ la filtration $\cF_n=\sigma\set{\xi_{i,m},m\leqs n}$. 

\noindent
On d\'efinit un processus $\set{Z_n}_{n\geqs0}$ par $Z_0=1$ et pour $n\geqs0$
\[
Z_{n+1} = 
\begin{cases}
\xi_{1,n+1} + \dots + \xi_{Z_n,n+1} & \text{si $Z_n>0$\;,} \\
0 & \text{sinon\;.}
\end{cases}
\]
Ce processus mod\'elise l'\'evolution d'une population avec initialement $Z_0=1$
individu, et dans laquelle chaque individu $i$ donne naissance au temps 
$n$ \`a un nombre al\'eatoire $\xi_{n,i}$ d'enfants, ind\'ependamment et avec
la m\^eme loi que tous les autres individus. 

\begin{enum}
\item	Montrer que $X_n=Z_n/\mu^n$ est une martingale par rapport \`a $\cF_n$.

\item	Montrer que si $\mu<1$, alors $\prob{Z_n=0}\to 1$ lorsque $n\to\infty$,
et donc $X_n\to0$. 

\item	On suppose $\mu>1$. Soit $\varphi(s)=\expec{s^{\xi_{i,n}}} =
\sum_{k=0}^\infty p_k s^k$ la fonction g\'en\'eratrice de la distribution
d'enfants.
\begin{enum}
\item	Montrer que $\varphi$ est croissante et convexe sur $[0,1]$.
\item	Soit $\theta_m=\prob{Z_m=0}$. Montrer que
$\pcond{Z_m=0}{Z_1=k}=\theta_{m-1}^k$ et en d\'eduire que
$\theta_m=\varphi(\theta_{m-1})$. 
\item	Montrer que $\varphi$ admet un unique point fixe $\rho$ sur $[0,1)$.
\item	Montrer que $\theta_m\nearrow\rho$ lorsque $m\to\infty$.
\end{enum}
En d\'eduire que $\prob{Z_n>0\;\forall n}=1-\rho>0$.

\item	Galton et Watson on introduit leur mod\`ele afin de d\'ecrire la survie
de noms de famille. Au XVIIIe si\`ecle, ces noms n'\'etaient transmis que par
les enfants de sexe masculin. On suppose que chaque famille a trois enfants,
dont le sexe est d\'etermin\'e par une loi de Bernoulli de param\`etre $1/2$.
Le nombre de descendants m\^ales est donc d\'ecrit par un processus de
Galton--Watson de loi binomiale $p_0=p_3=1/8$, $p_1=p_2=3/8$. D\'eterminer la
probabilit\'e de survie du nom de famille. 

\item	D\'eterminer le processus croissant $\braket{X}_n$ de $X_n$. Calculer
$\lim_{n\to\infty}\braket{X}_n$.  
\end{enum}
\end{exercice}


\chapter{Temps d'arr\^et}
\label{chap_arret}

Dans un jeu de hasard, un temps d'arr\^et est un temps lors duquel le joueur
d\'ecide d'arr\^eter de jouer, selon un crit\`ere ne d\'ependant que du pass\'e
et du pr\'esent. Il peut par exemple d\'ecider d'arr\^eter de jouer d\`es qu'il
a d\'epens\'e tout son capital, d\`es qu'il a gagn\'e une certaine somme, d\`es
qu'il a gagn\'e un certain nombre de fois successives, ou selon toute
combinaison de ces crit\`eres. Les temps d'arr\^et ont donc deux propri\'et\'es
importantes~: ils sont al\'eatoires, puisqu'ils d\'ependent du d\'eroulement
ant\'erieur du jeu, et ils ne peuvent pas d\'ependre du futur, puisque le
joueur doit \`a tout moment pouvoir d\'ecider s'il arr\^ete ou non.


\section{D\'efinition et exemples}
\label{sec_arretdef}

\begin{definition}[Temps d'arr\^et]
\label{def_tarret}  
Soit\/ $\set{X_n}_{n\in\N}$ un processus stochastique et soit\/ 
$\cF_n=\sigma(X_0,\dots,X_n)$ la filtration canonique. Une variable
al\'eatoire\/ $N$ \`a valeurs dans\/ $\overline{\N\!}=\N\cup\set{\infty}$ est un
\defwd{temps d'arr\^et} si\/ 
$\set{N=n}\in\cF_n$ pour tout\/ $n\in\N$. 
\end{definition}

La condition $\set{N=n}\in\cF_n$ signifie qu'avec l'information disponible au
temps $n$, on doit pouvoir d\'ecider si oui ou non l'\'ev\'enement $\set{N=n}$
est r\'ealis\'e. En d'autres termes, un temps d'arr\^et ne \lq\lq peut pas voir
dans le futur\rq\rq.

\begin{example}\hfill
\label{ex_tarret1} 
\begin{enum}
\item	Si $N$ est constante presque s\^urement, alors c'est un temps d'arr\^et.
\item	Supposons que $X_n$ soit r\'eel, et soit $B\subset\R$ un bor\'elien.
Alors le \defwd{temps de premi\`ere atteinte}\/ de $B$
\begin{equation}
 \label{tarret1}
N_B = \inf\setsuch{n\in\N}{X_n\in B} 
\end{equation}  
est un temps d'arr\^et (par convention, on pose $N_B=\infty$ si le processus
n'atteint jamais l'ensemble $B$). Pour le v\'erifier, il suffit d'observer qu'on
a la d\'ecomposition 
\begin{equation}
 \label{tarret2}
\set{N_B=n} = \bigcap_{k=0}^{n-1} \set{X_k\in B^c} \cap \set{X_n\in B} \in
\cF_n\;. 
\end{equation} 
\item	Le temps de dernier passage dans $B$ avant un temps fix\'e $m$, 
\begin{equation}
 \label{tarret3}
\sup\setsuch{n \leqs m}{X_n\in B}
\end{equation} 
n'est pas un temps d'arr\^et. En effet, si par exemple $X_n\in B$ pour un
$n<m$, l'information disponible au temps $n$ ne permet pas de dire si oui ou
non le processus repassera dans l'ensemble $B$ jusqu'au temps $m$. Un joueur ne
peut pas d\'ecider qu'il jouera jusqu'au dernier tour avant le centi\`eme lors
duquel il poss\'edera au moins 100 Euros! 
\end{enum}
\end{example}

\begin{prop}\hfill
\label{prop_tarret1}
\begin{enum}
\item	Soient $N$ et $M$ des temps d'arr\^et. Alors $N\wedge M$ et $N\vee M$
sont des temps d'arr\^et. 
\item	Soit $N_k$, $k\in\N$, une suite de temps d'arr\^et telle que
$N_k\nearrow N$. Alors $N$ est un temps d'arr\^et.
\end{enum}
\end{prop}

Nous laissons la preuve en exercice. 

\begin{definition}[Tribu des \'ev\'enements ant\'erieurs]
\label{def_tribu_anterieure} 
Soit $N$ un temps d'arr\^et. Alors la tribu 
\begin{equation}
 \label{tarret4}
\cF_N = \setsuch{A\in\cF}{A\cap\set{N=n}\in\cF_n \:\forall n<\infty} 
\end{equation} 
est appel\'ee la\/ \defwd{tribu des \'ev\'enements ant\'erieurs \`a $N$}.
\end{definition}

\begin{prop}
\label{prop_tarret2}
Soient $N$ et $M$ deux temps d'arr\^et tels que $M\leqs N$. Alors
$\cF_M\subset\cF_N$. 
\end{prop}
\begin{proof}
Soit $A\in \cF_M$. Alors pour tout $n\geqs 0$, $N=n$ implique $M\leqs n$ d'o\`u 
\begin{equation}
 \label{tarret5:1}
A \cap \set{N=n} = A \cap \set{M\leqs n} \cap \set{N=n} \in \cF_n\;, 
\end{equation} 
puisque $A \cap \set{M\leqs n} = \bigcup_{k=0}^n (A\cap\set{M=k})\in\cF_n$ et
$\set{N=n}\in\cF_n$.
\end{proof}


\section{Processus arr\^et\'e, in\'egalit\'e de Doob}
\label{sec_indoob}

\begin{definition}[Processus arr\^et\'e]
\label{def_proc_arr}
Soit $\set{\cF_n}_n$ une filtration, $\set{X_n}_n$ un processus adapt\'e \`a la
filtration et $N$ un temps d'arr\^et. On appelle\/ \defwd{processus arr\^et\'e
en $N$} le processus $X^N$ d\'efini par  
\begin{equation}
 X^N_n = X_{N\wedge n}\;.
\end{equation}  
\end{definition}

Si par exemple $N$ est le temps de premi\`ere atteinte d'un ensemble $B$, alors
$X^N$ est le processus obtenu en \lq\lq gelant\rq\rq\ $X$ \`a l'endroit o\`u il
atteint $B$ pour la premi\`ere fois. 

\begin{prop}
\label{prop_indoob1}
Si $N$ est un temps d'arr\^et et $X_n$ est une surmartingale (par rapport \`a
la m\^eme filtration $\cF_n$), alors le processus arr\^et\'e $X_{N\wedge n}$
est une surmartingale. 
\end{prop}
\begin{proof}
Consid\'erons le processus $H_n=\indexfct{N\geqs n}$. Puisque l'on a
$\set{N\geqs n}=\set{N\leqs n-1}^c\in\cF_{n-1}$, le processus $H_n$ est
pr\'evisible. De plus, il est \'evidemment born\'e et non-n\'egatif. Par la
proposition~\ref{prop_mpc1}, $(H\cdot X)_n$ est une surmartingale. Or nous
avons 
\begin{equation}
 \label{indoob1}
(H\cdot X)_n 
= \sum_{m=1}^n \indexfct{N\geqs m} (X_m - X_{m-1})
= \sum_{m=1}^{N\wedge n} (X_m - X_{m-1})
= X_{N\wedge n}- X_0\;,
\end{equation} 
donc $X_{N\wedge n} = X_0 + (H\cdot X)_n$ est une surmartingale.
\end{proof}

Il suit directement de ce r\'esultat que si $X_n$ est une sous-martingale
(respectivement une martingale), alors $X_{N\wedge n}$ est une sous-martingale
(respectivement une martingale). 

\begin{cor}
\label{cor_indoob}
Soit $X_n$ une sous-martingale et soit $N$ un temps d'arr\^et satisfaisant 
\mbox{$\prob{N\leqs k}=1$} pour un $k\in\N$. Alors 
\begin{equation}
 \label{indoob2}
\expec{X_0} \leqs \expec{X_N} \leqs \expec{X_k}\;. 
\end{equation} 
\end{cor}
\begin{proof}
Comme $X_{N\wedge n}$ est une sous-martingale, on a 
\begin{equation}
 \label{indoob2:1}
\expec{X_0} = \expec{X_{N\wedge0}} \leqs \expec{X_{N\wedge k}} = \expec{X_N}\;. 
\end{equation} 
Pour prouver la seconde in\'egalit\'e, nous imitons la preuve de la proposition
en utilisant $K_n=1-H_n$, c'est-\`a-dire $K_n=\indexfct{N<n}=\indexfct{N\leqs
n-1}$. C'est \'egalement un processus pr\'evisible, donc $(K\cdot X)_n$ est une
sous-martingale (car $-X_n$ est une surmartingale). Or 
\begin{equation}
 \label{indoob2:2}
(K\cdot X)_n 
= \sum_{m=1}^n \indexfct{N\leqs m-1} (X_m - X_{m-1})
= \sum_{m=N+1}^{n} (X_m - X_{m-1})
= X_n- X_{N\wedge n}\;,
\end{equation} 
donc, comme $N\wedge k=N$ presque s\^urement,
\begin{equation}
 \label{indoob2:3}
\expec{X_k} - \expec{X_N} = \expec{(K\cdot X)_k} 
\geqs \expec{(K\cdot X)_0} = 0\;,
\end{equation} 
d'o\`u la seconde in\'egalit\'e.
\end{proof}

Nous sommes maintenant en mesure de prouver l'in\'egalit\'e de Doob, qui est
tr\`es utile, notamment pour estimer le supremum d'un processus stochastique. 
Dans la suite, nous posons 
\begin{equation}
 \label{indoob3}
\Xbar_n = \max_{0\leqs m\leqs n} X^+_m 
= \max_{0\leqs m\leqs n} X_m \vee 0\;.
\end{equation} 

\begin{theorem}[In\'egalit\'e de Doob]
\label{thm_Doob}  
Soit $X_n$ une sous-martingale. Alors pour tout $\lambda>0$, on a 
\begin{equation}
 \label{indoob4}
\bigprob{\Xbar_n \geqs \lambda} 
\leqs \frac1\lambda \bigexpec{X_n\indexfct{\Xbar_n \geqs \lambda}}
\leqs \frac1\lambda \bigexpec{X^+_n}\;. 
\end{equation} 
\end{theorem}
\begin{proof}
Soit $A=\set{\Xbar_n \geqs \lambda}$ et
$N=\inf\setsuch{m}{X_m\geqs\lambda}\wedge n$. Alors 
\begin{equation}
 \label{indoob4:1}
\lambda\fP(A) = \expec{\lambda\indicator{A}} 
\leqs \expec{X_N\indicator{A}}
\leqs \expec{X_n\indicator{A}}
\leqs \expec{X^+_n}\;.
\end{equation} 
La premi\`ere in\'egalit\'e suit du fait que $X_N\geqs\lambda$ dans $A$.
La seconde vient du fait que $X_N=X_n$ sur $A^c$ et que
$\expec{X_N}\leqs\expec{X_n}$ par le corollaire~\ref{cor_indoob}. La derni\`ere
in\'egalit\'e est triviale. 
\end{proof}

L'int\'er\^et de cette in\'egalit\'e est qu'elle permet de majorer une
quantit\'e faisant intervenir tout le processus jusqu'au temps $n$ par une
quantit\'e ne d\'ependant que de $X_n$, qui est souvent beaucoup plus simple
\`a estimer.

\begin{example}
\label{ex_indoob1} 
Soient $\xi_1, \xi_2, \dots$ des variables ind\'ependantes d'esp\'erance nulle
et variance finie, et soit $X_n=\sum_{m=1}^n\xi_m$. Alors $X_n$ est une
martingale (la preuve est la m\^eme que pour la marche al\'eatoire sym\'etrique
sur $\Z$ dans l'exemple~\ref{ex_martingale}), donc $X_n^2$ est une
sous-martingale. Une application de l'in\'egalit\'e de Doob donne
l'\defwd{in\'egalit\'e de Kolmogorov}
\begin{equation}
 \label{indoob5}
\Bigprob{\max_{1\leqs m\leqs n}\abs{X_m}\geqs\lambda} 
= \Bigprob{\max_{1\leqs m\leqs n}X^2_m\geqs\lambda^2}
\leqs \frac1{\lambda^2} \bigexpec{X^2_n} 
= \frac1{\lambda^2} \variance(X_n)\;.
\end{equation} 
\end{example}

\goodbreak
Une autre application de l'in\'egalit\'e de Doob est l'\defwd{in\'egalit\'e du
maximum $L^p$}~:

\begin{theorem}
\label{thm_Lpmax}
Si $X_n$ est une sous-martingale, alors pour tout $p>1$ 
\begin{equation}
 \label{indoob6}
\bigexpec{\Xbar_n^p} \leqs \biggpar{\frac{p}{p-1}}^p
\bigexpec{\brak{X^+_n}^p}\;. 
\end{equation}  
\end{theorem}
\begin{proof}
Soit $M>0$ une constante, que nous allons faire tendre vers l'infini \`a la fin
de la preuve. Alors, par int\'egration par parties,  
\begin{align}
\nonumber
\bigexpec{\brak{\Xbar_n\wedge M}^p}
&= \int_0^\infty p\lambda^{p-1} \bigprob{\Xbar_n\wedge M \geqs \lambda}
\,\6\lambda \\
\nonumber
&\leqs \int_0^\infty p\lambda^{p-1} \frac1\lambda 
\bigexpec{X_n^+\indexfct{\Xbar_n\wedge
M\geqs\lambda}}
\,\6\lambda \\
\nonumber
&= \int_0^\infty p\lambda^{p-2} 
\int X_n^+ \indexfct{\Xbar_n\wedge
M\geqs\lambda}\,\6\fP
\,\6\lambda \\
\nonumber
&= \int X_n^+ \int_0^{\Xbar_n\wedge M} p\lambda^{p-2}\,\6\lambda\,\6\fP \\
\nonumber
&= \frac{p}{p-1}\int X_n^+ \brak{\Xbar_n\wedge M}^{p-1} \,\6\fP \\
\nonumber
&= \frac{p}{p-1}\bigexpec{ X_n^+ \brak{\Xbar_n\wedge M}^{p-1} } \\
&\leqs \frac{p}{p-1} \bigbrak{\expec{\abs{X_n^+}^p}}^{1/p} 
\bigbrak{\expec{\abs{\Xbar_n\wedge M}^p}}^{(p-1)/p}\;. 
\label{indoob6:1} 
\end{align}
La derni\`ere in\'egalit\'e provient de l'in\'egalit\'e de H\"older. Divisant
les deux c\^ot\'es de l'in\'egalit\'e par $\brak{\expec{\abs{\Xbar_n\wedge
M}^p}}^{(p-1)/p}$ et \'elevant \`a la puissance $p$, on trouve 
\begin{equation}
 \label{indoob6:2}
\bigexpec{\abs{\Xbar_n\wedge M}^p} \leqs \biggpar{\frac{p}{p-1}}^p
\bigexpec{\brak{X^+_n}^p}\;,
\end{equation} 
et le r\'esultat suit du th\'eor\`eme de la convergence domin\'ee, en faisant
tendre $M$ vers l'infini.
\end{proof}

\begin{cor}
\label{cor_indoob2}
Si $Y_n$ est une martingale, alors 
\begin{equation}
 \label{indoob7}
\biggexpec{\Bigbrak{\max_{0\leqs m\leqs n}\abs{Y_m}}^p}
\leqs \biggpar{\frac{p}{p-1}}^p \bigexpec{\abs{Y_n}^p}\;.
\end{equation}  
\end{cor}


\section{Exercices}
\label{sec_exo_doob}

\begin{exercice}
\label{exo_doob1} 
D\'emontrer la Proposition~\ref{prop_tarret1}~:
\begin{enum}
\item	Si $N$ et $M$ sont des temps d'arr\^et, alors $N\wedge M$ et $N\vee M$
sont des temps d'arr\^et. 
\item	Si $N_k$, $k\in\N$, est une suite de temps d'arr\^et telle que
$N_k\nearrow N$, alors $N$ est un temps d'arr\^et.
\end{enum}
\end{exercice}

\begin{exercice}
\label{exo_doob2}
Soient $\xi_1, \xi_2, \dots$ des variables al\'eatoires i.i.d.\ telles que
$\expec{\xi_m}=0$, et soit la martingale $X_n=\sum_{m=1}^n\xi_m$. On se donne
$\lambda>0$, et soit 
\[
 P_n(\lambda) = \Bigprob{\max_{1\leqs m\leqs n}\abs{X_m}\geqs\lambda}\;.
\]
\begin{enum}
\item	On suppose les $\xi_m$ de variance finie.
Donner une majoration de $P_n(\lambda)$ en appliquant
l'in\'egalit\'e de Doob \`a $X_n^2$. 
\item	Am\'eliorer la borne pr\'ec\'edente en appliquant l'in\'egalit\'e de
Doob \`a $(X_n+c)^2$ et en optimisant sur $c$.  
\item	On suppose que les $\xi_m$ suivent une loi normale centr\'ee r\'eduite. 
Majorer $P_n(\lambda)$ en appliquant l'in\'egalit\'e de Doob \`a
$\e^{cX_n^2}$ et en optimisant sur $c$. 
\item	Pour des $\xi_m$ normales centr\'ees r\'eduites, majorer 
$\prob{\max_{1\leqs m\leqs n}X_m\geqs\lambda}$ en appliquant l'in\'egalit\'e de
Doob \`a $\e^{cX_n}$ et en optimisant sur $c$.
\end{enum}
\end{exercice}


\chapter{Th\'eor\`emes de convergence}
\label{chap_conv}

Une propri\'et\'e importante des suites non-d\'ecroissantes de r\'eels est que
si elles sont born\'ees, alors elles convergent. Les sous-martingales sont un
analogue stochastique de telles suites~: si elles sont born\'ees, alors elles
convergent dans un sens appropri\'e. Ceci donne une m\'ethode tr\`es simple
pour \'etudier le comportement asymptotique de certains processus
stochastiques. 

Il faut cependant faire attention au
fait qu'il existe diff\'erentes notions de convergence de suites de variables
al\'eatoires. Dans cette section, nous allons d'abord examiner
la convergence de sous-martingales au sens presque s\^ur, qui est le plus
fort. Ensuite nous verrons ce qui se passe dans le cas plus faible de la
convergence $L^p$. Il s'av\`ere que le cas $p>1$ est relativement simple \`a
contr\^oler, alors que le cas $p=1$ est plus difficile.


\section{Rappel: Notions de convergence}
\label{sec_rapco}

Commen\c cons par rappeler trois des notions de convergence d'une suite de
variables al\'ea\-toires (une quatri\`eme notion, celle de convergence en loi,
ne jouera pas de r\^ole important ici).

\begin{definition}[Convergence d'une suite de variables al\'eatoires]
\label{def_conv} 
Soient\/ $X$ et $\set{X_n}_{n\geqs0}$ des variables al\'eatoires
r\'eelles, d\'efinies sur un m\^eme espace probabilis\'e\/ $(\Omega,\cF,\fP)$. 
\begin{enum}
\item	On dit que\/ $X_n$ \defwd{converge presque s\^urement}\/ vers\/ $X$ si 
\[
\fP\Bigpar{\bigsetsuch{\omega\in\Omega}{\lim_{n\to\infty}X_n(\omega)=X(\omega)}}
= 1\;. 
\]
\item	Si\/ $p>0$, on d\'enote par\/ $L^p(\Omega,\cF,\fP)$ l'ensemble des
variables al\'eatoires\/ $X$ telles que\/ $\expec{\abs{X}^p}<\infty$.
Si\/ $X_n, X\in L^p$, on
dit que\/ $X_n$ \defwd{converge dans $L^p$}\/ vers $X$ si
\[
\lim_{n\to\infty} \bigexpec{\abs{X_n-X}^p} = 0\;.
\]
\item	On dit que\/ $X_n$ \defwd{converge en probabilit\'e}\/ vers\/ $X$ si 
\[
\lim_{n\to\infty} \bigprob{\abs{X_n-X}\geqs\eps} = 0
\]
pour tout $\eps>0$.
\end{enum}
\end{definition}

Les principaux liens entre ces trois notions de convergence sont r\'esum\'ees
dans la proposition suivante.

\begin{prop}\hfill
\label{prop_rapco}
\begin{enum}
\item	La convergence presque s\^ure implique la convergence en probabilit\'e.
\item	La convergence dans\/ $L^p$ implique la convergence en probabilit\'e.
\item	La convergence dans\/ $L^p$ implique la convergence dans $L^q$ pour
tout $q<p$.
\item	Si\/ $X_n\to X$ presque s\^urement et\/ $\abs{X_n}\leqs Y$ $\forall n$
avec\/ $Y\in L^p$, alors\/ $X_n\to X$ dans\/ $L^p$.
\end{enum}
\end{prop}
\begin{proof}\hfill
\begin{enum}
\item	Soit $Y_n=\indexfct{\abs{X_n-X}\geqs\eps}$. Si $X_n\to X$ presque
s\^urement, alors $X_n\to X$ presque partout, donc $Y_n\to 0$ presque
s\^urement. Par le th\'eor\`eme de la convergence domin\'ee (qui s'applique car
$Y_n\leqs 1$), on a $\prob{\abs{X_n-X}\geqs\eps} = \expec{Y_n} \to 0$.

\item	Par l'in\'egalit\'e de Markov, 
$\prob{\abs{X_n-X}\geqs\eps} 
\leqs\eps^{-p}\,\expec{\abs{X_n-X}^p}\to0$.

\item	Si $q<p$, l'in\'egalit\'e de H\"older implique 
$\expec{\abs{X_n-X}^q}\leqs\expec{\abs{X_n-X}^p}^{q/p}\to0$. 

\item	On a $\abs{X_n-X}^p \leqs (\abs{X_n}+\abs{X})^p \leqs (2Y)^p \leqs
2^pY^p\in L^1$. Comme $\abs{X_n-X}^p\to 0$ presque s\^urement, le th\'eor\`eme
de la convergence domin\'ee permet d'\'ecrire 
\begin{equation}
\label{rconv0}
\lim_{n\to\infty} \bigexpec{\abs{X_n-X}^p} 
= \Bigexpec{\lim_{n\to\infty} \abs{X_n-X}^p} = 0\;.
\end{equation} 
\qed
\end{enum}
\renewcommand{\qed}{}
\end{proof}

Voici deux contre-exemples classiques \`a l'\'equivalence entre ces
diff\'erentes notions de convergence~:

\begin{example}
\label{ex_rapco} 
Soit $(\Omega,\cF,\fP)=([0,1],\cB_{[0,1]},\lambda)$, o\`u $\cB_{[0,1]}$ est la
tribu des bor\'eliens et $\lambda$ la mesure de Lebesgue. 
\begin{enum}
\item	Soit $X_n=n^{1/p}\indicator{[0,1/n]}$. Alors $X_n(\omega)\to0$ pour
tout $\omega>0$, donc $X_n\to0$ presque s\^urement. De plus, pour tout $\eps>0$
on a $\prob{\abs{X_n}>\eps}\leqs 1/n$, donc $X_n\to0$ en probabilit\'e. Par
contre, on a $\expec{\abs{X_n}^p} = \int_0^{1/n}(n^{1/p})^p\6x = 1$, donc $X_n$
ne tend pas vers $0$ dans $L^p$.

\item	D\'ecomposons tout entier $n$ comme $n=2^m+k$ avec $m\in\N$ et $0\leqs
k<2^m$, et soit $X_n=\indicator{[k2^{-m},(k+1)2^{-m}]}$. Alors
$\prob{\abs{X_n}>\eps}\leqs 2^{-m}$ pour tout $\eps>0$, donc $X_n\to0$ en
probabilit\'e. De plus, pour tout $p>0$, $\expec{\abs{X_n}^p}\leqs 2^{-m}$, donc
$X_n\to0$ dans $L^p$. Par contre, $X_n(\omega)$ ne converge pour aucun $\omega$
(on peut trouver des sous-suites de $n$ telles que $X_n(\omega)=1$ pour tout
$n$, et aussi telles que $X_n(\omega)=0$ pour tout $n$), donc $X_n$ ne converge
pas presque s\^urement.
\end{enum}
\end{example}


\section{Convergence presque s\^ure}
\label{sec_cops}

Consid\'erons un investisseur suivant la strat\'egie suivante. Il se fixe deux
seuils $a<b$. Il ach\`ete un nombre fix\'e de parts d'action lorsque leur prix
est inf\'erieur \`a $a$. Puis il attend que le prix de l'action atteigne ou
d\'epasse la valeur $b$ pour revendre ses parts. Ensuite, il attend que le prix
retombe au-dessous de $a$ pour racheter des parts, et ainsi de suite. Si $X_n$
d\'esigne le prix de l'action au temps $n$, la strat\'egie peut \^etre d\'ecrite
formellement comme suit. On pose $N_0=-1$ et, pour tout $k\geqs1$, on introduit
les temps d'arr\^et 
\begin{align}
\nonumber
N_{2k-1} &= \inf\setsuch{m>N_{2k-2}}{X_m\leqs a}\;,\\
N_{2k} &= \inf\setsuch{m>N_{2k-1}}{X_m\geqs b}\;.
\label{cops1} 
\end{align}
La strat\'egie correspond \`a la suite pr\'evisible 
\begin{equation}
 \label{cops2}
H_m = 
\begin{cases}
1 & \text{s'il existe $k$ tel que $N_{2k-1}<m\leqs N_{2k}$\;,} \\
0 & \text{sinon\;.}
\end{cases} 
\end{equation} 
Jusqu'au temps $n$, l'investisseur aura revendu ses actions un nombre $U_n$ de
fois, donn\'e par 
\begin{equation}
 \label{cops3}
U_n = \sup\bigsetsuch{k}{N_{2k}\leqs n}\;. 
\end{equation} 
Au temps $n$, le gain de l'investisseur sera donc sup\'erieur ou \'egal \`a
$(b-a)U_n$. Le r\'esultat suivant donne une borne sup\'erieure sur le
gain moyen dans le cas o\`u le processus est favorable ou \'equitable. 

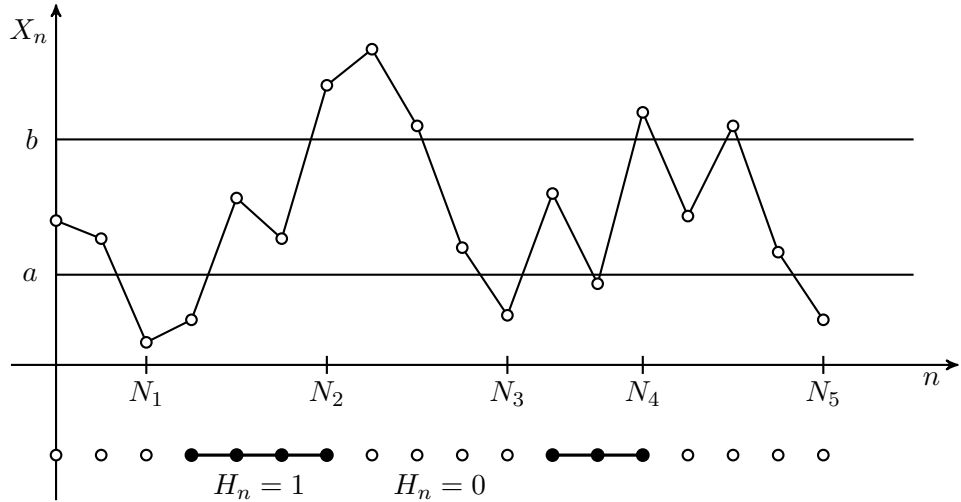
\begin{figure}
\begin{center}
\begin{tikzpicture}[-,auto,scale=0.6,node distance=1.0cm, thick,
main node/.style={draw,circle,fill=white,minimum size=4pt,inner sep=0pt},
full node/.style={draw,circle,fill=black,minimum size=4pt,inner sep=0pt}]

  \path[->,>=stealth'] 
     (-1,0) edge (20,0)
     (0,-3) edge (0,8)
  ;

  \path[-] 
     (0,2) edge (19,2)
     (0,5) edge (19,5)
  ;
  \node at (19.4,-0.3) {$n$};
  \node at (-0.6,7.4) {$X_n$};

  \draw 
  (0,2) node[left=0.1cm] {$a$} 
  -- (0,5) node[left=0.1cm] {$b$} 
   ;
   
  \draw 
  (2,-0.2) -- (2,0.2) node[below=0.2cm] {$N_1$} 
  (6,-0.2) -- (6,0.2) node[below=0.2cm] {$N_2$} 
  (10,-0.2) -- (10,0.2) node[below=0.2cm] {$N_3$} 
  (13,-0.2) -- (13,0.2) node[below=0.2cm] {$N_4$} 
  (17,-0.2) -- (17,0.2) node[below=0.2cm] {$N_5$} 
   ;

   \draw (0,3.2) node[main node] {} 
   -- (1,2.8) node[main node] {}
   -- (2,0.5) node[main node] {}
   -- (3,1) node[main node] {}
   -- (4,3.7) node[main node] {}
   -- (5,2.8) node[main node] {}
   -- (6,6.2) node[main node] {}
   -- (7,7) node[main node] {}
   -- (8,5.3) node[main node] {}
   -- (9,2.6) node[main node] {}
   -- (10,1.1) node[main node] {}
   -- (11,3.8) node[main node] {}
   -- (12,1.8) node[main node] {}
   -- (13,5.6) node[main node] {}
   -- (14,3.3) node[main node] {}
   -- (15,5.3) node[main node] {}
   -- (16,2.5) node[main node] {}
   -- (17,1) node[main node] {}
     ;
     
   \node[main node] at (0,-2) {};
   \node[main node] at (1,-2) {};
   \node[main node] at (2,-2) {};
   \node[main node] at (7,-2) {};
   \node[main node] at (8,-2) {};
   \node[main node] at (9,-2) {};
   \node[main node] at (10,-2) {};
   \node[main node] at (14,-2) {};
   \node[main node] at (15,-2) {};
   \node[main node] at (16,-2) {};
   \node[main node] at (17,-2) {};
   
   \draw[very thick,-] 
    (3,-2) node[full node] {}
    -- (4,-2) node[full node] {}
    -- (5,-2) node[full node] {}
    -- (6,-2) node[full node] {}
  ;

   \draw[very thick,-] 
    (11,-2) node[full node] {}
    -- (12,-2) node[full node] {}
    -- (13,-2) node[full node] {}
  ;
  
  \node at (4.5,-2.7) {$H_n=1$};
  \node at (8.5,-2.7) {$H_n=0$};

\end{tikzpicture}
\end{center}
\vspace{-4mm}
 \caption[]{Prix de l'action et strat\'egie de l'investisseur.}
 \label{fig_amost_sure}
\end{figure}

\begin{prop}
\label{prop_cops1}
Si $X_n$ est une sous-martingale, alors 
\begin{equation}
 \label{cops4}
(b-a) \expec{U_n} \leqs \bigexpec{(X_n-a)^+} - \bigexpec{(X_0-a)^+}\;.
\end{equation}  
\end{prop}
\begin{proof}
Soit $Y_n=a+(X_n-a)^+=X_n\vee a$. Par le corollaire~\ref{cor_mart}, $Y_n$ est
aussi une sous-martingale. De plus, $Y_n$ et $X_n$ effectuent le m\^eme nombre
de transitions de $a$ vers $b$. On a 
\begin{equation}
 \label{cops4:1}
(b-a) U_n \leqs (H\cdot Y)_n = \sum_{m=1}^n H_m (Y_m-Y_{m-1})\;. 
\end{equation} 
En effet, le membre de droite est sup\'erieur ou \'egal au profit fait au temps
$n$. Soit alors $K_m=1-H_m$. Comme $Y_n-Y_0=(H\cdot Y)_n + (K\cdot Y)_n$, et
que $(K\cdot Y)_n$ est une sous-martingale en vertu de la
Proposition~\ref{prop_mpc1}, on a 
\begin{equation}
 \label{cops4:2}
\bigexpec{(K\cdot Y)_n} \geqs \bigexpec{(K\cdot Y)_0}=0\;,
\end{equation} 
et donc $\expec{(H\cdot Y)_n}\leqs \expec{Y_n-Y_0}$, d'o\`u le r\'esultat.
\end{proof}

Cette majoration nous permet de d\'emontrer un premier r\'esultat de
convergence. 

\begin{theorem}[Convergence presque s\^ure]
\label{thm_conv_ps}
 Soit $X_n$ une sous-martingale telle que $\sup_n\expec{X_n^+}<\infty$.
Alors $X_n$ converge presque s\^urement, lorsque $n\to\infty$, vers une
variable al\'eatoire $X$ satisfaisant $\expec{\abs{X}}<\infty$.
\end{theorem}
\begin{proof}
Comme $(X-a)^+\leqs X^+ + \abs{a}$, la proposition pr\'ec\'edente implique 
\begin{equation}
 \label{cops5:1}
\expec{U_n} \leqs \frac{\abs{a}+\expec{X_n^+}}{b-a}\;. 
\end{equation} 
Lorsque $n\to\infty$, la suite croissante $U_n$ tend vers le nombre total $U$ de
passages de $a$ \`a $b$ de la suite des $X_n$. Si $\expec{X_n^+}$ est born\'e,
il suit que $\expec{U}<\infty$, et donc que $U$ est fini presque s\^urement.
Ceci \'etant vrai pour tout choix de rationnels $a<b$, l'\'ev\'enement 
\begin{equation}
 \label{cops5:2}
\bigcup_{a,b\in\Q} \bigsetsuch{\omega}{\liminf X_n(\omega) < a < b < \limsup
X_n(\omega)} 
\end{equation} 
a une probabilit\'e nulle. Il suit que $\liminf X_n=\limsup X_n$ presque
s\^urement, et donc que $\lim_{n\to\infty}X_n\bydef X$ existe presque
s\^urement. Le lemme de Fatou nous assure que
$\expec{X^+}\leqs\liminf\expec{X_n^+}<\infty$, et donc $X<\infty$ presque
s\^urement. Enfin, pour montrer que $X>-\infty$ presque s\^urement, on
d\'ecompose 
\begin{equation}
 \label{cops5:3}
\expec{X_n^-} = \expec{X_n^+} - \expec{X_n} \leqs \expec{X_n^+} -
\expec{X_0}\;, 
\end{equation}
et on applique \`a nouveau le lemme de Fatou pour montrer que
$\expec{X^-}<\infty$.
\end{proof}

\begin{cor}
\label{cor_convps}
Soit $X_n$ une surmartingale positive. Alors $X_n$ converge presque
s\^ure\-ment, lorsque $n\to\infty$, vers une variable al\'eatoire $X$ avec
$\expec{X}\leqs\expec{X_0}$. 
\end{cor}
\begin{proof}
$Y_n=-X_n$ est une sous-martingale born\'ee sup\'erieurement par $0$, avec
$\expec{Y_n^+}=0$, donc elle converge. Comme $\expec{X_0}\geqs\expec{X_n}$,
l'in\'egalit\'e suit du lemme de Fatou. 
\end{proof}

\begin{example}\hfill
\label{ex_convps}
\begin{enum}
\item	Consid\'erons l'urne de Polya introduite dans l'exemple
\ref{ex_martingale}. La proportion $X_n$ de boules rouges est une martingale
positive, qui est born\'ee sup\'erieurement par $1$. Par le th\'eor\`eme et le
corollaire, $X_n$ converge presque s\^urement vers une variable al\'eatoire $X$,
d'esp\'erance inf\'erieure ou \'egale \`a $1$. On peut montrer que selon le
nombre initial $r$ et $v$ de boules et le param\`etre $c$, la loi de $X$ est
soit uniforme, soit une loi b\^eta, de densit\'e proportionnelle \`a
$x^{r/c-1}(1-x)^{v/c-1}$.

\item	Soit $S_n$ une marche al\'eatoire sym\'etrique sur $\Z$, partant
de $1$, $N=N_0$ le
temps de premi\`ere atteinte de $0$, et $X_n=S_{n\wedge N}$ la marche
arr\^et\'ee en $0$. Par la Proposition~\ref{prop_indoob1}, $X_n$ est une
martingale non-n\'egative. Par le corollaire, elle converge vers une limite $X$
finie presque s\^urement. Cette limite doit \^etre nulle~: En effet, si
$X_n=k>0$, alors $X_{n+1}=k\pm1$, ce qui contredirait la convergence. On a donc
montr\'e que la marche al\'eatoire partant de $1$ atteint $0$ presque
s\^urement. On notera cependant que comme $X_n$ est une martingale, on a
$\expec{X_n}=\expec{X_0}=1$, donc que $X_n$ ne converge pas dans~$L^1$.
\end{enum}
\end{example}


\section{Convergence dans $L^p$, $p>1$}
\label{sec_coLp}

La convergence d'une martingale dans $L^p$ suit de mani\`ere simple 
de l'in\'egalit\'e du maximum~$L^p$~:

\begin{theorem}[Convergence d'une martingale dans $L^p$]
\label{thm_convLp}
Soit $X_n$ une martingale telle que $\sup_n\expec{\abs{X_n}^p}<\infty$ pour un
$p>1$. Alors $X_n$ converge vers une variable al\'eatoire $X$ presque
s\^urement et dans $L^p$. 
\end{theorem}
\begin{proof}
On a $(\expec{X_n^+})^p\leqs(\expec{\abs{X_n}})^p\leqs\expec{\abs{X_n}^p}$.
Donc par le Th\'eor\`eme~\ref{thm_conv_ps}, $X_n$ converge presque s\^urement
vers une variable $X$. Par le corollaire~\eqref{cor_indoob2}, 
\begin{equation}
 \label{coLp1}
\biggexpec{\biggbrak{\sup_{0\leqs m\leqs n}\abs{X_m}}^p} 
\leqs \biggpar{\frac{p}{p-1}}^p \bigexpec{\abs{X_n}^p}\;.
\end{equation} 
Faisant tendre $n$ vers l'infini, le th\'eor\`eme de la convergence monotone
montre que $\sup_n\abs{X_n}$ est dans $L^p$. Comme $\abs{X_n-X}^p\leqs
(2\sup_n\abs{X_n})^p$, le th\'eor\`eme de la convergence domin\'ee montre que
$\expec{\abs{X_n-X}^p}\to0$. 
\end{proof}

Dans la suite, nous consid\'erons plus particuli\`erement le cas $p=2$.
Rappelons que si une martingale $X_n$ est dans $L^2$, on peut d\'efinir son
processus croissant 
\begin{equation}
 \label{coLp2}
\braket{X}_n = \sum_{m=1}^n \econd{(X_m-X_{m-1})^2}{\cF_{m-1}}\;.
\end{equation} 
La croissance implique que 
\begin{equation}
 \label{coLp3} 
\lim_{n\to\infty} \braket{X}_n \bydef \braket{X}_\infty
\end{equation} 
existe dans $\R_+\cup\set{\infty}$. Cette quantit\'e s'interpr\`ete comme la
variance totale de la trajectoire $X_n(\omega)$. 

\begin{prop}
\label{prop_l2_1}
Soit $X_n$ une martingale dans $L^2$ telle que $X_0=0$. Alors 
\begin{equation}
\label{coLp4} 
\Bigexpec{\sup_n X_n^2} \leqs 4\bigexpec{\braket{X}_\infty}\;.
\end{equation} 
\end{prop}
\begin{proof}
L'in\'egalit\'e du maximum $L^2$ donne 
\begin{equation}
 \label{coLp4:1}
\biggexpec{\sup_{0\leqs m\leqs n}X_m^2} 
\leqs 4 \bigexpec{X_n^2} = 4\expec{\braket{X}_n}\;,
\end{equation} 
puisque $\expec{X_n^2}=\expec{M_n}+\expec{\braket{X}_n}$ et
$\expec{M_n}=\expec{M_0}=\expec{X_0^2}=0$. Le r\'esultat suit alors du
th\'eor\`eme de convergence monotone.
\end{proof}

\begin{theorem}
\label{thm_coL2}
La limite $\lim_{n\to\infty}X_n(\omega)$ existe et est finie presque s\^urement
sur l'en\-semble $\setsuch{\omega}{\braket{X}_\infty(\omega)<\infty}.$
\end{theorem}
\begin{proof}
Soit $a>0$. Comme $\braket{X}_{n+1}\measurable\cF_n$, 
$N=\inf\setsuch{n}{\braket{X}_{n+1}>a^2}$ est un temps d'arr\^et. Comme
$\braket{X}_{N\wedge n}<a^2$, la proposition ci-dessus appliqu\'ee \`a
$X_{N\wedge n}$ donne 
\begin{equation}
\label{colp5:1}
\biggexpec{\sup_n \abs{X_{N\wedge n}}^2} \leqs 4 a^2\;. 
\end{equation}
Par cons\'equent, le th\'eor\`eme~\ref{thm_convLp} avec $p=2$ implique que la
limite de $X_{N\wedge n}$ existe et est finie presque s\^urement. Le r\'esultat
suit alors du fait que $a$ est arbitraire. 
\end{proof}


\section{Convergence dans $L^1$}
\label{sec_coL1}

La discussion de la convergence d'une martingale dans $L^1$ n\'ecessite la
notion d'int\'egrabilit\'e uniforme. 

\begin{definition}[Int\'egrabilit\'e uniforme]
 Une collection $\set{X_i}_{i\in I}$ de variables al\'eatoires est dite\/ 
\defwd{uniform\'ement int\'egrable} si 
\begin{equation}
 \label{coL1_1}
\lim_{M\to\infty}
\biggpar{\sup_{i\in I} \bigexpec{\abs{X_i}\indexfct{\abs{X_i}>M}}} = 0\;. 
\end{equation} 
\end{definition}

On remarque qu'en prenant $M$ assez grand pour que le supremum soit inf\'erieur
\`a $1$, on obtient 
\begin{equation}
 \label{coL1_2}
\sup_{i\in I} \bigexpec{\abs{X_i}} 
\leqs \sup_{i\in I} \bigexpec{\abs{X_i}\indexfct{\abs{X_i}\leqs M}}
+ \sup_{i\in I} \bigexpec{\abs{X_i}\indexfct{\abs{X_i }>M}}
\leqs M + 1 < \infty\;. 
\end{equation} 

Nous commen\c cons par un r\'esultat g\'en\'eral montrant qu'il existe de
tr\`es vastes familles de variables uniform\'ement int\'egrables. 

\begin{prop}
\label{prop_coL1-0}
Soit $(\Omega,\cF,\fP)$ un espace probabilis\'e et $X\in L^1$. Alors la famille
\begin{equation}
 \label{col1-0} 
\setsuch{\econd{X}{\cF_i}}{\cF_i \text{ sous-tribu de } \cF}
\end{equation} 
est uniform\'ement int\'egrable.
\end{prop}
\begin{proof}
Pour tout $\eps>0$, le th\'eor\`eme de la convergence domin\'ee implique qu'on
peut trouver $\delta>0$ tel que si $\fP(A)\leqs\delta$, alors
$\expec{\abs{X}\indicator{A}}\leqs\eps$. Soit $M=\expec{\abs{X}}/\delta$ et 
$A_i=\set{\econd{\abs{X}}{\cF_i}>M}\in\cF_i$. 
Par l'in\'egalit\'e de Markov, 
\begin{equation*}
\fP(A_i) \leqs 
\frac1M \bigexpec{\econd{\abs{X}}{\cF_i}}
= \frac1M \bigexpec{\abs{X}} \leqs\delta\;.
\end{equation*}
Par l'in\'egalit\'e de Jensen, il suit 
\begin{equation*}
\Bigexpec{\bigabs{\econd{X}{\cF_i}}\indexfct{\abs{\econd{X}{\cF_i}}>M}}
\leqs 
\bigexpec{\econd{\abs{X}}{\cF_i}\indicator{A_i}}
=
\bigexpec{\abs{X}\indicator{A_i}}\leqs\eps\;.
\end{equation*} 
Comme $\eps$ est arbitraire (et $M$ ne d\'epend pas de $i$), le r\'esultat est
prouv\'e.
\end{proof}

Le lien entre int\'egrabilit\'e uniforme et convergence dans $L^1$ est
expliqu\'e par le r\'esultat suivant. 

\begin{theorem}
\label{thm_coL1-1}
Si $X_n\to X$ en probabilit\'e, alors les trois conditions suivantes sont
\'equivalentes~:
\begin{enum}
\item	La famille $\setsuch{X_n}{n\geqs0}$ est uniform\'ement int\'egrable;
\item	$X_n\to X$ dans $L^1$;
\item	$\expec{\abs{X_n}}\to\expec{\abs{X}}<\infty$.
\end{enum}
\end{theorem}

\goodbreak
\begin{proof}
\hfill
\begin{itemizz}
\item[1. $\Rightarrow$ 2.]
Soit 
\begin{equation*}
\ph_M(x) = 
\begin{cases}
M & \text{si $x\geqs M$\;,} \\
x & \text{si $\abs{x}\leqs M$\;,} \\
-M & \text{si $x\leqs-M$\;.}
\end{cases} 
\end{equation*} 
L'in\'egalit\'e triangulaire implique 
\begin{equation*}
\abs{X_n-X} \leqs 
\abs{X_n-\ph_M(X_n)} + \abs{\ph_M(X_n)-\ph_M(X)} + \abs{\ph_M(X)-X}\;.
\end{equation*}
Prenant l'esp\'erance, et utilisant le fait que
$\abs{\ph_M(Y)-Y}\leqs\abs{Y}\indexfct{\abs{Y}\geqs M}$, il vient  
\begin{equation*}
\expec{\abs{X_n-X}} \leqs 
\expec{\abs{\ph_M(X_n)-\ph_M(X)}} 
+ \expec{\abs{X_n}\indexfct{\abs{X_n}\geqs M}} 
+ \expec{\abs{X}\indexfct{\abs{X}\geqs M}}\;. 
\end{equation*} 
Le premier terme tend vers $0$ lorsque $n\to\infty$ par convergence domin\'ee et
le fait que $\ph_M(X_n)\to\ph_M(X)$ en probabilit\'e. L'int\'egrabilit\'e
uniforme implique que pour tout $\eps>0$, le second terme est plus petit que
$\eps$ pour $M$ assez grand. L'int\'egrabilit\'e uniforme implique aussi que
$\sup_n\expec{\abs{X_n}}<\infty$, donc par le lemme de Fatou que
$\expec{\abs{X}}<\infty$. Le troisi\`eme terme peut donc \'egalement \^etre
rendu plus petit que $\eps$ en prenant $M$ assez grand, ce qui prouve 2.

\item[2. $\Rightarrow$ 3.]
Par l'in\'egalit\'e de Jensen, 
\begin{equation*}
\bigabs{\expec{\abs{X_n}}-\expec{\abs{X}}} 
\leqs \bigexpec{\bigabs{\abs{X_n}-\abs{X}}}
\leqs \bigexpec{\abs{X_n-X}} \to 0\;.
\end{equation*} 

\item[3. $\Rightarrow$ 1.]
Soit 
\begin{equation*}
\psi_M(x) = 
\begin{cases}
x & \text{si $0\leqs x\leqs M-1$\;,} \\
(M-1)(M-x) & \text{si $M-1\leqs x\leqs M$\;,} \\
0 & \text{si $x\geqs M$\;.}
\end{cases} 
\end{equation*} 
Pour tout $\eps>0$, le th\'eor\`eme de la convergence domin\'ee montre que 
$\expec{\abs{X}}-\expec{\psi_M(\abs{X})}\leqs\eps/2$ pour $M$ assez grand. De
plus, $\expec{\psi_M(\abs{X_n})}\to\expec{\psi_M(\abs{X})}$, donc par 3.
\begin{equation*}
\bigexpec{\abs{X_n}\indexfct{\abs{X_n}>M}}
\leqs \expec{\abs{X_n}} - \expec{\psi_M(\abs{X_n})}
\leqs \expec{\abs{X}} - \expec{\psi_M(\abs{X})} + \frac{\eps}2 < \eps
\end{equation*} 
pour les $n$ plus grands qu'un $n_0(\eps)$. En augmentant encore $M$, on peut
rendre $\expec{\abs{X_n}\indexfct{\abs{X_n}>M}}$ inf\'erieur \`a $\eps$ pour
les $n\leqs n_0(\eps)$ \'egalement, ce qui montre l'int\'egrabilit\'e uniforme.
\qed
\end{itemizz}
\renewcommand{\qed}{}
\end{proof}

Il est maintenant ais\'e d'appliquer ce r\'esultat au cas des sous-martingales. 

\begin{theorem}[Convergence d'une sous-martingale dans $L^1$]
\label{thm_coL1-2} 
Si $X_n$ est une sous-martingale, alors les trois conditions suivantes sont
\'equivalentes~:
\begin{enum}
\item	$\set{X_n}_{n\geqs0}$ est uniform\'ement int\'egrable;
\item	$X_n$ converge presque s\^urement et dans $L^1$;
\item	$X_n$ converge dans $L^1$.
\end{enum}
\end{theorem}
\begin{proof}
\hfill
\begin{itemizz}
\item[1. $\Rightarrow$ 2.]
On a $\sup_n\expec{X_n}<\infty$, donc $X_n\to X$ presque s\^urement par le
th\'eor\`eme~\ref{thm_conv_ps}, et le th\'eor\`eme~\ref{thm_coL1-1} implique
que $X_n\to X$ dans $L^1$. 

\item[2. $\Rightarrow$ 3.]
Trivial.

\item[3. $\Rightarrow$ 1.]
La convergence dans $L^1$ implique la convergence en probabilit\'e, et donc le
th\'eo\-r\`eme~\ref{thm_coL1-1} permet de conclure. 
\qed
\end{itemizz}
\renewcommand{\qed}{}
\end{proof}

Dans le cas particulier des martingales, on a 

\begin{theorem}[Convergence d'une martingale dans $L^1$]
\label{thm_coL1-3} 
Si $X_n$ est une martingale, alors les quatre conditions suivantes sont
\'equivalentes~:
\begin{enum}
\item	$\set{X_n}_{n\geqs0}$ est uniform\'ement int\'egrable;
\item	$X_n$ converge presque s\^urement et dans $L^1$;
\item	$X_n$ converge dans $L^1$;
\item	il existe une variable al\'eatoire int\'egrable $X$ telle que
$X_n=\econd{X}{\cF_n}$. 
\end{enum}
Dans ce cas $X$ est la limite de $X_n$ dans $L^1$.
\end{theorem}
\begin{proof}
\hfill
\begin{itemizz}
\item[1. $\Rightarrow$ 2.]
Suit du th\'eor\`eme pr\'ec\'edent.

\item[2. $\Rightarrow$ 3.]
Trivial.

\item[3. $\Rightarrow$ 4.]
Si $X_n\to X$ dans $L^1$, alors pour tout $A\in\cF$, 
$\expec{X_n\indicator{A}}\to\expec{X\indicator{A}}$ puisque
\begin{equation*}
\abs{\expec{X_m\indicator{A}}-\expec{X\indicator{A}}}
\leqs\expec{\abs{X_m\indicator{A}-X\indicator{A}}}
\leqs\expec{\abs{X_m-X}}\to0\;.
\end{equation*} 
En particulier, si $A\in\cF_n$, alors pour tout $m>n$,
\begin{equation*}
\expec{X_n\indicator{A}}=\bigexpec{\econd{X_m}{\cF_n}\indicator{A}}
= \expec{X_m\indicator{A}}\;.
\end{equation*}  
Comme $\expec{X_m\indicator{A}}\to\expec{X\indicator{A}}$, on a en fait
$\expec{X_n\indicator{A}}=\expec{X\indicator{A}}$ pour tout $A\in\cF_n$. Mais
ceci est la d\'efinition de $X_n=\econd{X}{\cF_n}$. 

\item[4. $\Rightarrow$ 1.]
Suit de la proposition~\ref{prop_coL1-0}.
\qed
\end{itemizz}
\renewcommand{\qed}{}
\end{proof}


\section{Loi $0-1$ de L\'evy}
\label{levy01} 

Les lois $0-1$ jouent un r\^ole important pour les suites de variables
al\'eatoires, et ont parfois des cons\'equences surprenantes. La plus connue
est la loi $0-1$ de Kolmogorov. Soit $X_n$ une suite de variables al\'eatoires
ind\'ependantes, $\cF'_n=\sigma(X_n,X_{n+1},\dots)$, et $\cT=\bigcap_n\cF'_n$ la
tribu terminale. Intuitivement, les \'ev\'enements de $\cT$ sont ceux dont
l’occurrence n'est pas affect\'ee par la modification d'un nombre fini de $X_n$,
comme par exemple l'\'ev\'enement $\set{\sum_{i=1}^\infty X_i \text{ existe}}$.
La loi $0-1$ de Kolmogorov affirme alors que si $A\in\cT$, alors $\fP(A)=0$ ou
$1$. Ceci permet notamment de prouver la loi forte des grands nombres.

Nous \'ecrivons $\cF_n\nearrow\cF_\infty$ si $\cF_n$ est une filtration telle
que $\cF_\infty=\sigma(\bigcup_n\cF_n)$. 

\begin{theorem}
\label{thm_levy1}
Si $\cF_n\nearrow\cF_\infty$, alors 
\begin{equation}
 \label{levy1}
\econd{X}{\cF_n} \to \econd{X}{\cF_\infty} 
\end{equation} 
presque s\^urement et dans $L^1$ lorsque $n\to\infty$.
\end{theorem}
\begin{proof}
Comme pour $m>n$, $\econd{\econd{X}{\cF_m}}{\cF_n}=\econd{X}{\cF_n}$, 
$Y_n=\econd{X}{\cF_n}$ est une martingale. La Proposition~\ref{prop_coL1-0}
montre que $Y_n$ est uniform\'ement int\'egrable, donc par le
Th\'eor\`eme~\ref{thm_coL1-3} elle converge presque s\^urement et dans $L^1$
vers une variable al\'eatoire $Y_\infty$, telle que
$Y_n=\econd{Y_\infty}{\cF_n}$. Par d\'efinition de l'esp\'erance
conditionnelle, 
\[
\int_A X\6\fP = \int_A Y_\infty\6\fP
\]
pour tout $A\in\cF_n$. En prenant la limite $n\to\infty$, justifi\'ee par le
th\'eor\`eme $\pi$-$\lambda$ de th\'eorie de l'int\'egration, on obtient
$\econd{X}{\cF_\infty}=Y_\infty$ puisque $Y_\infty\measurable\cF_\infty$. 
\end{proof}

\begin{cor}[Loi $0-1$ de L\'evy]
\label{cor_Levy} 
Si $\cF_n\nearrow\cF_\infty$ et $A\in\cF_\infty$ alors
$\econd{\indicator{A}}{\cF_n}\to\indicator{A}$ presque s\^urement.
\end{cor}

\begin{example}
\label{ex_levy}
Supposons que $X_n$ est une suite de variables al\'eatoires ind\'ependantes, et
soit $A\in\cT$ un \'ev\'enement de la tribu terminale. Pour tout $n$, $A$ est
ind\'ependant de $\cF_n$, donc
$\econd{\indicator{A}}{\cF_n}=\expec{\indicator{A}}=\fP(A)$. Faisant tendre $n$
vers l'infini, on obtient $\indicator{A}=\fP(A)$, d'o\`u $\fP(A)=0$ ou $1$,
ce qui n'est autre que la loi $0-1$ de Kolmogorov. 
\end{example}


\section{Exercices}
\label{sec_exo_conv}

\begin{exercice}[Le processus de Galton--Watson]
\label{exo_conv1} 
On consid\`ere le processus de Galton--Watson $Z_n$ introduit dans
l'exercice~\ref{exo_mart5}.

\begin{enum}
\item	Montrer que $X_n=Z_n/\mu^n$ converge dans $L^2$
vers une variable al\'eatoire $X$ d'esp\'erance \'egale \`a $1$. 

\item	On suppose $p_0>0$. 
A l'aide de la loi $0-1$ de L\'evy, montrer que $Z_n$ converge presque
s\^urement vers une variable al\'eatoire $Z_\infty:\Omega\to\set{0,+\infty}$. 

{\bf Indications~:} Soit $A=\set{\lim_{n\to\infty} Z_n<\infty}$ et 
$B=\setsuch{\exists n}{Z_n=0}$. Montrer que dans $A$, on a
$\lim_{n\to\infty}\econd{\indicator{B}}{\cF_n}=\indicator{B}$, et en d\'eduire
que $A\subset B$. Conclure en \'etablissant que 
que $Z_\infty=\infty$ dans $A^c$ et $Z_\infty=0$ dans $B$. 
\end{enum}
\end{exercice}

\goodbreak

\begin{exercice} 
\label{exo_conv2} 
Un joueur dispose initialement de la somme $X_0=1$. Il joue \`a un jeu de
hasard, dans lequel il mise \`a chaque tour une proportion $\lambda$ de son
capital, avec $0<\lambda\leqs1$. Il a une chance sur deux de gagner le double
de sa mise, sinon il perd sa mise. 

L'\'evolution du capital $X_n$ en fonction du temps $n$ est d\'ecrite par 
\[
X_{n+1}=(1-\lambda)X_n + \lambda X_n \xi_n
\qquad (n\geqs0)
\]
o\`u les $\xi_n$ sont i.i.d., avec $\prob{\xi_n=2}=\prob{\xi_n=0}=1/2$. 

\bigskip

\begin{enum}
\item	Montrer que $X_n$ est une martingale. 
\item	Calculer $\expec{X_n}$. 
\item	Discuter la convergence presque s\^ure de $X_n$ lorsque $n\to\infty$.
\item	Calculer $\expec{X_n^2}$ par r\'ecurrence sur $n$.  
\item	Que peut-on en d\'eduire sur la convergence dans $L^2$ de $X_n$? 
\item	D\'eterminer le processus croissant $\braket{X}_n$. 
\item	On suppose que le joueur mise \`a chaque tour la totalit\'e de son
capital, c'est-\`a-dire $\lambda=1$. 
\begin{enum}
\item	Calculer explicitement la loi de $X_n$. 
\item	D\'eterminer la limite presque s\^ure de $X_n$. 
\item	Discuter la convergence de $X_n$ dans $L^1$.

Les $X_n$ sont-ils uniform\'ement int\'e\-grables?
\item	Commenter ces r\'esultats - est-ce que vous joueriez \`a ce jeu?
\end{enum}
\end{enum}
\end{exercice}

\goodbreak

\begin{exercice}[{\it La} Martingale]
\label{exo_conv3} 
Un joueur mise sur les r\'esultats des jets ind\'ependants d'une pi\`ece
\'equili\-br\'ee. A chaque tour, il mise une somme $S\geqs0$. Si la pi\`ece
tombe sur Pile, son capital augmente de $S$, si elle tombe sur Face, le joueur
perd sa mise et donc son capital diminue de $S$. 

Une strat\'egie populaire en France au XVIIIe si\`ecle est appel\'ee
\emph{La}\/ Martingale. Elle est d\'efinie comme suit:
\begin{itemiz}
\item	le joueur s'arr\^ete de jouer d\`es qu'il a gagn\'e la premi\`ere fois
(d\`es le premier Pile);
\item	il double sa mise \`a chaque tour, c'est-\`a-dire qu'il mise la somme
$S_n=2^n$ au $n$i\`eme tour, tant qu'il n'a pas gagn\'e. 
\end{itemiz}

Soit $Y_n$ le capital du joueur au temps $n$ (apr\`es $n$ jets de la pi\`ece). 
On admettra que le capital initial est nul, et que le joueur a le droit de
s'endetter d'une somme illimit\'ee, c'est-\`a-dire que $Y_n$ peut devenir
arbitrairement n\'egatif. Soit $X_n$ le capital au temps $n$ d'un joueur misant
un Euro \`a chaque tour. 
\begin{enum}
\item	Montrer que la strat\'egie est pr\'evisible, et \'ecrire $Y_n$ sous la
forme $Y_n=(H\cdot X)_n$ en fonction du processus $\set{X_n}_{n\geqs0}$.
\item	Montrer que $\set{Y_n}_{n\geqs0}$ est une martingale.
\item	D\'eterminer le processus croissant $\braket{Y}_n$. Calculer
$\expec{\braket{Y}_n}$ et discuter la convergence de $Y_n$ dans $L^2$. 
\item	D\'eterminer l'image de $Y_n$, sa loi, et discuter la convergence
presque s\^ure de $Y_n$. Quelle est sa limite?  
\item	$Y_n$ converge-t-elle dans $L^1$? 
\end{enum}
On suppose maintenant que la banque n'admet pas que le joueur s'endette de plus
qu'une valeur limite $L$ (on pourra supposer que $L=2^k$ pour un $k\geqs1$). Par
cons\'equent, le joueur est oblig\'e de s'arr\^eter d\`es que son capital au
temps $n$ est inf\'erieur \`a $-L+2^{n+1}$. Notons $Z_n$ ce capital. 
\begin{enum}
\setcounter{enumi}{5}
\item	Soit $N$ la dur\'ee du jeu (le nombre de fois que le joueur mise une
somme non nulle). Montrer que $N$ est un temps d'arr\^et et donner sa loi.
\item	Le processus $Z_n$ est-il une martingale?  
\item	Discuter la convergence presque s\^ure et dans $L^1$ de
$Z_n$ et commenter les r\'esultats. 
\end{enum}
\end{exercice}

\goodbreak

\begin{exercice} 
\label{exo_conv4} 
Soit $X_0=1$. On d\'efinit une suite $\set{X_n}_{n\in\N}$ r\'ecursivement en
posant que pour tout $n\geqs1$, $X_n$ suit la loi uniforme sur $]0,2X_{n-1}]$
(c-\`a-d $X_n=2U_nX_{n-1}$ o\`u les $U_n$ sont i.i.d. de loi uniforme sur
$]0,1]$). 

\begin{enum}
\item	Montrer que $X_n$ est une martingale. 
\item	Calculer le processus croissant $\braket{X}_n$ et discuter la
convergence de $X_n$ dans $L^2$. 
\item	Discuter la convergence presque s\^ure de $X_n$.
\item	D\'eterminer la limite presque
s\^ure de $X_n$. 

{\it Indication~:} Consid\'erer $Y_n=\log(X_n)$ ainsi que   
$Z_n = Y_n - \expec{Y_n}$.
\item	Discuter la convergence de $X_n$ dans $L^1$. 
\end{enum}
\end{exercice}

\goodbreak

\begin{exercice} 
\label{exo_conv5} 
Soit $U$ une variable al\'eatoire uniforme sur $[0,1]$. Soit $f:[0,1]\to\R$ une
fonction uniform\'ement lipschitzienne. Soit $I_{k,n}=[k2^{-n},(k+1)2^{-n}[$
et soit $\cF_n$ la tribu sur $\Omega=[0,1]$ engendr\'ee par les $I_{k,n}$ pour
$k=0,\dots, 2^n-1$.
On pose  
\[
X_n = \sum_{k=0}^{2^n-1} \frac{f((k+1)2^{-n})-f(k2^{-n})}{2^{-n}}
\indexfct{U\in I_{k,n}}\;.
\]

\begin{enum}
\item	Montrer que $X_n$ est une martingale par rapport \`a $\cF_n$.
\item	Montrer que $X_n$ converge presque s\^urement. On note la limite
$X_\infty$. 
\item	Discuter la convergence de $X_n$ dans $L^1$. 
\item	Montrer que pour tout $0\leqs a<b\leqs 1$, 
\[
\bigexpec{X_\infty\indexfct{U\in[a,b]}} = f(b) - f(a)\;.
\]
\item	On suppose $f$ de classe $C^1$. A l'aide de 4., 
expliciter $X_\infty(\omega)$ (si l'on
note $U(\omega)=\omega$ pour tout $\omega\in\Omega$). 
\end{enum}
\end{exercice}

\goodbreak

\begin{exercice}
\label{exo_conv7} 

Soit $\set{X_n}_{n\geqs1}$ une sous-martingale, et soit 
\[
\Xbar_n = \max_{1\leqs i\leqs n} X_i\;.
\]
On note $x^+ = 0\vee x$ et $\log^+(x) = 0\vee\log (x)$.

\begin{enum}
\item 	Soit $Y$ une variable al\'eatoire \`a valeurs dans $\R_+$, et soit
$f:\R_+\to\R$ une  fonction croissante, diff\'erentiable par morceaux. Montrer
que pour tout $c\geqs0$,
\[
\bigexpec{f(Y)} \leqs f(c) + 
\int_c^\infty f'(\lambda) \bigprob{Y>\lambda} \6\lambda\;.
\]

\item	
Montrer que pour tout $M>1$, 
\[
\expec{\Xbar_n^+\wedge M} \leqs 1 + \expec{X_n^+ \log^+(\Xbar_n^+\wedge M)}
\]
et en d\'eduire que
\[
\expec{\Xbar_n^+} \leqs 
\frac{1+\expec{X_n^+\log^+(X_n^+)}}{1-\e^{-1}}\;.
\]
On pourra utiliser le fait que $a\log b \leqs a\log a + b/\e$ pour tout
$a,b>0$. 

\item	Montrer que si $\sup_n \abs{X_n} \leqs Y$ pour une variable al\'eatoire
$Y$ telle que $\expec{Y}<\infty$ alors 
$\set{X_n}_{n\geqs1}$ est uniform\'ement int\'egrable. 

\item	En d\'eduire que si $\set{X_n}_{n\geqs1}$ est une martingale telle que 
$\sup_n\expec{\abs{X_n}\log^+\abs{X_n}} < \infty$,
alors $\set{X_n}_{n\geqs1}$ converge dans $L^1$. 

\end{enum}
\end{exercice}

\goodbreak

\begin{exercice}[La loi du logarithme it\'er\'e]
\label{exo_conv6} 

Soit $\set{X_n}_{n\geqs1}$ une suite de variables al\'eatoires r\'eelles
i.i.d., centr\'ees, de variance $1$. On suppose que la fonction g\'en\'eratrice 
\[
\psi(\lambda) = \bigexpec{\e^{\lambda X_1}}
\]
est finie pour tout $\lambda\in\R$. Soit 
\[
 S_n = \sum_{i=1}^n X_i\;.
\]
Le but de ce probl\`eme est de montrer que presque s\^urement 
\begin{equation}
\label{loglogn} 
\limsup_{n\to\infty} \frac{S_n}{\sqrt{2n \log(\log n)}} \leqs 1\;.
\end{equation}
Dans la suite, on pose $h(t)=\sqrt{2t\log(\log t)}$ pour $t>1$. 

\begin{enum}
\item	Montrer que
$\psi(\lambda)=1+\frac12\lambda^2+\order{\lambda^2}$.

\item	Montrer que $Y_n = \e^{\lambda S_n}/\psi(\lambda)^n$ est une martingale.

\item	Montrer que pour tout $N\geqs 1$ et tout $a>0$, 
\[
\biggprob{\exists n\leqs N \colon S_n > a + n \frac{\log\psi(\lambda)}{\lambda}}
\leqs \e^{-\lambda a}\;.
\]

\item	On fixe $t, \alpha > 1$. Pour tout $k\geqs 1$, on pose 
\[
a_k = \frac{\alpha}{2} h(t^k)\;, \qquad
\lambda_k = \frac{h(t^k)}{t^k}\;, \qquad
c_k = \frac{\alpha}{2} + t \frac{\log\psi(\lambda_k)}{\lambda_k^2}\;.
\]
Montrer que 
\[
\biggprob{\exists n\in]t^k,t^{k+1}] \colon S_n > h(n)c_k}
\leqs (k\log t)^{-\alpha}\;.
\]

\item	Montrer que presque s\^urement, 
\[
\biggprob{\frac{S_n}{h(n)}\leqs c_k \; \forall n\in]t^k,t^{k+1}], 
k\to\infty } = 1\;.
\]

\item	En utilisant le point 1.,
montrer que 
\[
c_k = \frac{\alpha+t}{2} + r(k)
\]
avec $\lim_{k\to\infty} r(k)=0$. 

\item	D\'emontrer~\eqref{loglogn}. 
\end{enum}
\end{exercice}


\part{Processus en temps continu}
\label{part_cont}


\chapter{Le mouvement Brownien}
\label{chap_mb}

Le mouvement Brownien joue un r\^ole fondamental dans la th\'eorie des processus
stochastiques en temps continu. Bien qu'il soit relativement simple \`a
d\'efinir, il jouit de nombreuses propri\'et\'es remarquables, et parfois
surprenantes. Le mouvement Brownien a \'et\'e \'etudi\'e de mani\`ere tr\`es
approfondie, et bien qu'il existe encore des questions ouvertes, c'est l'un des
processus les mieux compris. Il joue un r\^ole important dans la construction
de processus stochastiques plus g\'en\'eraux. 

Dans ce chapitre, nous pr\'esenterons une construction du mouvement Brownien,
et une petite s\'election de ses nombreuses propri\'et\'es, plus
particuli\`erement celles en rapport avec les martingales. 


\section{Limite d'\'echelle d'une marche al\'eatoire}
\label{sec_mbma}

Nous commen\c cons par introduire le mouvement Brownien de mani\`ere
heuristique. Consid\'erons pour cela la marche al\'eatoire sym\'etrique $X_n$
sur $\Z$. Elle peut \^etre repr\'esent\'ee sous la forme 
\begin{equation}
 \label{mbma1}
X_n = \sum_{i=1}^n \xi_i\;, 
\end{equation} 
o\`u les $\xi_i$ sont des variables al\'eatoires i.i.d., prenant valeurs $1$ et
$-1$ avec probabilit\'e $1/2$. On en d\'eduit ais\'ement les propri\'et\'es
suivantes:
\begin{enum}
\item	$\expec{X_n}=0$ pour tout $n$;
\item	$\variance(X_n)=n$;
\item	$X_n$ prend ses valeurs dans $\set{-n,-n+2,\dots,n-2,n}$ avec 
\begin{equation}
 \label{mbma2}
\prob{X_n=k} = \frac{1}{2^n}
\frac{n!}{\bigpar{\frac{n+k}{2}}!\bigpar{\frac{n-k}{2}}!} \;.
\end{equation} 
\item	{\it Propri\'et\'es des incr\'ements ind\'ependants:}\/ pour tout
$n>m\geqs0$, $X_n-X_m$ est ind\'epen\-dant de $X_1,\dots,X_m$;
\item	{\it Propri\'et\'es des incr\'ements stationnaires:}\/ pour tout
$n>m\geqs0$, $X_n-X_m$ a la m\^eme loi que $X_{n-m}$.
\end{enum}

Consid\'erons alors la suite de processus 
\begin{equation}
 \label{mbma3}
B_t^{(n)} = \frac{1}{\sqrt{n}} X_{\intpart{nt}}\;, 
\qquad
t\in\R_+\;, \quad
n\in\N^*\;.
\end{equation} 
Cela signifie que l'on acc\'el\`ere le temps d'un facteur $n$, tout en
comprimant l'espace d'un facteur $\sqrt{n}$, de sorte que $B_t^{(n)}$ effectue
des pas de $\pm1/\sqrt{n}$ sur des intervalles de temps de longueur $1/n$
(\figref{fig_bm}). 

\begin{figure}
 \centerline{
 \includegraphics*[clip=true,width=73mm]{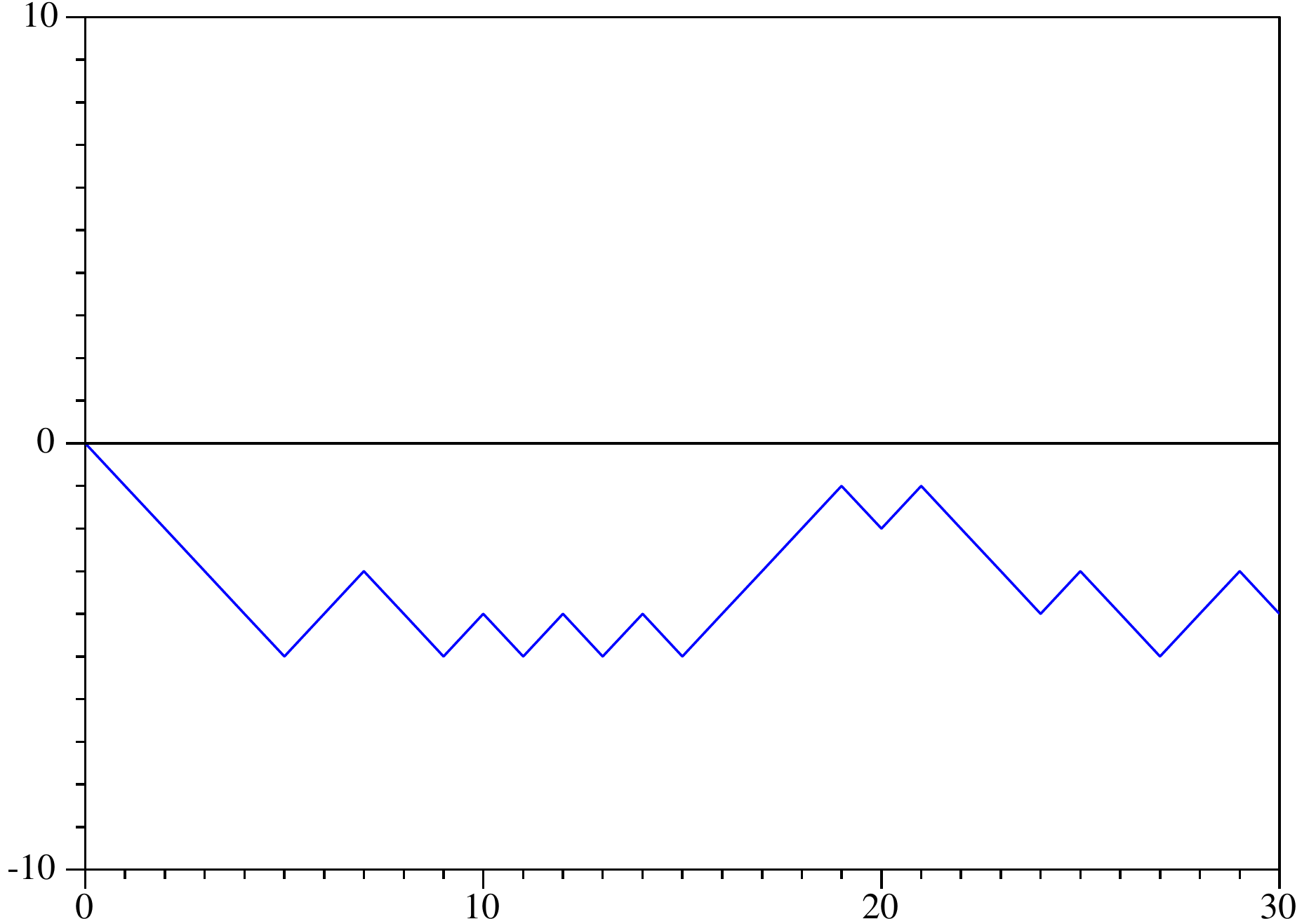}
 \includegraphics*[clip=true,width=73mm]{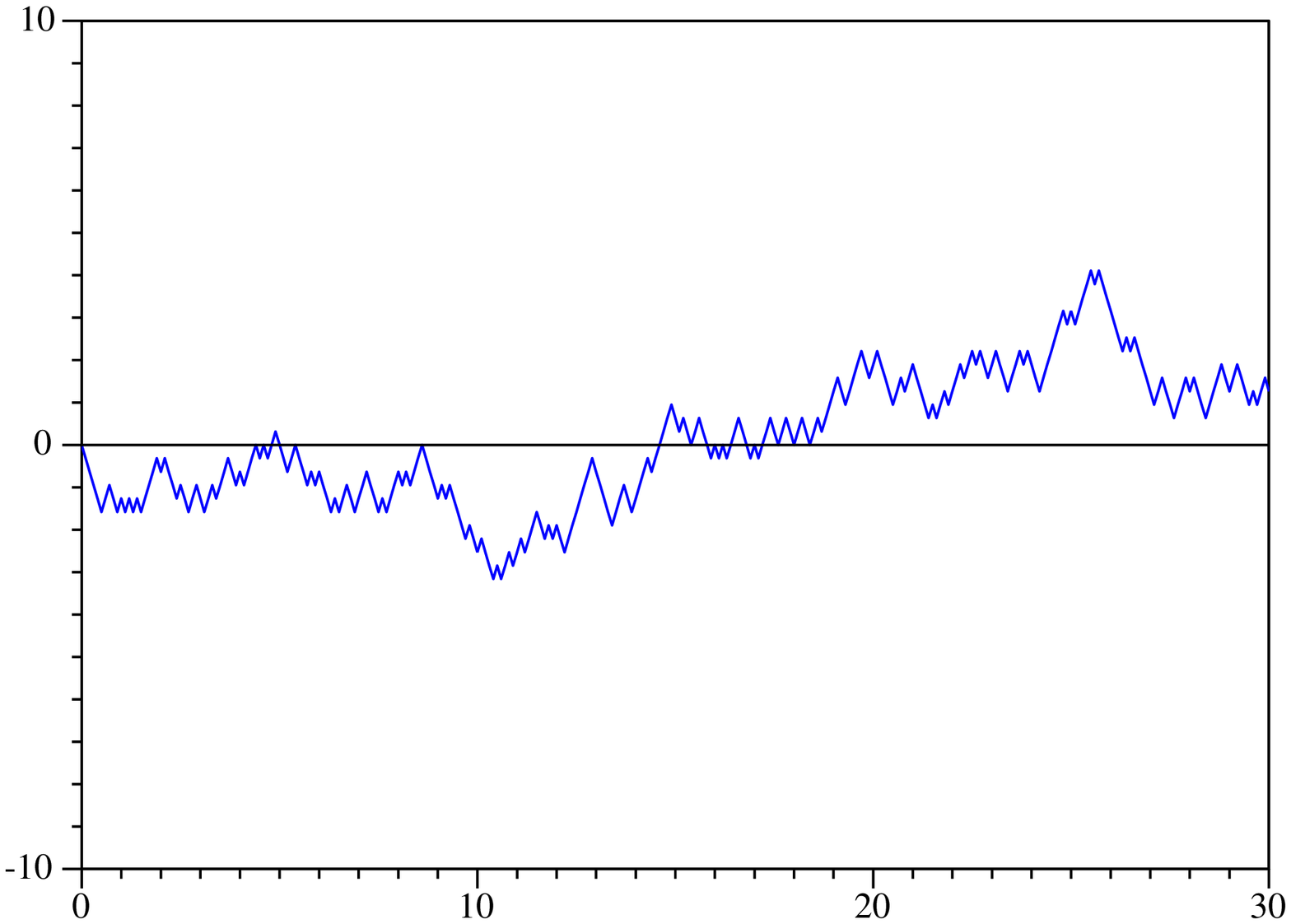}
 }
 \vspace{1mm}
 \centerline{
 \includegraphics*[clip=true,width=73mm]{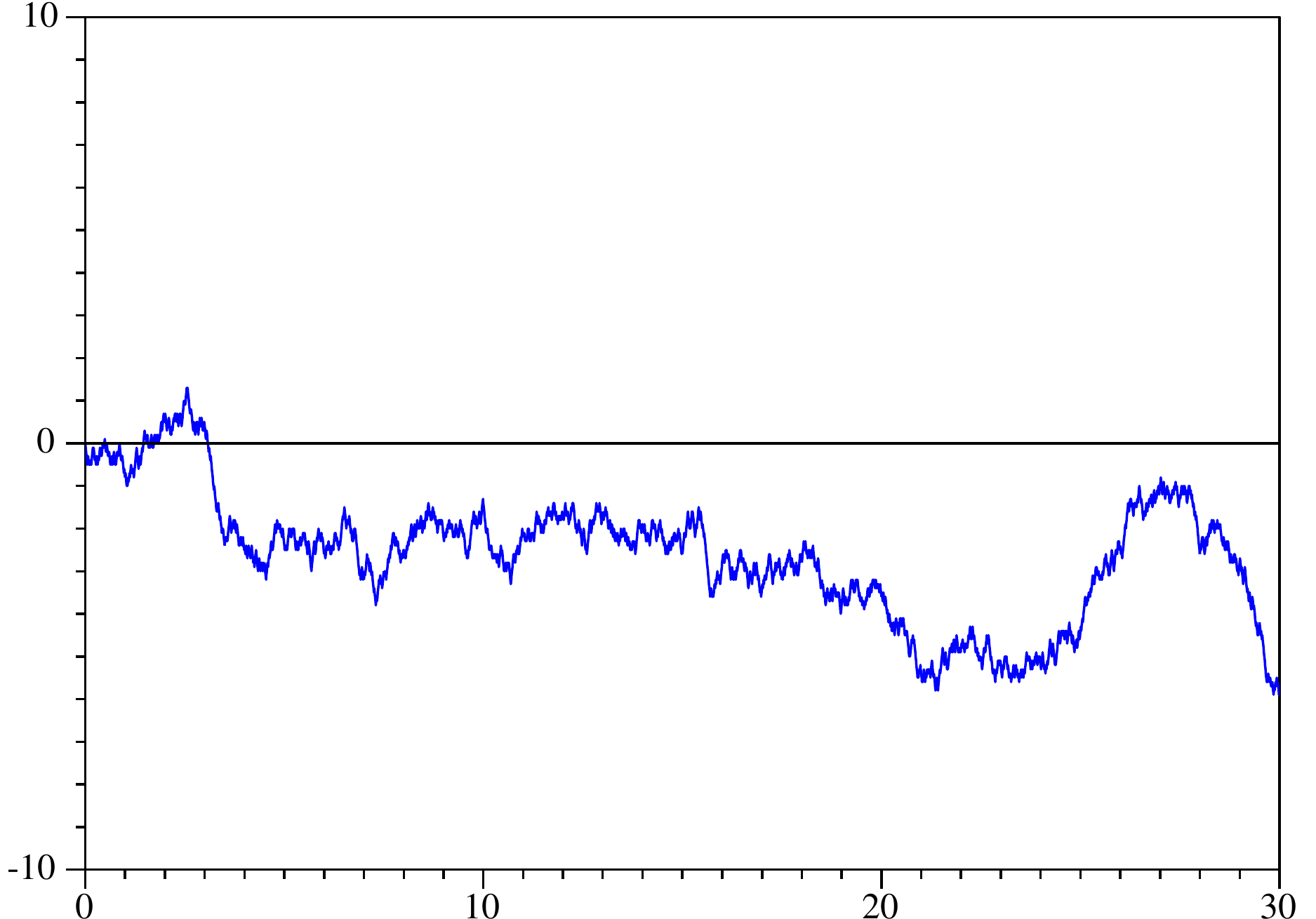}
 \includegraphics*[clip=true,width=73mm]{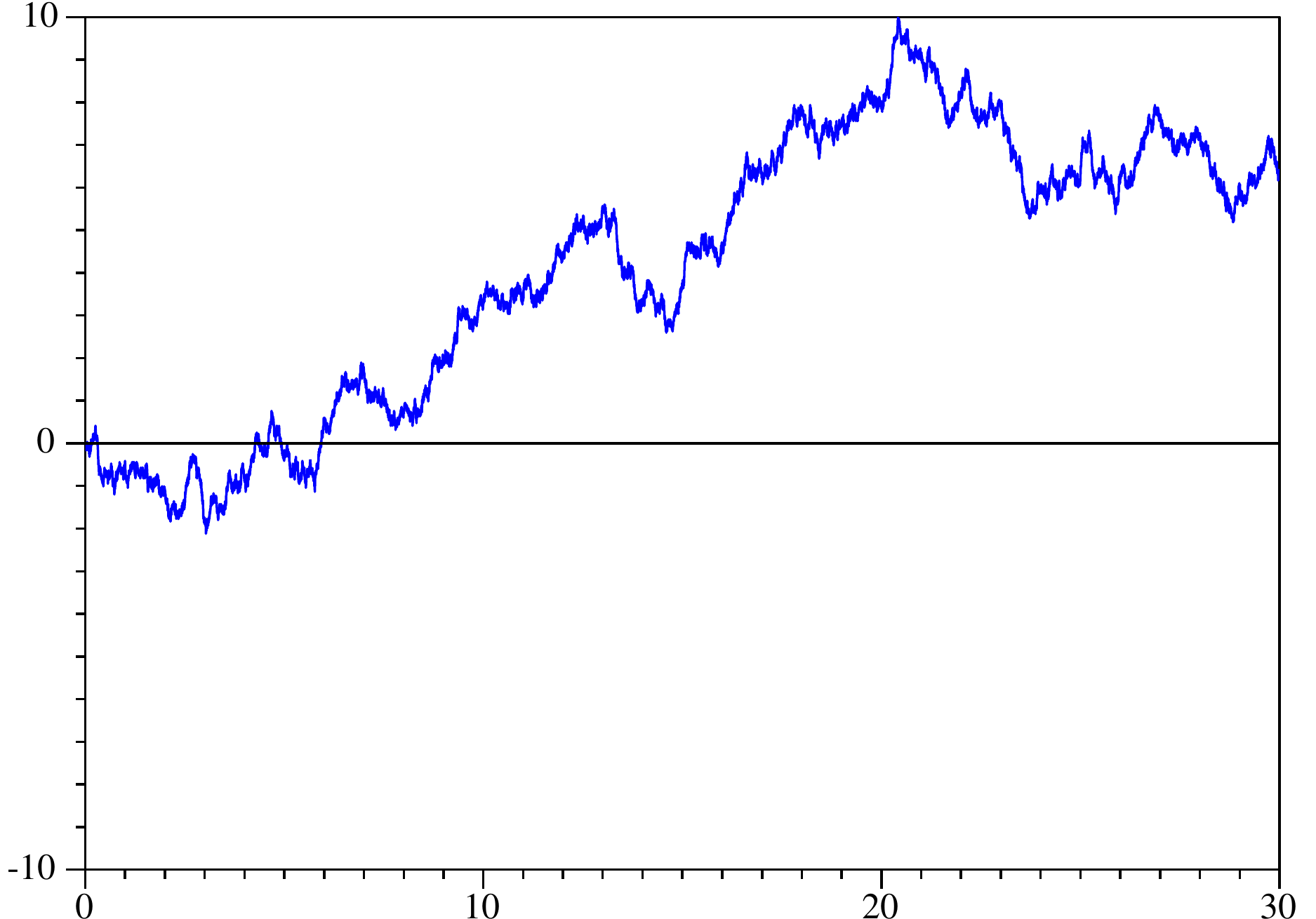}
 }
 \figtext{
 }
 \caption[]{Les versions r\'e\'echelonn\'ees $B_t^{(1)}$,
$B_t^{(10)}$, $B_t^{(100)}$ et $B_t^{(1000)}$ d'une m\^eme r\'ealisation
$\omega$ d'une marche al\'eatoire.}
 \label{fig_bm}
\end{figure}

Soit $B_t$ le processus obtenu en prenant la limite de $B_t^{(n)}$ lorsque
$n\to\infty$, au sens des distributions finies. Autrement dit, $B_t$ est
d\'efini par le fait que pour toute partition $0\leqs t_1 < t_2 < \dots < t_k=t$
de $[0,t]$ et tout $(x_1,x_2,\dots,x_k)\in\R^k$, 
\begin{equation}
\label{mbma4}
\bigprob{B_{t_1}\leqs x_1, \dots, B_{t_k}\leqs x_k}
= \lim_{n\to\infty}
\bigprob{B^{(n)}_{t_1}\leqs x_1, \dots, B^{(n)}_{t_k}\leqs x_k}\;. 
\end{equation}
Supposons pour l'instant que cette limite existe. Le processus $B_t$ aura 
les propri\'et\'es suivantes: 
\begin{enum}
\item	$\expec{B_t}=0$ pour tout $t\geqs0$;
\item	La variance de $B_t$ satisfait
\begin{equation}
 \label{mbma5}
\variance(B_t) = \lim_{n\to\infty} \biggpar{\frac{1}{\sqrt{n}}}^2\intpart{nt} =
t\;. 
\end{equation} 
\item	Par le th\'eor\`eme de la limite centrale (ou la formule de
de Moivre--Laplace appliqu\'ee \`a \eqref{mbma2}),
$X_{\intpart{nt}}/\sqrt{\intpart{nt}}$ converge en loi vers une variable normale
centr\'ee r\'eduite. Par cons\'equent, $B_t$ suit une loi normale $\cN(0,t)$.
\item	{\it Propri\'et\'es des incr\'ements ind\'ependants:}\/ pour tout
$t>s\geqs0$, $B_t-B_s$ est ind\'epen\-dant de $\set{B_u}_{0\leqs u\leqs s}$;
\item	{\it Propri\'et\'es des incr\'ements stationnaires:}\/ pour tout
$t>s\geqs0$, $B_t-B_s$ a la m\^eme loi que $B_{t-s}$.
\end{enum}


\section{Construction du mouvement Brownien}
\label{sec_mbc}

La construction heuristique ci-dessus motive la d\'efinition suivante. 

\begin{definition}
\label{def_mB}
Le {\em mouvement Brownien standard} ou {\em processus de Wiener standard}
est le processus stochastique $\set{B_t}_{t\geqs0}$
satisfaisant:
\begin{enum}
\item	$B_0 = 0$;
\item 	Incr\'ements ind\'ependants: pour tout $t>s\geqs 0$, $B_t-B_s$ est
ind\'ependant de $\set{B_u}_{u\leqs s}$;
\item	Incr\'ements gaussiens: pour tout $t>s\geqs0$, $B_t-B_s$ suit une 
loi normale $\cN(0,t-s)$. 
\end{enum}
\end{definition}

Nous allons maintenant donner une d\'emonstration de l'existence de ce
processus. De plus, nous montrons qu'il admet une version
continue, c'est-\`a-dire que ses trajectoires sont presque s\^urement
continues, et qu'on peut donc admettre qu'elles sont toutes continues quitte
\`a modifier le processus sur un ensemble de mesure nulle. 

\begin{figure}
\begin{center}
\begin{tikzpicture}[-,auto,node distance=1.0cm, thick,
main node/.style={draw,circle,fill=white,minimum size=4pt,inner sep=0pt},
full node/.style={draw,circle,fill=black,minimum size=4pt,inner sep=0pt}]

  \path[->,>=stealth'] 
     (-1,0) edge (9,0)
     (0,-2) edge (0,4)
  ;
  
  \path[color=yellow!50!orange]
     (0,0) edge (1,-1.5)
     node[main node] at(1,-1.5) {}
     edge (2,-1.2)
     (2,-1.2) edge (3,2.6)
     node[main node] at(3,2.6) {}
     edge (4,3.2)
     (4,3.2) edge (5,0.3)
     node[main node] at(5,0.3) {}
     edge[below right] 
     node[label={[label distance=-0.2cm]-45:$B_t^{(3)}$}] {} (6,1.2)
     (6,1.2) edge (7,3)
     node[main node] at(7,3) {}
     edge (8,2.4)
  ;

  \path[color=red] 
     (0,0) edge (2,-1.2) 
     node[main node] at(2,-1.2) {}
     edge (4,3.2)
     (4,3.2) edge 
     node[label={[label distance=-0.2cm]45:$B_t^{(2)}$}] {} (6,1.2)
     node[main node] at(6,1.2) {}
     edge (8,2.4) 
  ;

  \path[color=green!50!black] 
     (0,0) edge (4,3.2) 
     node[main node] at(4,3.2) {}
     edge[above]
     node[label={[label distance=0.0cm]90:$B_t^{(1)}$}] {} (8,2.4) 
  ;
  \path[color=blue] 
     (0,0) edge
     node[label={[label distance=-0.9cm]120:$B_t^{(0)}$}] {} (8,2.4) 
     node[main node] at(8,2.4) {}
  ;

  \draw (0,3.2) node[main node] {} node[left=0.1cm] {$X_{1/2}$} 
   -- (0,2.4) node[main node] {} node[left=0.1cm] {$X_{1}$} 
   -- (0,1.2) node[main node] {} node[left=0.1cm] {$X_{3/4}$} 
   -- (0,-1.2) node[main node] {} node[left=0.1cm] {$X_{1/4}$} 
   ;
     
  \draw (0,0) node[main node] {}
   -- (2,0) node[main node] {} node[below=0.1cm] {$\frac14$} 
   -- (4,0) node[main node] {} node[below=0.1cm] {$\frac12$} 
   -- (6,0) node[main node] {} node[below=0.1cm] {$\frac34$} 
   -- (8,0) node[main node] {} node[below=0.1cm] {$1$} 
   ;
\end{tikzpicture}
\end{center}
\vspace{-4mm}
 \caption[]{Construction du mouvement Brownien par interpolation.}
 \label{fig_bm_cons1}
\end{figure}
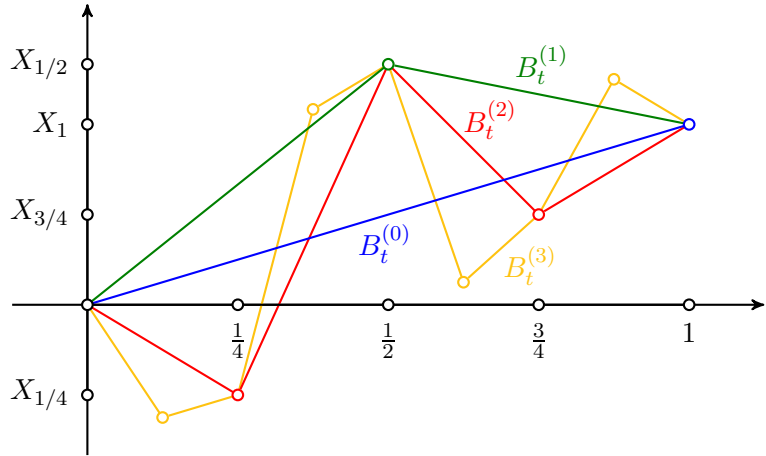

\begin{theorem}
\label{thm_cW}
Il existe un processus stochastique $\set{B_t}_{t\geqs0}$
satisfaisant la d\'efinition \ref{def_mB}, et dont les trajectoires
$t\mapsto B_t(\omega)$ sont continues. 
\end{theorem}
\begin{proof}\hfill
\begin{enum}
\item	Nous allons d'abord construire $\set{B_t}_{0\leqs t\leqs1}$ \`a
partir d'une collection de variables al\'ea\-toi\-res gaussiennes
ind\'ependantes $V_1, V_{1/2}, V_{1/4}, V_{3/4}, V_{1/8}, \dots$, toutes
centr\'ees, et avec $V_1$ et $V_{1/2}$ de variance $1$ et
$V_{k2^{-n}}$ de variance $2^{-(n-1)}$ ($k<2^n$ impair). 

Montrons d'abord que si $X_s$ et $X_t$ sont deux variables al\'eatoires
telles que $X_t-X_s$ soit gaussienne centr\'ee de variance $t-s$, alors il
existe une variable al\'eatoire $X_{(t+s)/2}$ telle que les variables 
$X_t - X_{(t+s)/2}$ et $X_{(t+s)/2} - X_s$
soient i.i.d.\ de loi $\cN(0,(t-s)/2)$. Si $U=X_t-X_s$ et $V$ est
ind\'ependante de $U$, de m\^eme distribution, il suffit de d\'efinir 
$X_{(t+s)/2}$ par 
\begin{align}
\nonumber
X_t - X_{(t+s)/2} &= \frac{U+V}2 \\
X_{(t+s)/2} - X_s &= \frac{U-V}2.
\label{cW9}
\end{align}
En effet, il est ais\'e de v\'erifier que ces variables ont la distribution
souhait\'ee, et qu'elles sont ind\'ependantes, puisque $\expec{(U+V)(U-V)} =
\expec{U^2}-\expec{V^2} = 0$, et que des variables al\'eatoires normales sont
ind\'ependantes si et seulement si elles sont non corr\'el\'ees.

Posons alors $X_0=0$, $X_1=V_1$, et construisons $X_{1/2}$ \`a l'aide de la
proc\'edure ci-dessus, avec $V=V_{1/2}$. Puis nous construisons $X_{1/4}$
\`a l'aide de $X_0$, $X_{1/2}$ et $V_{1/4}$, et ainsi de suite, pour obtenir
une familles de variables $\set{X_t}_{t=k2^{-n},n\geqs1,k<2^n}$ telles que
pour $t>s$, $X_t-X_s$ soit ind\'ependante de $X_s$ et de loi $\cN(0,t-s)$. 

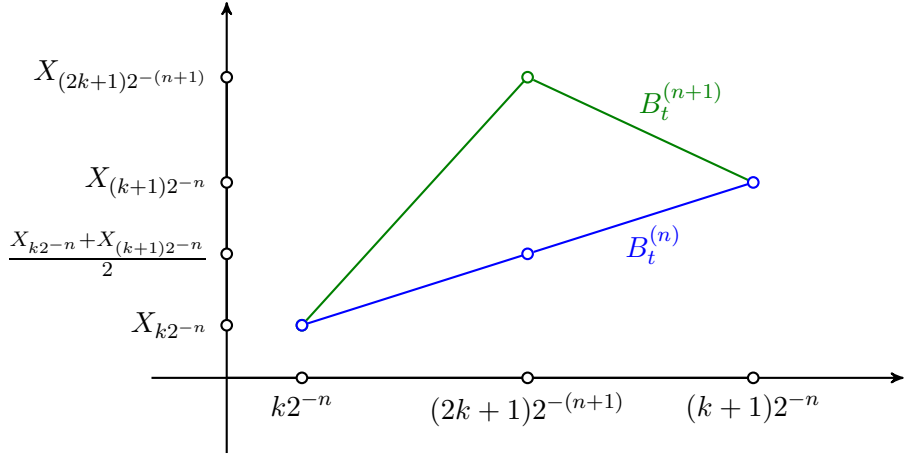
\begin{figure}
\begin{center}
\begin{tikzpicture}[-,auto,node distance=1.0cm, thick,
main node/.style={draw,circle,fill=white,minimum size=4pt,inner sep=0pt},
full node/.style={draw,circle,fill=black,minimum size=4pt,inner sep=0pt}]

  \path[->,>=stealth'] 
     (-1,0) edge (9,0)
     (0,-1) edge (0,5)
  ;
 
  \path[color=green!50!black] 
     node[main node] at(1,0.7) {} edge (4,4) 
     node[main node] at(4,4) {} edge 
     node[label={[label distance=-0.2cm]0:$B_t^{(n+1)}$}] {} (7,2.6) 
  ;
  \path[color=blue] 
     node[main node] at(1,0.7) {} edge (4,1.65) 
     node[main node] at(4,1.65) {} edge 
     node[label={[label distance=-1cm]130:$B_t^{(n)}$}] {} (7,2.6) 
     node[main node] at(7,2.6) {}
  ;

  \draw (0,4.0) node[main node] {} node[left=0.1cm] {$X_{(2k+1)2^{-(n+1)}}$} 
   -- (0,2.6) node[main node] {} node[left=0.1cm] {$X_{(k+1)2^{-n}}$} 
   -- (0,1.65) node[main node] {} node[left=0.1cm]
      {$\frac{X_{k2^{-n}}+X_{(k+1)2^{-n}}}{2}$} 
   -- (0,0.7) node[main node] {} node[left=0.1cm] {$X_{k2^{-n}}$} 
   ;
     
  \draw (0,0) 
   -- (1,0) node[main node] {} node[below=0.05cm] {$k2^{-n}$} 
   -- (4,0) node[main node] {} node[below=0.05cm] {$(2k+1)2^{-(n+1)}$} 
   -- (7,0) node[main node] {} node[below=0.05cm] {$(k+1)2^{-n}$} 
   ;
\end{tikzpicture}
\end{center}
\vspace{-4mm}
 \caption[]{Calcul de $\Delta^{(n)}$.}
 \label{fig_bm_cons2}
\end{figure}

\item	Pour $n\geqs0$, soit $\set{B^{(n)}_t}_{0\leqs t\leqs1}$ le processus
stochastique \`a trajectoires lin\'eaires par morceaux sur les intervalles
$[k2^{-n},(k+1)2^{-n}]$, $k<2^n$, et tel que
$\smash{B^{(n)}}_{k2^{-n}}=X_{k2^{-n}}$ (\figref{fig_bm_cons1}).
Nous voulons montrer que la suite des $\smash{B^{(n)}}(\omega)$ converge
uniform\'ement sur $[0,1]$ pour toute r\'ealisation $\omega$ des $V_i$.  
Il nous faut donc estimer
\begin{align}
\nonumber
\Delta^{(n)}(\omega) &= 
\sup_{0\leqs t\leqs 1} \bigabs{B^{(n+1)}_t(\omega)-B^{(n)}_t(\omega)} \\
\nonumber &=
\max_{0\leqs k\leqs 2^{n-1}}\max_{k2^{-n}\leqs t\leqs(k+1)2^{-n}}
\bigabs{B^{(n+1)}_t(\omega)-B^{(n)}_t(\omega)} \\
&=
\max_{0\leqs k\leqs 2^{n-1}}
\Bigabs{X_{(2k+1)2^{-(n+1)}}(\omega)
- \frac12\bigpar{X_{k2^{-n}}(\omega) 
+ X_{(k+1)2^{-n}}(\omega)}
}
\label{cW10}
\end{align}
(voir \figref{fig_bm_cons2}). 
Le terme en valeur absolue vaut $\frac12 V_{(2k+1)2^{-(n+1)}}$ par construction,
c.f.~\eqref{cW9}, qui est gaussienne de variance $2^{-n}$. Il suit que 
\begin{align}
\nonumber
\bigprob{\Delta^{(n)}>\sqrt{n2^{-n}}} 
&= \Bigprob{\max_{0\leqs k\leqs 2^{n-1}} \bigabs{V_{(2k+1)2^{-(n+1)}}} \geqs
2\sqrt{n2^{-n}}} \\
\nonumber
&\leqs 2\cdot2^n \int_{2\sqrt{n2^{-n}}}^\infty \e^{-x^2/2\cdot2^{-n}} 
\frac{\6x}{\sqrt{2\pi2^{-n}}} \\
&= 2\cdot2^n \int_{2\sqrt{n}}^\infty \e^{-y^2/2} 
\frac{\6x}{\sqrt{2\pi}} \leqs \const 2^n \e^{-2n}, 
\label{cW11}
\end{align}
et donc 
\begin{equation}
\label{cW12}
\sum_{n\geqs0} \bigprob{\Delta^{(n)}>\sqrt{n2^{-n}}} 
\leqs \const \sum_{n\geqs0} (2\e^{-2})^n < \infty. 
\end{equation}
Le lemme de Borel--Cantelli nous permet de conclure qu'avec probabilit\'e
$1$, il n'existe qu'un nombre fini de $n$ pour lesquels
$\Delta^{(n)}>\sqrt{n2^{-n}}$. Par cons\'equent, 
\begin{equation}
\label{cW13}
\Bigprob{\sum_{n\geqs0}\Delta^{(n)} < \infty} = 1. 
\end{equation}
La suite des $\set{B^{(n)}_t}_{0\leqs t\leqs1}$ est donc une suite de Cauchy
pour la norme sup avec probabilit\'e $1$, et alors elle converge
uniform\'ement. Nous posons pour $t\in[0,1]$
\begin{equation}
\label{cW14}
B^0_t = 
\begin{cases}
\lim_{n\to\infty} B^{(n)}_t
& \text{si la suite converge uniform\'ement} \\
0 & \text{sinon (avec probabilit\'e $0$).}
\end{cases}
\end{equation}
Il est facile de v\'erifier que $B^0$ satisfait les trois propri\'et\'es de
la d\'efinition. 

\item	Pour \'etendre le processus \`a des temps quelconques, nous
fabriquons des copies ind\'e\-pen\-dantes $\set{B^i}_{i\geqs0}$ et posons 
\begin{equation}
\label{cW15}
B_t = 
\begin{cases}
B^0_t & 0\leqs t < 1 \\ 
B^0_1 + B^1_{t-1} & 1\leqs t < 2 \\ 
B^0_1 + B^1_1 + B^2_{t-2} & 2\leqs t < 3 \\ 
\dots &
\end{cases}
\end{equation}
Ceci conclut la d\'emonstration. \qed
\end{enum}
\renewcommand{\qed}{}
\end{proof}


\section{Propri\'et\'es de base}
\label{sec_mbp}

Les propri\'et\'es suivantes sont des cons\'equences directes de la
d\'efinition~\ref{def_mB} et nous les donnons sans d\'emonstration
d\'etaill\'ee. 

\begin{enum}
\item	{\it Propri\'et\'e de Markov:}\/ Pour tout borelien $A\in\R$,
\begin{equation}
\label{pW1}
\bigpcond{B_{t+s}\in A}{B_t=x} = \int_A p(y,t+s|x,t)\6y,
\end{equation}
ind\'ependamment de $\set{B_u}_{u<t}$, avec des probabilit\'es de
transition gaussiennes
\begin{equation}
\label{pW2}
p(y,t+s|x,t) = \frac{\e^{-(y-x)^2/2s}}{\sqrt{2\pi s}}. 
\end{equation}
La preuve suit directement du fait que l'on peut d\'ecomposer 
$B_{t+s}=B_t + (B_{t+s}-B_t)$, le deuxi\`eme terme \'etant ind\'ependant du
premier et de loi $\cN(0,s)$. On v\'erifiera en particulier l'\'equation de
Chapman--Kolmogorov: Pour $t>u>s$, 
\begin{equation}
\label{pW3}
p(y,t|x,s) = \int_\R p(y,t|z,u)p(z,u|x,s)\6u.
\end{equation}

\item	{\it Propri\'et\'e diff\'erentielle:}\/ Pour tout $t\geqs0$,
$\set{B_{t+s}-B_t}_{s\geqs0}$ est un mouvement Brownien standard,
ind\'ependant de $\set{B_u}_{u<t}$. 

\item	{\it Propri\'et\'e d'\'echelle:}\/ Pour tout $c>0$,
$\set{cB_{t/c^2}}_{s\geqs0}$ est un mouvement Brownien standard. 

\item	{\it Sym\'etrie:}\/ $\set{-B_t}_{t\geqs0}$ est un mouvement Brownien
standard. 

\item	{\it Processus Gaussien:}\/ Le processus de Wiener est Gaussien de
moyenne nulle (c'est-\`a-dire que ses distributions jointes finies sont
normales centr\'ees), et il est caract\'eris\'e par sa covariance 
\begin{equation}
\label{pW4}
\cov\set{B_t,B_s} \equiv \expec{B_tB_s} = s \wedge t 
\end{equation}
($s \wedge t$ d\'esigne le minimum de $s$ et $t$). 
\begin{proof}
Pour $s<t$, nous avons 
\begin{equation}
\label{pW5}
\expec{B_tB_s} = \expec{B_s(B_s+B_t-B_s)} = \expec{B_s^2} +
\expec{B_s(B_t-B_s)} = s,
\end{equation}
puisque le deuxi\`eme terme est nul par la propri\'et\'e d'incr\'ements
ind\'ependants. 
\end{proof}
En fait, un processus Gaussien centr\'e, dont la covariance
satisfait~\eqref{pW4} est un processus de Wiener standard. 

\item	{\it Renversement du temps:}\/ Le processus $\set{X_t}_{t\geqs0}$
d\'efini par 
\begin{equation}
\label{pW6}
X_t = 
\begin{cases}
0 & \text{si $t=0$} \\
tB_{1/t} & \text{si $t>0$}
\end{cases}
\end{equation}
est un processus de Wiener standard.
\begin{proof}
$X_t$ est un processus Gaussien de moyenne nulle, et un calcul simple montre
que sa covariance est bien $\cov\set{X_t,X_s} = s \wedge t$. Le seul point
non trivial est de montrer la continuit\'e en z\'ero de $X_t$. Celle-ci est
\'equivalente \`a la loi forte des grands nombres. 
\end{proof}
\end{enum}

Voici enfin une propri\'et\'e importante de non-r\'egularit\'e des trajectoires
du mouvement Brownien.

\begin{theorem}
\label{thm_mB2}
Les trajectoires $t\mapsto B_t(\omega)$ sont presque s\^urement nulle part
lipschitziennes, donc aussi nulle part diff\'eren\-tiables.  
\end{theorem}
\begin{proof}
Fixons un $C<\infty$ et introduisons, pour $n\geqs1$, 
\begin{equation*}
A_n = \Bigsetsuch{\omega}{\exists s\in[0,1] \text{ t.q. }
\abs{B_t(\omega)-B_s(\omega)}\leqs C\abs{t-s} \text{ si } \abs{t-s}\leqs
\frac3n}\;.
\end{equation*}
Il s'agit de montrer que $\fP(A_n)=0$ pour tout $n$. 
Observons que si $n$ augmente, la condition s'affaiblit, donc $A_n\subset
A_{n+1}$. Soit encore, pour $n\geqs3$ et $1\leqs k\leqs n-2$, 
\begin{align*}
Y_{k,n}(\omega) &= \max_{j=0,1,2}\Bigset{\Bigabs{B_{(k+j)/n}(\omega) -
B_{(k+j-1)/n}(\omega)}}\;,\\
B_n &= \bigcup_{k=1}^{n-2} 
\Bigsetsuch{\omega}{Y_{k,n}(\omega)\leqs
\frac{5C}{n}}\;.
\end{align*}
L'in\'egalit\'e triangulaire implique $A_n\subset B_n$. En effet, 
soit $\omega\in A_n$. Si par exemple $s=1$, alors pour $k=n-2$
\begin{equation*}
\bigabs{B_{(n-3)/n}(\omega) - B_{(n-2)/n}(\omega)}
\leqs \bigabs{B_{(n-3)/n}(\omega) - B_1(\omega)} 
+ \bigabs{B_1(\omega) - B_{(n-2)/n}(\omega)}
\leqs C\biggpar{\frac 3n + \frac 2n}
\end{equation*} 
donc $\omega\in B_n$.
Il suit des propri\'et\'es des incr\'ements ind\'ependants et d'\'echelle
que 
\begin{equation*}
\fP(A_n) \leqs \fP(B_n) \leqs n \fP\biggpar{\abs{B_{1/n}}\leqs \frac{5C}n}^3
= n \fP\biggpar{\abs{B_1}\leqs \frac{5C}{\sqrt{n}}}^3
\leqs n \biggpar{\frac{10C}{\sqrt{2\pi n}}}^3\;.
\end{equation*} 
Il suit que $\fP(A_n)\to0$ pour $n\to\infty$. Mais comme
$\fP(A_n)\leqs\fP(A_{n+1})$ pour tout $n$, ceci implique $\fP(A_n)=0$ pour tout
$n$. 
\end{proof}


\section{Temps d'arr\^et}
\label{sec_mbta}

Soit $\set{B_t}_{t\geqs0}$ un mouvement Brownien standard. Nous allons laisser
de c\^ot\'e les d\'etails techniques de la construction de l'espace
probabilis\'e $(\Omega,\cF,\fP)$ sur lequel $B_t$ est d\'efini (la
construction utilise le th\'eor\`eme d'extension de Kolmogorov). Pr\'ecisons
simplement que l'univers $\Omega$ peut \^etre identifi\'e \`a l'ensemble des
fonctions continues $f:\R_+\to\R$ telles que $f(0)=0$. Introduisons pour tout
$t\geqs0$ la tribu 
\begin{equation}
\label{mbta1}
\cF_t = \sigma\bigpar{\set{B_s}_{0\leqs s\leqs t}}\;, 
\end{equation} 
contenant les \'ev\'enements ne d\'ependant que du mouvement Brownien jusqu'au
temps $t$. La suite $\set{\cF_t}_{t\geqs0}$ est une suite croissante de
sous-tribus de $\cF$, c'est la \defwd{filtration canonique}\/ engendr\'ee par le
mouvement Brownien.\footnote{Pour \^etre pr\'ecis, on pr\'ef\`ere consid\'erer
$\cF'_t=\bigcap_{s>t}\cF_s$, qui est continue \`a droite~:
$\cF'_t=\bigcap_{s>t}\cF'_s$. Pour des raisons techniques, on compl\`ete en
g\'en\'eral cette tribu par les ensembles de mesure nulle
(c.f.~\cite[Section~7.2]{Durrett}).}

\begin{definition}[Temps d'arr\^et]
\label{def_stopping_Bt}
On appelle\/ \defwd{temps d'arr\^et}\/ une variable al\'eatoire
$\tau:\Omega\to[0,\infty]$ telle que $\set{\tau<t}\in\cF_t$ pour tout
$t\geqs0$.
\end{definition}

Remarquons que si $\set{\tau\leqs t}\in\cF_t$, alors 
\begin{equation}
\label{mbta2} 
\bigset{\tau<t} = \bigcup_{n\geqs1}\biggset{\tau\leqs t-\frac1n} \in\cF_t\;.
\end{equation} 
La continuit\'e \`a droite de la filtration montre en plus que si
$\set{\tau<t}\in\cF_t$, alors 
\begin{equation}
\label{mbta3} 
\bigset{\tau\leqs t} = \bigcap_{n\geqs1}\biggset{\tau< t+\frac1n}
\in\cF_t\;.
\end{equation} 
On peut donc remplacer sans autres $\tau<t$ par $\tau\leqs t$ dans la
d\'efinition des temps d'arr\^et (ce qui n'est pas le cas pour les processus en
temps discret).

Commen\c cons par \'enoncer quelques propri\'et\'es de base permettant de
construire des exemples de temps d'arr\^et. 

\begin{prop}\hfill
\label{prop_mbta1}
\begin{enum}
\item	Soit $A$ un ouvert. Alors le temps de premier passage dans $A$,
$\tau_A=\inf\setsuch{t\geqs0}{B_t\in A}$ est un temps d'arr\^et.
\item	Si $\tau_n$ est une suite d\'ecroissante de temps d'arr\^et telle que
$\tau_n\to\tau$, alors $\tau$ est un temps d'arr\^et.
\item	Si $\tau_n$ est une suite croissante de temps d'arr\^et telle que
$\tau_n\to\tau$, alors $\tau$ est un temps d'arr\^et.
\item	Soit $K$ un ferm\'e. Alors le temps de premier passage dans $K$
est un temps d'arr\^et.
\end{enum}
\end{prop}
\begin{proof}\hfill
\begin{enum}
\item	Par continuit\'e de $B_t$,
$\set{\tau_A<t}=\bigcup_{q\in\Q,q<t}\set{B_q\in A}\in\cF_t$.
\item	On a $\set{\tau<t}=\bigcup_{n\geqs1}\set{\tau_n<t}\in\cF_t$.
\item	On a $\set{\tau\leqs t}=\bigcap_{n\geqs1}\set{\tau_n\leqs t}\in\cF_t$.
\item	Soit $A_n=\bigcup_{x\in K} (x-1/n,x+1/n)$ et $\tau_n$ le temps de
premier passage dans $A_n$. Comme $A_n$ est ouvert, $\tau_n$ est un temps
d'arr\^et par 1. Les $A_n$ \'etant d\'ecroissants, les $\tau_n$ sont croissants.
Soit $\sigma=\lim_{n\to\infty}\tau_n\in[0,\infty]$. Comme $\tau_K\geqs\tau_n$
pour tout $n$, on a $\sigma\leqs\tau_K$. Si $\sigma=\infty$, alors
$\tau_K=\infty=\sigma$. Supposons que $\sigma<\infty$. Comme
$B_{\tau_n}\in\overbar{A}_n$ et $B_{\tau_n}\to B_\sigma$, il suit que
$B_\sigma\in K$ et donc $\tau_K\leqs\sigma$. Ainsi $\tau_K=\sigma$ et le
r\'esultat suit de 3.
\qed
\end{enum}
\renewcommand{\qed}{}
\end{proof}

Comme dans le cas \`a temps discret, nous introduisons 

\begin{definition}[Tribu de \'ev\'enements ant\'erieurs]
\label{def_pretau_mb} 
Soit $\tau$ un temps d'arr\^et. Alors la tribu 
\begin{equation}
\label{mbta4}
\cF_\tau = \bigsetsuch{A\in\cF}{A\cap\set{\tau\leqs t}\in\cF_t\;\forall
t\geqs0} 
\end{equation} 
est appel\'ee la\/ \defwd{tribu des \'ev\'enements ant\'erieurs \`a $\tau$}.
\end{definition}

\begin{prop}\hfill
\label{prop_mbta2}
\begin{enum}
\item	Si $\sigma$ et $\tau$ sont des temps d'arr\^et tels que
$\sigma\leqs\tau$ alors $\cF_\sigma\subset\cF_\tau$. 
\item	Si $\tau_n\to\tau$ est une suite d\'ecroissante de temps d'arr\^et,
alors $\cF_\tau=\bigcap_n\cF_{\tau_n}$.
\end{enum}
\end{prop}
\begin{proof}\hfill
\begin{enum}
\item	Voir la preuve de la Proposition~\ref{prop_tarret2}.
\item	D'une part par 1., on a $\cF_{\tau_n}\supset\cF_\tau$ pour tout $n$.
D'autre part, soit $A\in\bigcap_n\cF_{\tau_n}$. Comme
$A\cap\set{\tau_n<t}\in\cF_t$ et $\tau_n$ d\'ecro\^\i t vers $\tau$, on a 
$A\cap\set{\tau<t}\in\cF_t$, donc $\bigcap_n\cF_{\tau_n}\subset\cF_\tau$.
\qed
\end{enum}
\renewcommand{\qed}{}
\end{proof}

L'une des propri\'et\'es les plus importantes du mouvement Brownien, li\'ee aux
temps d'arr\^et, est la propri\'et\'e de Markov forte. 

\begin{theorem}[Propri\'et\'e de Markov forte du mouvement Brownien]
\label{thm_markov_fort}
Si $\tau$ est un temps d'arr\^et tel que $\prob{\tau<\infty}=1$, alors
$\set{B_{\tau+t}-B_\tau}_{t\geqs0}$ est un mouvement Brownien standard,
ind\'ependant de $\cF_\tau$.  
\end{theorem}
\begin{proof}
Voir \cite[Section~7.3]{Durrett} ou \cite{McKean}, p.~10. 
\end{proof}

\begin{cor}[Principe de r\'eflexion]
\label{cor_apr1}
Pour $L>0$ et $\tau=\inf\setsuch{t\geqs0}{B_t\geqs L}$, le processus 
\begin{equation}
\label{apr3}
B^\star_t = 
\begin{cases}
B_t & \text{si $t\leqs\tau$} \\
2L - B_t & \text{si $t>\tau$}
\end{cases}
\end{equation}
est un mouvement Brownien standard. 
\end{cor}

Voici une application classique du principe de
r\'eflexion. 

\begin{cor}
\label{cor_apr2}
Pour tout $L\geqs0$, 
\begin{equation}
\label{apr4}
\Bigprob{\sup_{0\leqs s\leqs t}B_s\geqs L} = 2 \prob{B_t\geqs L}. 
\end{equation}
\end{cor}
\begin{proof}
Si $\tau=\inf\setsuch{t\geqs0}{B_t\geqs L}$, alors $\set{\sup_{0\leqs s\leqs
t}B_s\geqs L}=\set{\tau\leqs t}$. Les trajectoires du mouvement Brownien
\'etant continues, $B_t\geqs L$ implique $\tau\leqs t$. Ainsi, 
\begin{align}
\nonumber
\prob{B_t\geqs L} &= \prob{B_t\geqs L, \tau\leqs t} \\
\nonumber
&= \prob{2L - B_t \leqs L, \tau\leqs t} \\
\nonumber
&= \prob{B^\star_t \leqs L, \tau\leqs t} \\
&= \prob{B_t \leqs L, \tau\leqs t}.
\label{apr5}
\end{align}
Le r\'esultat suit de la d\'ecomposition $\prob{\tau\leqs t} = \prob{B_t \leqs
L, \tau\leqs t} + \prob{B_t > L, \tau\leqs t}$. 
\end{proof}

On remarquera que lorsque $t\to\infty$, l'expression~\eqref{apr4} tend vers
$1$, donc le mouvement Brownien atteint presque s\^urement tout niveau $L$.


\section{Mouvement Brownien et martingales}
\label{sec_mbm}

Le mouvement Brownien, ainsi que toute une s\'erie de processus d\'eriv\'es,
sont des martingales. Les martingales, sur- et sous-martingales sont d\'efinies
comme dans le cas discret, sauf qu'on consid\`ere tous les temps $t>s$, pour
lesquels $\cF_t\supset\cF_s$. 
Commen\c cons par consid\'erer le mouvement Brownien.

\begin{theorem}[Propri\'et\'e de martingale du mouvement Brownien]
\label{thm_bm_mart}
Le mouvement Brownien est une martingale par rapport \`a la filtration
canonique $\set{\cF_t}_{t\geqs0}$.  
\end{theorem}
\begin{proof}
Pour tout $t>s\geqs0$, on a 
\begin{equation*}
\econd{B_t}{\cF_s} = \econd{B_t-B_s}{\cF_s} + \econd{B_s}{\cF_s} 
= \expec{B_{t-s}} + B_s = B_s\;,
\end{equation*} 
en vertu de la propri\'et\'e diff\'erentielle et de la propri\'et\'e des
incr\'ements ind\'ependants. 
\end{proof}

L'in\'egalit\'e de Jensen implique que $B_t^2$ et $\e^{\gamma B_t}$ pour
$\gamma>0$ sont des sous-martingales. Le r\'esultat suivant montre qu'en les
modifiant de mani\`ere d\'eterministe, on obtient des martingales~:

\begin{prop}\hfill
\label{prop_mD}
\begin{enum}
\item	$B_t^2-t$ est une martingale.
\item	Pour tout $\gamma\in\R$, $\exp(\gamma B_t - \gamma^2t/2)$ est une
martingale. 
\end{enum}
\end{prop}
\begin{proof}\hfill
\begin{enum}
\item	$\econd{B_t^2}{\cF_s} = \econd{B_s^2+2B_s(B_t-B_s)+(B_t-B_s)^2}{\cF_s}
= B_s^2+0+(t-s)$.
\item	$\econd{\e^{\gamma B_t}}{\cF_s} 
= \e^{\gamma B_s}\econd{\e^{\gamma (B_t-B_s)}}{\cF_s}
= \e^{\gamma B_s}\expec{\e^{\gamma B_{t-s}}}$, et 
\[
\bigexpec{\e^{\gamma B_{t-s}}} = \int_{-\infty}^\infty \e^{\gamma x}
\frac{\e^{-x^2/2(t-s)}}{\sqrt{2\pi(t-s)}} \6x
= \e^{\gamma^2(t-s)/2}
\]
par compl\'etion du carr\'e. 
\qed
\end{enum}
\renewcommand{\qed}{}
\end{proof}

On peut donc \'ecrire 
\begin{equation}
 \label{mbm0}
B_t^2 = t + M_t\;,
\qquad
\text{$M_t$ martingale\;,} 
\end{equation} 
ce qui est l'analogue de la d\'ecomposition de Doob du cas en temps discret.
Le processus croissant associ\'e au mouvement Brownien est donc donn\'e par 
\begin{equation}
 \label{mbm2}
\braket{B}_t = t\;. 
\end{equation} 
Ceci peut \^etre vu intuitivement comme cons\'equence du fait que 
\begin{equation}
 \label{mbm1}
\bigecond{(B_t-B_s)^2}{\cF_s} = \bigexpec{(B_t-B_s)^2} = \bigexpec{B_{t-s}^2} =
t-s\;,
\end{equation} 
et donc que l’accroissement infinit\'esimal de $B_t^2$ est en moyenne 
$\econd{(B_{s+\6s}-B_s)^2}{\cF_s}=\6s$.

Nous donnons deux types d'applications de la propri\'et\'e de martingale.
Un premier type d'applications suit du fait qu'une martingale arr\^et\'ee est
encore une martingale (c.f.~Section~\ref{sec_indoob}). 

\begin{prop}
\label{prop_mbm1}
Soit $X_t$ une martingale continue \`a droite, adapt\'ee \`a une filtration
continue \`a droite. Si $\tau$ est un temps d'arr\^et born\'e, alors
$\expec{X_\tau}=\expec{X_0}$. 
\end{prop}
\begin{proof}
Soit $n$ un entier tel que $\prob{\tau\leqs n-1}=1$, et
$\tau_m=(\intpart{2^m\tau}+1)/2^m$. Le processus discr\'etis\'e 
$\set{Y^m_k}_{k\geqs1}=X_{k2^{-m}}$ est une martingale pour la filtration
$\cF^m_k=\cF_{k2^{-m}}$, et $\sigma_m=2^m\tau_m$ est un temps d'arr\^et pour
cette filtration. Il suit que 
\begin{equation}
\label{mbm3}
X_{\tau_m} = Y^m_{\sigma_m} = \econd{Y^m_{n2^m}}{\cF^m_{\sigma_m}} 
= \econd{X_n}{\cF_{\tau_m}}\;.
\end{equation} 
Lorsque $m\to\infty$, $X_{\tau_m}$ tend vers $X_\tau$ par continuit\'e \`a
droite, et $\cF_{\tau_m}\searrow\cF_\tau$ par la Proposition~\ref{prop_mbta2}. 
Ceci implique que $X_\tau=\econd{X_n}{\cF_\tau}$. Par la propri\'et\'e de
martingale, il suit que $\expec{X_\tau}=\expec{X_n}=\expec{X_0}$.
\end{proof}

Voici quelques applications de ce r\'esultat. 

\begin{cor}\hfill
\label{cor_mbm1} 
\begin{enum}
\item	Soit $\tau_a = \inf\setsuch{t\geqs0}{x+B_t=a}$. Alors pour $a<x<b$ on a 
\begin{equation}
 \label{mbm4}
\prob{\tau_a < \tau_b} = \frac{b-x}{b-a}\;. 
\end{equation} 
\item	Soit $a<0<b$, et $\tau=\inf\setsuch{t\geqs0}{B_t\not\in(a,b)}$. Alors 
\begin{equation}
 \label{mbm5}
\expec{\tau} = \abs{ab}\;. 
\end{equation}
\item	Soit $a>0$ et $\tau=\inf\setsuch{t\geqs0}{B_t\not\in(-a,a)}$. Alors
pour $\lambda>0$, 
\begin{equation}
 \label{mbm6}
\expec{\e^{-\lambda\tau}} = \frac{1}{\cosh(a\sqrt{2\lambda})}\;. 
\end{equation}  
\end{enum}
\end{cor}
\begin{proof}\hfill
\begin{enum}
\item	Soit $\tau=\tau_a\wedge\tau_b$. Alors $\expec{x+B_{\tau\wedge t}}=x$. 
Le Corollaire~\ref{cor_apr2} montre que $\tau_a$ et $\tau_b$ sont finis presque
s\^urement, donc aussi $\tau$.  
Faisant tendre $t$ vers l'infini, on obtient 
\begin{equation*}
x = \expec{x+B_\tau} = a \prob{\tau_a<\tau_b} + b \prob{\tau_b<\tau_a}\;,
\end{equation*} 
et le r\'esultat suit en r\'esolvant par rapport \`a $\prob{\tau_a<\tau_b}$. 

\item	Comme $B_t^2-t$ est une martingale, on a $\expec{B^2_{\tau\wedge t}} =
\expec{\tau\wedge t}$. Le th\'eor\`eme de convergence domin\'ee
et~\eqref{mbm4} permettent d'\'ecrire 
\[
\expec{\tau} 
= \lim_{t\to\infty} \expec{\tau\wedge t} 
= \lim_{t\to\infty} \expec{B^2_{\tau\wedge t}} 
= \expec{B^2_\tau} 
= a^2 \frac{b}{b-a} + b^2 \biggpar{1-\frac{b}{b-a}} 
= -ab\;.
\]

\item	Comme $\e^{\gamma B_t - \gamma t^2/2}$ est une martingale, on a 
$\expec{\e^{\gamma B_{\tau\wedge t}-\gamma^2 (\tau\wedge t)/2}}=1$. Faisant
tendre $t$ vers l'infini et invoquant le th\'eor\`eme de convergence domin\'ee,
on obtient 
\[
\expec{\e^{\gamma B_\tau - \gamma^2\tau/2}} = 1\;.
\]
Par sym\'etrie, on a $\prob{B_\tau=-a}=\prob{B_\tau=a}=1/2$ et $B_\tau$ est
ind\'ependant de $\tau$. Par cons\'equent, 
\[
\expec{\e^{\gamma B_\tau - \gamma^2\tau/2}} = 
\cosh(\gamma a) \expec{e^{-\gamma^2\tau/2}}\;,
\]
et le r\'esultat suit en prenant $\gamma=\sqrt{2\lambda}$. 
\qed
\end{enum}
\renewcommand{\qed}{}
\end{proof}

Remarquons que $\expec{\e^{-\lambda\tau}}$ est la transform\'ee de Laplace de
la densit\'e de $\tau$, donc l'expression~\eqref{mbm6} permet de trouver cette
densit\'e par transform\'ee de Laplace inverse. 

\goodbreak

Le second type d'applications que nous consid\'erons suit de l'in\'egalit\'e
de Doob. 

\begin{theorem}[In\'egalit\'e de Doob]
\label{thm_Doob_cont} 
Soit $X_t$ une sous-martingale \`a trajectoires continues. Alors 
pour tout $\lambda>0$, 
\begin{equation}
 \label{mbm7}
\biggprob{\sup_{0\leqs s\leqs t} X_s \geqs \lambda} 
\leqs \frac{1}{\lambda} \expec{X_t^+}\;.
\end{equation} 
\end{theorem}
\begin{proof}
Voir la preuve du Th\'eor\`eme~\ref{thm_Doob}.
\end{proof}

\begin{cor}
\label{cor_mbm2}
Pour tout $\gamma, \lambda>0$, 
\begin{equation}
 \label{mbm8}
\biggprob{\sup_{0\leqs s\leqs t} \Bigbrak{B_s-\gamma\frac s2}>\lambda} \leqs
\e^{-\gamma\lambda}\;. 
\end{equation}  
\end{cor}


\section{Exercices}
\label{sec_exo_mb}

\begin{exercice}
\label{exo_mb1} 
Montrer que pour tout $\gamma\in\R$, le processus
$X_t=\e^{-\gamma^2t/2}\cosh(\gamma B_t)$ est une martingale. 

\noindent
En d\'eduire une autre preuve du fait que
$\tau=\inf\setsuch{t\geqs0}{B_t\not\in(-a,a)}$ satisfait 
\[
\expec{\e^{-\lambda\tau}} = \frac{1}{\cosh(a\sqrt{2\lambda})}\;.
\]
\end{exercice}

\goodbreak

\begin{exercice}
\label{exo_mb2} 
\hfill 
\begin{enum}
\item	Montrer que si $X_t=f(B_t,t,\gamma)$ est une martingale, alors
(sous des conditions de r\'egularit\'e qu'on pr\'ecisera)
$\dtot{}{\gamma}f(B_t,t,\gamma)$ est \'egalement une martingale.
\item	Soit $f(x,t,\gamma)=\e^{\gamma x-\gamma^2 t/2}$. Calculer le
d\'eveloppement limit\'e de $f$ en $\gamma=0$ jusqu'\`a l'ordre $\gamma^4$. En
d\'eduire deux nouvelles martingales d\'eriv\'ees du mouvement Brownien. 
\item	Soit $\tau=\inf\setsuch{t\geqs0}{B_t\not\in(-a,a)}$. Calculer
$\expec{\tau^2}$ \`a l'aide du r\'esultat pr\'ec\'edent. 
\end{enum}
\end{exercice}

\goodbreak

\begin{exercice}
\label{exo_mb3} 
Pour $a,b>0$, on pose $X_t=B_t-bt$ et $\tau=\inf\setsuch{t\geqs0}{X_t=a}$. 

\begin{enum}
\item	En utilisant la martingale $\e^{\gamma B_t-\gamma^2 t^2/2}$, o\`u
$\gamma$ est la solution positive de $\gamma^2-2b\gamma-2\lambda=0$, calculer
\[
\expec{\e^{-\lambda\tau}\indexfct{\tau<\infty}}\;. 
\]
\item	Particulariser au cas $b=0$.
\item	Soit $b>0$. En choisissant une valeur convenable de $\lambda$,
d\'eterminer $\prob{\tau<\infty}$.
\end{enum}
\end{exercice}

\goodbreak

\begin{exercice}
\label{exo_mb4} 

Le but de ce probl\`eme est de montrer, de trois mani\`eres diff\'erentes, que
la premi\`ere intersection d'un mouvement Brownien bidimensionnel avec une 
droite suit une loi de Cauchy.

Soient $\set{B^{(1)}_t}_{t\geqs0}$ et $\set{B^{(2)}_t}_{t\geqs0}$ deux
mouvements Browniens standard ind\'ependants. On d\'enote par $\cF_t$ la
filtration engendr\'ee par $(B^{(1)}_t,B^{(2)}_t)$. Soit $\set{X_t}_{t\geqs0}$
le processus \`a valeurs dans $\C$ donn\'e par 
\[
X_t = \icx + B^{(1)}_t + \icx B^{(2)}_t\;.
\]
Soit 
\[
\tau = \inf\setsuch{t>0}{X_t\in\R} 
= \inf\setsuch{t>0}{1 + B^{(2)}_t = 0} \;.
\]

\begin{enumerate}
\item 	Approche martingale. 
\begin{enum}
\item 	Montrer que $Y_t=\e^{\icx\lambda X_t}$ est une martingale pour tout
$\lambda\geqs0$. 

\item 	Calculer $\expec{\e^{\icx
\lambda X_\tau}}$ pour tout $\lambda\geqs0$, puis pour tout $\lambda\in\R$. 

\item 	En d\'eduire la loi de $X_\tau$.
\end{enum}

\item	Approche principe de r\'eflexion.
\begin{enum}
\item 	D\'eterminer la densit\'e de $\smash{B^{(1)}_t}$. 
\item	En utilisant le principe de r\'eflexion, calculer 
$\prob{\tau < t}$ et en d\'eduire la densit\'e de $\tau$.
La propri\'et\'e d'\'echelle du Brownien permet de simplifier les calculs.
\item 	En d\'eduire la loi de $X_\tau = B^{(1)}_\tau$. 
\end{enum}

\item 	Approche invariance conforme. 

On rappelle qu'une application $f:\C\to\C$ est conforme si elle pr\'eserve les
angles. On admettra le r\'esultat suivant: le mouvement brownien bidimensionnel
est invariant conforme, c'est-\`a-dire que son image sous une application
conforme est encore un mouvement brownien bidimensionnel (\`a un changement de
temps pr\`es).  

\begin{enum}
\item	Soit l'application conforme $f:\C\to\C$ d\'efinie par 
\[
f(z) = \frac{z-\icx}{z+\icx}\;.
\]
V\'erifier que c'est une bijection du demi-plan sup\'erieur 
$\fH=\setsuch{z\in\C}{\im z >0}$ dans le disque unit\'e 
$\fD=\setsuch{z\in\C}{\abs{z}<1}$ qui envoie $\icx$ sur $0$. 

\item 	Par un argument de sym\'etrie, 
donner la loi du lieu de sortie de $\fD$ du mouvement Brownien issu de
$0$. En d\'eduire la loi de $X_\tau$. 
\end{enum}
\end{enumerate}
\end{exercice}


\chapter{L'int\'egrale d'It\^o}
\label{chap_is}

Le but de l'int\'egrale d'It\^o est de donner un sens \`a des
\'equations de la forme 
\begin{equation}
\label{is1}
\dtot Xt = f(X) + g(X) \dtot{B_t}{t}\;.
\end{equation}
Par exemple, si $f\equiv0$ et $g\equiv1$, on devrait retrouver
$X_t=X_0+B_t$, d\'ecrivant le mouvement suramorti d'une particule
Brownienne. 

Le probl\`eme est que, comme nous l'avons mentionn\'e, les trajectoires du
processus de Wiener ne sont pas diff\'erentiables, ni m\^eme \`a variations
born\'ees. 

Comme dans le cas des \'equations diff\'erentielles ordinaires, on
interpr\`ete une solution de l'\'equation diff\'erentielle~\eqref{is1} comme
une solution de l'\'equation int\'egrale 
\begin{equation}
\label{is2}
X_t = X_0 + \int_0^t f(X_s)\6s + \int_0^t g(X_s) \6B_s\;. 
\end{equation}
C'est \`a la deuxi\`eme int\'egrale qu'il s'agit de donner un sens
math\'ematique. Si $s\mapsto g(X_s)$ \'etait diff\'erentiable, on pourrait
le faire \`a l'aide d'une int\'egration par parties, mais ce n'est en
g\'en\'eral pas le cas. It\^o a donn\'e une autre d\'efinition de
l'int\'egrale stochastique, qui s'applique \`a une classe beaucoup plus
vaste d'int\'egrants (et donne le m\^eme r\'esultat que l'int\'egration par
parties dans le cas diff\'erentiable). 


\section{D\'efinition}
\label{sec_ito}

Notre but est de d\'efinir l'int\'egrale stochastique 
\begin{equation}
\label{ito1}
\int_0^t X_s \6B_s
\end{equation}
simultan\'ement pour tous les $t\in[0,T]$, o\`u $X_t$ est lui-m\^eme un
processus stochastique. Plus pr\'ecis\'ement, nous supposerons que $X_t$ est
une {\em fonctionnelle Brownienne non-anticipative}, c'est-\`a-dire
($\set{\cF_t}_{t\geqs0}$ d\'esignant la filtration canonique engendr\'ee par
$\set{B_t}_{t\geqs0}$)
\begin{enum}
\item	$X$ est mesurable par rapport \`a $\cF$;
\item	$X_t$ est adapt\'e \`a $\cF_t$, c'est-\`a-dire mesurable par rapport \`a
$\cF_t$ pour tout $t\in[0,T]$. 
\end{enum}
Ceci revient \`a exiger que $X_t$ ne d\'epende que de l'histoire du
processus de Wiener jusqu'au temps $t$, ce qui est raisonnable au vu
de~\eqref{is2}. En outre, nous allons supposer que 
\begin{equation}
\label{ito2}
\biggprob{\int_0^T X_t^2 \6t < \infty} = 1\;. 
\end{equation}

\begin{remark}
\label{rem_ito1}
On peut admettre que $X_t$ d\'epende de variables al\'eatoires
suppl\'emen\-taires, ind\'ependantes de $B_t$; par exemple, la condition
initiale peut \^etre al\'eatoire. Il convient alors d'\'etendre les
tribus $\cF$ et $\cF_t$ dans la d\'efinition ci-dessus \`a des
tribus plus grandes $\cA$ et $\cA_t$, o\`u $\cA_t$ ne doit pas
d\'ependre de la tribu engendr\'ee par $\set{B_{t+s}-B_t}_{s\geqs0}$. 
\end{remark}

Dans un premier temps, nous allons d\'efinir l'int\'egrale stochastique pour
un int\'egrant {\em simple}. 

\begin{definition}
\label{def_ito1}
Une fonctionnelle Brownienne non-anticipative $\set{e_t}_{t\in[0,T]}$ est
dite {\em simple} ou {\em \'el\'ementaire} s'il existe une partition
$0=t_0<t_1<\dots<t_N=T$ de $[0,T]$ telle que 
\begin{equation}
\label{ito3}
e_t = \sum_{k=1}^N e_{t_{k-1}} \indicator{[t_{k-1},t_k)}(t)\;. 
\end{equation}
\end{definition}

Pour une telle fonctionnelle, nous d\'efinissons l'int\'egrale stochastique
par 
\begin{equation}
\label{ito4}
\int_0^t e_s \6B_s = \sum_{k=1}^m e_{t_{k-1}} \bigbrak{B_{t_k}-B_{t_{k-1}}}
+ e_{t_m} \bigbrak{B_{t}-B_{t_{m}}}\;, 
\end{equation}
o\`u $m$ est tel que $t\in[t_m,t_{m+1})$. 

Il est ais\'e de v\'erifier les propri\'et\'es suivantes:

\begin{enum}
\item	Pour deux fonctionnelles simples $e^{(1)}$ et $e^{(2)}$, 
\begin{equation}
\label{ito5}
\int_0^t \bigpar{e^{(1)}_s + e^{(2)}_s} \6B_s = 
\int_0^t e^{(1)}_s \6B_s + \int_0^t e^{(2)}_s \6B_s\;.
\end{equation}

\item	Pour toute constante $c$, 
\begin{equation}
\label{ito6}
\int_0^t \bigpar{ce_s} \6B_s = c\int_0^t e_s \6B_s\;.
\end{equation}

\item	L'int\'egrale~\eqref{ito4} est une fonction continue de $t$.

\item	Si $\int_0^t \expec{\abs{e_s}} \6s<\infty$, alors 
\begin{equation}
\label{ito7A}
\biggexpec{\int_0^t e_s \6B_s} = 0\;.
\end{equation}
\begin{proof}
Posons $t_{m+1}=t$. On a 
\begin{align}
\nonumber
\biggexpec{\int_0^t e_s \6B_s} 
&= \biggexpec{\sum_{k=1}^{m+1} e_{t_{k-1}}
\bigpar{B_{t_k}-B_{t_{k-1}}}} \\
\nonumber
&= \sum_{k=1}^{m+1} \bigexpec{e_{t_{k-1}}} 
\underbrace{\bigexpec{B_{t_k}-B_{t_{k-1}}}}_{0} 
= 0\;,
\label{ito8A}
\end{align}
en vertu des propri\'et\'es des incr\'ements ind\'ependants et gaussiens. 
\end{proof}

\item	Si $\int_0^t \expec{e_s^2} \6s<\infty$, on a l'{\em isom\'etrie d'It\^o}
\begin{equation}
\label{ito7}
\biggexpec{\biggpar{\int_0^t e_s \6B_s}^2} = 
\int_0^t \expec{e_s^2} \6s\;. 
\end{equation}
\begin{proof}
Posons $t_{m+1}=t$. On a 
\begin{align}
\nonumber
\biggexpec{\biggpar{\int_0^t e_s \6B_s}^2} 
&= \biggexpec{\sum_{k,l=1}^{m+1} e_{t_{k-1}}e_{t_{l-1}}
\bigpar{B_{t_k}-B_{t_{k-1}}}\bigpar{B_{t_l}-B_{t_{l-1}}}} \\
\nonumber
&= \sum_{k=1}^{m+1} \bigexpec{e_{t_{k-1}}^2} 
\underbrace{\bigexpec{\bigpar{B_{t_k}-B_{t_{k-1}}}^2}}_{t_k-t_{k-1}} \\
&= \int_0^t \expec{e_s^2} \6s\;.
\label{ito8}
\end{align}
Nous avons utilis\'e la propri\'et\'e des incr\'ements ind\'ependants afin
d'\'eliminer les termes $k\neq l$ de la double somme, et le fait que $e_s$
est non-anticipative.  
\end{proof}
\end{enum}

L'id\'ee d'It\^o pour d\'efinir l'int\'egrale stochastique d'une
fonctionnelle non-anticipative g\'en\'erale $X$ est de trouver une suite de
fonctionnelles simples $e^{(n)}$ approchant $X$ dans $L^2(\fP)$,
c'est-\`a-dire 
\begin{equation}
\label{ito9}
\lim_{n\to\infty} \int_0^T \bigexpec{\bigpar{X_s-e^{(n)}_s}^2} \6s = 0\;.
\end{equation}
L'isom\'etrie~\eqref{ito7} nous permet alors d'affirmer que la limite
suivante existe dans $L^2(\fP)$:
\begin{equation}
\label{ito10}
\lim_{n\to\infty} \int_0^t  e^{(n)}_s \6B_s \bydef 
\int_0^t  X_s \6B_s\;. 
\end{equation}
C'est par d\'efinition l'int\'egrale d'It\^o de $X_s$. 

Nous allons maintenant prouver que cette construction est bien possible,
ind\'ependante de la suite des $e^{(n)}$, et que l'int\'e\-grale
r\'esultante est une fonction continue de $t$. Nous commen\c cons par
\'enoncer deux lemmes pr\'eparatoires. Dans ce qui suit, nous utilisons la
notation abr\'eg\'ee
\begin{equation}
 \label{ito10B}
\Bigprob{A_n,n\to\infty} = 1
\quad
\Leftrightarrow
\quad
\biggprob{\sum_{n=1}^\infty \indicator{A_n^c}<\infty}=1
\quad
\Leftrightarrow
\quad
\biggprob{\limsup_n A_n^c}=0\;. 
\end{equation} 
Autrement dit, avec probabilit\'e $1$, une infinit\'e de $A_n$ sont
r\'ealis\'es, c'est-\`a-dire que pour presque tout $\omega\in\Omega$, il existe
$n_0(\omega)<\infty$ tel que $A_n\ni\omega$ pour tout $n\geqs n_0(\omega)$. Le
lemme de Borel--Cantelli affirme que c'est le cas si la somme des
$\prob{A_n^c}$ converge. 

\begin{lemma}
\label{lem_Ito1}
Pour toute fonctionnelle non-anticipative $X$ satisfaisant~\eqref{ito2}, il
existe une suite $\set{e^{(n)}}_{n\geqs1}$ de fonctionnelles simples telles
que 
\begin{equation}
\label{ito11}
\biggprob{\int_0^T \bigpar{X_t-e^{(n)}_t}^2 \6t \leqs 2^{-n}, n\to\infty} =
1\;. 
\end{equation}
\end{lemma}
\begin{proof}
Consid\'erons les fonctionnelles simples 
\begin{equation}
\label{ito12}
e^{(m,k)}_t =
2^k\int_{(2^{-m}\intpart{2^mt}-2^{-k})\vee0}^{2^{-m}\intpart{2^mt}} X_s\6s\;. 
\end{equation}
Faisant tendre d'abord $m$, puis $k$ vers l'infini, on s'aper\c coit que
$\int_0^T (X_t-e^{(m,k)}_t)^2\6t \to 0$. On peut donc trouver des suites
$m_n$ et $k_n$ telles que $e^{(n)}=e^{(m_n,k_n)}$ satisfasse
\begin{equation}
\label{into13}
\biggprob{\int_0^T (X_t-e^{(n)}_t)^2 \6t > 2^{-n}} \leqs 2^{-n}\;. 
\end{equation}
La relation \eqref{ito11} suit alors du lemme de Borel--Cantelli. 
\end{proof}

\begin{lemma}
\label{lem_Ito2}
Pour une suite $\set{f^{(n)}}_{n\geqs1}$ de fonctionnelles non-anticipatives
simples satisfaisant $\prob{\int_0^T (f^{(n)}_t)^2 \6t \leqs 2^{-n},
n\to\infty} = 1$ et tout $\theta>1$, 
\begin{equation}
\label{ito14}
\biggprob{\sup_{0\leqs t\leqs T} \biggabs{\int_0^t f^{(n)}_s \6B_s} < \theta
\sqrt{\frac{\log n}{2^{n-1}}}, n\to\infty} = 1\;. 
\end{equation}
\end{lemma}
\begin{proof}
Pour tout $\gamma\in\R$, le processus stochastique 
\begin{equation}
\label{ito15}
M^{(n)}_t = \exp\biggset{\gamma\int_0^t f^{(n)}_s \6B_s - \frac{\gamma^2}2
\int_0^t
(f^{(n)}_s)^2 \6s}
\end{equation}
est une martingale par rapport \`a $\set{B_t}_{t\geqs0}$. En effet, nous
avons d\'emontr\'e cette propri\'et\'e pour $f^{(n)}\equiv1$ dans la 
proposition~\ref{prop_mD}. La m\^eme d\'emonstration montre que si 
$f^{(n)}\equiv c$ o\`u $c$ est une variable al\'eatoire mesurable par rapport
\`a $\cF_s$, alors $\econd{M_t}{\cF_s} = M_s$. Le cas d'un $f^{(n)}$ simple
arbitraire est alors trait\'e par r\'ecurrence. 

L'in\'egalit\'e de Doob implique
\begin{equation}
\label{into16}
\biggprob{\sup_{0\leqs t\leqs T} \biggpar{\int_0^t f^{(n)}_s \6B_s -
\frac\gamma2 \int_0^t (f^{(n)}_s)^2 \6s} > L} \leqs \e^{-\gamma L}\;.  
\end{equation}
Posons alors $\gamma=\sqrt{2^{n+1}\log n}$ et $L =
\theta\sqrt{2^{-(n+1)}\log n}$. Utilisant l'hypoth\`ese sur les $f^{(n)}$,
nous obtenons 
\begin{equation}
\label{ito17}
\biggprob{\sup_{0\leqs t\leqs T} \biggpar{\int_0^t f^{(n)}_s \6B_s} >
(1+\theta) \sqrt{2^{-(n+1)}\log n}} \leqs \e^{-\theta\log n} = n^{-\theta}\;. 
\end{equation}
Le lemme de Borel--Cantelli nous permet alors de conclure. 
\end{proof}

\begin{theorem}
\label{thm_ito}
La limite~\eqref{ito10} existe, est ind\'ependante de la suite des
$\set{e^{(n)}}_{n\geqs1}$ convergeant vers $X$, et est une fonction continue
de $t$. 
\end{theorem}
\begin{proof}
Quitte \`a passer \`a une sous-suite, nous pouvons choisir les $e^{(n)}$ de
mani\`ere \`a satisfaire~\eqref{ito11}. Mais ceci implique aussi que 
\begin{equation}
\label{ito18}
\biggprob{\int_0^T \bigpar{e^{(n)}_t-e^{(n-1)}_t}^2 \6t \leqs \const
2^{-n}, n\to\infty} = 1\;.
\end{equation}
Le lemme~\ref{lem_Ito2} montre alors que  
\begin{equation}
\label{ito19}
\biggprob{\sup_{0\leqs t\leqs T} \biggabs{\int_0^t
\bigpar{e^{(n)}_s-e^{(n-1)}_s} \6B_s} < \const \theta \sqrt{\frac{\log
n}{2^{n-1}}}, n\to\infty} = 1\;.
\end{equation}
Ainsi la suite des $\int_0^t e^{(n)}_s \6B_s$ est de Cauchy presque
s\^urement, et dans ce cas elle converge uniform\'ement. 
\end{proof}


\section{Propri\'et\'es \'el\'ementaires}
\label{sec_elito}

Les propri\'et\'es suivantes sont prouv\'ees ais\'ement pour des processus
$X$ et $Y$ non-anticipatifs satisfaisant la condition
d'int\'egrabilit\'e~\eqref{ito2}.  

\begin{enum}
\item	{\bf Lin\'earit\'e:}
\begin{equation}
\label{elito1}
\int_0^t (X_s+Y_s)\6B_s = \int_0^t X_s \6B_s + \int_0^t Y_s \6B_s
\end{equation}
et
\begin{equation}
\label{elito2}
\int_0^t (cX_s)\6B_s = c\int_0^t X_s \6B_s\;.
\end{equation}

\item	{\bf Additivit\'e:} Pour $0\leqs s<u<t\leqs T$, 
\begin{equation}
\label{elito3a}
\int_s^t X_v\6B_v = \int_s^u X_v \6B_v + \int_u^t X_v \6B_v\;.
\end{equation}

\item	Pour un temps d'arr\^et $\tau$, 
\begin{equation}
\label{elito3}
\int_0^{\tau\wedge T} X_t \6B_t = \int_0^T \indexfct{t\leqs\tau} X_t \6B_t\;.
\end{equation}

\item	Si $\int_0^T\expec{X_t^2}\6t < \infty$, alors pour tout $t\leqs T$, 
\begin{equation}
\label{elito4}
\biggexpec{\int_0^t X_s\6B_s} = 0
\end{equation}
et
\begin{equation}
\label{elito5}
\biggexpec{\biggpar{\int_0^t X_s\6B_s}^2} = 
\int_0^t\expec{X_s^2}\6s\;.
\end{equation}
De plus, le processus $\bigset{\int_0^t X_s\6B_s}_{t\geqs0}$ est une
martingale. 

\item	Le processus 
\begin{equation}
\label{elito6}
M_t = \exp\biggset{\int_0^t X_s\6B_s - \frac12 \int_0^t X_s^2 \6s}
\end{equation}
est une surmartingale. 
\end{enum}


\section{Un exemple}
\label{sec_exito}

Nous donnons ici un exemple de calcul explicite d'une int\'egrale
stochastique par la m\'ethode d'It\^o:
\begin{equation}
\label{exito1}
\int_0^t B_s \6B_s = \frac12 B_t^2 - \frac t2\;.
\end{equation}
Le r\'esultat, quelque peu surprenant au premier abord, prendra tout son
sens lorsque nous aurons vu la formule d'It\^o. 

Consid\'erons la suite de fonctionnelles simples d\'efinies par
$e^{(n)}_t=B_{2^{-n}\intpart{2^nt}}$. Il suffit alors de v\'erifier que 
\begin{equation}
\label{exito2}
\lim_{n\to\infty} \int_0^t e^{(n)}_s \6B_s = \frac12 B_t^2 - \frac t2\;. 
\end{equation}
Notons $t_k = k2^{-n}$ pour $k\leqs m=\intpart{2^nt}$ et $t_{m+1}=t$. Il
suit de la d\'efinition~\eqref{ito4} que 
\begin{align}
\nonumber
2\int_0^t e^{(n)}_s \6B_s &= 2\sum_{k=1}^{m+1}
B_{t_{k-1}}\bigpar{B_{t_k}-B_{t_{k-1}}} \\
\nonumber
&= \sum_{k=1}^{m+1}\Bigbrak{B_{t_k}^2-B_{t_{k-1}}^2
-\bigpar{B_{t_k}-B_{t_{k-1}}}^2} \\
&= B_t^2 - \sum_{k=1}^{m+1} \bigpar{B_{t_k}-B_{t_{k-1}}}^2\;.
\label{exito3}
\end{align}
Consid\'erons la martingale 
\begin{equation}
\label{exito4}
M^{(n)}_t = \sum_{k=1}^{m+1} \bigpar{B_{t_k}-B_{t_{k-1}}}^2 - t
= \sum_{k=1}^{m+1} \bigbrak{\bigpar{B_{t_k}-B_{t_{k-1}}}^2 -
\bigpar{t_k-t_{k-1}}}\;. 
\end{equation}
Les termes de la somme sont ind\'ependants et d'esp\'erance nulle. Nous
avons donc 
\begin{align}
\nonumber
\bigexpec{\bigpar{M^{(n)}_t}^2} 
&= \sum_{k=1}^{m+1} \bigexpec{\bigbrak{\bigpar{B_{t_k}-B_{t_{k-1}}}^2 -
\bigpar{t_k-t_{k-1}}}^2} \\
\nonumber
&\leqs (m+1)
\bigexpec{\bigbrak{\bigpar{B_{t_1}-B_{t_0}}^2 - \bigpar{t_1-t_0}}^2} \\
\nonumber
&\leqs \const2^n\bigexpec{\bigbrak{\bigpar{B_{2^{-n}}}^2 - 2^{-n}}^2} \\
\nonumber
&= \const2^{-n}\bigexpec{\bigbrak{\bigpar{B_1}^2 - 1}^2} \\
&\leqs \const 2^{-n}\;,
\label{exito5}
\end{align}
\`a cause de la propri\'et\'e d'\'echelle. $\bigpar{M^{(n)}_t}^2$ \'etant
une sous-martingale, l'in\'egalit\'e de Doob nous donne 
\begin{equation}
\label{exito6}
\biggprob{\sup_{0\leqs s\leqs t} \bigpar{M^{(n)}_s}^2 > n^2 2^{-n}} 
\leqs 2^n n^{-2} \bigexpec{\bigpar{M^{(n)}_t}^2} 
\leqs \const n^{-2}\;.
\end{equation}
Le lemme de Borel--Cantelli nous permet alors de conclure que 
\begin{equation}
\label{exito7}
\biggprob{\sup_{0\leqs s\leqs t} \bigabs{M^{(n)}_s} < n2^{-n/2},
n\to\infty}=1\;,
\end{equation}
ce qui prouve~\eqref{exito2}. 


\section{La formule d'It\^o}
\label{sec_fito}

Consid\'erons une int\'egrale stochastique de la forme 
\begin{equation}
\label{fito1}
X_t = X_0 + \int_0^t f_s \6s + \int_0^t g_s \6B_s\;,
\qquad t\in[0,T]
\end{equation}
o\`u $X_0$ est ind\'ependante du mouvement Brownien, et $f_s$ et $g_s$ sont
des fonctionnelles non-anticipatives satisfaisant 
\begin{align}
\nonumber
\biggprob{\int_0^T \abs{f_s}\6s < \infty} &= 1 \\
\biggprob{\int_0^T g_s^2 \6s < \infty} &= 1\;.
\label{fito2}
\end{align}
Le processus~\eqref{fito1}  s'\'ecrit \'egalement sous forme diff\'erentielle
\begin{equation}
\label{fito3}
\6X_t = f_t \6t + g_t \6B_t\;.
\end{equation}
Par exemple, la relation~\eqref{exito1} est \'equivalente \`a
\begin{equation}
\label{fito4}
\6\,(B_t^2) = \6t + 2 B_t \6B_t\;.
\end{equation}
La formule d'It\^o permet de d\'eterminer de mani\`ere g\'en\'erale l'effet
d'un changement de variables sur une diff\'erentielle stochastique. 

\begin{lemma}[Formule d'It\^o]
\label{lem_Ito}
Soit $u : [0,\infty)\times\R \to \R, (t,x)\mapsto u(t,x)$ une fonction
contin\^ument diff\'erentiable par rapport \`a $t$ et deux fois 
contin\^ument diff\'erentiable par rapport \`a $x$. Alors le processus
stochastique $Y_t = u(t,X_t)$ satisfait l'\'equation 
\begin{align}
\nonumber
Y_t ={}& Y_0 + \int_0^t \dpar ut(s,X_s) \6s 
+ \int_0^t \dpar ux(s,X_s) f_s \6s 
+ \int_0^t \dpar ux(s,X_s) g_s \6B_s \\ 
&{}+ \frac12 \int_0^t \dpar{^2u}{x^2}(s,X_s) g_s^2 \6s\;. 
\label{fito5A}
\end{align}
\end{lemma}
\begin{proof}
Il suffit de prouver le r\'esultat pour des fonctionnelles simples, et par
l'additivit\'e des int\'egrales, on peut se ramener au cas de fonctionnelles
constantes. Mais alors $X_t = f_0 t+g_0 B_t$ et $Y_t=u(t,f_0 t+g_0 B_t)$
peut s'exprimer comme une fonction de $(t,B_t)$. Il suffit en d\'efinitive
de consid\'erer le cas $X_t=B_t$. Or pour une partition
$0=t_0<t_1<\dots<t_n=t$, on a 
\begin{align}
\nonumber
u(t,B_t) - u(0,0) 
&= \sum_{k=1}^n \bigbrak{u(t_k,B_{t_k}) - u(t_{k-1},B_{t_k})} + 
\bigbrak{u(t_{k-1},B_{t_k}) - u(t_{k-1},B_{t_{k-1}})} \\
\nonumber
&=\sum_{k=1}^n \dpar ut(t_{k-1},B_{t_k}) (t_k-t_{k-1}) 
+ \dpar ux(t_{k-1},B_{t_{k-1}})(B_{t_k}-B_{t_{k-1}}) \\
\nonumber
&\phantom{=} {}+ \frac12 \dpar
{^2u}{x^2}(t_{k-1},B_{t_{k-1}})(B_{t_k}-B_{t_{k-1}})^2 + 
\order{t_k-t_{k-1}} + \order{(B_{t_k}-B_{t_{k-1}})^2} \\
\nonumber
&= \int_0^t \dpar ut(s,B_s)\6s + \int_0^t\dpar ux(s,B_s)\6B_s + 
\frac12 \int_0^t\dpar{^2u}{x^2}(s,B_s) \6s \\
&\phantom{=} {}+ \sum_{k=1}^n \frac12 
\dpar{^2u}{x^2}(t_{k-1},B_{t_{k-1}})
\bigbrak{(B_{t_k}-B_{t_{k-1}})^2-(t_k-t_{k-1})} +\order{1}\;.
\label{fito10}
\end{align}
La somme se traite comme $M^{(n)}_t$ dans la section~\ref{sec_exito} lorsque
$t_k-t_{k-1}\to0$, c.f.~\eqref{exito4}.   
\end{proof}

\begin{remark}\hfill
\label{rem_fIto}
\begin{enum}
\item	La formule d'It\^o s'\'ecrit aussi sous forme diff\'erentielle,  
\begin{equation}
\label{fito5}
\6Y_t = \dpar ut(t,X_t) \6t + \dpar ux(t,X_t) \bigbrak{f_t \6t + g_t \6B_t} 
+ \frac12 \dpar{^2u}{x^2}(t,X_t) g_t^2 \6t\;.
\end{equation}
\item	Un moyen mn\'emotechnique pour retrouver la formule est de
l'\'ecrire sous la forme 
\begin{equation}
\label{fito6}
\6Y_t = \dpar ut \6t + \dpar ux \6X_t + \frac12 \dpar{^2u}{x^2}\6X_t^2\;,
\end{equation}
o\`u $\6X_t^2$ se calcule en utilisant les r\`egles
\begin{equation}
\label{fito7}
\6t^2 = \6t\6B_t=0, \qquad
\6B_t^2 = \6t\;.
\end{equation}

\item	La formule se g\'en\'eralise \`a des fonctions
$u(t,X^{(1)}_t,\dots,X^{(n)}_t)$, d\'ependant de $n$ processus d\'efinis
par $\6X^{(i)}_t=f^{(i)}_t\6t+g^{(i)}_t\6B_t$, en 
\begin{equation}
\label{fito8}
\6Y_t = \dpar ut \6t + \sum_i\dpar u{x_i} \6X^{(i)}_t + \frac12
\sum_{i,j}\dpar{^2u}{x_i\partial x_j} \6X^{(i)}_t\6X^{(j)}_t\;,
\end{equation}
o\`u $\6X^{(i)}_t\6X^{(j)}_t = g^{(i)}_tg^{(j)}_t\6t$\;.
\end{enum}
\end{remark}

\goodbreak
\begin{example}\hfill
\label{ex_ito1}
\begin{enum}
\item	Si $X_t = B_t$ et $u(x)=x^2$, on retrouve la relation~\eqref{fito4}.
\item	Si $\6X_t = g_t\6B_t - \frac12 g_t^2\6t$ et $u(x)=\e^x$, on obtient
\begin{equation}
\label{fito9}
\6\,(\e^{X_t}) = g_t e^{X_t}\6B_t\;.
\end{equation}
Ainsi la martingale $M_t = \exp\bigset{\gamma B_t - \gamma^2 \frac t2}$ de
la proposition~\ref{prop_mD} est la solution de l'\'equation $\6M_t=\gamma
M_t\6B_t$. 
\end{enum}
\end{example}


\section{Exercices}
\label{sec_exo_ito}

\begin{exercice}
\label{exo_ito1} 
On consid\`ere les deux processus stochastiques 
\[
X_t = \int_0^t \e^s \6B_s\;,
\qquad
Y_t = \e^{-t} X_t\;.
\]

\begin{enum}
\item	D\'eterminer $\expec{X_t}$, $\Variance(X_t)$, $\expec{Y_t}$ et
$\Variance(Y_t)$.
\item	Sp\'ecifier la loi de $X_t$ et de $Y_t$. 
\item	Montrer que $Y_t$ converge en loi vers une variable $Y_\infty$ lorsque
$t\to\infty$ et sp\'ecifier sa loi. 
\item	Exprimer $\6Y_t$ en fonction de $Y_t$ et de $B_t$.
\end{enum}
\end{exercice}

\goodbreak

\begin{exercice}
\label{exo_ito2} 
Soit 
\[
X_t = \int_0^t s \6B_s\;.
\]

\begin{enum}
\item	Calculer $\expec{X_t}$ et $\Variance(X_t)$. 
\item	Quelle est la loi de $X_t$?
\item	Calculer $\6\,(tB_t)$ \`a l'aide de la formule d'It\^o.
\item	En d\'eduire une relation entre $X_t$ et 
\[
Y_t = \int_0^t B_s \6s\;.
\]
\item	Calculer la variance de $Y_t$, 
\begin{enum}
\item	directement \`a partir de sa d\'efinition;
\item	en calculant d'abord la covariance de $B_t$ et $X_t$, \`a l'aide d'une
partition de $[0,t]$. 
\end{enum}
En d\'eduire la loi de $Y_t$. 
\end{enum}
\end{exercice}

\goodbreak

\begin{exercice}[In\'egalit\'e de Bernstein]
\label{exo_ito3}
\hfill 
\begin{enum}
\item	Soit $Y$ une variable al\'eatoire normale centr\'ee, de variance
$\sigma^2$. Montrer que 
\[
\expec{\e^Y}=\e^{\sigma^2/2}
\]

\item	Soit $B_t$ un mouvement Brownien standard, et 
$\varphi:[0,T]\to\R$ une fonction ind\'epen\-dante de $B_t$. Pour $t\in[0,T]$ on
pose
\[
X_t = \int_0^t \varphi(s)\6B_s
\]
Calculer $\expec{X_t}$ et $\Variance{X_t}$. On pr\'ecisera les hypoth\`eses
faites sur la fonction $\varphi$.

\item	Montrer que 
\[
M_t = \exp\biggset{X_t - \frac12 \int_0^t \varphi(s)^2 \6s}
\]
est une martingale. 

\item	D\'emontrer l'{\it in\'egalit\'e de Bernstein}~: Pour tout $\lambda>0$,
\[
\Bigprob{\sup_{0\leqs s\leqs t}X_s > \lambda} \leqs 
\exp\biggset{-\frac{\lambda^2}{2\Phi(t)}}
\]
o\`u $\Phi(t) = \displaystyle\int_0^t \varphi(s)^2 \6s$. 

\end{enum}
\end{exercice}

\goodbreak

\begin{exercice}[Int\'egrale de Stratonovich]
\label{exo_ito4}
Soit $\set{B_t}_{t\in[0,T]}$ un mouvement Brownien standard. 
Soit $0=t_0<t_1<\dots<t_N=T$ une partition de $[0,T]$, et soit 
\[
e_t = \sum_{k=1}^N e_{t_{k-1}} \indicator{[t_{k-1},t_k)}(t)
\]
une fonction simple, adapt\'ee \`a la filtration canonique du mouvement
Brownien. 

L'int\'egrale de Stratonovich de $e_t$ est d\'efinie par 
\[
\int_0^T e_t \circ \6B_t = \sum_{k=1}^N \frac{e_{t_k}+e_{t_{k-1}}}{2} \Delta B_k
\qquad
\text{o\`u } \Delta B_k = B_{t_k} - B_{t_{k-1}}\;.
\]
L'int\'egrale de Stratonovich $\int_0^T X_t\circ \6B_t$ d'un processus adapt\'e
$X_t$ est d\'efinie comme la limite de la suite $\int_0^T e^{(n)}_t\circ \6B_t$,
o\`u $e^{(n)}$ est une suite de fonctions simples convergeant vers $X_t$ dans
$L^2$. On admettra que cette limite existe et est ind\'ependante de la suite
$e^{(n)}$. 

\begin{enum}
\item 	Calculer
\[
\int_0^T B_t \circ \6B_t\;. 
\]

\item	Soit $g:\R\to\R$ une fonction de classe $C^2$, et soit $X_t$ un
processus adapt\'e satisfaisant 
\[
X_t = \int_0^t g(X_s) \circ \6B_s 
\qquad
\forall t\in[0,T]\;.
\]
Soit $Y_t$ l'int\'egrale d'It\^o 
\[
Y_t = \int_0^t g(X_s) \6B_s\;.
\]
Montrer que 
\[
X_t - Y_t = \frac12 \int_0^t g'(X_s)g(X_s) \6s
\qquad
\forall t\in[0,T]\;.
\]
\end{enum}
\end{exercice}


\chapter{Equations diff\'erentielles stochastiques}
\label{chap_sde}

Nous avons maintenant tous les \'el\'ements en main pour d\'efinir la notion
de solution d'une \'equation diff\'erentielle stochastique (EDS), de la forme
\begin{equation}
\label{edsf1}
\6X_t = f(X_t,t)\6t + g(X_t,t)\6B_t, 
\end{equation}
o\`u $f, g:\R\times[0,T]\to\R$ sont des fonctions d\'eterministes mesurables. 
La fonction $f$ est commun\'ement appel\'ee \defwd{coefficient de d\'erive}\/,
alors que $g$ est appel\'ee \defwd{coefficient de diffusion}\/.

Dans tout ce chapitre, nous supposons que 
\begin{itemiz}
\item	Soit $X_0\in\R$ est une constante, et alors $\set{\cF_t}_{t\in[0,T]}$
d\'esigne la filtration engendr\'ee par le mouvement Brownien.
\item	Soit $X_0:\Omega\to\R$ est une variable al\'eatoire, de carr\'e
int\'egrable, et ind\'ependante du mouvement Brownien. Dans ce cas, 
$\set{\cF_t}_{t\in[0,T]}$ d\'esignera la filtration engendr\'ee par le
mouvement Brownien et par $X_0$. 
\end{itemiz}


\section{Solutions fortes}
\label{sec_sdef}

\begin{definition}
\label{def_edsf}
Un processus stochastique $\set{X_t}_{t\in[0,T]}$ est appel\'e une {\em solution
forte\/} de l'EDS~\eqref{edsf1} avec condition initiale $X_0$ si 
\begin{itemiz}
\item	$X_t$ est $\cF_t$-mesurable pour tout $t\in[0,T]$; 
\item	on a les conditions de r\'egularit\'e 
\begin{equation}
\label{edsf2}
\biggprob{\int_0^T \abs{f(X_s,s)}\6s < \infty} =
\biggprob{\int_0^T g(X_s,s)^2\6s < \infty} = 1;
\end{equation}
\item	pour tout $t\in[0,T]$, on a 
\begin{equation}
\label{edsf3}
X_t = X_0 + \int_0^t f(X_s,s)\6s + \int_0^t g(X_s,s)\6B_s 
\end{equation}
avec probabilit\'e $1$. 
\end{itemiz}
\end{definition}

\begin{remark}
\label{rem_edsf}
On d\'efinit de mani\`ere similaire la solution forte d'une EDS
multidimensionnelle, c'est-\`a-dire qu'on peut supposer que $X\in\R^n$,
$f(x,t)$ prend des valeurs dans $R^n$, et, si $B_t$ est un mouvement
Brownien de dimension $k$, $g(x,t)$ prend des valeurs dans les matrices
$n\times k$. 
\end{remark}



Avant de montrer l'existence de solutions fortes dans un cadre g\'en\'eral,
nous donnons quelques exemples d'\'equations pour lesquelles une solution peut
\^etre donn\'ee explicitement sous forme d'int\'egrales. Il n'est pas
surprenant que de tels cas solubles sont extr\^emement rares. 

\begin{example}
\label{ex_edslin}
Consid\'erons l'EDS lin\'eaire avec \lq\lq bruit additif\rq\rq\  
\begin{equation}
\label{edsf8}
\6X_t = a(t) X_t\6t + \sigma(t)\6B_t\;,
\end{equation}
o\`u $a$ et $\sigma$ sont des fonctions d\'eterministes. Dans le cas
particulier $\sigma\equiv0$, la solution peut s'\'ecrire simplement 
\begin{equation}
 \label{edsf8A}
X_t = \e^{\alpha(t)} X_0\;, \qquad
\alpha(t) = \int_0^t a(s)\,\6s\;. 
\end{equation} 
Ceci sugg\`ere d'appliquer la m\'ethode de la variation de la constante,
c'est-\`a-dire de chercher une solution de la forme 
$X_t = \e^{\alpha(t)} Y_t$. La formule d'It\^o appliqu\'ee \`a
$Y_t=u(X_t,t)=\e^{-\alpha(t)}X_t$ nous donne 
\begin{equation}
 \label{edsf8B}
\6Y_t = -a(t) \e^{-\alpha(t)}X_t \,\6t 
+ \e^{-\alpha(t)}\,\6X_t
= \e^{-\alpha(t)} \sigma(t) \,\6B_t\;,
\end{equation} 
d'o\`u en int\'egrant et en tenant compte du fait que $Y_0=X_0$, 
\begin{equation}
 \label{edsf8C}
Y_t = X_0 + \int_0^t \e^{-\alpha(s)} \sigma(s) \,\6B_s\;. 
\end{equation} 
Ceci donne finalement la solution forte de l'\'equation~\eqref{edsf8} 
\begin{equation}
\label{edsf8D}
X_t = X_0 \e^{\alpha(t)} + 
\int_0^t \e^{\alpha(t)-\alpha(s)} \sigma(s) \,\6B_s\;. 
\end{equation}
On v\'erifie effectivement~\eqref{edsf3} en appliquant encore une fois la
formule d'It\^o. On notera que si la condition initiale $X_0$ est
d\'eterministe, alors $X_t$ suit une loi normale, d'esp\'erance
$\expec{X_t} = X_0\e^{\alpha(t)}$ et de variance 
\begin{equation}
 \label{edsf8E}
\Variance(X_t) = 
\int_0^t  \e^{2(\alpha(t)-\alpha(s))} \sigma(s)^2 \,\6s\;, 
\end{equation}  
en vertu de l'isom\'etrie d'It\^o. 
\end{example}

\begin{example}
\label{ex_edslinm}
Soit l'EDS lin\'eaire avec \lq\lq bruit multiplicatif\rq\rq\  
\begin{equation}
\label{edsf9}
\6X_t = a(t) X_t\6t + \sigma(t)X_t\6B_t\;,
\end{equation}
avec \`a nouveau $a$ et $\sigma$ des fonctions d\'eterministes. Nous pouvons
alors \'ecrire 
\begin{equation}
 \label{edsf9A}
\frac{\6X_t}{X_t} =  a(t) \6t + \sigma(t)\6B_t\;.
\end{equation} 
En int\'egrant le membre de gauche, on devrait trouver $\log(X_t)$, mais ceci
est-il compatible avec le calcul d'It\^o? Pour s'en assurer, posons
$Y_t=u(X_t)=\log(X_t)$. Alors la formule d'It\^o donne 
\begin{align}
\nonumber
\6Y_t &= \frac{1}{X_t}\,\6X_t - \frac{1}{2X_t^2}\,\6X_t^2 \\ 
&= a(t) \6t + \sigma(t)\6B_t - \frac12 \sigma(t)^2 \6t\;.
 \label{edsf9B}
\end{align} 
En int\'egrant et en reprenant l'exponentielle, on obtient donc la solution
forte  
\begin{equation}
 \label{edsf9C} 
X_t = X_0 \exp\biggset{\int_0^t \bigbrak{a(s)-\frac12 \sigma(s)^2}\6s 
+ \int_0^t \sigma(s)\6B_s}\;.
\end{equation} 
En particulier, si $a\equiv0$ et $\sigma\equiv\gamma$, on retrouve la
martingale $X_t=X_0\exp\set{\gamma B_t-\gamma^2 t/2}$, appel\'ee
\defwd{mouvement Brownien exponentiel}.
\end{example}


\section{Existence et unicit\'e de solutions}
\label{sec_sdeu}

Nous donnons d'abord un r\'esultat d'existence et d'unicit\'e d'une solution
forte sous des conditions un peu restrictives sur les coefficients $f$ et $g$. 

\begin{theorem}
\label{thm_edsf}
Supposons que les fonctions $f$ et $g$ satisfont les deux conditions
suivantes: 
\begin{enum}
\item	{\em Condition de Lipschitz globale:} Il existe une constante $K$
telle que 
\begin{equation}
\label{edsf5}
\abs{f(x,t) - f(y,t)} + \abs{g(x,t)-g(y,t)} \leqs K\abs{x-y}
\end{equation}
pour tous les $x, y\in\R$ et $t\in[0,T]$. 

\item	{\em Condition de croissance:} Il existe une constante $L$ telle que
\begin{equation}
\label{edsf6}
\abs{f(x,t)} + \abs{g(x,t)} \leqs L(1+\abs{x})
\end{equation}
pour tous les $x\in\R$ et $t\in[0,T]$. 
\end{enum}
Alors l'EDS~\eqref{edsf1} admet, pour toute condition initiale $X_0$ de carr\'e
int\'egrable, une solution
forte $\set{X_t}_{t\in[0,T]}$, presque s\^urement continue. Cette solution
est unique dans le sens que si $\set{X_t}_{t\in[0,T]}$ et
$\set{Y_t}_{t\in[0,T]}$ sont deux solutions presque s\^urement continues,
alors 
\begin{equation}
\label{edsf7}
\biggprob{\sup_{0\leqs t\leqs T}\abs{X_t-Y_t}>0} = 0. 
\end{equation}
\end{theorem}
\begin{proof}
Commen\c cons par d\'emontrer l'unicit\'e. Soient $X_t$ et $Y_t$ deux solutions
fortes de condition initiale $X_0$. Posons $\phi(t)=f(X_t,t)-f(Y_t,t)$ et
$\gamma(t)=g(X_t,t)-g(Y_t,t)$. Alors en utilisant l'in\'egalit\'e
de Cauchy--Schwartz et l'isom\'etrie d'It\^o on trouve 
\begin{align}
\nonumber
\bigexpec{\abs{X_t-Y_t}^2}
&= \E\biggbrak{\biggpar{\int_0^t \phi(s)\,\6s + \int_0^t \gamma(s)\,\6B_s}^2} \\
\nonumber
&\leqs 2\E\biggbrak{\biggpar{\int_0^t \phi(s)\,\6s}^2}
+ 2\E\biggbrak{\biggpar{\int_0^t \gamma(s)\,\6B_s}^2} \\
\nonumber
&\leqs 2\E\biggbrak{t\int_0^t \phi(s)^2\,\6s}
+ 2\E\biggbrak{\int_0^t \gamma(s)^2\,\6s} \\
\nonumber
&\leqs 2(1+t)K^2\int_0^t \bigexpec{\abs{X_s-Y_s}^2}\,\6s\;.
\label{edsf8:1} 
\end{align}
Ainsi la fonction $v(t)=\expec{\abs{X_t-Y_t}^2}$ satisfait l'in\'egalit\'e 
\begin{equation}
 \label{edsf8:2}
v(t) \leqs 2(1+T)K^2 \int_0^t v(s)\,\6s\;, 
\end{equation} 
donc par l'in\'egalit\'e de Gronwall, 
\begin{equation}
 \label{edsf8:3}
v(t) \leqs v(0) \e^{2(1+T)K^2 t} = 0\;, 
\end{equation} 
puisque $v(0)=0$. Il suit que 
\begin{equation}
 \label{edsf8:4}
\bigprob{\abs{X_t-Y_t}=0 \;\forall t\in\Q\cap[0,T]}= 1\;, 
\end{equation} 
ce qui implique~\eqref{edsf7} en vertu de la continuit\'e des trajectoires. 

Afin de prouver l'existence, nous posons $X^{(0)}_t=X_0$ et d\'efinissons une
suite $\set{X^{(k)}_t}_{k\in\N}$ r\'ecursivement par 
\begin{equation}
 \label{edsf8:5}
 X^{(k+1)}_t = X_0 + \int_0^t f(X^{(k)}_s,s) \,\6s
+ \int_0^t g(X^{(k)}_s,s) \,\6B_s\;.
\end{equation} 
Un calcul similaire au calcul ci-dessus montre que 
\begin{equation}
 \label{edsf8:6}
\bigexpec{\bigabs{X^{(k+1)}_t - X^{(k)}_t}^2}
\leqs 2(1+T)K^2 \int_0^t  \bigexpec{\abs{X^{(k)}_s - X^{(k-1)}_s}^2}\,\6s
\end{equation} 
pour tout $k\geqs1$ et $t\in[0,T]$, alors que 
\begin{equation}
 \label{edsf8:7}
 \bigexpec{\bigabs{X^{(1)}_t - X^{(0)}_t}^2}
\leqs A_1 t
\end{equation} 
pour une constante $A_1=A_1(L,T,X_0)$. Par r\'ecurrence, nous avons donc 
\begin{equation}
 \label{edsf8:8}
\bigexpec{\bigabs{X^{(k+1)}_t - X^{(k)}_t}^2}
\leqs \frac{A_2^{k+1}t^{k+1}}{(k+1)!}
\end{equation} 
pour tout $k\geqs1$ et $t\in[0,T]$ et une constante $A_2=A_2(K,L,T,X_0)$. Il
est alors facile de v\'erifier que la suite des $X^{(k)}$ est une suite de
Cauchy dans $L^2(\fP\times\lambda)$, c'est-\`a-dire par rapport au produit
scalaire d\'efini par la mesure $\fP$ sur $\Omega$ et la mesure de Lebesgue
$\lambda$ sur $[0,T]$. Il existe donc un processus limite 
\begin{equation}
 \label{edsf8:9}
X_t = \lim_{n\to\infty} X^{(n)}_t \in L^2(\fP\times\lambda)\;. 
\end{equation} 

Il reste \`a montrer que c'est la solution forte cherch\'ee. 
Faisons tendre $k$ vers l'infini dans~\eqref{edsf8:5}. Le membre de gauche tend
vers $X_t$. La premi\`ere int\'egrale dans le membre de droite tend dans 
$L^2(\fP\times\lambda)$ vers l'int\'egrale de $f(X_s,s)\6s$ par l'in\'egalit\'e
de H\"older. La seconde tend dans  $L^2(\fP\times\lambda)$ vers l'int\'egrale de
$g(X_s,s)\6Bs$ par l'isom\'etrie d'It\^o. Ceci montre que $X_t$
satisfait~\eqref{edsf3}. Les relations~\eqref{edsf2} sont satisfaites en
cons\'equence de la condition de croissance et du fait que $X_t$ est de carr\'e
int\'egrable. 
Finalement, le Th\'eor\`eme~\ref{thm_ito} implique la continuit\'e presque
s\^ure de $X_t$. 
\end{proof}

Les conditions~\eqref{edsf5} et~\eqref{edsf6} sont en fait trop restrictives,
et peuvent \^etre remplac\'ees par les conditions plus faibles suivantes~:
\begin{enum}
\item	{\em Condition de Lipschitz locale:} Pour tout compact
$\cK\in\R$, il existe une constante $K=K(\cK)$ telle que 
\begin{equation}
\label{edsf5B}
\abs{f(x,t) - f(y,t)} + \abs{g(x,t)-g(y,t)} \leqs K\abs{x-y}
\end{equation}
pour $x, y\in\cK$ et $t\in[0,T]$. 

\item	{\em Condition de croissance:} Il existe une constante $L$ telle que
\begin{equation}
\label{edsf6B}
xf(x,t) + g(x,t)^2 \leqs L^2(1+x^2)
\end{equation}
pour tous $x,t$. 
\end{enum}
Par exemple, le coefficient de d\'erive $f(x,t)=-x^2$ satisfait les
conditions~\eqref{edsf5B} et \eqref{edsf6B}, alors qu'il ne satisfait pas les  
conditions~\eqref{edsf5} et~\eqref{edsf6}. Par contre, la condition de
Lipschitz locale est n\'ecessaire pour garantir l'unicit\'e, comme le montre le
contre-exemple classique $f(x,t)=3x^{2/3}$ (toute fonction valant $0$ jusqu'\`a
un temps arbitraire $a$, puis $(t-a)^3$ pour les temps ult\'erieurs est
solution). De m\^eme, le contre-exemple $f(x,t)=x^2$, dont les solutions
divergent pour un temps fini, montre la n\'ecessit\'e de la condition de
croissance~\eqref{edsf6B}. 

Sans entrer dans les d\'etails, pour montrer l'existence d'une unique solution
forte sous les conditions~\eqref{edsf5B} et \eqref{edsf6B}, on proc\`ede en
deux \'etapes, tout \`a fait similaires \`a celles du cas d\'eterministe~:
\begin{itemiz}
\item	On montre que sous la condition de Lipschitz locale, toute
trajectoire solution $X_t(\omega)$ soit existe jusqu'au temps $T$, soit quitte
tout compact $\cK$ en un temps $\tau(\omega)<T$. Par cons\'equent, il existe un
temps d'arr\^et $\tau$, appel\'e \defwd{temps d'explosion}\/, tel que soit
$\tau(\omega)=+\infty$ et alors $X_t(\omega)$ existe jusqu'au temps $T$, soit
$\tau(\omega)\leqs T$, et alors $X_t(\omega)\to\pm\infty$ lorsque
$t\to\tau(\omega)$.
\item	On montre que sous la condition de croissance, les trajectoires
$X_t(\omega)$ ne peuvent pas exploser (car le terme de d\'erive ne cro\^\i t
pas assez vite, ou ram\`ene les trajectoires vers l'origine si $xf(x,t)$ est
n\'egatif). 
\end{itemiz}


\section{Exercices}
\label{sec_ex_sde}

\begin{exercice}
\label{exo_sde1}
On consid\`ere l'\'equation 
\[
\6X_t = a(t)X_t\6t + b(t)\6t + c(t)\6B_t\;,
\]
o\`u $a(t)$, $b(t)$ et $c(t)$ sont des processus adapt\'es.

R\'esoudre cette \'equation par la m\'ethode de la variation de la constante,
c'est-\`a-dire
\begin{enum}
\item	Soit $\alpha(t)=\int_0^t a(s)\6s$. V\'erifier que $X_0\e^{\alpha(t)}$ 
est la solution de l'\'equation homog\`ene, c-\`a-d avec $b=c=0$.
\item	Poser $Y_t=\e^{-\alpha(t)}X_t$ et calculer $\6Y_t$ \`a l'aide de la
formule d'It\^o.
\item	En d\'eduire $Y_t$ puis $X_t$ sous forme int\'egrale.
\item	R\'esoudre l'EDS 
\[
\6X_t = -\frac1{1+t} X_t \6t + \frac{1}{1+t} \6B_t\;,
\qquad X_0 = 0\;.
\]
\end{enum}
\end{exercice}

\begin{exercice}
\label{exo_sde2}
R\'esoudre l'EDS 
\[
\6X_t = -\frac12 X_t \6t + \sqrt{1-X_t^2} \6B_t\;,
\qquad 
X_0 = 0
\]
\`a l'aide du changement de variable $Y=\Arcsin(X)$.
\end{exercice}

\goodbreak

\begin{exercice}\hfill
\label{exo_sde3}
\begin{enum}
\item	En utilisant l'exercice~\ref{exo_sde1}, r\'esoudre l'EDS:
\[
\6X_t = \frac{b-X_t}{1-t} \6t + \6B_t\;,
\qquad 0\leqs t<1\;, \quad X_0 = a\;.
\]
\item	
D\'eterminer la variance de $X_t$. 
Calculer $\lim_{t\to 1_-} X_t$ dans $L^2$.

\item 	Soit $M_t=\int_0^t (1-s)^{-1}\6B_s$. Montrer que 
\[
\biggprob{\sup_{1-2^{-n}\leqs t\leqs 1-2^{-n-1}}(1-t)\abs{M_t}>\eps}
\leqs \frac{2}{\eps^2 2^n}\;.
\]
A l'aide du lemme de Borel--Cantelli, en d\'eduire $\lim_{t\to 1_-} X_t$
au sens presque s\^ur. 

\item	Le processus $X_t$ est appel\'e un pont Brownien --- expliquer pourquoi.
\end{enum}
\end{exercice}

\begin{exercice}
\label{exo_sde4}
On se donne $r, \alpha\in\R$. R\'esoudre l'\'equation diff\'erentielle
stochastique 
\[
\6Y_t = r\6t + \alpha Y_t \6B_t\;,
\qquad Y_0=1\;.
\]

\medskip
\noindent
{\it Indication~:}\/ Soit le \lq\lq facteur int\'egrant\rq\rq\ 
\[
F_t = \exp\biggset{-\alpha B_t + \frac12 \alpha^2 t}\;.
\]
Consid\'erer $X_t=F_tY_t$. 
\end{exercice}


\chapter{Diffusions}
\label{chap_diff}

On appelle \defwd{diffusion} un processus stochastique ob\'eissant \`a une
\'equation diff\'erentielle stochastique de la forme
\begin{equation}
\label{diff00}
\6X_t = f(X_t)\6t + g(X_t)\6B_t\;.
\end{equation}
Le terme $f(x)$ peut s'interpr\'eter comme la force d\'eterministe agissant sur
une particule dans un fluide au point $x$, et s'appelle donc le coefficient de
d\'erive. Le terme $g(x)$ mesure l'effet de l'agitation thermique des
mol\'ecules du fluide en $x$, et s'appelle le coefficient de diffusion. 

Dans l'\'etude des diffusions, il est particuli\`erement int\'eressant de
consid\'erer la d\'epen\-dance des solutions dans la condition initiale $X_0=x$.
La propri\'et\'e de Markov affirme que l'\'etat $X_t$ en un temps donn\'e $t$
d\'etermine univoquement le comportement \`a tous les temps futurs. Ceci permet
de d\'emontrer la propri\'et\'e de semi-groupe, qui g\'en\'eralise celle du
flot d'une \'equation diff\'erentielle ordinaire. Un semi-groupe de Markov peut
\^etre caract\'eris\'e par son g\'en\'erateur, qui s'av\`ere \^etre un
op\'erateur diff\'erentiel du second ordre dans le cas des diffusions. Il en
r\'esulte un ensemble de liens importants entre \'equations diff\'erentielles
stochastiques et \'equations aux d\'eriv\'ees partielles. 


\section{La propri\'et\'e de Markov}
\label{sec_diffMarkov}

\begin{definition}[Diffusion d'It\^o]
Une \defwd{diffusion d'It\^o homog\`ene dans le temps} est un processus
stochastique $\set{X_t(\omega)}_{t\geqs0}$ satisfaisant une \'equation
diff\'erentielle stochastique de la forme 
\begin{equation}
\label{diffM01}
\6X_t = f(X_t)\6t + g(X_t)\6B_t\;, 
\qquad t\geqs s>0\;, \quad X_s = x\;,
\end{equation}
o\`u $B_t$ est un mouvement Brownien standard de dimension $m$, et
le \defwd{coefficient de d\'erive} 
$f:\R^n\to\R^n$ et le \defwd{coefficient de diffusion} $g:\R^n\to\R^{n\times m}$
sont tels que l'EDS~\eqref{diffM01} admette une unique solution en tout temps.
\end{definition}

Nous noterons la solution de~\eqref{diffM01} $X^{s,x}_t$. L'homog\'en\'eit\'e
en temps, c'est-\`a-dire le fait que $f$ et $g$ ne d\'ependent pas du temps, a
la cons\'equence importante suivante. 

\begin{lemma}
\label{lem_law_diffusion} 
Les processus $\set{X^{s,x}_{s+h}}_{h\geqs0}$ et $\set{X^{0,x}_h}_{h\geqs0}$
ont la m\^eme loi. 
\end{lemma}
\begin{proof}
Par d\'efinition, $X^{0,x}_h$ satisfait l'\'equation int\'egrale 
\begin{equation}
 \label{diffM02:1} 
X^{0,x}_h = x + \int_0^h f(X^{0,x}_v)\6v + \int_0^h g(X^{0,x}_v)\6B_v\;.
\end{equation} 
De m\^eme, $X^{s,x}_{s+h}$ satisfait l'\'equation
\begin{align}
\nonumber
X^{s,x}_{s+h} &= x + \int_s^{s+h} f(X^{s,x}_u)\6u + \int_s^{s+h}
g(X^{s,x}_u)\6B_u \\
&= x + \int_0^h f(X^{s,x}_{s+v})\6v + \int_0^h
g(X^{s,x}_{s+v})\6\widetilde B_v
\label{diffM02:2} 
\end{align} 
o\`u nous avons utilis\'e le changement de variable $u=s+v$, et $\widetilde B_v
= B_{s+v} - B_s$. Par la propri\'et\'e diff\'erentielle, $\widetilde B_v$ est
un mouvement Brownien standard, donc par unicit\'e des solutions de
l'EDS~\eqref{diffM01}, les int\'egrales~\eqref{diffM02:2} et~\eqref{diffM02:1}
ont la m\^eme loi. 
\end{proof}

Nous noterons $\fP^{\mskip1.5mu x}$ la mesure de probabilit\'e sur la tribu
engendr\'ee par toutes les variables al\'eatoires $X^{0,x}_t$, $t\geqs0$,
$x\in\R^n$, d\'efinie par 
\begin{equation}
 \label{diffM03}
\bigprobin{x}{X_{t_1}\in A_1,\dots,X_{t_k}\in A_k} 
= \bigprob{X^{0,x}_{t_1}\in A_1,\dots,X^{0,x}_{t_k}\in A_k} 
\end{equation} 
pour tout choix de temps $0\leqs t_1 < t_2 < \dots < t_k$ et de bor\'eliens
$A_1, \dots,A_k\subset\R^n$. Les esp\'erances par rapport \`a $\fP^{\mskip1.5mu
x}$ seront not\'ees $\E^{\mskip1.5mu x}$. 

\begin{theorem}[Propri\'et\'e de Markov pour les diffusions d'It\^o]
\label{thm_Markov_diffusion} 
Pour toute fonction mesurable born\'ee $\varphi:\R^n\to\R$, 
\begin{equation}
 \label{diffM04}
\bigecondin{x}{\varphi(X_{t+h})}{\cF_t}(\omega) 
= \bigexpecin{X_t(\omega)}{\varphi(X_h)}\;,
\end{equation} 
le membre de droite d\'esignant la fonction 
$\expecin{y}{\varphi(X_h)}$ \'evalu\'ee en $y=X_t(\omega)$. 
\end{theorem}
\begin{proof}
Consid\'erons pour $y\in\R^n$ et $s\geqs t$ la fonction 
\begin{equation}
 \label{diffM04:1}
F(y,t,s,\omega) = X^{t,y}_s(\omega) 
= y + \int_t^s f(X_u(\omega))\6u + \int_t^s g(X_u(\omega))\6B_u(\omega)\;. 
\end{equation} 
On notera que $F$ est ind\'ependante de $\cF_t$. Par unicit\'e des solutions de
l'EDS~\eqref{diffM01}, on a 
\begin{equation}
 \label{diffM04:2}
X_s(\omega) = F(X_t(\omega),t,s,\omega)\;. 
\end{equation} 
Posons $g(y,\omega) = \varphi\circ F(y,t,t+h,\omega)$. On v\'erifie que cette
fonction est mesurable. 
La relation~\eqref{diffM04} est alors \'equivalente \`a
\begin{equation}
 \label{diffM04:3}
\bigecond{g(X_t,\omega)}{\cF_t} = \bigexpec{\varphi\circ F(y,0,h,\omega)}
\Bigevalat{y=X_t(\omega)}\;.
\end{equation} 
On a 
\begin{equation}
 \label{diffM04:4}
 \bigecond{g(X_t,\omega)}{\cF_t} = \bigecond{g(y,\omega)}{\cF_t}
\Bigevalat{y=X_t(\omega)}\;.
\end{equation} 
En effet, cette relation est vraie pour des fonctions de la forme
$g(y,\omega)=\phi(y)\psi(\omega)$, puisque 
\begin{equation}
 \label{diffM04:5}
 \bigecond{\phi(X_t)\psi(\omega)}{\cF_t}  
= \phi(X_t)\bigecond{\psi(\omega)}{\cF_t}
= \bigecond{\phi(y)\psi(\omega)}{\cF_t}
\Bigevalat{y=X_t(\omega)}\;.
\end{equation} 
Elle s'\'etend alors \`a toute fonction mesurable born\'ee en approximant
celle-ci par une suite de combinaisons lin\'eaires de fonctions comme
ci-dessus. Or il suit de l'ind\'ependance de $F$ et de $\cF_t$ que  
\begin{align}
\nonumber
\bigecond{g(y,\omega)}{\cF_t}
&= \bigexpec{g(y,\omega)} \\
\nonumber
&= \bigexpec{\varphi\circ F(y,t,t+h,\omega)} \\
&= \bigexpec{\varphi\circ F(y,0,h,\omega)}\;,
 \label{diffM04:6}
\end{align}
la derni\`ere \'egalit\'e suivant du Lemme~\ref{lem_law_diffusion}. Le
r\'esultat s'obtient alors en \'evaluant la derni\`ere \'egalit\'e en $y=X_t$. 
\end{proof}

Comme pour le mouvement Brownien, la propri\'et\'e de Markov se g\'en\'eralise
\`a des temps d'arr\^et. 

\begin{theorem}[Propri\'et\'e de Markov forte pour les diffusions d'It\^o]
\label{thm_strong_Markov_diffusion} 
Pour toute fonction mesurable born\'ee $\varphi:\R^n\to\R$ et tout temps
d'arr\^et $\tau$ fini presque s\^urement, 
\begin{equation}
 \label{diffM05}
\bigecondin{x}{\varphi(X_{\tau+h})}{\cF_\tau}(\omega) 
= \bigexpecin{X_\tau(\omega)}{\varphi(X_h)}\;.
\end{equation} 
\end{theorem}
\begin{proof}
La preuve est une adaptation relativement directe de la preuve pr\'ec\'e\-dente.
Voir par exemple~\cite[Theorem~7.2.4]{Oksendal}.
\end{proof}


\section{Semigroupes et g\'en\'erateurs}
\label{sec_diffgen}

\begin{definition}[Semi-groupe de Markov]
\label{def_semigroup}
A toute fonction mesurable born\'ee $\varphi:\R^n\to\R$, on associe pour tout
$t\geqs0$ la fonction $T_t\varphi$ d\'efinie par 
\begin{equation}
 \label{diffsg1}
(T_t\varphi)(x) = \bigexpecin{x}{\varphi(X_t)}\;. 
\end{equation}  
L'op\'erateur lin\'eaire $T_t$ est appel\'e le \defwd{semi-groupe de Markov}
associ\'e \`a la diffusion. 
\end{definition}

Par exemple, si $\varphi(x)=\indicator{A}(x)$ est la fonction indicatrice d'un
Bor\'elien $A\subset\R^n$, on a 
\begin{equation}
 \label{diffsg2}
(T_t\indicator{A})(x) = \bigprobin{x}{X_t\in A}\;. 
\end{equation} 

Le nom de semi-groupe est justifi\'e par le r\'esultat suivant.

\begin{lemma}[Propri\'et\'e de semi-groupe]
Pour tous $t, h\geqs 0$, on a 
\begin{equation}
 \label{diffsg3}
T_h\circ T_t = T_{t+h}\;. 
\end{equation} 
\end{lemma}
\begin{proof}
On a 
\begin{align}
\nonumber
(T_h\circ T_t)(\varphi)(x)
&= (T_h(T_t\varphi))(x) \\
\nonumber
&= \bigexpecin{x}{(T_t\varphi)(X_h)} \\
\nonumber
&= \bigexpecin{x}{\bigexpecin{X_h}{\varphi(X_t)}} \\
\nonumber
&= \bigexpecin{x}{\bigecondin{x}{\varphi(X_{t+h})}{\cF_t}} \\
\nonumber
&= \bigexpecin{x}{\varphi(X_{t+h})} \\
&= (T_{t+h}\varphi)(x)\;,
\label{diffsg3:1} 
\end{align}
o\`u l'on a utilis\'e la propri\'et\'e de Markov pour passer de la troisi\`eme
\`a la quatri\`eme ligne. 
\end{proof}

De plus, on v\'erifie facilement les propri\'et\'es suivantes:
\begin{enum}
\item	$T_t$ pr\'eserve les fonctions constantes: $T_t (c\indicator{\R^n}) =
c\indicator{\R^n}$;
\item	$T_t$ pr\'eserve les fonctions non-n\'egatives: $\varphi(x)\geqs0
\;\forall x \Rightarrow (T_t\varphi)(x)\geqs 0 \;\forall x$;
\item	$T_t$ est contractante par rapport \`a la norme $L^\infty$:
\begin{equation}
 \label{diffsg4}
\sup_{x\in\R^n} \bigabs{(T_t\varphi)(x)} = \sup_{x\in\R^n} 
\bigabs{\bigexpecin{x}{\varphi(X_t)}}
\leqs \sup_{y\in\R^n}\bigabs{\varphi(y)} 
\sup_{x\in\R^n} \bigexpecin{x}{1} = \sup_{y\in\R^n}\bigabs{\varphi(y)}\;.
\end{equation} 
\end{enum}
Le semi-groupe de Markov est donc un op\'erateur lin\'eaire positif, born\'e par
rapport \`a la norme~$L^\infty$. En fait, il est de norme op\'erateur $1$. 

La propri\'et\'e de semi-groupe implique que le comportement de $T_t$ sur tout
intervalle $[0,\eps]$, avec $\eps>0$ arbitrairement petit, d\'etermine son
comportement pour tout $t\geqs0$. Il est donc naturel de consid\'erer la
d\'eriv\'ee de $T_t$ en $t=0$. 

\begin{definition}[G\'en\'erateur d'une diffusion d'It\^o]
\label{def_generateur}
Le \defwd{g\'en\'erateur infinit\'esimal} $L$ d'une diffusion d'It\^o est
d\'efini par son action sur une fonction test $\varphi$ via 
\begin{equation}
 \label{diffsg5}
(L\varphi)(x) = \lim_{h\to0_+}
\frac{(T_h\varphi)(x) - \varphi(x)}{h}\;.
\end{equation}  
Le domaine de $L$ est par d\'efinition l'ensemble des fonctions $\varphi$ pour
lesquelles la limite~\eqref{diffsg5} existe pour tout $x\in\R^n$. 
\end{definition}

\begin{remark}
\label{rem_generator} 
Formellement, la relation~\eqref{diffsg5} peut s'\'ecrire
\begin{equation}
 \label{diffsg5A}
L = \dtot{T_t}{t}\Bigevalat{t=0}\;.
\end{equation} 
Par la propri\'et\'e de Markov, cette relation se g\'en\'eralise en 
\begin{equation}
 \label{diffsg5B} 
\dtot{}{t} T_t = \lim_{h\to0_+} \frac{T_{t+h}-T_t}{h}
= \lim_{h\to0_+} \frac{T_h-\id}{h}T_t
= LT_t\;,
\end{equation}
et on peut donc \'ecrire formellement 
\begin{equation}
 \label{diffsg5C}
T_t = \e^{tL}\;. 
\end{equation} 
Nous pr\'eciserons ce point dans la Section~\ref{sec_diffKolmogorov}.
\end{remark}

\begin{prop}
\label{prop_generator_Ito}
Le g\'en\'erateur de la diffusion d'It\^o~\eqref{diffM01} est l'op\'erateur
diff\'eren\-tiel 
\begin{equation}
 \label{diffsg6}
L = \sum_{i=1}^n f_i(x) \dpar{}{x_i}
+ \frac{1}{2} \sum_{i,j=1}^n (gg^T)_{ij}(x) \dpar{^2}{x_i\partial x_j}\;. 
\end{equation} 
Le domaine de $L$ contient l'ensemble des fonctions deux fois contin\^ument
diff\'erentiables \`a support compact. 
\end{prop}
\begin{proof}
Consid\'erons le cas $n=m=1$. Soit $\varphi$ une fonction deux fois
contin\^ument diff\'erentiable \`a support compact, et soit $Y_t=\varphi(X_t)$.
Par la formule d'It\^o, 
\begin{equation}
 \label{diffsg7:1}
Y_h = \varphi(X_0) + \int_0^h \varphi'(X_s)f(X_s)\6s 
+ \int_0^h \varphi'(X_s)g(X_s)\6B_s + \frac12\int_0^h
\varphi''(X_s)g(X_s)^2 \6s\;.
\end{equation}  
En prenant l'esp\'erance, comme l'esp\'erance de l'int\'egrale d'It\^o est
nulle, on trouve 
\begin{equation}
 \label{diffsg7:2}
\bigexpecin{x}{Y_h} = \varphi(x) + \biggexpecin{x}{
\int_0^h \varphi'(X_s) f(X_s)\6s + \frac12\int_0^h \varphi''(X_s) g(X_s)^2\6s
 }\;,
\end{equation}
d'o\`u 
\begin{equation}
 \label{diffsg7:3}
\frac{\bigexpecin{x}{\varphi(X_h)} - \varphi(x)}{h} 
= \frac 1h
\int_0^h \bigexpecin{x}{\varphi'(X_s) f(X_s)}\6s 
+ \frac1{2h}\int_0^h
\bigexpecin{x}{\varphi''(X_s) g(X_s)^2}\6s\;.
\end{equation}  
En prenant la limite $h\to0_+$, on obtient 
\begin{equation}
 \label{diffsg7:4}
(L\varphi)(x) = \varphi'(x)f(x) + \frac12 \varphi''(x)g(x)^2\;. 
\end{equation} 
Les cas o\`u $n\geqs2$ ou $m\geqs2$ se traitent de mani\`ere similaire, en
utilisant la formule d'It\^o multidimensionnelle. 
\end{proof}

\begin{example}[G\'en\'erateur du mouvement Brownien]
Soit $B_t$ le mouvement Brownien de dimension $m$. C'est un cas particulier de
diffusion, avec $f=0$ et $g=\one$. Son g\'en\'erateur est donn\'e par 
\begin{equation}
 \label{diffsg8}
L = \frac{1}{2}\sum_{i=1}^m \dpar{^2}{x_i^2} = \frac12\Delta\;. 
\end{equation} 
C'est donc le Laplacien \`a un facteur $1/2$ pr\`es. 
\end{example}


\section{La formule de Dynkin}
\label{sec_diffDynkin}

La formule de Dynkin est essentiellement une g\'en\'eralisation de
l'expression~\eqref{diffsg7:2} \`a des temps d'arr\^et. Elle fournit une
premi\`ere classe de liens entre diffusions et \'equations aux d\'eriv\'ees
partielles. 

\begin{prop}[Formule de Dynkin]
Soit $\set{X_t}_{t\geqs0}$ une diffusion de g\'en\'erateur $L$,
$x\in\R^n$, $\tau$ un temps d'arr\^et tel que $\expecin{x}{\tau}<\infty$, et
$\varphi:\R^n\to\R$ une fonction deux fois contin\^ument diff\'erentiable \`a
support compact. Alors 
\begin{equation}
 \label{diffD01}
\bigexpecin{x}{\varphi(X_\tau)} 
= \varphi(x) + \biggexpecin{x}{\int_0^\tau (L\varphi)(X_s)\6s}\;. 
\end{equation}  
\end{prop}
\begin{proof}
Consid\'erons le cas $n=m=1$, $m$ \'etant la dimension du mouvement Brownien.
En proc\'edant comme dans la preuve de la Proposition~\ref{prop_generator_Ito},
on obtient 
\begin{equation}
 \label{diffD01:1}
\bigexpecin{x}{\varphi(X_\tau)} = \varphi(x) 
+  \biggexpecin{x}{\int_0^\tau (L\varphi)(X_s)\6s}
+  \biggexpecin{x}{\int_0^\tau g(X_s)\varphi'(X_s)\6B_s}\;.
\end{equation} 
Il suffit donc de montrer que l'esp\'erance de l'int\'egrale stochastique est
nulle. Or pour toute fonction $h$ born\'ee par $M$ et tout $N\in\N$, on a 
\begin{equation}
 \label{diffD01:2}
\biggexpecin{x}{\int_0^{\tau\wedge N} h(X_s)\6B_s}
= \biggexpecin{x}{\int_0^N \indexfct{s<\tau}h(X_s)\6B_s} = 0\;,
\end{equation} 
en vertu de la $\cF_s$-mesurabilit\'e de $\indexfct{s<\tau}$ et $h(X_s)$.
De plus, 
\begin{align}
\nonumber
\biggexpecin{x}{\biggbrak{\int_0^\tau h(X_s)\6B_s - \int_0^{\tau\wedge
N} h(X_s)\6B_s}^2}
&= \biggexpecin{x}{\int_{\tau\wedge N}^\tau h(X_s)^2\6s} \\
&\leqs M^2 \bigexpecin{x}{\tau - \tau\wedge N}\;,
 \label{diffD01:3}
\end{align}
qui tend vers $0$ lorsque $N\to\infty$, en vertu de l'hypoth\`ese
$\expecin{x}{\tau}<\infty$, par convergence domin\'ee. On peut donc \'ecrire 
\begin{equation}
 \label{diffD01:4}
0 = \lim_{N\to\infty}  \biggexpecin{x}{\int_0^{\tau\wedge N} h(X_s)\6B_s}
= \biggexpecin{x}{\int_0^{\tau} h(X_s)\6B_s}\;,
\end{equation} 
ce qui conclut la preuve, en substituant dans~\eqref{diffD01:1}. La preuve du
cas g\'en\'eral est analogue. 
\end{proof}

Consid\'erons le cas o\`u le temps d'arr\^et $\tau$ est le temps de premi\`ere
sortie d'un ouvert born\'e $D\subset\R^n$. Supposons que le probl\`eme avec
conditions au bord 
\begin{align}
\nonumber
(Lu)(x) &= \theta(x) & x&\in D \\
u(x) &= \psi(x) & x&\in\partial D
\label{diffD02}
\end{align}
admet une unique solution. C'est le cas si $D$, $\theta$ et $\psi$ 
sont suffisamment r\'eguliers. Substituant $\varphi$ par $u$ dans la formule de
Dynkin, on obtient la relation 
\begin{equation}
 \label{diffD03}
u(x) = \biggexpecin{x}{\psi(X_\tau) - \int_0^\tau \theta(X_s)\6s}\;. 
\end{equation} 
Pour $\psi=0$ et $\theta=-1$, $u(x)$ est \'egal \`a l'esp\'erance de $\tau$,
partant de $x$. Pour $\theta=0$ et $\psi$ l'indicatrice d'une partie $A$ du
bord $\partial D$, $u(x)$ est la probabilit\'e de quitter $D$ par $A$.
Ainsi, si l'on sait r\'esoudre le probl\`eme~\eqref{diffD02}, on obtient des
informations sur le temps et le lieu de sortie de $D$. Inversement, en
simulant l'expression~\eqref{diffD03} par une m\'ethode de Monte-Carlo, on
obtient une approximation num\'erique de la solution du
probl\`eme~\eqref{diffD02}. 

\begin{example}[Temps de sortie moyen du mouvement Brownien d'une boule]
\label{ex_MB_exit_ball}
Soit $K=\setsuch{x\in\R^n}{\norm{x}<R}$ la boule de rayon $R$ centr\'ee \`a
l'origine. Soit 
\begin{equation}
 \label{diffD04}
\tau_K = \inf\setsuch{t>0}{x+B_t\not\in K} 
\end{equation} 
et soit 
\begin{equation}
 \label{diffD05}
\tau(N) = \tau_K\wedge N\;. 
\end{equation} 
La fonction $\varphi(x)=\norm{x}^2\indexfct{\norm{x}\leqs R}$ est \`a
support compact et satisfait $\Delta \varphi(x)=2n$ pour $x\in K$. On peut par
ailleurs la prolonger en dehors de $K$ de mani\`ere qu'elle soit lisse et \`a
support compact. En substituant dans la formule de Dynkin, on obtient 
\begin{align}
 \nonumber
\bigexpecin{x}{\norm{x+B_{\tau(N)}}^2} &= \norm{x}^2 + 
\biggexpecin{x}{\int_0^{\tau(N)} \frac12\Delta\varphi(B_s)\6s} \\
&= \norm{x}^2 + n \bigexpecin{x}{\tau(N)}\;.
 \label{diffD06}
\end{align} 
Comme $\norm{x+B_{\tau(N)}}\leqs R$, 
faisant tendre $N$ vers l'infini, on obtient par convergence domin\'ee 
\begin{equation}
 \label{diffD07}
\bigexpecin{x}{\tau_K} = \frac{R^2 - \norm{x}^2}{n}\;. 
\end{equation} 
\end{example}

\begin{example}[R\'ecurrence/transience du mouvement Brownien]
\label{ex_MB_recurrent_transient} 
Soit \`a nouveau $K=\setsuch{x\in\R^n}{\norm{x}<R}$. Nous consid\'erons
maintenant le cas o\`u $x\not\in K$, et nous voulons d\'eterminer si le
mouvement Brownien partant de $x$ touche $K$ presque s\^urement, on dit
alors qu'il est \defwd{r\'ecurrent}, ou s'il touche $K$ avec une probabilit\'e
strictement inf\'erieure \`a $1$, on dit alors qu'il est \defwd{transient}.
Comme dans le cas des marches al\'eatoires, la r\'eponse d\'epend de la
dimension $n$ de l'espace. 

Nous d\'efinissons
\begin{equation}
 \label{diffD08}
\tau_K = \inf\setsuch{t>0}{x+B_t\in K}\;. 
\end{equation} 
Pour $N\in\N$, soit $A_N$ l'anneau 
\begin{equation}
 \label{diffD09}
A_N = \setsuch{x\in\R^n}{R<\norm{x}<2^N R}\;, 
\end{equation} 
et soit $\tau$ le temps de premi\`ere sortie de $x+B_t$ de $A_N$. On a donc 
\begin{equation}
 \label{diffD10}
\tau = \tau_K \wedge \tau'\;, 
\qquad
\tau' = \inf\setsuch{t>0}{\norm{x+B_t}=2^NR}\;. 
\end{equation} 
Soit enfin 
\begin{equation}
 \label{diffD11}
p = \bigprobin{x}{\tau_K < \tau'} 
= \bigprobin{x}{\norm{x+B_\tau} = R}
= 1 - \bigprobin{x}{\norm{x+B_\tau} = 2^NR}\;.  
\end{equation} 
Les solutions \`a sym\'etrie sph\'erique de $\Delta\varphi=0$ sont de la forme
\begin{equation}
 \label{diffD12}
\varphi(x) = 
\begin{cases}
\abs{x} & \text{si $n=1$\;,} \\
-\log\norm{x} & \text{si $n=2$\;,} \\
\norm{x}^{2-n} & \text{si $n>2$\;.}
\end{cases}
\end{equation} 
Pour un tel $\varphi$, la formule de Dynkin donne 
\begin{equation}
 \label{diffD13}
\bigexpecin{x}{\varphi(x+B_\tau)} = \varphi(x)\;. 
\end{equation} 
Par ailleurs, on a 
\begin{equation}
 \label{diffD14}
\bigexpecin{x}{\varphi(x+B_\tau)} = \varphi(R) p + \varphi(2^N R) (1-p)\;. 
\end{equation} 
En r\'esolvant par rapport \`a $p$, on obtient 
\begin{equation}
 \label{diffD15}
p = \frac{\varphi(x) - \varphi(2^N R)}{\varphi(R) - \varphi(2^N R)}\;. 
\end{equation} 
Lorsque $N\to\infty$, on a $\tau'\to\infty$, d'o\`u 
\begin{equation}
 \label{diffD16}
\bigprobin{x}{\tau_K < \infty} = \lim_{N\to\infty}
\frac{\varphi(x) - \varphi(2^N R)}{\varphi(R) - \varphi(2^N R)}\;. 
\end{equation} 
Consid\'erons alors s\'epar\'ement les cas $n=1$, $n=2$ et $n>2$. 
\begin{enum}
\item	Pour $n=1$, on a 
\begin{equation}
 \label{diffD17}
\bigprobin{x}{\tau_K < \infty} = \lim_{N\to\infty}
\frac{2^N R - \abs{x}}{2^N R - R} = 1\;, 
\end{equation} 
donc le mouvement Brownien est r\'ecurrent en dimension $1$. 
\item	Pour $n=2$, on a 
\begin{equation}
 \label{diffD18}
\bigprobin{x}{\tau_K < \infty} = \lim_{N\to\infty}
\frac{\log\norm{x} + N\log2 - \log R}{N\log2} = 1\;, 
\end{equation} 
donc le mouvement Brownien est \'egalement r\'ecurrent en dimension $2$. 
\item	Pour $n>2$, on a 
\begin{equation}
 \label{diffD19}
\bigprobin{x}{\tau_K < \infty} = \lim_{N\to\infty}
\frac{(2^N R)^{2-n} + \norm{x}^{2-n}}{(2^N R)^{2-n} + R^{2-n}} =
\biggpar{\frac{R}{\norm{x}}}^{n-2} < 1\;. 
\end{equation} 
Le mouvement Brownien est donc transient en dimension $n>2$. 
\end{enum}
\end{example}


\section{Les \'equations de Kolmogorov}
\label{sec_diffKolmogorov}

La seconde classe de liens entre \'equations diff\'erentielles stochastiques et
\'equations aux d\'eriv\'ees partielles est constitu\'ee par les \'equations
de Kolmogorov, qui sont des probl\`emes aux valeurs initiales. 

On remarque qu'en d\'erivant par rapport \`a $t$ la formule de Dynkin, dans le
cas particulier $\tau=t$, on obtient 
\begin{equation}
 \label{diffK01}
\dpar{}{t} (T_t\varphi)(x) = 
\dpar{}{t} \bigexpecin{x}{\varphi(X_t)}
= \bigexpecin{x}{(L\varphi)(X_t)} = (T_tL\varphi)(x)\;,
\end{equation} 
que l'on peut abr\'eger sous la forme 
\begin{equation}
 \label{diffK02}
\dtot{}{t} T_t = T_t L\;. 
\end{equation} 
Or nous avons vu dans le remarque~\ref{rem_generator} que l'on pouvait aussi
\'ecrire formellement $\dtot{}{t} T_t = L T_t$. Par cons\'equent, les
op\'erateurs $L$ et $T_t$ commutent, du moins formellement. Le th\'eor\`eme
suivant rend ce point rigoureux. 

\begin{theorem}[Equation de Kolmogorov r\'etrograde]
Soit $\varphi:\R^n\to\R$ une fonction deux fois contin\^ument diff\'erentiable
\`a support compact. 
\begin{enum}
\item	La fonction 
\begin{equation}
 \label{diffK03}
u(t,x) = (T_t\varphi)(x) = \bigexpecin{x}{\varphi(X_t)} 
\end{equation} 
satisfait le probl\`eme aux valeurs initiales 
\begin{align}
\nonumber
\dpar ut(t,x) &= (Lu)(t,x)\;, && t>0\;, \quad x\in\R^n\;, \\
u(0,x) &= \varphi(x)\;, && x\in\R^n\;.
\label{diffK04} 
\end{align} 
\item	Si $w(t,x)$ est une fonction born\'ee, contin\^ument diff\'erentiable
en $t$ et deux fois contin\^ument diff\'erentiable en $x$, satisfaisant le
probl\`eme~\eqref{diffK04}, alors $w(t,x)=(T_t\varphi)(x)$.
\end{enum}
\end{theorem}
\begin{proof} \hfill
\begin{enum}
\item	On a $u(0,x)=(T_0\varphi)(x)=\varphi(x)$ et 
\begin{align}
\nonumber
(Lu)(t,x) 
&= \lim_{h\to0_+} \frac{(T_h\circ T_t\varphi)(x) - (T_t\varphi)(x)}{h} \\
\nonumber
&= \lim_{h\to0_+} \frac{(T_{t+h}\varphi)(x) - (T_t\varphi)(x)}{h} \\
&= \dpar{}{t} (T_t\varphi)(x) = \dpar{}{t} u(t,x)\;.
\label{diffK04:1}
\end{align}
\item	Si $w(t,x)$ satisfait le probl\`eme~\eqref{diffK04}, alors on a 
\begin{equation}
 \label{diffK04:2}
\widetilde L w = 0 
\qquad 
\text{o\`u } \quad \widetilde L w = -\dpar{w}{t} + L w\;. 
\end{equation} 
Fixons $(s,x)\in\R_+\times\R^n$. Le processus $Y_t=(s-t,X^{0,x}_t)$ admet
$\widetilde L$ comme
g\'en\'erateur. Soit 
\begin{equation}
 \label{diffK04:3}
\tau_R = \inf\setsuch{t>0}{\norm{X_t}\geqs R}\;. 
\end{equation} 
La formule de Dynkin montre que 
\begin{equation}
 \label{diffK04:4}
\bigexpecin{s,x}{w(Y_{t\wedge\tau_R})} 
= w(s,x) + \biggexpecin{s,x}{\int_0^{t\wedge\tau_R} (\widetilde L w)(Y_u)\6u}
= w(s,x)\;. 
\end{equation} 
Faisant tendre $R$ vers l'infini, on obtient 
\begin{equation}
 \label{diffK04:5}
w(s,x) = \bigexpecin{s,x}{w(Y_t)} \qquad \forall t\geqs0\;.
\end{equation} 
En particulier, prenant $t=s$, on a 
\begin{equation}
 \label{diffK04:6}
w(s,x) = \bigexpecin{s,x}{w(Y_s)} 
= \bigexpec{w(0,X^{0,x}_s)}
= \bigexpec{\varphi(X^{0,x}_s)} 
= \bigexpecin{x}{\varphi(X_s)}\;.
\end{equation} 
\end{enum}
\end{proof}

On remarquera que dans le cas du mouvement Brownien, dont le g\'en\'erateur est 
$L=\frac12\Delta$, l'\'equation de Kolmogorov r\'etrograde~\eqref{diffK04} est
l'\'equation de la chaleur. 

La lin\'earit\'e de l'\'equation de Kolmogorov r\'etrograde implique qu'il
suffit de la r\'esoudre pour une famille compl\`ete de conditions initiales
$\varphi$, pour conna\^\i tre la solution pour toute condition initiale. 

Un premier cas important est celui o\`u l'on conna\^\i t toutes les fonctions
propres et valeurs propres de $L$. Dans ce cas, la solution g\'en\'erale se
d\'ecompose sur les fonctions propres, avec des coefficients d\'ependant
exponentiellement du temps.

\begin{example}[Mouvement Brownien]
Les fonctions propres du g\'en\'erateur $L=\frac12\dtot{^2}{x^2}$ du mouvement
Brownien unidimensionnel sont de la forme $\e^{\icx kx}$. D\'ecomposer la
solution sur la base de ces fonctions propres revient \`a r\'esoudre
l'\'equation de la chaleur par transformation de Fourier. On sait que la
solution s'\'ecrit 
\begin{equation}
 \label{diffK05}
u(t,x) = \frac{1}{\sqrt{2\pi}} \int_\R \e^{-k^2t/2}\hat \varphi(k) \e^{\icx kx}
\6k\;, 
\end{equation}
o\`u $\hat \varphi(k)$ est la transform\'ee de Fourier de la condition
initiale. 
\end{example}

Un second cas important revient \`a d\'ecomposer formellement la condition
initiale sur une \lq\lq base\rq\rq\ de distributions de Dirac. En pratique,
cela revient \`a utiliser la notion de densit\'e de transition. 

\begin{definition}[Densit\'e de transition]
\label{transition_density}
On dit que la diffusion $\set{X_t}_t$ admet la \defwd{densit\'e de transition}
$p_t(x,y)$, aussi not\'ee $p(y,t|x,0)$, si 
\begin{equation}
 \label{diffK06}
\bigexpecin{x}{\varphi(X_t)} = \int_{\R^n} \varphi(y)p_t(x,y)\6y  
\end{equation} 
pour toute fonction mesurable born\'ee $\varphi:\R^n\to\R$. 
\end{definition}

Par lin\'earit\'e, la densit\'e de transition, si elle existe et est lisse,
satisfait l'\'equation de Kolmogorov r\'etrograde (le g\'en\'erateur $L$
agissant sur la variable $x$), avec la condition initiale
$p_0(x,y)=\delta(x-y)$. 

\begin{example}[Mouvement Brownien et noyau de la chaleur]
Dans le cas du mouvement Brownien unidimensionnel, nous avons vu
(c.f.~\eqref{pW2}) que la densit\'e de transition \'etait donn\'ee par 
\begin{equation}
 \label{diffK07}
p(y,t|x,0) = \frac{1}{\sqrt{2\pi t}} \e^{-(x-y)^2/2t}\;, 
\end{equation} 
qui est appel\'e le \defwd{noyau de la chaleur}. C'est \'egalement
la valeur de l'int\'egrale~\eqref{diffK05} avec $\hat\varphi(k)=\e^{-\icx k
y}/\sqrt{2\pi}$, qui est bien la transform\'ee de Fourier de
$\varphi(x)=\delta(x-y)$. 
\end{example}

L'adjoint du g\'en\'erateur $L$ est par d\'efinition l'op\'erateur lin\'eaire
$L^*$ tel que 
\begin{equation}
 \label{diffK08}
\pscal{L\phi}{\psi} = \pscal{\phi}{L^*\psi}
\end{equation} 
pour tout choix de fonctions $\phi,\psi:\R^n\to\R$ deux fois contin\^ument
diff\'erentiables, avec $\phi$ \`a support compact, o\`u $\pscal{\cdot}{\cdot}$
d\'esigne le produit scalaire usuel de $L^2$. En int\'egrant
$\pscal{L\phi}{\psi}$ deux fois par parties, on obtient 
\begin{equation}
 \label{diffK09}
(L^*\psi)(y) = \frac12\sum_{i,j=1}^n \dpar{^2}{y_i\partial y_j} 
\bigpar{(gg^T)_{ij}\psi}(y) - \sum_{i=1}^n \dpar{}{y_i} 
\bigpar{f_i\psi}(y)\;.
\end{equation} 

\begin{theorem}[Equation de Kolmogorov progressive]
Si $X_t$ poss\`ede une densit\'e de transition lisse $p_t(x,y)$, alors celle-ci
satisfait l'\'equation 
\begin{equation}
 \label{diffK10}
\dpar{}{t} p_t(x,y) = L^*_y p_t(x,y)\;, 
\end{equation} 
la notation $L^*_y$ signifiant que $L^*$ agit sur la variable $y$. 
\end{theorem}
\begin{proof}
La formule de Dynkin, avec $\tau=t$, implique 
\begin{align}
\nonumber
\int_{\R^n} \varphi(y) p_t(x,y)\6y
&= \bigexpecin{x}{\varphi(X_t)} \\
\nonumber
&= \varphi(x) + \int_0^t \bigexpecin{x}{(L\varphi)(X_s)} \6s \\
&= \varphi(x) + \int_0^t \int_{\R^n} (L\varphi)(y) p_s(x,y)\6y\;.
 \label{diffK10:1}
\end{align}
En d\'erivant par rapport au temps, et en utilisant~\eqref{diffK08}, il vient 
\begin{equation}
 \label{diffK10:2}
\dpar{}{t} \int_{\R^n} \varphi(y) p_t(x,y)\6y
= \int_{\R^n} (L\varphi)(y) p_t(x,y)\6y
= \int_{\R^n} \varphi(y) (L^*_yp_t)(x,y)\6y\;,
\end{equation} 
d'o\`u le r\'esultat. 
\end{proof}

Supposons que la loi $X_0$ admette une densit\'e $\rho$ par rapport \`a la
mesure de Lebesgue. Alors $X_t$ aura une densit\'e donn\'ee par 
\begin{equation}
 \label{diffK10:3}
\rho(t,y) = (S_t\rho)(y) \defby \int_{\R^n} p_t(x,y)\rho(x)\6x\;. 
\end{equation} 
En appliquant l'\'equation de Kolmogorov progressive~\eqref{diffK10}, on
obtient l'\defwd{\'equation de Fokker--Planck}
\begin{equation}
 \label{diffK10:4}
\dpar{}{t} \rho(t,y) = L^*_y \rho(t,y)\;,
\end{equation} 
que l'on peut aussi \'ecrire formellement 
\begin{equation}
 \label{diffK10:5}
\dtot{}{t} S_t = L^* S_t\;. 
\end{equation} 
Le g\'en\'erateur adjoint $L^*$ est donc le g\'en\'erateur du semi-groupe
adjoint $S_t$. 

\begin{cor}
\label{cor_stationary}
Si $\rho_0(y)$ est la densit\'e d'une mesure de probabilit\'e satisfaisant
$L^*\rho_0=0$, alors $\rho_0$ est une mesure stationnaire de la diffusion. En
d'autres termes, si la loi de $X_0$ admet la densit\'e $\rho_0$, alors $X_t$
admettra la densit\'e $\rho_0$ pour tout $t\geqs0$.  
\end{cor}


\section{La formule de Feynman--Kac}
\label{sec_diffFeynmanKac}

Jusqu'ici nous avons rencontr\'e des probl\`emes \`a valeurs au bord elliptiques
de la forme $Lu=\theta$, et des \'equations d'\'evolution paraboliques de la
forme $\sdpar ut=Lu$. Le formule de Feynman--Kac montre qu'on peut \'egalement
lier des propri\'et\'es d'une diffusion \`a celles d'\'equations paraboliques
o\`u le g\'en\'erateur contient un terme lin\'eaire en $u$.  

L'ajout d'un terme lin\'eaire dans le g\'en\'erateur peut
s'interpr\'eter comme le fait de \lq\lq tuer\rq\rq\ la diffusion avec un
certain taux. Le cas le plus simple est celui d'un taux constant. Soit $\zeta$
une variable al\'eatoire de loi exponentielle de param\`etre $\lambda$,
ind\'ependante de $B_t$. Posons
\begin{equation}
 \label{FK01} 
\widetilde X_t = 
\begin{cases}
X_t & \text{si $t<\zeta$\;,} \\
\Delta & \text{si $t\geqs\zeta$\;,}
\end{cases}
\end{equation} 
o\`u $\Delta$ est un \lq\lq \'etat cimeti\`ere\rq\rq\ que l'on a ajout\'e \`a
$\R^n$. On v\'erifie que gr\^ace au caract\`ere exponentiel de $\zeta$,
$\widetilde X_t$ est un processus de Markov sur $\R^n\cup\set{\Delta}$. 
Si $\varphi:\R^n\to\R$ est une fonction test mesurable born\'ee, on aura
(si l'on pose $\varphi(\Delta)=0$)
\begin{equation}
 \label{FK02}
\bigexpecin{x}{\varphi(\widetilde X_t)}
=  \bigexpecin{x}{\varphi(X_t) \indexfct{t<\zeta}} 
= \prob{\zeta>t} \bigexpecin{x}{\varphi(X_t)} 
= \e^{-\lambda t} \bigexpecin{x}{\varphi(X_t)}\;.
\end{equation} 
Il suit que 
\begin{equation}
 \label{FK03}
\lim_{h\to0} \frac{\bigexpecin{x}{\varphi(\widetilde X_h)} - \varphi(x)}{h} 
= -\lambda \varphi(x) + (L\varphi)(x)\;,
\end{equation} 
ce qui montre que le g\'en\'erateur infinit\'esimal de $\widetilde X$ est
l'op\'erateur diff\'erentiel 
\begin{equation}
 \label{FK04}
\widetilde L = L - \lambda\;. 
\end{equation} 
Plus g\'en\'eralement, si $q:\R^n\to\R$ est une fonction continue, born\'ee
inf\'erieurement, on peut construire une variable al\'eatoire $\zeta$ telle que 
\begin{equation}
 \label{FK05}
 \bigexpecin{x}{\varphi(\widetilde X_t)}
= \bigexpecin{x}{\varphi(X_t) \e^{-\int_0^t q(X_s)\6s}}\;.
\end{equation} 
Dans ce cas le g\'en\'erateur de $\widetilde X_t$ sera 
\begin{equation}
 \label{FK06}
 \widetilde L = L - q\;, 
\end{equation} 
c'est-\`a-dire $(\widetilde L\varphi)(x)=(L\varphi)(x) - q(x)\varphi(x)$. 

\begin{theorem}[Formule de Feynman--Kac]
Soit $\varphi:\R^n\to\R$ une fonction deux fois contin\^ument diff\'erentiable
\`a support compact, et soit $q:\R^n\to\R$ une fonction continue et born\'ee
inf\'erieurement. 
\begin{enum}
\item	La fonction 
\begin{equation}
 \label{FK07}
v(t,x) = \Bigexpecin{x}{\e^{-\int_0^t q(X_s)\6s}\varphi(X_t)} 
\end{equation} 
satisfait le probl\`eme aux valeurs initiales 
\begin{align}
\nonumber
\dpar vt(t,x) &= (Lv)(t,x) - q(x)v(x)\;, && t>0\;, \quad x\in\R^n\;, \\
v(0,x) &= \varphi(x)\;, && x\in\R^n\;.
\label{FK08} 
\end{align} 
\item	Si $w(t,x)$ est une fonction contin\^ument diff\'erentiable
en $t$ et deux fois contin\^ument diff\'erentiable en $x$, born\'ee pour $x$
dans un compact, satisfaisant le
probl\`eme~\eqref{FK08}, alors $w(t,x)$ est \'egale au membre de droite
de~\eqref{FK07}.
\end{enum}
\end{theorem}

\begin{proof}\hfill
\begin{enum}
\item	Soit $Y_t=\varphi(X_t)$ et $Z_t=\e^{-\int_0^t q(X_s)\6s}$, et soit
$v(t,x)$ donn\'ee par~\eqref{FK07}. Alors 
\begin{align}
\nonumber
\frac{1}{h} \Bigbrak{\bigexpecin{x}{v(t,X_h)} - v(t,x)}
={}& \frac{1}{h} \Bigbrak{\Bigexpecin{x}{\bigexpecin{X_h}{Y_tZ_t}} -
\bigexpecin{x}{Y_tZ_t}} \\
\nonumber
={}& \frac{1}{h} \Bigbrak{\Bigexpecin{x}{\econdin{x}{Y_{t+h}\e^{-\int_0^t
q(X_{s+h})\6s}}{\cF_h} - Y_tZ_t}} \\
\nonumber
={}& \frac{1}{h} \Bigexpecin{x}{Y_{t+h}Z_{t+h}\e^{\int_0^h q(X_s)\6s}
-Y_tZ_t} \\
\nonumber
={}& \frac{1}{h} \Bigexpecin{x}{Y_{t+h}Z_{t+h}-Y_tZ_t} \\
{}& - \frac{1}{h} \Bigexpecin{x}{Y_{t+h}Z_{t+h}
\Bigbrak{\e^{\int_0^h q(X_s)\6s}-1}}\;.
\label{FK09:1} 
\end{align}
Lorsque $h$ tend vers $0$, le premier terme de la derni\`ere expression tend
vers $\sdpar vt(t,x)$, alors que le second tend vers $q(x)v(t,x)$. 

\item	Si $w(t,x)$ satisfait le probl\`eme~\eqref{FK08}, alors on a 
\begin{equation}
 \label{FK09:2}
\widetilde L w = 0 
\qquad 
\text{o\`u } \quad \widetilde L w = -\dpar{w}{t} + L w - qw\;. 
\end{equation} 
Fixons $(s,x,z)\in\R_+\times\R^n\times\R^n$ et posons $Z_t=z+\int_0^t
q(X_s)\6s$. Le processus $Y_t=(s-t,X^{0,x}_t,Z_t)$ est une diffusion admettant 
comme g\'en\'erateur 
\begin{equation}
 \label{FK09:3}
\widehat L = -\dpar{}{s} + L + q\dpar{}{z}\;. 
\end{equation} 
Soit $\phi(s,x,z)=\e^{-z}w(s,x)$. Alors $\widehat L\phi=0$, et la formule de
Dynkin montre que si $\tau_R$ est le temps de sortie d'une boule de rayon $R$,
on a 
\begin{equation}
 \label{FK09:4}
\bigexpecin{s,x,z}{\phi(Y_{t\wedge\tau_R})} = \phi(s,x,z)\;. 
\end{equation} 
Il suit que 
\begin{align}
\nonumber
w(s,x) = \phi(s,x,0) 
&= \bigexpecin{s,x,0}{\phi(Y_{t\wedge\tau_R})} \\
\nonumber
&=
\Bigexpecin{x}{\phi
\bigpar{s-t\wedge\tau_R,X^{0,x}_{t\wedge\tau_R},Z_{t\wedge\tau_R}}} \\
&=
\Bigexpecin{x}{\e^{-\int_0^{t\wedge\tau_R} q(X_u)\6u}w(s-t\wedge\tau_R,X^{0,x}_{
t\wedge\tau_R})}\;,
 \label{FK09:5}
\end{align}
qui tend vers $\expecin{x}{\e^{-\int_0^t q(X_u)\6u}w(s-t,X^{0,x}_t)}$
lorsque $R$ tend vers l'infini. En particulier, pour $t=s$ on trouve 
\begin{equation}
 \label{FK09:6}
w(s,x) = \Bigexpecin{x}{\e^{-\int_0^s q(X_u)\6u}w(0,X^{0,x}_s)}\;,
\end{equation} 
qui est bien \'egal \`a la fonction $v(t,x)$ d\'efinie dans~\eqref{FK07}.
\qed
\end{enum}
\renewcommand{\qed}{}
\end{proof}

En combinaison avec la formule de Dynkin, la formule de Feynman--Kac admet peut
\^etre g\'en\'eralis\'ee \`a des temps d'arr\^et. Si par exemple $D\subset\R^n$
est un domaine r\'egulier, et que $\tau$ d\'esigne le temps de premi\`ere
sortie de $D$, alors sous des conditions de r\'egularit\'e sur les fonctions
$q, \varphi, \theta:\overbar D\to\R$, la quantit\'e 
\begin{equation}
 \label{FK10}
v(t,x) = \biggexpecin{x}{\e^{-\int_0^{t\wedge\tau}q(X_s)\6s}
\varphi(X_{t\wedge\tau})
- \int_0^{t\wedge\tau} \e^{-\int_0^s q(X_u)\6u}
\theta(X_{s})\6s} 
\end{equation} 
satisfait le probl\`eme avec valeurs initiales et aux bords 
\begin{align}
\nonumber
\dpar vt(t,x) &= (Lv)(t,x) - q(x)v(t,x) - \theta(x)\;, && t>0\;, \quad x\in D\;,
\\
\nonumber
v(0,x) &= \varphi(x)\;, && x\in D\;,\\
v(t,x) &= \varphi(x)\;, && x\in\partial D\;.
\label{FK11} 
\end{align} 
En particulier, si $\tau$ est fini presque s\^urement, prenant la limite
$t\to\infty$, on obtient que 
\begin{equation}
 \label{FK12}
v(x) = \biggexpecin{x}{\e^{-\int_0^{\tau}q(X_s)\6s}
\varphi(X_{\tau})
- \int_0^{\tau} \e^{-\int_0^s q(X_u)\6u}
\theta(X_{s})\6s} 
\end{equation} 
satisfait le probl\`eme 
\begin{align}
\nonumber
 (Lv)(x) &= q(x)v(x) + \theta(x)\;, && x\in D\;,
\\
v(x) &= \varphi(x)\;, && x\in\partial D\;.
\label{FK13} 
\end{align} 
On remarquera que dans le cas $q=0$, on retrouve les relations~\eqref{diffD02},
\eqref{diffD03}.

\begin{example}
\label{ex_FeynmanKac}
Soit $D=]-a,a[$ et $X_t=x+B_t$. Alors $v(x)=\expecin{x}{\e^{-\lambda\tau}}$
satisfait  
\begin{align}
\nonumber
\frac12 v''(x) &= \lambda v(x)\;, && x\in D\;,
\\
v(-a) &= v(a) = 1\;.
\label{FK14} 
\end{align} 
La solution g\'en\'erale de la premi\`ere \'equation est de la forme 
$v(x)=c_1\e^{\sqrt{2\lambda}x} + c_2\e^{-\sqrt{2\lambda}x}$. Les constantes
d'int\'egration $c_1$ et $c_2$ sont d\'etermin\'ees par les conditions aux
bords, et on trouve 
\begin{equation}
 \label{FK15}
 \bigexpecin{x}{\e^{-\lambda\tau}} = 
\frac{\cosh(\sqrt{2\lambda}\,x)}{\cosh(\sqrt{2\lambda}\,a)}\;,
\end{equation} 
qui g\'en\'eralise~\eqref{mbm6}. En \'evaluant la d\'eriv\'ee en $\lambda=0$,
on retrouve $\expecin{x}{\tau}=a^2-x^2$, qui est un cas particulier
de~\eqref{diffD07}, mais~\eqref{FK15} d\'etermine tous les autres moments de
$\tau$ ainsi que sa densit\'e. 

En r\'esolvant l'\'equation avec les conditions aux bords $v(-a)=0$ et $v(a)=1$
on obtient
\begin{equation}
 \label{FK16}
 \bigexpecin{x}{\e^{-\lambda\tau}\indexfct{\tau_a<\tau_{-a}}} = 
\frac{\sinh(\sqrt{2\lambda}\,(x+a))}{\sinh(\sqrt{2\lambda}\cdot2a)}\;,
\end{equation} 
qui nous permet de retrouver $\probin{x}{\tau_a<\tau_{-a}}=(x+a)/(2a)$, mais
aussi
\begin{align}
\nonumber
 \bigexpecin{x}{\tau\indexfct{\tau_a<\tau_{-a}}} &=
\frac{(a^2-x^2)(3a+x)}{6a}\;, 
\\
 \bigecondin{x}{\tau}{\tau_a<\tau_{-a}} &=
\frac{(a-x)(3a+x)}{3}\;.
 \label{FK17}
\end{align}  
\end{example}


\section{Exercices}
\label{sec_exo_diff}

\begin{exercice}
\label{exo_diff1}
On consid\`ere la diffusion d\'efinie par l'\'equation 
\[
\6X_t = -X_t\6t + \6B_t
\]
(processus d'Ornstein--Uhlenbeck).

\begin{enum}
\item	Donner le g\'en\'erateur $L$ associ\'e et son adjoint $L^*$.
\item	Soit $\rho(x)=\pi^{-1/2}\e^{-x^2}$. Calculer $L^*\rho(x)$. 
Que peut-on en conclure?
\end{enum}
\end{exercice}

\goodbreak

\begin{exercice}
\label{exo_diff2}
On consid\`ere la diffusion d\'efinie par l'\'equation 
\[
\6X_t = X_t \6B_t\;.
\]

\begin{enum}
\item	Donner le g\'en\'erateur $L$ associ\'e. 
\item	Trouver la solution g\'en\'erale de l'\'equation $Lu=0$. 
\item	En d\'eduire $\probin{x}{\tau_a<\tau_b}$, o\`u $\tau_a$ d\'enote le
temps de premier passage de $X_t$ en $a$.

{\it Indication:} Il s'agit de calculer $\expecin{\!x}{\psi(X_\tau)}$, o\`u
$\tau$ est le temps de premi\`ere sortie de $[a,b]$, et $\psi(a)=1$,
$\psi(b)=0$.
\end{enum}
\end{exercice}

\goodbreak

\begin{exercice}[Mouvement Brownien g\'eom\'etrique]
\label{exo_diff3}
On consid\`ere plus g\'en\'eralement la diffusion d\'efinie par l'\'equation 
\[
\6X_t = r X_t \6t + X_t \6B_t\;, 
\qquad r\in\R
\]
(mouvement Brownien g\'eom\'etrique).

\begin{enum}
\item	Calculer son g\'en\'erateur $L$. 

\item	Montrer que si $r\neq1/2$, la solution g\'en\'erale de l'\'equation
$Lu=0$ s'\'ecrit 
\[
u(x) = c_1 x^\gamma + c_2\;,
\]
o\`u $\gamma$ est une fonction de $r$ qu'on d\'eterminera. 

\item	On suppose $r<1/2$. Calculer $\probin{x}{\tau_b<\tau_a}$ pour $0<a<x<b$,
puis $\probin{x}{\tau_b<\tau_0}$ en faisant tendre $a$ vers $0$. On remarquera
que si $X_{t_0}=0$ alors $X_t=0$ pour tout $t\geqs t_0$. Par cons\'equent si
$\tau_0<\tau_b$, alors $X_t$ n'atteindra jamais $b$. Quelle est la
probabilit\'e que cela arrive?

\item	On suppose maintenant $r>1/2$. 
\begin{enum}
\item	Calculer $\probin{x}{\tau_a<\tau_b}$ pour $0<a<x<b$, et montrer que
cette probabilit\'e tend vers~$0$ pour tout $x\in]a,b[$ lorsque $a\to0_+$. En
conclure que presque s\^urement, $X_t$ n'atteindra jamais $0$ dans cette
situation. 
\item	Trouver $\alpha$ et $\beta$ tels que $u(x)=\alpha\log x+\beta$
satisfasse le probl\`eme 
\[
\begin{cases}
(Lu)(x) = -1 & \text{si $0<x<b$\;,}\\
u(x) = 0 &\text{si $x=b$\;.}
\end{cases}
\]
\item	En d\'eduire $\expecin{\!x}{\tau_b}$. 
\end{enum}
\end{enum}
\end{exercice}

\goodbreak

\begin{exercice}
\label{exo_diff6}

On appelle processus d'Ornstein--Uhlenbeck la solution de l'EDS 
\[
\6X_t = -X_t \6t + \sigma \6B_t\;, 
\qquad 
X_0 = x\;.
\]

\begin{enum}
\item	R\'esoudre cette \'equation, c'est-\`a-dire \'ecrire $X_t$ \`a l'aide 
d'une int\'egrale stochastique d'une fonction explicite.  
\item	Donner le g\'en\'erateur infinit\'esimal $L$ de $X_t$. 
\item	On se donne $a < 0 < b$. On note $\tau_a$, respectivement $\tau_b$, le
temps de premier passage de $X_t$ en $a$, respectivement $b$. 
A l'aide de la formule de Dynkin, exprimer 
\[
h(x) = \probin{x}{\tau_a < \tau_b}
\]
pour $x\in]a,b[$ comme un rapport de deux int\'egrales. 

\item 	Etudier $h(x)$ lorsque $\sigma\to0$, sachant que  
$
\int_a^b \e^{f(y)/\sigma^2}\6y \simeq
\exp\bigset{\sup_{y\in[a,b]}f(y)/\sigma^2}$.
\end{enum}

\end{exercice}

\goodbreak

\begin{exercice}\hfill
\label{exo_diff4}
\begin{enum}
\item	On consid\`ere une diffusion d'\'equation 
\[
\6X_t = f(X_t)\6t + g(X_t)\6B_t\;. 
\]
Les fonctions $f, g:\R\to\R$ sont suppos\'ees suffisamment r\'eguli\`eres pour
assurer l'existence d'une unique solution pour tout temps $t\geqs0$. 
\begin{enum}
\item	Calculer 
\[
\dtot{}{t} \bigexpecin{x}{X_t} \Bigr|_{t=0_+} := 
\lim_{h\to0_+} \frac{\bigexpecin{x}{X_h} - x}h \;.
\]
\item	Calculer 
\[
\dtot{}{t} 
\Bigexpecin{x}{\e^{\gamma \brak{X_t - \expecin{x}{X_t}}}} \Bigr|_{t=0_+} = 
\dtot{}{t} \biggbrak{
\e^{-\gamma \expecin{x}{X_t}}
\Bigexpecin{x}{ \e^{\gamma X_t}}}
\biggr|_{t=0_+}\;.
\]
\item	En d\'eduire 
\[
\dtot{}{t} 
\Bigexpecin{x}{\bigbrak{X_t - \expecin{x}{X_t}}^k} \Bigr|_{t=0_+}
\]
pour $k=2,3,\dots$. 
\end{enum}

\item	
On se donne une suite d'ensembles d\'enombrables $\cX^{(N)}$, $N\in\N^*$. Sur
chaque $\cX^{(N)}$ on d\'efinit une cha\^ine de Markov
$\set{Y^{(N)}_n}_{n\geqs0}$, de matrice de transition $P^{(N)}$. 
On pose 
\[
v^{(N)}(y) = \bigexpecin{y}{Y^{(N)}_1-y} := \sum_{z\in\cX^{(N)}} (z-y)
P^{(N)}(y,z)
\]
et, pour $k=2,3,\dots$, 
\[
m_k^{(N)}(y) = \Bigexpecin{y}{\brak{Y^{(N)}_1-\expecin{y}{Y^{(N)}_1}}^k}\;.
\]
On d\'efinit une suite de processus $\set{X^{(N)}_t}_{t\geqs0}$, \`a
trajectoires continues, lin\'eaires par morceaux sur tout intervalle
$]k/N,(k+1)/N[$, telles que 
\[
 X^{(N)}_{n/N} = N^{-\alpha} Y^{(N)}_n\;, 
\qquad n\in\N\, 
\]
pour un $\alpha>0$. 

\begin{enum}
\item	Exprimer, en fonction de $v^{(N)}$ et $m_k^{(N)}$, 
\[
\lim_{h\to0_+} \frac{\bigexpecin{x}{ X^{(N)}_h-x}}h
\qquad
\text{et}
\qquad 
\lim_{h\to0_+} \frac{\bigexpecin{x}{ \brak{X^{(N)}_h - 
\expecin{x}{X^{(N)}_h}}^k}}h\;.
\]

\item	Donner des conditions n\'ecessaires sur les $v^{(N)}$ et $m_k^{(N)}$
pour que la suite des $X^{(N)}_t$ converge vers la diffusion $X_t$. 
\end{enum}

\item	Montrer que ces conditions sont v\'erifi\'ees, pour un $\alpha$
appropri\'e, dans le cas o\`u chaque $Y^{(N)}$ est la marche al\'eatoire simple
sur $\Z$ et $X_t=B_t$ est le mouvement Brownien.

\item	On rappelle que le mod\`ele d'Ehrenfest \`a $N$ boules est la cha\^ine
de Markov $Y^{(N)}_n$ sur $\set{0,1,\dots,N}$ de probabilit\'es de transition 
\[
P^{(N)}(y, y-1) = \frac{y}{N}\;, \qquad
P^{(N)}(y, y+1) = 1 -  \frac{y}{N}\;.
\]
En supposant que la suite de processus d\'efinis par 
\[
X^{(N)}_{n/N}=N^{-1/2}\biggpar{Y^{(N)}_n-\frac N2}\;, 
\qquad n\in\N\;,
\]
converge vers une diffusion $X_t$, 
d\'eterminer les coefficients $f(x)$ et $g(x)$ de cette diffusion.

\end{enum}
\end{exercice}

\goodbreak

\begin{exercice}[La loi de l'arcsinus]
\label{exo_diff5}

Soit $\set{B_t}_{t\geqs0}$ un mouvement Brownien standard dans $\R$. On
consid\`ere le processus 
\[
 X_t = \frac{1}{t} \int_0^t \indexfct{B_s>0}\6s\;, 
\qquad
t>0\;.
\]
Le but de ce probl\`eme est de d\'emontrer la \emph{loi de l'arcsinus} :
\begin{equation}
\label{arcsinus} 
\bigprob{X_t < u} = \frac{2}{\pi} \Arcsin\bigpar{\sqrt{u}}\;, 
\qquad
0\leqs u\leqs 1\;.
\end{equation}

\begin{enum}
\item	Que repr\'esente la variable $X_t$?
\item	Montrer que $X_t$ est \'egal en loi \`a $X_1$ pour tout $t>0$. 
\item	On fixe $\lambda>0$. 
Pour $t>0$ et $x\in\R$, on d\'efinit la fonction 
\[
v(t,x) = \Bigexpec{\e^{-\lambda\int_0^t \indexfct{x+ B_s>0}\6s}}
\]
et sa transform\'ee de Laplace 
\[
g_\rho(x) = \int_0^\infty v(t,x) \e^{-\rho t} \6t\;,
\qquad
\rho>0\;.
\]
Montrer que 
\[
g_\rho(0) = \biggexpec{\frac{1}{\rho+\lambda X_1}}\;.
\]

\item	Calculer $\dpar vt(t,x)$ \`a l'aide de la formule de Feynman--Kac. 

\item	Calculer $g_\rho''(x)$. En d\'eduire que $g_\rho(x)$ satisfait une
\'equation diff\'erentielle ordinaire lin\'eaire du second ordre \`a
coefficients constants par morceaux. Montrer que sa solution g\'en\'erale
s'\'ecrit 
\[
 g_\rho(x) = A_\pm + B_\pm \e^{\gamma_\pm x} + C_\pm \e^{-\gamma_\pm x}
\]
avec des constantes $A_\pm, B_\pm, C_\pm, \gamma_\pm$ d\'ependant du
signe de $x$. 

\item	D\'eterminer les constantes en utilisant le fait que $g_\rho$ doit
\^etre born\'ee, continue en $0$, et que $g_\rho'$ doit \^etre continue en $0$. 
En conclure que $g_\rho(0)=1/\sqrt{\rho(\lambda+\rho)}$. 

\item	D\'emontrer~\eqref{arcsinus} en utilisant l'identit\'e 
\[
\frac{1}{\sqrt{1+\lambda}} = 
\sum_{n=0}^\infty (-\lambda)^n 
\frac1\pi \int_0^1 \frac{x^n}{\sqrt{x(1-x)}}\6x\;.
\]
\end{enum}
\end{exercice}

\goodbreak



\appendix

\chapter{Corrig\'es des exercices}


\section{Exercices du Chapitre~\ref{chap_ec}}

\subsection*{Exercice~\ref{exo_ec1}}

\begin{enum}
\item	
Il suffit de prendre $\Omega=\set{1,2,3,4}^2$ avec la probabilit\'e uniforme.

\item	
Les lois de $X(\omega)=\omega_1+\omega_2$ et
$Y(\omega)=\abs{\omega_1-\omega_2}$ sont donn\'ees dans les tableaux suivants~:

\medskip\noindent
\begin{tabular}{c|cccc}
$X$ & 1 & 2 & 3 & 4 \\
\hline 
1 & 2 & 3 & 4 & 5 \\
2 & 3 & 4 & 5 & 6\\
3 & 4 & 5 & 6 & 7\\
4 & 5 & 6 & 7 & 8\\
\end{tabular}
\hspace{10mm}
\begin{tabular}{c|cccc}
$Y$ & 1 & 2 & 3 & 4 \\
\hline 
1 & 0 & 1 & 2 & 3 \\
2 & 1 & 0 & 1 & 2\\
3 & 2 & 1 & 0 & 1\\
4 & 3 & 2 & 1 & 0\\
\end{tabular}
\medskip\noindent

Par simple d\'enombrement, on obtient 
leur loi conjointe et les marginales~:

\medskip\noindent
\begin{tabular}{l|ccccccc|c}
$Y  \backslash  X$ & 2 & 3 & 4 & 5 & 6 & 7 & 8 & \\
\hline
$0$ & $1/16$ & $0$ & $1/16$ & $0$ & $1/16$ & $0$ & $1/16$ & $4/16$ \\
$1$ & $0$ & $2/16$ & $0$ & $2/16$ & $0$ & $2/16$ & $0$ &    $6/16$ \\
$2$ & $0$ & $0$ & $2/16$ & $0$ & $2/16$ & $0$ & $0$ & $4/16$ \\
$3$ & $0$ & $0$ & $0$ & $2/16$ & $0$ & $0$ & $0$ & $2/16$ \\
\hline
    & $1/16$ & $2/16$ & $3/16$ & $4/16$ & $3/16$ & $2/16$ & $1/16$ & 
\end{tabular}
\medskip\noindent

On en d\'eduit les esp\'erances
\[
\expec{X} = \sum_{x=2}^8 x \prob{X=x} = 5\;, 
\qquad
\expec{Y} = \sum_{y=0}^3 y \prob{Y=y} = \frac54\;.
\]

\item	
Notons $\econd{X}{Y=y}$ la valeur constante de $\econd{Y}{X}$ sur l'ensemble
$\set{Y=y}$. Il suit de la relation (3.2.11) du cours que 
\[
\econd{X}{Y=y} = \frac{\expec{X\indexfct{Y=y}}}{\prob{Y=y}}
= \sum_{x} x\frac{\prob{X=x,Y=y}}{\prob{Y=y}}
= \sum_{x} x\pcond{X=x}{Y=y}\;.
\]
En appliquant \`a notre cas, on obtient 
\[
\econd{X}{Y=y} = 5 \quad \forall y\in\set{0,1,2,3}\;,
\]
ce qui traduit le fait que la distribution de $X$ est sym\'etrique autour de
$5$ pour tout $Y$. 
De mani\`ere similaire, on trouve 
\begin{align*}
\econd{Y}{X=2} &= \econd{Y}{X=8} = 0\;, \\
\econd{Y}{X=3} &= \econd{Y}{X=7} = 1\;, \\
\econd{Y}{X=4} &= \econd{Y}{X=6} = \tfrac43\;, \\
\econd{Y}{X=5} &=  2\;. 
\end{align*}
Cela permet en particulier de v\'erifier que 
$\expec{\econd{Y}{X}}=\expec{Y}$. 

\end{enum}

\subsection*{Exercice~\ref{exo_ec2}}

Un choix possible d'espace probabilis\'e est
$\Omega=\set{1,\dots,6}\times\set{0,1}^6$, avec la probabilit\'e uniforme et 
\[
N(\omega) = \omega_1\;, 
\qquad
X(\omega) = \sum_{i=1}^{\omega_1} \omega_{i+1}\;.
\]
Avant de calculer $\expec{X}$, commen\c cons par calculer $\econd{X}{N}$.
Conditionnellement \`a $N=n$, $X$ suit une loi binomiale de param\`etres $n$ et
$1/2$, d'o\`u 
\[
\econd{X}{N} = \frac{N}{2}\;.
\]
Il suffit alors de prendre l'esp\'erance pour conclure~:
\[
\expec{X} = \expec{\econd{X}{N}} = \frac12 \expec{N} = \frac12 \cdot \frac72 =
\frac{7}{4}\;.
\]

\subsection*{Exercice~\ref{exo_ec3}}

La monotonie de l'esp\'erance conditionnelle nous permet d'\'ecrire 
\[
\econd{X^2}{\cF_1} \geqs \bigecond{X^2\indexfct{X^2\geqs a^2}}{\cF_1}
\geqs \bigecond{a^2\indexfct{X^2\geqs a^2}}{\cF_1}
= a^2 \bigpecond{\set{\abs{X}\geqs a}}{\cF_1}\;.
\]

\subsection*{Exercice~\ref{exo_ec4}}

Soit $X=\indicator{A}$ et $Y=\indicator{B}$. 
\begin{itemiz}
\item	Si $X$ et $Y$ sont ind\'ependantes, alors $\pprob{A\cap
B}=\expec{XY}=\expec{X}\expec{Y}=\pprob{A}\pprob{B}$. Or 
\[
\int_A Y\6\fP = \int_{A\cap B} \6\fP = \pprob{A\cap
B}=\pprob{A}\pprob{B} = \pprob{A} \expec{Y} = \int_A \expec{Y} \6\fP\;.
\]
On v\'erifie facilement des relations analogues avec $A$ rempla\c c\'e par
$A^c$, par $\emptyset$ et
par $\Omega$. Comme $\sigma(X)=\set{\emptyset,A,A^c,\Omega}$ et
$\expec{Y}\measurable\cF_0\subset\sigma(X)$, on a bien
que $\econd{Y}{X} = \expec{Y}$. 

\item	Supposons que $\econd{Y}{X} = \expec{Y}$. Comme $A\in\sigma(X)$ on a 
\[
\expec{XY} = \int_A Y\6\fP = \int_A \econd{Y}{X}\6\fP
= \int_A \expec{Y}\6\fP = \expec{Y} \pprob{A} = \expec{Y} \expec{X}\;.
\]
\end{itemiz}
Le r\'esultat s'\'etend \`a des variables al\'eatoires quelconques de la
mani\`ere habituelle, en les d\'ecomposant en partie positive et n\'egative et
en approchant chaque partie par des fonctions \'etag\'ees.

\subsection*{Exercice~\ref{exo_ec5}}

Notons la loi conjointe
\[
p_{xy} = \prob{X=x,Y=y}\;, 
\]
et ses marginales 
\[
p_{x\bullet} = \prob{X=x} = \sum_y p_{xy}\;, 
\qquad
p_{\bullet y} = \prob{Y=y} = \sum_x p_{xy}\;.
\]
Les variables $X$ et $Y$ sont ind\'ependantes si et seulement si 
\[
p_{xy} = p_{x\bullet} p_{\bullet y} \quad \forall x, y\;.
\]
Les esp\'erances conditionnelles sont donn\'ees par 
\[
\econd{Y}{X=x} = \frac{1}{p_{x\bullet}} \sum_y y p_{xy}\;,
\]
donc on aura $\econd{Y}{X}=\expec{Y}$ si et seulement si
\[
\sum_y y p_{xy} = \sum_y y p_{x\bullet} p_{\bullet y} \quad \forall x\;.
\]
Enfin la condition $\expec{XY}=\expec{X}\expec{Y}$ s'\'ecrit 
\[
\sum_{xy} xy p_{xy} = \sum_{xy} xy p_{x\bullet}p_{\bullet y}\;.
\]
Si $X$ et $Y$ prennent leurs valeurs dans $\set{-1,0,1}$ et $\expec{Y}=0$, 
on aura n\'ecessairement $p_{\bullet-}=p_{\bullet+}$. Si de plus on veut avoir
$\econd{Y}{X}=\expec{Y}$, alors il faut que $p_{x-}=p_{x+}$ pour tout $x$.

Il est alors facile de construire des contre-exemples aux implications
inverses. Le tableau suivant donne un exemple de loi conjointe pour laquelle
$\econd{Y}{X}=\expec{Y}=0$, mais $X$ et $Y$ ne sont pas ind\'ependantes~:

\bigskip
\begin{tabular}{c|ccc|c}
$Y\backslash X$ & $-1$ & $0$ & $1$ & \\
\hline 
$-1$ & $1/10$ & $0$ & $2/10$ & $3/10$ \\
$0$ & $0$ & $4/10$ & $0$ & $4/10$ \\
$1$ & $1/10$ &$ 0$ & $2/10$ & $3/10$ \\
\hline
& $2/10$ & $4/10$ & $4/10$ & 
\end{tabular}

\bigskip

\noindent
Et voici un exemple o\`u 
$\expec{XY}=\expec{X}\expec{Y}=0$, mais $\econd{Y}{X}\neq\expec{Y}=0$~:

\bigskip
\begin{tabular}{c|ccc|c}
$Y\backslash X$ & $-1$ & $0$ & $1$ & \\
\hline 
$-1$ & $1/10$ & $2/10$ & $1/10$ & $4/10$ \\
$0$ & $0$ & $2/10$ & $0$ & $2/10$ \\
$1$ & $2/10$ &$0$ & $2/10$ & $4/10$ \\
\hline
& $3/10$ & $4/10$ & $3/10$ & 
\end{tabular}

\subsection*{Exercice~\ref{exo_ec6}}

\begin{enum}
\item 
Notons $X_2=\econd{X}{\cF_2}$ et $X_1=\econd{X}{\cF_1}=\econd{X_2}{\cF_1}$. 
On a 
\[
\expec{XX_1} = \bigexpec{\econd{XX_1}{\cF_1}}
= \bigexpec{X_1\econd{X}{\cF_1}} = \expec{X_1^2}\;.
\]
Par cons\'equent, en d\'eveloppant le carr\'e on obtient 
\[
\bigexpec{\brak{X-X_1}^2} = \expec{X^2} - \expec{X_1^2}\;.
\]
De mani\`ere similaire, on montre que $\expec{XX_2}=\expec{X_2^2}$ et
$\expec{X_1X_2}=\expec{X_1^2}$, d'o\`u
\begin{align*}
\bigexpec{\brak{X-X_2}^2} &= \expec{X^2} - \expec{X_2^2}\;,\\
\bigexpec{\brak{X_2-X_1}^2} &= \expec{X_2^2} - \expec{X_1^2}\;.
\end{align*}
Ceci implique le r\'esultat (qui est \'equivalent au th\'eor\`eme de Pythagore
appliqu\'e \`a $X$, $X_1$ et $X_2$, consid\'er\'es comme des vecteurs de
$L^2(\cF)$). 

\item	
On peut proc\'eder par un calcul direct. Une autre m\'ethode est de commencer
par observer que 
\[
\varcond{X}{\cF_1} = \bigecond{\brak{X-\econd{X}{\cF_1}}^2}{\cF_1}\;.
\]
Appliquons alors l'\'egalit\'e montr\'ee en 1.\ avec $\cF_1$ rempla\c c\'e par
$\cF_0$, et $\cF_2$ rempla\c c\'e par $\cF_1$. Comme
$\econd{X}{\cF_0}=\expec{X}$, le premier terme du membre de
gauche est \'egal \`a $\expec{\varcond{X}{\cF_1}}$, le second \`a 
$\Variance(\econd{X}{\cF_1})$, alors que le membre de droite vaut
$\Variance(X)$.

\item	
Appliquons le r\'esultat pr\'ec\'edent avec $\cF_1=\sigma(N)$. 
On a $\econd{X}{N}=\mu N$ et en d\'evelop\-pant la somme on trouve 
$\econd{X^2}{N} = \sigma^2 N + \mu^2 N^2$. Il suit que 
\[
\varcond{X}{N} = \sigma^2 N\;.
\]
Comme d'autre part on a $\Variance(\econd{X}{N})=\Variance(\mu N)=\mu^2
\Variance(N)$, le r\'esultat est montr\'e.

\item	
C'est une application du r\'esultat pr\'ec\'edent, avec $\mu=1/2$,
$\sigma^2=1/4$ et $\Variance(N)=35/12$. On trouve donc $\Variance(X)=77/48$.

\end{enum}

\subsection*{Exercice~\ref{exo_ec7}}

\begin{enum}
\item	
On a $\econd{Y-X}{\cG} = X - \econd{X}{\cG} = 0$, et 
\[
\bigecond{(Y-X)^2}{\cG} = \econd{Y^2}{\cG} - \econd{X^2}{\cG}
= \econd{Y^2}{\cG} - X^2\;.
\]
Par cons\'equent, 
\[
\varcond{Y-X}{\cG} = \econd{Y^2}{\cG} - X^2\;.
\]

\item	
\begin{align*}
\Variance(Y-X) &= \bigexpec{\varcond{Y-X}{\cG}} +
\Variance\bigpar{\econd{Y-X}{\cG}} \\
&= \bigexpec{\econd{Y^2}{\cG}} - \expec{X^2} \\
&= \expec{Y^2} - \expec{X^2} \\
&= 0\;.
\end{align*}

\item	
Elles sont \'egales presque s\^urement.

\end{enum}

\subsection*{Exercice~\ref{exo_ec8}}

\begin{enum}
\item	
Les $\xi_k$ \'etant ind\'ependants, on a 
\[
 \econd{X_t}{N_t} = N_t \,\expec{\xi_0}
\]
et donc 
\[
 \expec{X_t} = \bigexpec{\econd{X_t}{N_t}} 
= \expec{N_t} \expec{\xi_0} = \lambda t \,\expec{\xi_0}\;.
\]

\item	
La variance \'etant une forme quadratique, on a 
\[
 \Variance\bigpar{\econd{X_t}{N_t}} = \Variance(N_t) \, \expec{\xi_0}^2 
= \lambda t \,\expec{\xi_0}^2\;.
\]
Afin de d\'eterminer $\varcond{X_t}{N_t}$, on commence par calculer 
\[
 \econd{X_t^2}{N_t} = \sum_{k,l=1}^{N_t} \expec{\xi_k\xi_l} 
= N_t\,\expec{\xi_0^2} + (N_t^2-N_t) \,\expec{\xi_0}^2 
= N_t \Variance(\xi_0) + N_t^2 \,\expec{\xi_0}^2\;.
\]
Par cons\'equent,
\[
 \varcond{X_t}{N_t} = \econd{X_t^2}{N_t} - \econd{X_t}{N_t}^2 
= N_t\Variance(\xi_0)\;,
\]
d'o\`u il suit 
\[
 \bigexpec{\varcond{X_t}{N_t}} = \expec{N_t}\Variance(\xi_0)
= \lambda t \Variance(\xi_0)\;.
\]
En appliquant l'Exercice 3.6, on conclut que 
\[
 \Variance(X_t) =  \Variance\bigpar{\econd{X_t}{N_t}} + 
\bigexpec{\varcond{X_t}{N_t}}
= \lambda t \,\expec{\xi_0^2}\;.
\]

\end{enum}


\section{Exercices du Chapitre~\ref{chap_mart}}

\subsection*{Exercice~\ref{exo_mart1}}

On choisit la filtration canonique. 
Comme $\expec{\abs{X_n}}=\expec{\abs{Y_1}}^n$, une premi\`ere condition est
$\expec{\abs{Y_1}}<\infty$, c'est-\`a-dire $Y_1\in L^1$.

\noindent
D'autre part, $\econd{X_{n+1}}{\cF_n} = \econd{Y_{n+1}X_n}{\cF_n} =
X_n\econd{Y_{n+1}}{\cF_n} = X_n\expec{Y_1}$, o\`u nous avons utilis\'e
$X_n\measurable\cF_n$, $Y_{n+1}\perp\cF_n$ et $\expec{Y_{n+1}}=\expec{Y_1}$. La
suite $X_n$ est donc une surmartingale, une sous-martingale ou une martingale
selon que $\expec{Y_1}\leqs 1$, $\expec{Y_1}\geqs 1$ ou $\expec{Y_1}=1$.

\subsection*{Exercice~\ref{exo_mart3}}

\begin{enum}
\item	
Le processus est clairement adapt\'e et int\'egrable. 

De plus, 
$\econd{X_{n+1}}{\cF_n}=\econd{\indicator{B_{n+1}}+X_n}{\cF_n}=
\econd{\indicator{B_{n+1}}}{\cF_n}+X_n\geqs X_n$. 

\item	
La d\'ecomposition de Doob donne $X_n=M_n+A_n$ avec 
\begin{align*}
A_n &= \sum_{m=1}^n \econd{\indicator{B_m}}{\cF_{m-1}}\;, \\
M_n &= \sum_{m=1}^n \indicator{B_m} - \econd{\indicator{B_m}}{\cF_{m-1}}\;.
\end{align*}

\item	
Comme
$\cF_1=\sigma(\indicator{B_1})=\set{\emptyset,\Omega,B_1,B_1^c}$, 
l'\'equation~\eqref{ec5} donne 
\[
\econd{\indicator{B_2}}{\cF_1}(\omega) = 
\begin{cases}
\vrule height 6pt depth 16pt width 0pt
\dfrac{\expec{\indicator{B_2}\indicator{B_1}}}{\fP(B_1)}
= \ppcond{B_2}{B_1} & \text{si $\omega\in B_1$\;,} \\
\vrule height 16pt depth 6pt width 0pt
\dfrac{\expec{\indicator{B_2}\indicator{B_1^c}}}{\fP(B_1^c)}
= \ppcond{B_2}{B_1^c} & \text{si $\omega\in B_1^c$\;.}
\end{cases}
\]
De mani\`ere g\'en\'erale, 
$\econd{\indicator{B_m}}{\cF_{m-1}}=\ppcond{B_m}{C^{m-1}_i}$ pour tout
$\omega\in C^{m-1}_i$, o\`u les $C^{m-1}_i$ sont les \'el\'ements de la
partition
engendrant $\cF_{m-1}$ (c'est-\`a-dire
$\cF_{m-1}=\sigma(\bigsqcup_iC^{m-1}_i)$), obtenus par intersection des $B_i$ et
$B_i^c$ pour $i\leqs m-1$.

\end{enum}

\subsection*{Exercice~\ref{exo_mart4}}

\begin{enum}
\item	
Apr\`es $n$ tirages, l'urne contient $n+2$ boules, dont $m+1$ vertes,
o\`u $m\in\set{0,1,\dots,n}$. La probabilit\'e de tirer d'abord les $m$ boules
vertes, puis les $l=n-m$ rouges est 
\[
\frac12 \cdot \frac23 \dots \frac{m}{m+1} \cdot \frac{1}{m+2} \dots
\frac{l}{n+1} = \frac{m!l!}{(n+1)!}\;.
\]
Pour tout autre ordre des boules, les termes sont arrang\'es diff\'eremment,
mais le r\'esultat est le m\^eme. On a donc 
\[
\Bigprob{X_n=\frac{m+1}{n+2}} = \binom{n}{m}\frac{m!l!}{(n+1)!} 
= \frac1{n+1}\;,
\qquad m=1,2,\dots,n\;,
\]
c'est-\`a-dire que la distribution de $X_n$ est uniforme sur
$\bigset{\frac1{n+2},\frac2{n+2},\dots,\frac{n+1}{n+2}}$.

\item	
Apr\`es $n$ tirages, le nombre de boules est $N_n=r_n+v_n=r+v+nc$. On a 
\[
X_n = 
\begin{cases}
\vrule height 6pt depth 16pt width 0pt
\dfrac{v_{n-1}+c}{v_{n-1}+r_{n-1}+c}
& \text{avec probabilit\'e $\dfrac{v_{n-1}}{v_{n-1}+r_{n-1}}=X_{n-1}$\;,} \\
\vrule height 16pt depth 6pt width 0pt
\dfrac{v_{n-1}}{v_{n-1}+r_{n-1}+c}
& \text{avec probabilit\'e
$\dfrac{r_{n-1}}{v_{n-1}+r_{n-1}}=1-X_{n-1}$\;.} \\
\end{cases}
\]
On en d\'eduit que 
\[
X_n - X_{n-1} = 
\begin{cases}
\vrule height 6pt depth 16pt width 0pt
\dfrac{c(1-X_{n-1})}{N_n}
& \text{avec probabilit\'e $X_{n-1}$\;,} \\
\vrule height 16pt depth 6pt width 0pt
\dfrac{-cX_{n-1}}{N_n}
& \text{avec probabilit\'e $1-X_{n-1}$\;,} \\
\end{cases}
\]
et donc 
\begin{align*}
\bigecond{(X_n-X_{n-1})^2}{\cF_{n-1}} &= 
\frac{c^2}{N_n^2} \bigbrak{X_{n-1}(1-X_{n-1})^2 + X_{n-1}^2(1-X_{n-1})} \\
&= \frac{c^2}{N_n^2} X_{n-1}(1-X_{n-1})\;. 
\end{align*}
Le processus croissant vaut donc 
\[
\braket{X}_n = c^2\sum_{m=1}^n \frac{X_{m-1}(1-X_{m-1})}{N_m^2}\;. 
\]

\item	
Comme\/ $N_m\geqs cm$ et\/ $X_{m-1}(1-X_{m-1})$ est born\'e, le crit\`ere
de Riemann montre que\/ $\braket{X}_n$ converge. 
\end{enum}

\subsection*{Exercice~\ref{exo_mart6}}

\begin{enum}
\item 	On a 
\[
\bigecond{Y_{n+1}}{\cF_n} = 
\bigecond{f(X_n+U_{n+1})}{\cF_n} =
\frac{1}{2\pi} \int_0^{2\pi} f(X_n + r\e^{\icx\theta}) \6\theta\;.
\]
Par cons\'equent, $Y_n$ est une sous-martingale si $f$ est sous-harmonique, une
surmartingale si $f$ est surharmonique, et une martingale si $f$ est
harmonique. 
\item 	Si $f$ est la partie r\'eelle d'une fonction
analytique, le th\'eor\`eme de Cauchy s'\'ecrit 
\[
f(z) = \frac{1}{2\pi\icx} \int_{\cC} \frac{f(w)}{w-z} \6w
\]
o\`u $\cC$ est un contour entourant $z$. En prenant un contour de la forme 
$w=z+r\e^{\icx\theta}$, avec $0\leqs\theta<2\pi$, on obtient que $f$ est
harmonique, donc que $Y_n$ est une martingale.  
\end{enum}

\subsection*{Exercice~\ref{exo_mart5}}

\begin{enum}
\item	
On a 
$\econd{Z_{n+1}}{\cF_n} = \econd{\xi_{1,n+1}}{\cF_n} + \dots +
\econd{\xi_{Z_n,n+1}}{\cF_n} = Z_n\mu$. 

\item	
$X_n$ \'etant une martingale, on a $\expec{X_n}=\expec{X_0}=\expec{Z_0}=1$,
donc $\expec{Z_n}=\mu^n\to 0$ lorsque $n\to\infty$. 

Par cons\'equent, $\prob{Z_n>0} = \sum_{k\geqs1}\prob{Z_n=k} \leqs
\expec{Z_n}\to 0$, d'o\`u $\prob{Z_n=0}\to 1$. $Z_n$ ayant valeurs enti\`eres,
cela signifie que pour toute r\'ealisation $\omega$, il existe $n_0(\omega)$ tel
que $Z_n(\omega)=0$ pour tout $n\geqs n_0(\omega)$, et donc aussi
$X_n(\omega)=0$ pour ces $n$.

\item	
\begin{enum}
\item	
On a $\varphi'(s)=\sum_{k=1}^\infty kp_k s^{k-1} \geqs 0$. En fait, $\varphi$
est m\^eme strictement croissante pour $s>0$. En effet, on a n\'ecessairement
$p_0<1$, car sinon on aurait $\mu=0$, donc il existe au moins un $k\geqs1$ tel
que $p_k>0$.  

De m\^eme, $\varphi''(s)=\sum_{k=2}^\infty k(k-1)p_k s^{k-2} \geqs 0$. En fait,
$\varphi$ est strictement convexe pour $s>0$. En effet, si tous les $p_k$
pour $k\geqs2$ \'etaient nuls, on aurait $\mu=p_1<1$. Il existe donc
n\'ecessairement un $k\geqs2$ tel que $p_k>0$, d'o\`u $\varphi''(s)>0$ pour
$s>0$.

\item	
Si $Z_1=k$, on a au temps $1$ $k$ individus dont la descendance \'evolue de
mani\`ere ind\'ependante. Chacune des $k$ descendances se sera \'eteinte au
temps $m$ avec probabilit\'e $\theta_{m-1}$. Par ind\'ependance, toutes les $k$
lign\'ees se seront \'eteintes \`a la fois au temps $m$ avec probabilit\'e
$\theta_{m-1}^k$. 

Il suit $\theta_m = \sum_{k=0}^\infty \pcond{Z_m=0}{Z_1=k} \prob{Z_1=k} 
= \sum_{k=0}^\infty \theta_{m-1}^k p_k = \varphi(\theta_{m-1})$.

\item	
Notons d'abord que $\varphi(1)=1$, $\varphi'(1)=\expec{\xi}=\mu$,
$\varphi(0)=p_0\geqs0$ et $\varphi'(0)=p_1<1$. La fonction
$\psi(s)=\varphi(s)-s$ satisfait donc $\psi(1)=0$, $\psi'(1)=\mu-1>0$,
$\psi(0)\geqs0$ et $\psi'(0)<0$. Etant strictement convexe sur $(0,1]$, $\psi$
admet un unique minimum en $s_0=(0,1)$. Elle s'annule donc exactement deux
fois: une fois en un $\rho\in[0,s_0)$, et une fois en $1$. 

\item	
On a $\theta_0=0$. Si $p_0=0$, $\theta_m=0$ pour tout $m$, mais dans ce cas on
a \'egalement $\rho=0$. Si $p_0>0$, on a $\theta_1=p_0>0$, et
$\rho=\varphi(\rho)>p_0$, donc $0<\theta_1<\rho$. Par r\'ecurrence, on voit que
la suite des $\theta_m$ est croissante et major\'ee par $\rho$. Elle admet donc
une limite, qui est n\'ecessairement $\rho$.
\end{enum}
Le fait que $Z_n=0$ implique $Z_m=0$ pour tout $m>n$ permet d'\'ecrire 

$\prob{\exists n \colon Z_n=0} = \fP \bigpar{\bigcup_n\set{Z_n=0}}
= \lim_{n\to\infty}\prob{Z_n=0} = \rho < 1$.

\item	
On a $\varphi(s)=\frac18 (1+3s+3s^2+s^3)$. En utilisant le fait que
$\varphi(s)-s$ s'anulle en $s=1$, on obtient par division euclidienne 
$8(\varphi(s)-s)=(s-1)(s^2+4s-1)$. La seule racine dans $[0,1)$ est
$\rho=\sqrt{5}-2$. La probabilit\'e de survie vaut donc $3-\sqrt{5}\simeq0.764$.

\item	
Il est commode d'introduire les variables centr\'ees $\eta_{i,n}=\xi_{i,n}-\mu$.

On a $Z_n-\mu Z_{n-1} = \eta_{1,n} + \dots + \eta_{Z_{n-1},n}$, ce qui
implique, par ind\'ependance des $\eta_{i,n}$, 
$\econd{(Z_n-\mu Z_{n-1})^2}{\cF_{n-1}} = \econd{(\eta_{1,n} + \dots +
\eta_{Z_{n-1},n})^2}{\cF_{n-1}} = Z_{n-1}\sigma^2$, et donc 

$\econd{(X_n-X_{n-1})^2}{\cF_{n-1}} = Z_{n-1}\sigma^2/\mu^{2n}
= X_{n-1}\sigma^2/\mu^{n+1}$. Il suit 
\[
\braket{X}_\infty = \sigma^2 \sum_{m=1}^\infty \frac{X_{m-1}}{\mu^{m+1}}\;.
\]
De plus, 
\[
\bigexpec{\braket{X}_\infty} = \sigma^2 \sum_{m=1}^\infty
\frac{\expec{X_{m-1}}}{\mu^{m+1}} = \sigma^2 \sum_{m=1}^\infty
\frac{1}{\mu^{m+1}} = \frac{\sigma^2}{\mu(\mu-1)} < \infty\;.
\]
\end{enum}



\section{Exercices du Chapitre~\ref{chap_arret}}

\subsection*{Exercice~\ref{exo_doob1}}

\begin{enum}
\item	
La d\'ecomposition 
\begin{align*}
\set{N\wedge M=n} 
&= \Bigpar{\set{N=n}\cap\set{M\geqs n}} \bigcup 
\Bigpar{\set{N\geqs n}\cap\set{M=n}} \\
&= \Bigpar{\set{N=n}\cap\set{M<n}^c} \bigcup 
\Bigpar{\set{N<n}^c\cap\set{M=n}}
\end{align*}
montre que $\set{N\wedge M=n} \in \cF_n$ pour tout $n$, et donc que $N\wedge M$
est un temps d'arr\^et. De m\^eme, la d\'ecomposition
\[
\set{N\vee M=n}
= \Bigpar{\set{N=n}\cap\set{M\leqs n}} \bigcup 
\Bigpar{\set{N\leqs n}\cap\set{M=n}}
\]
montre que $\set{N\vee M=n} \in \cF_n$ pour tout $n$, et donc que $N\vee M$
est un temps d'arr\^et. 

{\bf Remarque~:} On peut \'egalement observer que $N$ est un temps d'arr\^et si
et seulement si $\set{N\leqs n}\in\cF_n$ pour tout $n$. Cela permet d'utiliser
la d\'ecomposition 
\[
\set{N\wedge M\leqs n} = \set{N\leqs n}\bigcap\set{M\leqs n}
\]
pour montrer que $N\wedge M$ est un temps d'arr\^et. 

\item	
Il suffit d'observer que le fait que $N_k\nearrow N$ implique 
\[
\set{N\leqs n} = \bigcap_{k\in\N} \set{N_k\leqs n} \in \cF_n\;. 
\]
\end{enum}

\subsection*{Exercice~\ref{exo_doob2}}

\begin{enum}
\item	
Notons $\sigma^2$ la variance des $\xi_m$.
$X_n^2$ \'etant une sous-martingale, on a 
\[
P_n(\lambda) = \Bigprob{\max_{1\leqs m\leqs n}X_m^2\geqs\lambda^2}
\leqs \frac1{\lambda^2} \expec{X_n^2} = \frac{n\sigma^2}{\lambda^2}\;.
\]

\item	
Pour tout $c>0$, $(X_n+c)^2$ est une sous-martingale et on a
($x\mapsto(x+c)^2$ \'etant croissante sur $\R_+$)
\[
P_n(\lambda) = \Bigprob{\max_{1\leqs m\leqs n}(X_m+c)^2\geqs(\lambda+c)^2}
\leqs \frac1{(\lambda+c)^2} \bigexpec{(X_n+c)^2} =
\frac{n\sigma^2+c^2}{(\lambda+c)^2}\;,
\]
puisque $\expec{X_n}=0$. Cette borne est minimale pour $c=n\sigma^2/\lambda$,
et donne 
\[
P_n(\lambda) \leqs 
\frac{n\sigma^2}{\lambda^2+n\sigma^2}\;.
\]
Contrairement \`a la premi\`ere borne, celle-ci est toujours inf\'erieure \`a
$1$.

\item	
Pour tout $c>0$, $\e^{cX_n^2}$ est une sous-martingale et on a 
\[
P_n(\lambda) = \Bigprob{\max_{1\leqs m\leqs n}\e^{cX_m^2}\geqs\e^{c\lambda^2}}
\leqs \e^{-c\lambda^2} \bigexpec{\e^{cX_n^2}}\;.
\]
Or comme $X_n$ est normale centr\'ee de variance $n$, on a 
\[
\bigexpec{\e^{cX_n^2}} = \int_{-\infty}^\infty 
\e^{cx^2}\frac{\e^{-x^2/2n}}{\sqrt{2\pi n}}\6x 
= \frac{1}{\sqrt{1-2nc}}\;,
\]
et donc 
\[
P_n(\lambda) \leqs \frac{\e^{-c\lambda^2}}{\sqrt{1-2nc}}\;.
\]
Le meilleure borne est obtenue pour $2nc=1-n/\lambda^2$ et donne 
\[
P_n(\lambda) \leqs 
\frac{\lambda}{\sqrt{n}} \e^{-(\lambda^2-n)/2n}\;.
\]
Cette borne est utile lorsque $\lambda^2\gg n$. Dans ce cas elle fournit une
d\'ecroissance exponentielle en $\lambda^2/2n$. 

\item	
Pour tout $c>0$, $\e^{cX_n}$ est une sous-martingale et on a 
\[
\Bigprob{\max_{1\leqs m\leqs n}X_m\geqs\lambda} =
\Bigprob{\max_{1\leqs m\leqs n}\e^{cX_m}\geqs\e^{c\lambda}}
\leqs \e^{-c\lambda} \bigexpec{\e^{cX_n}}\;.
\]
Par compl\'etion du carr\'e on trouve 
\[
\bigexpec{\e^{cX_n}} = \int_{-\infty}^\infty 
\e^{cx}\frac{\e^{-x^2/2n}}{\sqrt{2\pi n}}\6x 
= \e^{c^2n/2}\;, 
\]
d'o\`u 
\[
\Bigprob{\max_{1\leqs m\leqs n}X_m\geqs\lambda} 
\leqs \e^{c^2n/2 - c\lambda}\;.
\]
La borne est optimale pour $c=\lambda/n$ et vaut 
\[
\Bigprob{\max_{1\leqs m\leqs n}X_m\geqs\lambda} 
\leqs \e^{-\lambda^2/2n}\;.
\]
\end{enum}


\section{Exercices du Chapitre~\ref{chap_conv}}

\subsection*{Exercice~\ref{exo_conv1}}

\begin{enum}
\item	
Nous avons montr\'e pr\'ec\'edemment que 
\[
\bigexpec{\braket{X}_\infty} = \sigma^2 \sum_{m=1}^\infty
\frac{\expec{X_{m-1}}}{\mu^{m+1}} = \sigma^2 \sum_{m=1}^\infty
\frac{1}{\mu^{m+1}} = \frac{\sigma^2}{\mu(\mu-1)} < \infty\;.
\]
Par la proposition~6.3.2 du cours, $\expec{X_n^2}$ est uniform\'ement born\'ee,
donc $X_n$ converge dans $L^2$ vers une variable al\'eatoire $X$. Comme $X_n$
converge \`a fortiori dans $L^1$, on a $\expec{X}=1$, et \'egalement 
$\prob{X>0}=1-\rho$. 

\item	
S'il y a $Z_n$ individus au temps $n$, alors $Z_{n+1}=0$ si et seulement si
chacun des $Z_n$ individus n'a aucun descendant, ce qui arrive avec
probabilit\'e $p_0^{Z_n}$. Ceci montre que
$\econd{\indexfct{Z_{n+1}=0}}{\cF_n}=p_0^{Z_n}$, et donc que
$\econd{\indicator{B}}{\cF_n}\geqs p_0^{Z_n}$. Or si $\omega\in A$, alors il
existe $L=L(\omega)$ tel que $Z_n(\omega)\leqs L(\omega)$ pour tout $n$. 
Il suit que $\lim_{n\to\infty}\econd{\indicator{B}}{\cF_n}(\omega)\geqs p_0^L$
pour ces $\omega$. Mais par la loi 0--1 de L\'evy, cette limite vaut
$\indicator{B}(\omega)$. Par cons\'equent, $\indicator{B}(\omega)=1$ pour tout
$\omega\in A$, ou encore $A\subset B$. 

D'une part, par d\'efinition, $Z_\infty=\infty$ dans $A^c$. D'autre
part, $B\subset\set{\lim_{n\to\infty}Z_n=0}$ et donc $Z_\infty=0$ dans
$B$. Ceci montre qu'en fait $A=B=\set{Z_\infty=0}$ et
$A^c=\set{Z_\infty=\infty}$. 
\end{enum}

\subsection*{Exercice~\ref{exo_conv2}}

\begin{enum}
\item	
Par construction, $\xi_n$ est ind\'ependant de $\cF_n =
\sigma(\xi_0,\dots,\xi_{n-1})$. Par cons\'equent \\
 $\econd{\xi_n}{\cF_n}=\expec{\xi_n}=1$, et il suit que 
$\econd{X_{n+1}}{\cF_n}=(1-\lambda)X_n+\lambda X_n=X_n$.

\item	
Comme $X_n$ est une martingale, 
$\expec{X_n}=\expec{\econd{X_n}{\cF_0}}=\expec{X_0}=1$.

\item	
$X_n$ \'etant une surmartingale positive (donc $-X_n$ une sous-martingale
born\'ee sup\'erieu\-rement), elle converge presque s\^urement vers une variable
al\'eatoire int\'egrable $X$.

\item	
Comme $\xi_n$ est ind\'ependante de $\cF_n$, avec $\expec{\xi_n}=1$ et
$\expec{\xi_n^2}=2$, on obtient \\
$\expec{X_{n+1}^2}=(1+\lambda^2)\expec{X_n^2}$
donc $\expec{X_n^2}=(1+\lambda^2)^n$.

\item	
La suite $\expec{X_n^2}$ diverge, donc la suite des $X_n$ ne converge pas dans
$L^2$.

\item	
On a $\econd{(X_{n+1}-X_n)^2}{\cF_n} = \lambda^2 X_n^2 \expec{(\xi_n-1)^2} =
\lambda^2 X_n^2$, d'o\`u 
\[
\braket{X}_n = \sum_{m=0}^{n-1} \bigecond{(X_{m+1}-X_m)^2}{\cF_m}
= \lambda^2 \sum_{m=0}^{n-1} X_m^2\;.
\]

\item	
\begin{enum}
\item	
Comme $X_{n+1}=X_n\xi_n$, on v\'erifie par r\'ecurrence que
$X_n(\Omega)=\set{0,2^n}$ avec 
\begin{align*}
\prob{X_n=2^n} &= \frac{1}{2^n}\;, \\
\prob{X_n=0}   &= 1 - \frac{1}{2^n}\;.
\end{align*}

\item	
Comme $\prob{X_n=0}\to 1$ lorsque $n\to\infty$ et $X_n(\omega)=0$ implique
$X_m(\omega)=0$ pour tout $m\geqs n$, $X_n$ converge presque
s\^urement vers $X=0$. 

\item	
On a $\expec{\abs{X_n-X}}=\expec{\abs{X_n}}=\expec{X_n}=1$ pour tout $n$, donc
$X_n$ ne converge pas dans $L^1$. Les $X_n$ ne peuvent donc pas \^etre
uniform\'ement int\'egrables. 

On peut aussi le voir directement \`a partir de la d\'efinition
d'int\'egrabilit\'e uniforme~: on a 
\[
\bigexpec{\abs{X_n}\indexfct{X_n>M}} = 2^n \bigprob{X_n>M} = 
\begin{cases}
1 & \text{si $2^n > M$\;,} \\
0 & \text{sinon\;.}
\end{cases}
\]
Par cons\'equent, $\sup_n \bigexpec{\abs{X_n}\indexfct{X_n>M}} = 1$ pour tout
$M$.

\item	
C'est \`a vous de voir --- sachant que vous allez perdre votre mise presque
s\^urement. Toutefois, il y a une probabilit\'e non nulle de gagner beaucoup
d'argent apr\`es tout nombre fini de tours. 
\end{enum}
\end{enum}

\subsection*{Exercice~\ref{exo_conv3}}

\begin{enum}
\item	
On a $Y_n = Y_{n-1} + H_n (X_n-X_{n-1})$ avec
$H_n = 2^n \indexfct{X_1=-1,X_2=-2,\dots,X_{n-1}=1-n}$ qui est clairement
pr\'evisible.

\item	
$\econd{Y_n}{\cF_{n-1}} = Y_{n-1} + H_n \econd{X_n-X_{n-1}}{\cF_{n-1}}
= Y_{n-1} + H_n \expec{X_n-X_{n-1}} = Y_{n-1}$.

\item	
On a $\econd{(Y_n-Y_{n-1})^2}{\cF_{n-1}} = H_n^2$ donc 
$\braket{Y}_n = \sum_{m=1}^n H_m^2$. \\
Comme $\expec{H_m^2} = 2^{2m} \prob{X_1=-1,X_2=-2,\dots,X_{n-1}=1-n} =
2^{m+1}$, on obtient 
$\expec{\braket{Y}_n} = \sum_{m=1}^n 2^{m+1} = 4(2^n-1)$. 
Par cons\'equent, $Y^n$ ne converge pas dans $L^2$. 

\item	
Par inspection, on voit que $Y_n$ prend les valeurs $1$ et $1-2^n$, avec 
$\prob{Y_n=1-2^n}=2^{-n}$ (si l'on a perdu $n$ fois) et donc 
$\prob{Y_n=1}=1-2^{-n}$. 

On observe que $\expec{Y_n^+} \leqs 2$ pour tout $n$, donc $Y_n$ converge
presque s\^urement vers une variable $Y_\infty$. L'expression de la loi de $Y_n$
montre que $Y_\infty=1$ presque s\^urement, c'est-\`a-dire qu'on aura gagn\'e
un Euro avec probabilit\'e $1$.  

\item	
Non, car $\expec{Y_n}=0\neq1=\expec{Y_\infty}$. 

\item	
On a $N = \inf\setsuch{n\geqs1}{Z_n-2^{n-1}<-L \text{ ou }Z_n=1}$. C'est un
temps d'arr\^et puisqu'il s'agit d'un temps de premier passage. Sa loi est
donn\'ee par $\prob{N=n}=2^{-n}$ pour $n=1,\dots,k-1$ et
$\prob{N=k}=2^{-(k-1)}$. 

\item	
Oui car $Z_n=Y_{n\wedge N}$ est une martingale arr\^et\'ee.

\item	
Comme dans le cas de $Y_n$, $Z_n$ est une martingale telle que $\expec{Z_n^+}$
est born\'ee, donc elle converge vers une variable al\'eatoire $Z_\infty$. On
trouve $\prob{Z_\infty=1}=1-2^{-k}$ et $\prob{Z_\infty=1-2^k}=2^{-k}$. En
particulier, $\expec{Z_\infty}=0=\expec{Z_n}$ donc $Z_n$ converge dans $L^1$. 
Avec la contrainte de la banque, la grande probabilit\'e de faire un petit gain
est donc compens\'ee par la petite probabilit\'e de faire une grande perte! 
\end{enum}

\subsection*{Exercice~\ref{exo_conv4}}

\begin{enum}
\item	
$\econd{X_n}{\cF_{n-1}} = \econd{2U_nX_{n-1}}{\cF_{n-1}} 
= 2X_{n-1}\econd{U_n}{\cF_{n-1}} = 2X_{n-1}\expec{U_n} = X_{n-1}$.

\item	
On a 
\[
\bigecond{X_n^2}{\cF_{n-1}} = \bigecond{4U_n^2X_{n-1}^2}{\cF_{n-1}}
= 4X_{n-1}^2 \bigexpec{U_n^2} = \frac{4}{3} X_{n-1}^2\;. 
\]
Par cons\'equent, le processus croissant est donn\'e par 
\[
\braket{X}_n = \sum_{m=1}^n \bigecond{X_m^2 - X_{m-1}^2}{\cF_{m-1}}
= \frac{1}{3} \sum_{m=1}^n X_{m-1}^2\;.
\]
Comme 
\[
\bigexpec{\braket{X}_n} = \frac{1}{3} \sum_{m=1}^n \bigexpec{X_{m-1}^2} 
= \frac{1}{3} \sum_{m=1}^n\biggpar{\frac{4}{3}}^{m-1}\;,
\]
on a $\braket{X}_\infty = \infty$, donc la suite ne peut pas converger dans
$L^2$ (on peut aussi observer directement que $\expec{X_n^2}$ diverge).

\item	
$X_n$ est une surmartingale positive, donc elle converge presque s\^urement (on
peut aussi observer que c'est une sous-martingale telle que
$\expec{X_n^+}=\expec{X_n}=1\;\forall n$). 

\item	
Comme $Y_n=Y_{n-1} + \log 2 +  \log(U_n)$, on a 
\[
\bigexpec{Y_n} = \bigexpec{Y_{n-1}} + \log 2 + \bigexpec{\log(U_n)}
= \bigexpec{Y_{n-1}} - \bigbrak{1-\log2}\;,
\]
et donc $\expec{Y_n} = -n[1-\log2]$ tend vers $-\infty$ lorsque $n\to\infty$.
On peut donc s'attendre \`a ce que $Y_n$ converge vers $-\infty$, donc que
$X_n$ converge vers $0$ presque s\^urement.

Pour le montrer, nous devons contr\^oler les fluctuations de 
$Z_n = Y_n + n [1-\log2]$. On peut l'\'ecrire sous la forme 
\[
Z_n = \sum_{k=1}^n V_k 
\qquad
\text{o\`u}
\qquad
V_k = 1 + \log(U_k)\;.
\]
Nous allons montrer que pour $n$ assez grand, $Z_n \leqs c n$ presque
s\^urement pour tout $c\in]0,1[$. En prenant $0<c<1-\log2$, cela montrera qu'en
effet $Y_n$ converge presque s\^urement vers $-\infty$, et donc $X_n\to0$ p.s.

\begin{itemiz}
\item 
Une premi\`ere m\'ethode consiste \`a \'ecrire 
\[
\bigprob{Z_n>cn} = \bigprob{\e^{\gamma Z_n} > \e^{\gamma cn}}
\leqs \e^{-\gamma cn} \bigexpec{\e^{\gamma Z_n}}
= \bigbrak{\e^{-\gamma c}\expec{\e^{\gamma V_1}}}^n\;.
\]
Comme $\e^{\gamma V_1} = \e^\gamma U_1^\gamma$, on a 
\[
\bigexpec{\e^{\gamma V_1}} = \e^\gamma \int_0^1 x^\gamma\6x  
= \frac{\e^\gamma}{\gamma+1}\;.
\]
Il suit donc que 
\[
\bigprob{Z_n>cn} \leqs
\biggbrak{\frac{\e^{-\gamma(c-1)}}{\gamma+1}}^n
\]
pour tout $\gamma>0$. La meilleure borne est obtenue pour $\gamma=c/(1-c)$, et
a la forme 
\[
\bigprob{Z_n>cn} \leqs
\bigbrak{\e^c(1-c)}^n\;.
\]
La s\'erie de terme g\'en\'eral $\brak{\e^c(1-c)}^n$ converge pour tout
$c\in]0,1[$, donc le lemme de Borel--Cantelli montre que 
\[
\limsup_{n\to\infty} \frac{Z_n}{n} \leqs c 
\]
presque s\^urement. C'est le r\'esultat cherch\'e.

\item
Une seconde m\'ethode de montrer que $Z_n \leqs c n$ presque
s\^urement pour $n$ assez grand consiste \`a \'ecrire  
\[
\bigprob{\abs{Z_n}>cn} = \bigprob{Z_n^4 > (cn)^4} 
\leqs \frac{1}{(cn)^4} \bigexpec{Z_n^4}\;.
\]
On a 
\[
Z_n^4 = \sum_{i,j,k,l=1}^n V_iV_jV_kV_l\;.
\]
On trouve facilement les moments  
\[
\bigexpec{V_i} = 0\;, \quad
\bigexpec{V_i^2} = 1\;, \quad
\bigexpec{V_i^3} = -2\;, \quad
\bigexpec{V_i^4} = 9\;.
\]
Comme les $V_i$ sont ind\'ependants, 
les seuls termes contribuant \`a $\expec{Z_n^4}$ sont ceux qui contiennent
soit quatre indices identiques, soit deux paires d'indices identiques. Un peu de
combinatoire nous fournit alors 
\[
\bigexpec{Z_n^4} = n \bigexpec{V_1^4} + \binom{4}{2} \frac{n(n-1)}{2}
\bigexpec{V_1^2}^2 = 3n^2+6n\;.
\]
Ceci implique 
\[
\bigprob{\abs{Z_n}>cn} \leqs \frac{3}{c^4 n^2}\;,
\]
et le lemme de Borel--Cantelli permet de conclure. 
\end{itemiz}

\item	
Nous avons montr\'e que $X_n\to0$ presque s\^urement. Par cons\'equent, nous
avons aussi $X_n\to0$ en probabilit\'e. Comme par ailleurs,
$\expec{\abs{X_n}}=1$ pour tout $n$, le Th\'eor\`eme~6.4.3 du cours montre que
$X_n$ ne peut pas converger dans $L^1$. 
\end{enum}

\subsection*{Exercice~\ref{exo_conv5}}

\begin{enum}
\item	
Commen\c cons par montrer que les $X_n$ sont int\'egrables. Si $K$ d\'enote la
constante de Lipschitz, on a 
$\abs{f((k+1)2^{-n})-f(k2^{-n})} \leqs K2^{-n}$, d'o\`u 
\[
\bigexpec{\abs{X_n}} \leqs \sum_{k=0}^{2^n-1} K \,\bigprob{U\in I_{k,n}} \leqs
K\;.
\]
Pour calculer les esp\'erances conditionnelles, on observe que chaque
intervalle $I_{\ell,n}$ est la r\'eunion disjointe des deux intervalles
$I_{2\ell,n+1}$ et $I_{2\ell+1,n+1}$ de m\^eme longueur. Par cons\'equent 
\[
\bigecond{\indexfct{U\in I_{2\ell,n+1}}}{\cF_n} = 
\bigecond{\indexfct{U\in I_{2\ell+1,n+1}}}{\cF_n} = 
\frac12 \indexfct{U\in I_{\ell,n}}\;.
\]
En s\'eparant les termes pairs et impairs, on a 
\begin{align*}
\bigecond{X_{n+1}}{\cF_n} 
={}& \sum_{\ell=0}^{2^n-1} \biggl[
\frac{f((2\ell+1)2^{-(n+1)}) - f(2\ell\,2^{-(n+1)})}{2^{-(n+1)}}
\bigecond{\indexfct{U\in I_{2\ell,n+1}}}{\cF_n} \\
&{} \phantom{\sum_{\ell=0}^{2^n-1}}
+ \frac{f((2\ell+2)2^{-(n+1)}) - f((2\ell+1)\,2^{-(n+1)})}{2^{-(n+1)}}
\bigecond{\indexfct{U\in I_{2\ell+1,n+1}}}{\cF_n}
\biggr] \\
={}& \sum_{\ell=0}^{2^n-1} 
\frac{f((2\ell+2)2^{-(n+1)}) - f(2\ell\,2^{-(n+1)})}{2^{-(n+1)}} 
\frac12\indexfct{U\in I_{\ell,n}} \\
={}& \sum_{\ell=0}^{2^n-1} 
\frac{f((2\ell+2)2^{-n}) - f(2\ell\,2^{-(n+1)})}{2^{-n}} 
\indexfct{U\in I_{\ell,n}} \\
={}& X_n\;.
\end{align*}

\item	
On a $\expec{X_n^+} \leqs \expec{\abs{X_n}} \leqs K$, donc $X_n$ converge
presque s\^urement vers une variable al\'eatoire $X_\infty$. 

\item	
En bornant chaque fonction indicatrice par $1$ et en utilisant \`a nouveau le
caract\`ere lipschitzien, on voit que $\abs{X_n(\omega)}$ est born\'e par $K$
pour tout $\omega$. Par cons\'equent, 
\[
\indexfct{\abs{X_n}>M} = 0 
\qquad
\text{pour $M > K$\;,}
\]
donc les $X_n$ sont uniform\'ement int\'egrables. Il suit que $X_n$ converge
vers $X_\infty$ dans $L^1$. 

\item	
Pour tout $n$, on a 
\begin{align*}
X_n \indexfct{U\in[a,b]}
&= 
\sum_{k=0}^{2^n-1} \frac{f((k+1)2^{-n})-f(k2^{-n})}{2^{-n}}
\indexfct{U\in I_{k,n}\cap[a,b]} \\
&= 
\sum_{k \colon I_{k,n}\cap[a,b]\neq\emptyset} 
\frac{f((k+1)2^{-n})-f(k2^{-n})}{2^{-n}}
\indexfct{U\in I_{k,n}}\;. 
\end{align*}
Soient $k_-(n)$ et $k_+(n)$ le plus petit et le plus grand $k$ tel que 
$I_{k,n}\cap[a,b]\neq\emptyset$. Prenant l'esp\'erance, comme 
$\prob{U\in I_{k,n}}=2^{-n}$ on voit que 
\[
\bigexpec{X_n \indexfct{U\in[a,b]}}
= f((k_+(n)+1)2^{-n}) - f(k_-(n)2^{-n})\;.
\]
Lorsque $n\to\infty$, on a $k_-(n)2^{-n}\to a$ et $(k_+(n)+1)2^{-n}\to b$. 

\item	
D'une part, 
\[
\bigexpec{X_\infty\indexfct{U\in[a,b]}}
= \int_0^1 X_\infty(\omega) \indexfct{\omega\in[a,b]} \6\omega
= \int_a^b X_\infty(\omega) \6\omega\;.
\]
D'autre part, par le th\'eor\`eme fondamental du calcul int\'egral, 
\[
f(b) - f(a) = \int_a^b f'(\omega)\6\omega\;.
\]
On en conclut que 
\[
X_\infty(\omega) = f'(\omega)\;.
\]
\end{enum}

\subsection*{Exercice~\ref{exo_conv7}}

\begin{enum}
\item 	Comme 
\[
f(Y) = \indexfct{Y<c}f(c) + \indexfct{Y\geqs c}\int_c^Y f'(\lambda)\6\lambda
\leqs f(c) + \int_c^\infty \indexfct{\lambda<Y} f'(\lambda)\6\lambda\;,
\]
on obtient, en prenant l'esp\'erance, 
\[
\bigexpec{f(Y)} \leqs f(c) 
+ \int_c^\infty \bigprob{Y>\lambda} f'(\lambda)\6\lambda\;. 
\]

\item	
Pour tout $M>1$ nous pouvons \'ecrire 
\[
\bigexpec{\Xbar_n \wedge M}
\leqs 1 + \int_1^\infty \bigprob{\Xbar_n \wedge M > \lambda} \6\lambda\;.
\]
Par l'in\'egalit\'e de Doob, 
\begin{align*}
\int_1^\infty \bigprob{\Xbar_n \wedge M > \lambda} \6\lambda
&\leqs \int_1^\infty \frac1\lambda \bigexpec{\Xbar_n \indexfct{\Xbar_n\wedge
M\geqs\lambda}} \6\lambda \\
&= \int_1^\infty \frac1\lambda \int_\Omega X_n^+ \indexfct{\Xbar_n\wedge
M\geqs\lambda}\6\fP \6\lambda \\
&= \int_\Omega X_n^+ \int_1^{\Xbar_n\wedge M} \frac1\lambda \6\lambda
\indexfct{\Xbar_n\wedge M > 1} \6\fP \\
&= \int_\Omega X_n^+ \log^+(\Xbar_n\wedge M) \6\fP \\
&= \bigexpec{X_n^+ \log^+(\Xbar_n\wedge M)} \\
&\leqs \bigexpec{X_n^+ \log^+(X_n^+)} + \frac{1}{\e}\bigexpec{\Xbar_n\wedge
M}\;.
\end{align*}
Il suit que 
\[
\biggpar{1-\frac{1}{\e}} \bigexpec{\Xbar_n\wedge M}
\leqs 1 + \bigexpec{X_n^+ \log^+(X_n^+)}\;.
\]
Le r\'esultat s'obtient en faisant tendre $M$ vers l'infini, et en invoquant
le th\'eor\`eme de convergence domin\'ee. 

\item	Pour une variable $Y$ int\'egrable, on a n\'ecessairement 
\[
\lim_{M\to\infty} \expec{Y \indexfct{Y>M}} 
= \lim_{M\to\infty} \Bigbrak{\expec{Y} - \expec{Y \indexfct{Y\leqs M}}}
= 0
\]
en vertu du th\'eor\`eme de convergence monotone. Comme par ailleurs
\[
\abs{X_n} \indexfct{\abs{X_n} > M} \leqs Y \indexfct{Y>M}
\]
pour tout $n$, 
le r\'esultat suit en prenant le sup puis l'esp\'erance des deux c\^ot\'es. 

\item	
Soit $Y = \sup_n \abs{X_n}$. On a $Y \leqs \Xbar_n^+ + \Xbar_n^-$,
o\`u $\Xbar_n^- = \sup_n (-X_n)$. Alors le r\'esultat du point 2.\
implique $\expec{Y} < \infty$, et le point 3.\ montre que que la suite des
$X_n$ est uniform\'ement int\'egrable. Le th\'eor\`eme de convergence $L^1$
permet de conclure. 

uniforme. 

%
%


\end{enum}

\subsection*{Exercice~\ref{exo_conv6}}

\begin{enum}
\item	
Le th\'eor\`eme de d\'erivation sous le signe int\'egral de Lebesgue montre que
$\psi(\lambda)$ est de classe $\cC^\infty$ avec 
\[
\psi^{(n)}(0) = \expec{X_1^n} 
\]
pour tout $n\geqs0$. Le r\'esultat suit alors de la formule de Taylor et des
hypoth\`eses sur la loi de $X_1$. 

\item	
La relation 
\[
Y_{n+1} = \frac{\e^{\lambda X_{n+1}}}{\psi(\lambda)}Y_n
\]
montre que $\expec{\abs{Y_n}}$ est born\'ee par $1$ et que 
$\econd{Y_{n+1}}{Y_n}=Y_n$. 

\item	
Le probabilit\'e peut se r\'e\'ecrire 
\[
\biggprob{\sup_{n\leqs N} \biggpar{S_n - n \frac{\log\psi(\lambda)}{\lambda}} 
> a} = \biggprob{\sup_{n\leqs N} Y_n > \e^{\lambda a}}
\]
et la borne suit de l'in\'egalit\'e de Doob. 

\item	
On remarque que pour $n\in]t^k,t^{k+1}]$, on a 
\[
h(n)c_k \geqs h(t^k)c_k = \frac{\alpha}{2}h(t^k) +
t^{k+1}\frac{\log\psi(\lambda_k)}{\lambda_k}
\geqs a_k + n \frac{\log\psi(\lambda_k)}{\lambda_k}\;.
\]
La probabilit\'e cherch\'ee est donc born\'ee par
$\e^{-\lambda_ka_k}=\e^{-\alpha\log(k\log t)}=(k\log t)^{-\alpha}$ en vertu du
r\'esultat pr\'ec\'edent. 

\item	
Cela suit du lemme de Borel--Cantelli, et du fait que 
\[
\sum_{k=1}^\infty \frac{1}{(k\log t)^{\alpha}}
= \frac{1}{(\log t)^{\alpha}}\sum_{k=1}^\infty \frac{1}{k^\alpha}
< \infty\;.
\]

\item	
\[
c_k = \frac{\alpha}{2} + \frac{t}{\lambda_k^2}
\log\biggpar{1+\frac12\lambda_k^2+\order{\lambda_k^2}}
= \frac{\alpha+t}{2} + \order{1}\;.
\]

\item	
Soit $\varepsilon>0$. Prenons $\alpha=t=1+\varepsilon/2$. Les r\'esultats
pr\'ec\'edents montrent que pour presque tout $\omega$, il existe
$k_0(\omega,\varepsilon)<\infty$ tel que 
\[
\frac{S_n(\omega)}{h(n)} \leqs 1+\frac{\varepsilon}{2} + r(k)
\]
pour tout $k\geqs k_0(\omega,\varepsilon)$ et $n\in]t^k,t^{k+1}]$. 
En prenant de plus $k$ assez grand pour que $r(k)<\varepsilon/2$, on obtient 
$S_n/h(n) \leqs 1+\varepsilon$ pour tous les $n$ assez grands, d'o\`u le
r\'esultat.
\end{enum}


\section{Exercices du Chapitre~\ref{chap_mb}}

\subsection*{Exercice~\ref{exo_mb1}}

Pour $t>s\geqs0$, 
\begin{align*}
\econd{\cosh(\gamma B_t)}{\cF_s} 
&= \frac{\e^{\gamma B_s}}{2} \econd{\e^{\gamma(B_t-B_s)}}{\cF_s}
+ \frac{\e^{-\gamma B_s}}{2} \econd{\e^{-\gamma(B_t-B_s)}}{\cF_s} \\
&= \frac{\e^{\gamma B_s}}{2} \e^{\gamma^2(t-s)/2}
+ \frac{\e^{-\gamma B_s}}{2} \e^{\gamma^2(t-s)/2} \\
&= \cosh(\gamma B_s)\e^{\gamma^2(t-s)/2}\;.
\end{align*}
On a $1=\expec{X_{t\wedge\tau}}$. Faisant tendre $t$ vers l'infini, on obtient
par le th\'eor\`eme de convergence domin\'ee 
$1=\expec{X_{\tau}}=\cosh(\gamma a)\expec{\e^{-\gamma^2\tau/2}}$. Il suffit
alors de prendre $\gamma=\sqrt{2\lambda}$.

\subsection*{Exercice~\ref{exo_mb2}}

\begin{enum}
\item	
$X_t$ \'etant une martingale, $\expec{\abs{X_t}}<\infty$. 
Pour tout $A\in\cF_s$, on a 
\[
\expec{X_t\indicator{A}} = \expec{X_s\indicator{A}}\;.
\]
Si $f$ est contin\^ument diff\'erentiable en $\gamma$ et
$\expec{\abs{\partial_\gamma f(B_t,t,\gamma)}}<\infty$, on peut prendre la
d\'eriv\'ee des deux c\^ot\'es et permuter esp\'erance et d\'eriv\'ee, ce qui
montre le r\'esultat. 

\item	
En d\'eveloppant l'exponentielle ou en calculant des d\'eriv\'ees successives,
on trouve
\[
f(x,t,\gamma) = 1 + \gamma x + \frac{\gamma^2}{2!} (x^2-t) 
+ \frac{\gamma^3}{3!} (x^3-3tx)
+ \frac{\gamma^4}{4!} (x^4-6tx^2+3t^2)
+ \Order{\gamma^5}\;.
\]
On retrouve le fait que $B_t$ et $B_t^2-t$ sont des martingales, mais on trouve
aussi deux nouvelles martingales: $B_t^3-3tB_t$ et $B_t^4-6tB_t^2+3t^2$. 

\item	
On a $\expec{B_{t\wedge\tau}^4-6(t\wedge\tau)B_{t\wedge\tau}^2} = 
-3\expec{(t\wedge\tau)^2}$. Nous avons d\'ej\`a \'etabli que
$\expec{\tau}=a^2<\infty$. Lorsque $t$ tend vers l'infini, 
$\expec{(t\wedge\tau)^2}$ tend vers $\expec{\tau^2}$ par le th\'eor\`eme de
convergence monotone. D'autre part, par le th\'eor\`eme de convergence
domin\'ee, $\expec{B_{t\wedge\tau}^4-6(t\wedge\tau)B_{t\wedge\tau}^2}$ tend
vers $\expec{B_\tau^4-6\tau B_\tau^2}=a^4-6a^2\expec{\tau}=-5a^4$. Ainsi,
\[
\expec{\tau^2} = \frac53 a^4\;.
\]
\end{enum}

\subsection*{Exercice~\ref{exo_mb3}}

\begin{enum}
\item	
Soit la martingale $M_t=\e^{\gamma B_t-\gamma^2 t^2/2}=\e^{\gamma X_t-\lambda
t}$. Dans ce cas, nous ne savons pas si le temps d'arr\^et $\tau$ est born\'e.
Mais la preuve de la Proposition~\ref{prop_mbm1} peut \^etre adapt\'ee
afin de montrer que 
\[
1 = \bigexpec{M_{t\wedge\tau}\indexfct{\tau<\infty}}
= \bigexpec{\e^{\gamma
X_{t\wedge\tau}-\lambda(t\wedge\tau)}\indexfct{\tau<\infty}}\;.
\]
Faisant tendre $t$ vers l'infini, on obtient 
\[
1 = \bigexpec{M_\tau\indexfct{\tau<\infty}}
= \bigexpec{\e^{\gamma
X_\tau-\lambda\tau}\indexfct{\tau<\infty}}
= \e^{\gamma a}\bigexpec{\e^{-\lambda\tau}\indexfct{\tau<\infty}}\;,
\]
et donc 
\[
 \bigexpec{\e^{-\lambda\tau}\indexfct{\tau<\infty}} = 
\e^{-\gamma a} = \e^{-a(b+\sqrt{b^2+2\lambda})}\;.
\]

\item	
On a $\expec{\e^{-\lambda\tau}\indexfct{\tau<\infty}} =
\e^{-a\sqrt{2\lambda}}$. 

En particulier, prenant $\lambda=0$, on obtient $\prob{\tau<\infty}=1$. 
Donc en fait $\expec{\e^{-\lambda\tau}} = \e^{-a\sqrt{2\lambda}}$.
Remarquons toutefois qu'en prenant la d\'eriv\'ee par rapport \`a $\lambda$, on
voit que $\expec{\tau}=\lim_{\lambda\to0+}\expec{\tau\e^{-\lambda\tau}}
=-\lim_{\lambda\to 0+}\dtot{}{\lambda}\e^{-a\sqrt{2\lambda}}=+\infty$. Cela est
li\'e au fait que, comme la marche al\'eatoire, le mouvement Brownien est
r\'ecurrent nul.

\item	
Prenant $\lambda=0$, on obtient $\prob{\tau<\infty}=\e^{-ab}$. Le mouvement
Brownien a donc une probabilit\'e $1-\e^{-ab}$ de ne jamais atteindre la droite
$x=a+bt$. 
\end{enum}

\subsection*{Exercice~\ref{exo_mb4}}

\begin{enum}
\item 	Approche martingale. 
\begin{enum}
\item 	On a $\expec{\abs{\e^{\icx \lambda X_t}}} = \expec{\e^{-\lambda\im X_t}}
< \infty$ pour $\lambda\geqs0$ et $t<\infty$. De plus, si $t>s\geqs0$, alors 
\begin{align*}
\pecond{\e^{\icx \lambda X_t}}{\cF_s}
&= \e^{\icx \lambda X_s} \expec{\e^{\icx \lambda (X_t-X_s)}} \\
&= \e^{\icx \lambda X_s} \expec{\e^{\icx \lambda (B^{(1)}_t - B^{(1)}_s)}}
\expec{\e^{- \lambda (B^{(2)}_t - B^{(2)}_s)}} \\
&= \e^{\icx \lambda X_s} \e^{-(t-s)\lambda^2/2}\e^{(t-s)\lambda^2/2} \\
&= \e^{\icx \lambda X_s}\;.
\end{align*}
Par cons\'equent, $\e^{\icx \lambda X_t}$ est bien une martingale. 

\item 	Comme $\e^{\icx \lambda X_t}$, le th\'eor\`eme d'arr\^et montre que 
\[
\expec{\e^{\icx \lambda X_{t\wedge\tau}}} = \e^{\icx \lambda X_0} =
\e^{-\lambda}\;.
\]
Si $\lambda\geqs 0$, le membre de gauche est born\'e par $1$ puisque $\re (\icx
\lambda X_{t\wedge\tau})\leqs 0$. On peut donc prendre la limite $t\to\infty$. 
Si $\lambda<0$, on peut appliquer un raisonnement similaire \`a la
martingale $\e^{-\icx \lambda X_t}$. Par cons\'equent on obtient 
\[
\expec{\e^{\icx \lambda X_{t\wedge\tau}}} = \e^{-\abs{\lambda}}\;.
\]

\item 	La densit\'e de $X_\tau$ s'obtient par transform\'ee de Fourier
inverse, 
\[
f_{X_\tau}(x) = \frac{1}{2\pi} \int_{-\infty}^\infty 
\e^{-\abs{\lambda}} \e^{-\icx\lambda x}\6\lambda 
= \frac{1}{2\pi} \biggpar{\frac{1}{1+\icx x} + \frac{1}{1-\icx x}}
= \frac{1}{\pi(1+x^2)}\;.
\]
$X_\tau$ suit donc une loi de Cauchy (de param\`etre $1$). 

\end{enum}

\item	Approche principe de r\'eflexion.
\begin{enum}
\item 	$B^{(1)}_t$ suit une loi normale centr\'ee de variance $t$, donc sa
densit\'e est 
\[
\frac{1}{\sqrt{2\pi t}} \e^{-x^2/2t}\;.
\]

\item	Le principe de r\'eflexion implique 
\[
\bigprob{\tau < t} = 2 \bigprob{1+B^{(2)}_t < 0} 
= 2 \biggprob{B^{(2)}_1 < -\frac{1}{\sqrt{t}}} = 2
\int_{-\infty}^{-1/\sqrt{t}}
\frac{\e^{-y^2/2}}{\sqrt{2\pi}}\6y
\]
ce qui donne pour la densit\'e de $\tau$
\[
\dtot{}{t} \bigprob{\tau < t} 
= \frac{1}{\sqrt{2\pi t^3}}\e^{-1/2t}\;.
\]

\item 	La densit\'e de $X_\tau = B^{(1)}_\tau$ vaut donc 
\[
\int_0^\infty \frac{1}{\sqrt{2\pi t^3}}\e^{-1/2t} \frac{1}{\sqrt{2\pi t}}
\e^{-x^2/2t}\6t 
= \frac{1}{2\pi} \int_0^\infty \frac{\e^{-(1+x^2)/2t}}{t^2}\6t
= \frac{1}{\pi(1+x^2)}
\]
o\`u on a utilis\'e le changement de variable $u=1/t$. 

\end{enum}

\item 	Approche invariance conforme. 

\begin{enum}
\item	On v\'erifie que $f'(z)\neq 0$, que $f$ admet une r\'eciproque, et que
$\abs{f(x)}=1$ pour $x$ r\'eel. 

\item 	Le lieu de sortie du disque $\fD$ du mouvement Brownien issu de
$0$ a une distribution uniforme, par sym\'etrie. Par le th\'eor\`eme de
transfert, la loi de $X_\tau$ est alors donn\'ee par 
\[
\bigprob{X_\tau\in\6x} = \frac{1}{2\pi} \abs{f'(x)}\6x =
\frac{1}{\pi(1+x^2)}\6x\;.
\]

\end{enum}

\end{enum}


\section{Exercices du Chapitre~\ref{chap_is}}

\subsection*{Exercice~\ref{exo_ito1}}

\begin{enum}
\item	
$X_t$ \'etant l'int\'egrale d'un processus adapt\'e, on a $\expec{X_t}=0$.

Par cons\'equent, l'isom\'etrie d'It\^o donne 
$\Variance(X_t) = \expec{X_t^2} = \int_0^t \e^{2s}\6s =
\frac12\brak{\e^{2t}-1}$.

Enfin, par lin\'earit\'e $\expec{Y_t}=0$ et par bilin\'earit\'e
$\Variance(Y_t) = \e^{-2t}\Variance(X_t) = \frac12\brak{1-\e^{-2t}}$. 

\item	
Etant des int\'egrales stochastiques de fonctions d\'eterministes, $X_t$ et
$Y_t$ suivent des lois normales (centr\'ees, de variance calcul\'ee ci-dessus). 

\item	
La fonction caract\'eristique de $Y_t$ est $\expec{\e^{\icx u
Y_t}}=\e^{-u^2\Variance(Y_t)/2}$. Elle converge donc vers $\e^{-u^2/4}$ lorsque
$t\to\infty$. Par cons\'equent, $Y_t$ converge en loi vers une
variable $Y_\infty$, de loi normale centr\'ee de variance $1/2$. 

\item	
La formule d'It\^o avec $u(t,x) = \e^{-t}x$ donne 
\[
\6Y_t = -\e^{-t}X_t\6t + \e^{-t}\6X_t = -Y_t\6t + \6B_t\;. 
\]
$Y_t$ est appel\'e \defwd{processus d'Ornstein--Uhlenbeck}. 
\end{enum}

\subsection*{Exercice~\ref{exo_ito2}}

\begin{enum}
\item	
$X_t$ \'etant l'int\'egrale d'un processus adapt\'e, on a $\expec{X_t}=0$.

Par cons\'equent, l'isom\'etrie d'It\^o donne 
$\Variance(X_t) = \expec{X_t^2} = \int_0^t s^2\6s =
\frac13t^3$.

\item	
$X_t$ suit une loi normale centr\'ee de variance $\frac13t^3$.

\item	
La formule d'It\^o avec $u(t,x)=tx$ donne $\6\,(tB_t)=B_t\6t+t\6B_t$. 

\item	
Comme\/ $B_s\6s = \6\,(sB_s) - s\6B_s$, on a la formule d'int\'egration par
parties 
\[
Y_t = \int_0^t \6\,(sB_s) - \int_0^t s\6B_s
= tB_t - X_t\;.
\]
$Y_t$ suit donc une loi normale de moyenne nulle.

\item	
\begin{enum}
\item	
Comme $\expec{B_sB_u}=s\wedge u$, 
\begin{align*}
\expec{Y_t^2} &= \E \int_0^t\int_0^t B_s B_u \,\6s\6u 
= \int_0^t\int_0^t(s\wedge u)\,\6s\6u \\
&= \int_0^t \biggbrak{\int_0^u s\6s + \int_u^t u\6s}\6u
= \int_0^t \biggbrak{\frac12 u^2 + ut - u^2}\6u
= \frac13t^3\;.
\end{align*}

\item	
Pour calculer la covariance, on introduit une partition $\set{t_k}$ de $[0,t]$,
d'espacement $1/n$. Alors 
\begin{align*}
\cov(B_t,X_t) &= \expec{B_tX_t} \\
&= \E \int_0^t sB_t\6B_s \\
&= \lim_{n\to\infty} \sum_k t_{k-1} \expec{B_t (B_{t_k}-B_{t_{k-1}})} \\
&= \lim_{n\to\infty} \sum_k t_{k-1} (t_k-t_{k-1}) \\
&= \int_0^t s\6s = \frac12t^2\;.
\end{align*}
Il suit que 
\[
\Variance(Y_t) = \Variance(tB_t) + \Variance(X_t) - 2\cov(tB_t,X_t) 
= t^3 + \frac13t^3 - 2t\cov(B_t,X_t) = \frac13t^3\;. 
\]
\end{enum}

$Y_t=tB_t-X_t$ \'etant une combinaison lin\'eaire de variables normales
centr\'es, elle suit \'egalement une loi normale centr\'ee, en l'occurrence de
variance $t^3/3$. Remarquons que $Y_t$ repr\'esente l'aire (sign\'ee) entre la
trajectoire Brownienne et l'axe des abscisses.  
\end{enum}

\subsection*{Exercice~\ref{exo_ito3}}

\begin{enum}
\item	
En compl\'etant le carr\'e ($y-y^2/2\sigma^2 = \sigma^2/2 -
(y-\sigma^2)^2/2\sigma^2$), il vient  
\[
\expec{\e^Y}= \int_{-\infty}^\infty \e^y
\frac{\e^{-y^2/2\sigma^2}}{\sqrt{2\pi\sigma^2}} \6y 
= \e^{\sigma^2/2} \int_{-\infty}^\infty
\frac{\e^{-(y-\sigma^2)^2/2\sigma^2}}{\sqrt{2\pi\sigma^2}} \6y 
=\e^{\sigma^2/2}\;.
\]

\item	
$\varphi(s)$ \'etant adapt\'e, on a $\expec{X_t}=0$ pourvu que $\varphi$ soit
int\'egrable. L'isom\'etrie d'It\^o montre que 
\[
\Variance{X_t} = \int_0^t \varphi(s)^2 \6s =: \Phi(t)\;,
\]
pourvu que $\varphi$ soit de carr\'e int\'egrable. 

\item	
Soit $t>s\geqs0$. 
La diff\'erence $X_t-X_s=\int_s^t \varphi(u)\6B_u$ est ind\'ependante de
$\cF_s$, et suit une loi normale centr\'ee de variance $\Phi(t)-\Phi(s)$. Par
cons\'equent, 
\[
\econd{\e^{X_t}}{\cF_s} = 
\econd{\e^{X_s}\e^{X_t-X_s}}{\cF_s} = 
\e^{X_s} \expec{\e^{X_t-X_s}} = 
\e^{(\Phi(t)-\Phi(s))/2}
\]
en vertu de 1., ce qui \'equivaut \`a la propri\'et\'e de martingale pour
$M_t$. 

\item	
Soit $\gamma>0$. 
En rempla\c cant $\varphi$ par $\gamma\varphi$ dans la d\'efinition de $X_t$,
on voit que 
\[
M^{(\gamma)}_t = \exp\biggset{\gamma X_t - \frac{\gamma^2}2 \int_0^t
\varphi(s)^2 \6s}
\]
est \'egalement une martingale. Il suit que 
\begin{align*}
\biggprob{\sup_{0\leqs s\leqs t} X_s > \lambda}
&= \biggprob{\sup_{0\leqs s\leqs t} \e^{\gamma X_s} > \e^{\gamma\lambda}} \\
&= \biggprob{\sup_{0\leqs s\leqs t} M^{(\gamma)}_s \e^{\gamma^2\Phi(s)/2} >
\e^{\gamma\lambda}} \\
&\leqs \biggprob{\sup_{0\leqs s\leqs t} M^{(\gamma)}_s  >
\e^{\gamma\lambda - \gamma^2\Phi(t)/2}} \\
&\leqs \e^{\gamma^2\Phi(t)/2 - \gamma\lambda} \expec{M^{(\gamma)}_t}\;,
\end{align*}
la derni\`ere in\'egalit\'e d\'ecoulant de l'in\'egalit\'e de Doob. Comme
$\expec{M^{(\gamma)}_t}=1$ par la propri\'et\'e de martingale, le r\'esultat
suit en optimisant sur $\gamma$, c'est-\`a-dire en prenant
$\gamma=\lambda/\Phi(t)$. 
\end{enum}

\subsection*{Exercice~\ref{exo_ito4}}

\begin{enum}
\item 	
Soit $t_k=t_k(n) = k2^{-n}$ et $N=\intpart{2^nT}$. Alors 
\begin{align*}
\int_0^T B_t \circ \6B_t 
&= \lim_{n\to\infty}
\sum_{k=1}^N \frac{B_{t_k} + B_{t_{k-1}}}{2} \Delta B_k \\
&= \frac12\lim_{n\to\infty}
\sum_{k=1}^N \bigpar{B_{t_k}^2 - B_{t_{k-1}}^2}  \\
&= \frac12 B_T^2\;.
\end{align*}

\item
On choisit une partition $t_k(n)$ comme ci-dessus. Alors les processus 
\begin{align*}
X_t^{(n)} &= \sum_{k=1}^{N} \frac{g(X_{t_k})+g(X_{t_{k-1}})}{2} \Delta B_k \\
Y_t^{(n)} &= \sum_{k=1}^{N} g(X_{t_{k-1}}) \Delta B_k 
\end{align*}
avec $N=N(t)$ correspondant \`a l'intervalle de la partition contenant $t$,
convergent respectivement vers $X_t$ et $Y_t$ lorsque $n\to\infty$. 
Leur diff\'erence s'\'ecrit 
\[
X_t^{(n)} - Y_t^{(n)} 
= \frac12 \sum_{k=1}^{N} \bigbrak{g(X_{t_k})-g(X_{t_{k-1}})} \Delta B_k\;.
\]
La formule de Taylor implique 
\[
g(X_{t_k})-g(X_{t_{k-1}})
= g'(X_{t_{k-1}})(X_{t_k}-X_{t_{k-1}}) 
+ \order{X_{t_k}-X_{t_{k-1}}}\;.
\]
Or 
\begin{align*}
X_{t_k} - X_{t_{k-1}}
&= \int_{t_{k-1}}^{t_k} g(X_s) \circ \6B_s \\
&= g(X_{t_{k-1}})\Delta B_k 
+ \int_{t_{k-1}}^{t_k} \bigbrak{g(X_s)-g(X_{t_{k-1}})}\circ\6B_s \\
&= g(X_{t_{k-1}})\Delta B_k + \order{t_k-t_{k-1}}\;.
\end{align*}
On en conclut que 
\[
g(X_{t_k}) - g(X_{t_{k-1}})
= g'(X_{t_{k-1}}) g(X_{t_{k-1}}) \Delta B_k
+ \order{\Delta B_k} 
+ \order{t_k-t_{k-1}}\;.
\]
En substituant dans l'expression de $X_t^{(n)}-Y_t^{(n)}$, on obtient donc 
\[
X_t^{(n)}-Y_t^{(n)} 
= \frac12 \sum_{k=1}^N \bigbrak{g'(X_{t_{k-1}}) g(X_{t_{k-1}}) \Delta B_k^2 
 + \order{\Delta B_k^2} + 
\order{(t_k-t_{k-1})\Delta B_k}}\;.
\]
En proc\'edant comme dans la preuve de la formule d'It\^o, on peut remplacer
$\Delta B_k^2$ par $\Delta t_k$, et il suit 
\[
\lim_{n\to\infty} \bigbrak{X_t^{(n)}-Y_t^{(n)}} 
= \frac12 \int_0^t g'(X_s)g(X_s)\6s\;. 
\]
\end{enum}


\section{Exercices du Chapitre~\ref{chap_sde}}

\subsection*{Exercice~\ref{exo_sde1}}

\begin{enum}
\addtocounter{enumi}{1}
\item	
\begin{align*}
\6Y_t 
&= -a(t) \e^{-\alpha(t)} X_t \6t + \e^{-\alpha(t)}\6X_t \\ 
&= \e^{-\alpha(t)}b(t)\6t + \e^{-\alpha(t)}c(t)\6B_t\;.
\end{align*}

\item	
En int\'egrant, il vient 
\[
Y_t = Y_0 + \int_0^t \e^{-\alpha(s)}b(s)\6s + \int_0^t
\e^{-\alpha(s)}c(s)\6B_s\;,
\]
puis, comme $X_0=Y_0$, 
\[
X_t = X_0\e^{\alpha(t)} + \int_0^t \e^{\alpha(t)-\alpha(s)}b(s)\6s + \int_0^t
\e^{\alpha(t)-\alpha(s)}c(s)\6B_s\;.
\]

\item	
Dans ce cas, $\alpha(t)=-\log(1+t)$, donc $\e^{\alpha(t)}=(1+t)^{-1}$ et 
\[
X_t = \int_0^t \frac{1+s}{1+t} \frac{1}{1+s} \6B_s 
= \frac{B_t}{1+t}\;.
\]
\end{enum}

\subsection*{Exercice~\ref{exo_sde2}}

La formule d'It\^o donne 
\[
\6Y_t = \frac{\6X_t}{\sqrt{1-X_t^2}} 
+ \frac12 \frac{X_t}{(1-X_t^2)^{3/2}} \6X_t^2
= \6B_t\;.
\]
Par cons\'equent, $X_t=\sin(B_t)$ pour tous les
$t\leqs\inf\setsuch{s>0}{\abs{B_s}=\frac\pi2}$.

\subsection*{Exercice~\ref{exo_sde3}}

\begin{enum}
\item	
Avec $\alpha(t)=\log(1-t)$, il vient 
\[
X_t = a(1-t) + bt + (1-t)\int_0^t \frac{1}{1-s} \6B_s\;.
\]

\item	
Par l'isom\'etrie d'It\^o, 
\[
\Variance(X_t) = (1-t)^2 \int_0^t \frac{1}{(1-s)^2}\6s 
= (1-t)^2 \biggbrak{\frac1{1-t}-1} = t(1-t)\;.
\]
Par cons\'equent, $\Variance(X_t-b)\to0$ lorsque $t\to 1_-$, donc $X_t\to b$
dans $L^2$ lorsque $t\to 1_-$. 

\item 	
Comme $M_t$ est une martingale, $M_t^2$ est une sous-martingale, et
l'in\'egalit\'e de Doob nous permet d'\'ecrire 
\begin{align*}
\biggprob{\sup_{1-2^{-n}\leqs t\leqs 1-2^{-n-1}}(1-t)\abs{M_t}>\varepsilon}
&\leqs \biggprob{\sup_{1-2^{-n}\leqs t\leqs 1-2^{-n-1}} \abs{M_t} >
2^n\varepsilon} \\
&= \biggprob{\sup_{1-2^{-n}\leqs t\leqs 1-2^{-n-1}} M_t^2 >
2^{2n}\varepsilon^2} \\
&\leqs \frac{1}{\varepsilon^2 2^{2n}} \bigexpec{M_{1-2^{-n-1}}^2} \\
&\leqs \frac{2}{\varepsilon^2 2^n}\;.
\end{align*}
Soit alors l'\'ev\'enement 
\[
A_n = \biggsetsuch{\omega}{\sup_{1-2^{-n}\leqs t\leqs
1-2^{-n-1}}(1-t)\abs{M_t}>2^{-n/4}}\;.
\]
Nous avons $\pprob{A_n}\leqs 2\cdot2^{-n/2}$, qui est sommable. Le lemme de
Borel--Cantelli montre alors que 
\[
\biggprob{\sup_{1-2^{-n}\leqs t\leqs 1-2^{-n-1}}(1-t)\abs{M_t}\leqs 2^{-n/4}, 
n\to\infty} = 1\;,
\]
et donc que $(1-t)M_t\to0$ presque s\^urement lorsque $t\to1_-$. Par
cons\'equent, $X_t\to b$ presque s\^urement lorsque $t\to1_-$. 

\item	
On a $X_0=a$ et $X_1=b$ presque s\^urement. De plus,  
les incr\'ements de $X_t$ sont ind\'ependants et gaussiens. La seule
diff\'erence par rapport au mouvement Brownien est que la variance de $X_t-X_s$
est donn\'ee par $t(1-t)-s(1-s)$ au lieu de $t-s$. 
\end{enum}

\subsection*{Exercice~\ref{exo_sde4}}

En appliquant la formule d'It\^o \`a $F_t=u(t,B_t)$ avec 
$u(t,x)=\e^{-\alpha x + \alpha^2 t/2}$, on obtient 
\begin{align*}
\6F_t 
&= \frac12 \alpha^2 F_t\6t - \alpha F_t\6B_t + \frac12 \alpha^2 F_t\6B_t^2 \\
&= \alpha^2 F_t \6t - \alpha F_t \6B_t\;.
\end{align*}
Soit alors $X_t=F_tY_t=u(F_t,Y_t)$. On applique maintenant la formule d'It\^o
\`a plusieurs variables, avec la fonction $u(x_1,x_2)=x_1x_2$. Cela nous donne
\begin{align*}
\6X_t &= F_t\6Y_t + Y_t\6F_t + \6F_t\6Y_t \\
&= rF_t\6t + \alpha F_tY_t\6B_t - \alpha F_tY_t\6B_t + \alpha^2 F_tY_t\6t 
- \alpha^2 F_t Y_t \6t \\
&= rF_t\6t\;.
\end{align*}
Comme $X_0=F_0Y_0=1$, on a donc 
\[
X_t = 1 + r\int_0^t F_s\6s \;,
\]
et finalement 
\[
Y_t = F_t^{-1} X_t 
= \e^{\alpha B_t - \frac12 \alpha^2 t}
+ r \int_0^t \e^{\alpha (B_t-B_s) - \frac12 \alpha^2 (t-s)} \6s\;.
\]


\section{Exercices du Chapitre~\ref{chap_diff}}

\subsection*{Exercice~\ref{exo_diff1}}

\begin{enum}
\item	
\[
L = -x \dpar{}{x} + \frac12 \dpar{^2}{x^2}\;, \qquad
L^*\rho = \dpar{}{x} \bigpar{x \rho} + \frac12 \dpar{^2\rho}{x^2}\;.
\]
\item	
On trouve $L^*\rho=0$. Par cons\'equent, $\rho(x)$ est une solution
stationnaire de l'\'equation de Kolmogorov progressive (ou de Fokker--Planck)
$\sdpar{u}{t}=L^*u$, ce qui signifie que c'est une mesure invariante du
syst\`eme~: Si $X_0$ suit la loi $\rho$, alors $X_t$ suit la m\^eme loi pour
tout $t>0$. 

Remarquons que $\rho$ est la densit\'e d'une variable al\'eatoire normale,
centr\'ee, de variance $1/2$. Nous avons d\'ej\`a obtenu dans
l'exercice~\ref{exo_ito1},
que $\rho$ est la loi asymptotique de la solution de la m\^eme EDS
avec $X_0=0$. En fait on peut montrer que pour toute distribution initiale, la
loi de $X_t$ tend  vers la distribution stationnaire $\rho$. 
\end{enum}

\subsection*{Exercice~\ref{exo_diff2}}

\begin{enum}
\item	
\[
L = \frac12 x^2 \dpar{^2}{x^2}\;.
\]

\item	
On a $Lu=0$ si $u''(x)=0$, dont la solution g\'en\'erale est $u(x)=c_1
x+c_2$. 

\item	
On sait que $u(x)=\probin{x}{\tau_a<\tau_b}$ est solution du probl\`eme 
\[
\begin{cases}
Lu(x) = 0 & \text{pour $x\in[a,b]$\;,} \\
u(a) = 1\;, & \\
u(b) = 0\;. &
\end{cases}
\]
En substituant la solution g\'en\'erale dans les conditions aux bords, on peut
d\'eterminer les constantes d'int\'egration $c_1$ et $c_2$, d'o\`u la
solution 
\[
\probin{x}{\tau_a<\tau_b} = \frac{b-x}{b-a}\;.
\]
\end{enum}

\subsection*{Exercice~\ref{exo_diff3}}

\begin{enum}
\item	
\[
L = r x \dpar{}{x} + \frac12 x^2 \dpar{^2}{x^2}\;.
\]

\item	
En substituant, on obtient $Lu(x)=c_1\gamma(r+\frac12(\gamma-1))x^\gamma$, donc
$Lu=0$ \`a condition de prendre $\gamma=1-2r$.

{\it Remarque~:} La solution g\'en\'erale s'obtient en observant que
$v(x)=u'(x)$ satisfait l'\'equation
\[
\frac{v'(x)}{v(x)} = -\frac{2r}{x}\;,
\]
que l'on peut int\'egrer. 

\item	
Dans ce cas, on a $\gamma>0$. En proc\'edant comme \`a l'exercice
pr\'ec\'edent, on obtient
\[
 \probin{x}{\tau_b<\tau_a} = \frac{x^\gamma - a^\gamma}{b^\gamma - a^\gamma}\;.
\]
Comme $a^\gamma\to0$ lorsque $a\to0$, il suit que 
\[
 \probin{x}{\tau_b<\tau_0} = \Bigpar{\frac{x}{b}}^\gamma\;.
\]
La probabilit\'e que $X_t$ n'atteigne jamais $b$ est donc $1-(X_0/b)^\gamma$.

\item	
\begin{enum}
\item	
Dans ce cas, on a $\gamma<0$. En proc\'edant comme \`a l'exercice
pr\'ec\'edent, on obtient
\[
 \probin{x}{\tau_a<\tau_b} = \frac{x^\gamma - b^\gamma}{a^\gamma - b^\gamma}\;.
\]
Comme $a^\gamma\to+\infty$ lorsque $a\to0$, toutes les autres grandeurs \'etant
constantes, on obtient en faisant tendre $a$ vers $0$ 
\[
 \probin{x}{\tau_0<\tau_b} = 0 
\qquad \forall x\in]a,b[\;. 
\]

\item	
L'\'equation $Lu=-1$ donne $\alpha=-1/(r-1/2)$ et la condition au bord donne
$\beta=\log b/(r-1/2)$. 

\item	
C'est pr\'ecis\'ement la solution du probl\`eme ci-dessus, donc 
\[
 \expecin{\!x}{\tau_b} = \frac{1}{r-1/2} \log\biggpar{\frac bx}\;.
\]
Cela montre que les solutions tendent \`a cro\^\i tre exponentiellement. En
effet, on a $\expecin{\!x}{\tau_b}=T$ pour 
\[
 b = x \e^{(r-1/2)T}\;. 
\]
\end{enum}
\end{enum}

\subsection*{Exercice~\ref{exo_diff6}}

\begin{enum}
\item	Par variation de la constante, 
\[
X_t = x \e^{-t} + \sigma \int_0^t \e^{-(t-s)} \6B_s\;.
\]

\item	Le g\'en\'erateur infinit\'esimal $L$ de $X_t$ est  
\[
L = - x\dtot{}{x} + \frac{\sigma^2}{2} \dtot{^2}{x^2}\;.
\]

\item	En admettant que $\tau$ est fini presque s\^urement, 
la formule de Dynkin implique que $h(x) =
\expecin{x}{\indicator{a}(X_\tau)}$ satisfait le syst\`eme 
\begin{align*}
-xh'(x) + \frac{\sigma^2}{2} h''(x) &= 0
&& \text{pour $a<x<b$\;,} \\
h(a) &= 1\;, \\
h(b) &= 0\;.
\end{align*}
Soit $g(x) = h'(x)$. On a 
\[
g'(x) = \frac{2}{\sigma^2} xg(x)\;,
\]
dont la solution g\'en\'erale est, par s\'eparation des variables, 
\[
g(x) = c\e^{x^2/\sigma^2}\;.
\]
Il suit que 
\[
h(x) = 1 + c\int_a^x \e^{y^2/\sigma^2} \6y\;,
\]
et la condition $h(b)=0$ implique 
\[
c = -\frac{1}{\displaystyle\int_a^b \e^{y^2/\sigma^2} \6y}\;.
\]
On peut \'ecrire la solution sous la forme 
\[
h(x) = \frac{\displaystyle\int_x^b \e^{y^2/\sigma^2} \6y}
{\displaystyle\int_a^b \e^{y^2/\sigma^2} \6y}\;.
\]

\item	Lorsque $\sigma\to0$, chaque int\'egrale est domin\'ee par la valeur
maximale de $y$ dans le domaine d'int\'egration. Il suit que pour $a<x<b$, 
\[
h(x) \simeq \frac{\e^{x^2/\sigma^2} + \e^{b^2/\sigma^2}}
{\e^{a^2/\sigma^2} + \e^{b^2/\sigma^2}}\;.
\]
Par cons\'equent, 
\begin{itemiz}
\item 	si $\abs{a} > b$, $h(x)\to0$;
\item 	si $a=-b$, $h(x)\to 1/2$;
\item 	si $\abs{a} < b$, $h(x)\to1$. 
\end{itemiz}
\end{enum}

\subsection*{Exercice~\ref{exo_diff4}}

\begin{enum}
\item	
\begin{enum}
\item	
Soit $\varphi(x)=x$. Alors
$(L\varphi)(x)=f(x)\varphi'(x)+\frac12 g(x)^2\varphi''(x) = f(x)$, et par la
d\'efinition du g\'en\'erateur, 
\[
\lim_{h\to0_+} \frac{\bigexpecin{x}{X_h} - x}h 
=(L\varphi)(x)=f(x)\;.
\]
Ainsi le coefficient de d\'erive $f(x)$ s'interpr\`ete comme la d\'eriv\'ee de
la position moyenne. 

\item	
On a 
\begin{align*}
 \dtot{}{t} \biggbrak{
\e^{-\gamma \expecin{x}{X_t}}
\Bigexpecin{x}{ \e^{\gamma X_t}}}
\biggr|_{t=0_+}
&= 
 \dtot{}{t} \biggbrak{
\e^{-\gamma \expecin{x}{X_t}}}
\biggr|_{t=0_+} \e^{\gamma x}
+  \e^{-\gamma x}\dtot{}{t} \biggbrak{
\Bigexpecin{x}{ \e^{\gamma X_t}}}
\biggr|_{t=0_+} \\
&= 
 -\gamma \dtot{}{t} \biggbrak{
\expecin{x}{X_t}}
\biggr|_{t=0_+} \e^{-\gamma x}\e^{\gamma x}
+  \e^{-\gamma x}
\bigpar{L\e^{\gamma x}}(x) \\
&= -\gamma f(x) + \e^{-\gamma x}
\biggbrak{f(x)\gamma\e^{\gamma x} + \frac{1}{2} g(x)^2\gamma^2\e^{\gamma x}} \\
&= \frac{1}{2}\gamma^2 g(x)^2\;.
\end{align*}
\item	
Comme 
\[
\dtot{}{t} 
\Bigexpecin{x}{\e^{\gamma \brak{X_t - \expecin{x}{X_t}}}} \Bigr|_{t=0_+} 
= \sum_{k\geqs 0} \frac{\gamma^k}{k!}
\dtot{}{t} 
\Bigexpecin{x}{\bigbrak{X_t - \expecin{x}{X_t}}^k} \Bigr|_{t=0_+}
\]
on obtient, par identification terme \`a terme des s\'eries,
\[
\dtot{}{t} 
\Bigexpecin{x}{\bigbrak{X_t - \expecin{x}{X_t}}^k} \Bigr|_{t=0_+}
= 
\begin{cases}
0 & \text{si $k=1$\;,} \\
g(x)^2 &\text{si $k=2$\;,} \\
0 &\text{si $k\geqs3$\;.}
\end{cases}
\]
En particulier, $g(x)^2$ s'interpr\`ete comme la vitesse de croissance de la
variance du processus. 
\end{enum}

\item	
\begin{enum}
\item	
Le processus $X^{(N)}_t$ \'etant lin\'eaire sur l'intervalle $[0,1/N]$, on a 
\begin{align*}
\lim_{h\to0_+} \frac{\bigexpecin{x}{ X^{(N)}_h-x}}h
&=  \frac{\bigexpecin{x}{ X^{(N)}_{1/N}-x}}{1/N} \\
&= \frac{\bigexpecin{y}{N^{-\alpha} Y^{(N)}_{1}-N^{-\alpha}y}}{1/N}
\biggr|_{y=N^\alpha x}  \\
&= N^{-\alpha+1} v^{(N)}(N^\alpha x)\;.
\end{align*}
De m\^eme, on obtient 
\[
\lim_{h\to0_+} \frac{\bigexpecin{x}{ \brak{X^{(N)}_h - 
\expecin{x}{X^{(N)}_h}}^k}}h
= N^{-k\alpha+1} m^{(N)}_k(N^\alpha x)\;.
\]
\item	
En vertu du point 1., il faut que 
\begin{align*}
\lim_{N\to\infty} N^{-\alpha+1} v^{(N)}(N^\alpha x) &= f(x)\;, \\
\lim_{N\to\infty} N^{-2\alpha+1} m^{(N)}_2(N^\alpha x) &= g(x)^2\;, \\
\lim_{N\to\infty} N^{-k\alpha+1} m^{(N)}_k(N^\alpha x) &= 0\;,
& k\geqs 3\;. 
\end{align*}
\end{enum}
\item	
Dans ce cas on a $v^{(N)}(y)=0$ et $m^{(N)}_k(y)=1$ pour $k$ pair et $0$ pour
$k$ impair. En prenant $\alpha=1/2$, les conditions sont donc satisfaites avec
$f(x)=0$ et $g(x)=1$. 

\item	
On a 
\[
v^{(N)}(y) 
= (y-1) \frac{y}{N} + (y+1) \Bigpar{1-\frac{y}{N}} - y
= 1 - \frac{2y}{N}
\]
et
\begin{align*}
m^{(N)}_k(y)
&= \Bigpar{y-1-y-1+\frac{2y}{N}}^k \frac{y}{N}
+ \Bigpar{y+1-y-1+\frac{2y}{N}}^k\Bigpar{1-\frac{y}{N}}\\
&= 2^k \frac{y}{N}\Bigpar{1-\frac{y}{N}}
\biggbrak{\Bigpar{\frac{y}{N}}^{k-1} + (-1)^k \Bigpar{1-\frac{y}{N}}^{k-1}}
\;.
\end{align*}
Par cons\'equent, 
\begin{align*}
\lim_{N\to\infty} N^{1/2}  v^{(N)}\Bigpar{N^{1/2}x + \frac N2} &= -2x\;,
\\
\lim_{N\to\infty} m^{(N)}_2\Bigpar{N^{1/2}x + \frac N2}& =
1\;,
\\
\lim_{N\to\infty} N^{1-k/2}  m^{(N)}_k\Bigpar{N^{1/2}x + \frac N2}& =
0\;, & k\geqs3\;.
\end{align*}
La diffusion limite est donc le processus d'Ornstein--Uhlenbeck 
\[
\6X_t = -2X_t\6t + \6B_t\;.
\]
\end{enum}

\subsection*{Exercice~\ref{exo_diff5}}

\begin{enum}
\item	
$X_t$ est la proportion de temps avant $t$ pendant laquelle le mouvement
Brownien est positif. Le r\'esultat montre qu'elle est ind\'ependante de $t$,
ce qui est une cons\'equence de l'invariance d'\'echelle du mouvement Brownien.

\item	
La propri\'et\'e diff\'erentielle du mouvement Brownien montre que $B_{tu}$ est
\'egal \`a $t^{-1/2}B_u$ en loi, donc que $\indexfct{B_{tu}>0}$ est
\'egal \`a $\indexfct{B_{u}>0}$ en loi. Le changement de variable $s=tu$ montre
que 
\[
X_t = \frac1t\int_0^t \indexfct{B_{tu}>0} t\6u 
= \int_0^1 \indexfct{B_{tu}>0} \6u
\]
qui est \'egal \`a $X_1$ en loi. 

\item	
On a $v(t,0)=\expec{\e^{-\lambda tX_t}}=\expec{\e^{-\lambda tX_1}}$. 
Par cons\'equent 
\[
g_\rho(0) = \E \int_0^\infty \e^{-(\rho+\lambda X_1)t}\6t 
= \biggexpec{\frac{1}{\rho+\lambda X_1}}\;.
\]

\item	
\[
 \dpar vt(t,x) = \frac12 \dpar{^2v}{x^2}(t,x) - \lambda \indexfct{x>0} v(t,x)\;.
\]

\item	
\begin{align*}
g_\rho''(x) &= \int_0^\infty \dpar{^2v}{x^2}(t,x) \e^{-\rho t}\6t \\
&= 2\int_0^\infty \dpar vt(t,x)\e^{-\rho t}\6t 
+ 2\lambda \indexfct{x>0}\int_0^\infty v(t,x)\e^{-\rho t}\6t \\
&= 2v(t,x) \e^{-\rho t}\Bigr|_0^\infty + 2\rho \int_0^\infty
v(t,x)\e^{-\rho t}\6t + 2\lambda \indexfct{x>0} g_\rho(x) \\
&= -2 + 2\rho g_\rho(x)+ 2\lambda \indexfct{x>0} g_\rho(x)\;.
\end{align*}
On a donc l'\'equation diff\'erentielle ordinaire 
\[
g_\rho''(x) = 
\begin{cases}
2(\rho+\lambda)g(x) - 2 & \text{si $x>0$\;,} \\
2\rho g(x) - 2 & \text{si $x<0$\;.} 
\end{cases}
\]
Consid\'erons s\'epar\'ement les \'equations sur $\R_+$ et $\R_-$. Les
\'equations homog\`enes admettent les solutions lin\'eairement ind\'ependantes
$\e^{\gamma_\pm x}$ et $\e^{-\gamma_\pm x}$, o\`u
$\gamma_+=\sqrt{2(\rho+\gamma)}$ et $\gamma_-=\sqrt{2\rho}$. De plus chaque
\'equation admet une solution particuli\`ere constante, \'egale \`a
$A_+=1/(\rho+\lambda)$, respectivement $A_-=1/\rho$.  

\item	
Pour que $g_\rho$ soit born\'ee, il faut que $B_+=C_-=0$. La continuit\'e de
$g_\rho$ et $g'_\rho$ permet de d\'eterminer $B_-$ et $C_+$. En particulier 
\[
C_+ =
\frac{\lambda}{\sqrt{\rho}\,(\sqrt{\rho+\lambda}+\sqrt{\rho}\,)(\rho+\lambda)}
= \frac{\sqrt{\rho+\lambda}-\sqrt{\rho}}{\sqrt{\rho}\,(\rho+\lambda)}\;.
\]
Il suit que 
\[
g_\rho(0) = \frac{1}{\rho+\lambda} + C_+ 
=  \frac{1}{\sqrt{\rho(\lambda+\rho)}}\;.
\]

\item	
Nous avons 
\[
\frac{1}{\sqrt{1+\lambda}} = g_1(0) = 
\biggexpec{\frac{1}{1+\lambda X_1}} = 
\sum_{n=0}^\infty (-\lambda)^n \expec{X_1^n}\;.
\]
L'identit\'e montre que 
\[
\expec{X_1^n} = \frac{1}{\pi} \int_0^1 \frac{x^n}{\sqrt{x(1-x)}}\6x
\]
pour tout $n\geqs0$, donc que $X_1$ admet la densit\'e $1/\sqrt{x(1-x)}$. Le
r\'esultat suit en int\'egrant cette densit\'e et en utilisant l'\'egalit\'e en
loi de $X_t$ et $X_1$.
\end{enum}



\newpage
%
%

\nocite{Durrett}
\nocite{Bolthausen_skript}
\nocite{Friedman}
\nocite{Gihman-Skorohod}
\nocite{Resnick}

\bibliography{eds}
\bibliographystyle{amsalpha}               

\end{document}